\documentclass[11pt]{jfm}
\usepackage{verbatim}
\usepackage[usenames,dvipsnames,svgnames,table]{xcolor}
\usepackage{amsfonts}
\usepackage{fancyhdr}
\usepackage[parfill]{parskip}
\usepackage{graphicx}
\usepackage{amssymb}
\usepackage{amsmath}
\usepackage{epstopdf}
\usepackage{natbib}
\usepackage{bm,footnpag}
\usepackage{caption}
\usepackage{float}
\usepackage{xcolor}
\usepackage{epstopdf}
\usepackage{ifthen}
\usepackage{pgf}
\usepackage{parskip}
\usepackage{csquotes}
\usepackage[nodisplayskipstretch]{setspace}
\usepackage[toc]{appendix}
\usepackage[section]{placeins}
\usepackage[latin1]{inputenc}

\ifCUPmtlplainloaded \else
  \checkfont{eurm10}
  \iffontfound
   \IfFileExists{upmath.sty}
      {\typeout{^^JFound AMS Euler Roman fonts on the system,
                   using the 'upmath' package.^^J}%
       \usepackage{upmath}}
      {\typeout{^^JFound AMS Euler Roman fonts on the system, but you
                   dont seem to have the}%
       \typeout{'upmath' package installed. JFM.cls can take advantage
                 of these fonts,^^Jif you use 'upmath' package.^^J}%
      }
  \else
  \fi
\fi

\ifCUPmtlplainloaded \else
  \checkfont{msam10}
  \iffontfound
    \IfFileExists{amssymb.sty}
      {\typeout{^^JFound AMS Symbol fonts on the system, using the
                'amssymb' package.^^J}%
       \usepackage{amssymb}%
       \let\le=\leqslant  \let\leq=\leqslant
       \let\ge=\geqslant  \let\geq=\geqslant
      }{}
  \fi
\fi

\ifCUPmtlplainloaded \else
  \IfFileExists{amsbsy.sty}
    {\typeout{^^JFound the 'amsbsy' package on the system, using it.^^J}%
     \usepackage{amsbsy}}
    {\providecommand\boldsymbol[1]{\mbox{\boldmath $##1$}}}
\fi

\providecommand\bnabla{\boldsymbol{\nabla}}
\providecommand\bcdot{\boldsymbol{\cdot}}

\newsavebox{\astrutbox}
\sbox{\astrutbox}{\rule[-5pt]{0pt}{20pt}}

\newtheorem{lemma}{Lemma}[section]
\newtheorem{corollary}[lemma]{Corollary}
\allowdisplaybreaks[1]
\newtheorem{theorem}[lemma]{Theorem}
\newtheorem{proposition}[lemma]{Proposition}
\title[$n$-tuple-periodic Navier--Stokes Solutions: Existence and Obstructions]{$n$-tuple-periodic Solutions of the Navier--Stokes Equations: Existence and Obstructions}
\author[R. K. Michael Thambynayagam]%
{R. K. Michael Thambynayagam%
  \thanks{email: michael.thambynayagam@gmail.com}}
\affiliation{Managing Director (Retired), Schlumberger Cambridge Research
\includegraphics[width=0.5cm]{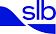}}
\newif\iffullAppendixA
\fullAppendixAtrue
\numberwithin{equation}{section}
\begin{document}
\maketitle
\begin{abstract}
We introduce a phase-angle framework for constructing and classifying $n$-tuple-periodic solutions
of the incompressible Navier--Stokes equations, and use it to determine, dimension by dimension, when
such solutions exist and why they fail when they do not.

The velocity field is a difference of two $n$-fold cyclic trigonometric products, made
divergence-free for every phase and every $n\ge3$ by a telescoping identity. The problem then
reduces to two explicit conditions on the $n$ phase angles: a \emph{reduction condition}, necessary
and algebraic, and a \emph{pressure-integrability condition}, sufficient and differential. Separating
them is what allows a failed construction to be located rather than merely recorded. Unlike the
stream-function route on which earlier tri-periodic solutions rest, the framework is not confined to
three dimensions, and it does not presuppose the Beltrami property.

For $3\le n\le8$ we prove that unforced solutions exist at $n=3$ and $n=4$ and nowhere else. In
three dimensions the admissible phases form exactly two families, mirror images distinguished by the
sign of the helicity. These two families are known; what is established here is that they are the
only ones the construction admits, a completeness question raised and left open when they were
first obtained. In four dimensions integrability reduces a one-parameter curve to a single
family, so the quarter-period offset is forced, not assumed. The obstructions divide by parity:
arithmetic at $n=5,7$, where every admissible phase vector annihilates the field; differential at
$n=6,8$, where the field survives but its convective term is not a gradient.

Every obstructed dimension nevertheless carries an exact forced solution in closed form. These are
not mere self-balancing: the distinguished phases are those at which strain and rotation are in exact
pointwise balance, so that no pressure is generated at all. What obstructs the unforced problem
becomes, in the forced problem, an exactly prescribable input.
\end{abstract}
\protect{\bfseries{2020 Mathematics Subject Classification:}} 35Q30, 76D05, 35C05
\\\\
\protect{\bfseries{ORCID:}} https://orcid.org/0000-0002-1778-7327
\vspace*{-0.15 in}
\section{Introduction, Problem Statement, and Literature Review}
\label{sec:intro}

\subsection{Introduction}

Exact analytical solutions of the incompressible Navier--Stokes equations occupy
a special place in fluid mechanics. Beyond providing insight into the nonlinear
structure of the governing equations, they serve as rigorous benchmarks for
validating numerical algorithms and assessing the accuracy and stability of
computational fluid dynamics solvers. While numerous exact solutions have been
developed in two and three spatial dimensions, comparatively little attention
has been devoted to the systematic construction and analysis of exact solutions
in higher-dimensional spaces.

Among the best-known analytical solutions is the periodic flow introduced by
\citet{tay}, commonly known as the Taylor--Green vortex. This work demonstrated how carefully chosen
trigonometric velocity fields can satisfy both the incompressibility condition
and the Navier--Stokes equations, leading to exact closed-form solutions.
Subsequent investigations extended this general philosophy to increasingly
complex flow configurations. Of particular relevance to the present work is the
fully three-dimensional triple-periodic solution introduced by
\citet{ant}, which established the existence of a non-trivial cyclic
velocity field exhibiting dependence on all spatial coordinates.

Despite these advances, two fundamental questions remain largely unexplored.
First, can the phase angles governing periodic (Taylor-type) velocity fields be
determined systematically rather than through ad hoc construction or
case-specific analysis? Second, do analogous solutions within such periodic
constructions persist in arbitrary spatial dimensions, or do the compatibility
conditions impose dimension-dependent obstructions?

The present work addresses both questions. Building upon the framework
developed in~\cite{tha4}, we develop a systematic methodology for determining
admissible phase angles directly from the pressure-integrability conditions
associated with the nonlinear inertial terms. The resulting framework reduces
the search for exact solutions to the determination of a finite set of
phase-angle parameters, provides explicit analytical expressions for both the
velocity and pressure fields whenever admissible phase angles exist, and
establishes a unified $n$-dimensional approach for investigating periodic solutions.

Application of the methodology reveals a striking dimension-dependent behaviour. The procedure
recovers the known exact three-dimensional solutions and determines completely the admissible phase
sets within the periodic constructions considered here, in the two solvable dimensions. In three
dimensions the admissible phases form exactly two families, mirror images of one another and
distinguished by the sign of the helicity. These are the two families of~\citet{ant}; within the
present construction, the classification establishes that no others occur---a question that work
raised and left open. In four dimensions the reduction condition admits a one-parameter curve of
phase vectors, of which pressure-integrability retains exactly one non-trivial member, so that the
quarter-period offset is forced rather than assumed. The pressure is obtained in closed form in
each case.

Beyond four dimensions the construction becomes obstructed, but the obstruction is itself
constructive: it identifies the part of the nonlinear momentum flux that cannot be balanced by the
reconstructed pressure and, when retained as a body force, yields an exact forced solution. What
distinguishes these solutions from the formal cancellation of the convective term by an equal body
force is the Helmholtz character of that term. For the fields constructed here it is entirely
longitudinal in the dimensions admitting an unforced solution; at the selected odd-dimensional
forced phases it is entirely transverse, so that no pressure gradient is required and the whole
convective term becomes the body force; at the symmetric even-dimensional obstructed phases both
longitudinal and transverse parts survive. The principal outcome is therefore a dimension-dependent
solvability theory, within the periodic constructions studied here, in which existence,
obstruction, and obstruction-forced solutions arise from the same framework.

The higher dimensions, as will be seen, prove less accommodating than their
three- and four-dimensional predecessors; fortunately, they are quite explicit
about their objections.

The principal contributions of this paper, in the order in which they are developed below, are
therefore:
(i)~the development of a unified $n$-dimensional analytical framework for constructing and
analysing periodic solutions of the incompressible Navier--Stokes
equations;
(ii)~a systematic phase-angle methodology that reduces the search for exact
solutions from the nonlinear partial differential equations to a finite system of
trigonometric equations in the phase angles, equivalently to a finite polynomial
system in their sine and cosine variables. Two tests arise. The \emph{reduction
condition} requires that the $x_i$-independent Fourier component of the $i$th
convective component vanish identically; it is algebraic and readily applied, but
only necessary. The \emph{pressure-integrability condition} requires in addition
that the whole convective field be a gradient, so that it can be balanced by a
single-valued periodic pressure. We characterise both conditions and show that the
first is in general not sufficient for the second;

(iii)~explicit analytical velocity and pressure solutions in three and four
dimensions, including the complete admissible phase classification within the
alternating sine--cosine product construction in those dimensions, together with a
second four-dimensional unit-periodic family whose pressure is likewise obtained in
closed form;

(iv)~the identification of two distinct obstruction mechanisms in the higher
dimensions investigated here: an arithmetic obstruction at $n=5,7$ and a
differential, pressure-integrability obstruction at $n=6,8$;

(v)~a dimension-dependent classification, within the periodic constructions
considered here, separating the solvable cases $n=3,4$ from the obstructed cases
$n=5,\ldots,8$ and identifying the distinct mechanism responsible for each
obstruction; and

(vi)~for every obstructed dimension in the range studied, $5\leq n\leq8$, the
construction of an exact solution of the corresponding forced Navier--Stokes
equations, with the body force identified explicitly as the remainder left by the
pressure-reconstruction procedure applied to the advective momentum flux. The
resulting force is a zero-mean, non-zero, closed-form quantity in each obstructed
dimension considered. These forced solutions are not merely a by-product of the
obstruction analysis: they provide exact high-dimensional benchmark fields with
velocity, pressure and body force all available in closed form, and are therefore
well suited to the validation of high-dimensional numerical solvers.

By the \emph{forced} equations we mean the incompressible Navier--Stokes
equations~\eqref{1.1} carrying a prescribed external body force
$\boldsymbol{f}(\boldsymbol{x},t)$ on the right-hand side, as opposed to the
\emph{unforced} equations in which $\boldsymbol{f}\equiv\boldsymbol{0}$. The
forced solutions constructed here are of a particular kind, which we call
\emph{obstruction-forced}: the prescribed force is not chosen freely but is
exactly the object that remains unbalanced by the pressure reconstruction in the
corresponding unforced construction---the remainder that procedure leaves in the
convective field. In this sense the
obstruction and the forcing are one and the same quantity, seen first as a
barrier to existence and then as the source term that an exact solution
requires.

The distinction between the reduction condition and the full
pressure-integrability condition is central to the present work. A phase vector
satisfying the reduction condition is a \emph{candidate}, not an exact
solution, and in six and eight dimensions the symmetric candidates that satisfy
the reduction condition fail the complete integrability requirement. This distinction explains why the obstructed even-dimensional
cases $n=6,8$ satisfy the reduction condition exactly while nevertheless admitting no
scalar pressure whose gradient can balance the full convective field in the unforced equations. The precise relationship between these two
conditions is established in
Subsection~\ref{subsec:nec-suf}, where the pressure-integrability condition is
given its Fourier characterisation.
\subsection{Problem Statement}
The objective of the present work is to investigate the existence of exact periodic
solutions of the incompressible Navier--Stokes equations within a unified $n$-dimensional framework. Particular attention is given to a class of periodic velocity fields for which both the velocity and pressure fields may be determined analytically. The problem is formulated in an unbounded domain \( \mathbb{R}^n \) and serves as the foundation for the phase-angle determination methodology developed in the subsequent sections.

We consider the Navier--Stokes equations governing the motion of an incompressible, viscous fluid in an unbounded domain \( \mathbb{R}^n \), for arbitrary spatial dimension \( n \geq 2 \). The governing equations express conservation of momentum and mass, and take the form
\begin{equation}
\label{1.1}
\frac{\partial \boldsymbol{v}}{\partial t} + \boldsymbol{g} = \kappa \Delta \boldsymbol{v} - \frac{1}{\rho} \nabla p + \boldsymbol{f}, \qquad \text{for } \boldsymbol{x} \in \mathbb{R}^n, \, t \geq 0,
\end{equation}
where \( \boldsymbol{x} = (x_1, x_2, \dots, x_n) \in \mathbb{R}^n \) denotes the spatial coordinate vector. The vector fields are defined as follows:
\begin{itemize}
    \item \( \boldsymbol{v}(\boldsymbol{x}, t) = (v_1, \dots, v_n) \) is the velocity field;
    \item \( \boldsymbol{g}(\boldsymbol{x}, t) = (g_1, \dots, g_n) \) is the nonlinear inertial term;
    \item \( \boldsymbol{f}(\boldsymbol{x}, t) = (f_1, \dots, f_n) \) is an externally applied body force per unit mass;
    \item \( p(\boldsymbol{x}, t) \) is the scalar pressure field.
\end{itemize}
The constants \( \rho > 0 \) and \( \kappa > 0 \) denote the fluid density and kinematic viscosity, respectively. The Laplacian operator in \( \mathbb{R}^n \) is given by \( \Delta = \sum_{i=1}^n \partial^2 / \partial x_i^2 \).

The nonlinear inertial term is given componentwise by:
\begin{equation}
\label{1.2}
g_i(\boldsymbol{x}, t) = \sum_{j=1}^n v_j(\boldsymbol{x}, t) \frac{\partial v_i(\boldsymbol{x}, t)}{\partial x_j}, \qquad i = 1, \dots, n.
\end{equation}
For constant density, conservation of mass reduces to the incompressibility condition:
\begin{equation}
\label{1.3}
\nabla \cdot \boldsymbol{v} = \sum_{i=1}^n \frac{\partial v_i(\boldsymbol{x}, t)}{\partial x_i} = 0, \qquad \boldsymbol{x} \in \mathbb{R}^n, \, t \geq 0.
\end{equation}
The system evolves from a smooth, solenoidal initial velocity field:
\begin{equation}
\label{1.4}
\boldsymbol{v}(\boldsymbol{x}, 0) = \boldsymbol{v}^0(\boldsymbol{x}), \qquad \boldsymbol{x} \in \mathbb{R}^n,
\end{equation}
which is assumed to be spatially periodic. The corresponding inertial term at \( t = 0 \) is:
\begin{equation}
\label{1.5}
g_i(\boldsymbol{x}, 0) = g_i^0(\boldsymbol{x}) = \sum_{j=1}^n v_j^0(\boldsymbol{x}) \frac{\partial v_i^0(\boldsymbol{x})}{\partial x_j}, \qquad i = 1, \dots, n.
\end{equation}
When a body force is present we shall likewise write
\[
\boldsymbol{f}^0(\boldsymbol{x})=\boldsymbol{f}(\boldsymbol{x},0),
\]
so that the superscript $0$ denotes evaluation at $t=0$. The time dependence of the body force,
where required, is determined subsequently as part of the construction of the corresponding exact
solution.

We further require that the kinetic energy over one periodic cell $\Omega=[0,L]^n$ remain finite,
\begin{equation}
\label{1.6}
\int\limits_{\Omega}\left|\boldsymbol{v}(\boldsymbol{x},t)\right|^2\,d\boldsymbol{x} < E
\end{equation}
for some finite constant $E$, on the time interval under consideration.

The primary existence problem considered first is the unforced case
\( \boldsymbol{f}(\boldsymbol{x}, t) = \boldsymbol{0} \), with the velocity field
\( \boldsymbol{v}(\boldsymbol{x}, t) \) smooth and spatially periodic. The
Navier--Stokes system~\eqref{1.1}--\eqref{1.5} is then solved forward in time
\( t \geq 0 \), with a prescribed solenoidal and periodic initial velocity field defined
over \( \mathbb{R}^n \). The forced problem is considered subsequently: when an obstruction to the unforced construction
arises, we investigate whether that obstruction can be retained as an explicitly prescribed
source, so that the corresponding velocity field solves the forced equations exactly.
\subsection{Literature Review}
An encompassing perspective on analytical solutions to the Navier--Stokes equations is provided in the works of \cite{wan}, \cite{oka}, and \cite{lan}. Two broad classes of exact solution are particularly relevant to the present work. The
first class encompasses solutions in which the nonlinearity is either reduced or completely eliminated from the solution structure, as discussed in \cite{fox}. Notable instances of this simplified analysis are the \cite{cou} and \cite{poi} flows. The second class pertains to the Beltrami and generalized-Beltrami solutions, in which the nonlinear convective term is retained but reduces to a gradient and can therefore be absorbed
into the pressure, so that no separate integrability condition for the convective term is required
\citep{bel}. The cyclic fields constructed in the present work are, in general, \emph{not} of this type: their convective term is not a gradient \emph{a priori}, and the phase angles must be chosen to make it one, a distinction developed in Section~\ref{sec:solutions-3D-4D}. An early unsteady viscous example of the Beltrami type is the decaying Beltrami flow of \cite{trk}, whose construction underlies the unsteady Beltrami
solutions considered below. Two-dimensional solutions in $\mathbb{R}^2$ to the Navier--Stokes equation belonging to the class of generalized Beltrami flows---those for which the convective term reduces to a gradient, in the sense of \cite{wan2}---for incompressible flows were developed by \cite{tay}:
\begin{eqnarray}
\label{1.7}
\frac{v_1}{v_r} = \sin \left( \alpha x_1 \right)\cos \left( \alpha x_2 \right)e^{-2 \alpha^2 \kappa t}
\end{eqnarray}
\begin{eqnarray}
\label{1.8}
\frac{v_2}{v_r} = -\cos \left( \alpha x_1 \right)\sin \left( \alpha x_2 \right)e^{-2 \alpha^2 \kappa t}
\end{eqnarray}
\begin{eqnarray}
\label{1.9}
p = \frac{\rho v_r^2 e^{-4 \alpha^2 \kappa t}}{4} \left[ \cos \left( 2 \alpha x_1 \right) + \cos \left( 2 \alpha x_2 \right) \right]
\end{eqnarray}
where $\alpha = \frac{2\pi}{L}$ is the wave number and $L$ is the wavelength, and $v_r$ is the reference velocity.

\citet{tay} arrived at this solution through a physical rather than a purely formal argument. He considered a doubly-periodic array of counter-rotating vortices and cast the problem in the vorticity--stream-function form, which removes the pressure. For this array the stream function is a Laplacian eigenfunction, $\nabla^2\psi=-2\alpha^2\psi$, so the vorticity is constant along streamlines and the nonlinear convective term reduces to a
gradient, which is absorbed into the pressure. The vorticity equation then reduces to linear diffusion, and the whole field decays exponentially. This is the mechanism that the present
paper investigates in higher dimensions: whether a periodic velocity field can be constructed
whose self-advection is a gradient and can therefore be absorbed into the pressure, leaving a
purely diffusive decay of the velocity field.

A word on nomenclature: although the two-dimensional solution~\eqref{1.7}--\eqref{1.9} is commonly called the \emph{Taylor--Green} vortex, and we follow that usage here, the exact two-dimensional solution is due to \citet{tay} alone. The compound name derives from the later \citet{taygre}, who introduced the \emph{three-dimensional} vortex as an initial condition for studying the production of small eddies from large ones, without an exact solution. Strictly, then, the ``Green'' attribution belongs to the three-dimensional problem.

A three-dimensional steady solution in $\mathbb{R}^3$ was initially introduced by \cite{arn}, the \emph{unit-periodic} solution, widely known as the Arnold--Beltrami--Childress (ABC) flow. Unsteady analytical solutions encompassing all three velocity components in Cartesian coordinates were subsequently presented by \cite{eth}; a related class of exact unsteady solutions, combining vortex stretching, convection and viscous diffusion, was given by \cite{kam}. Building upon \cite{eth}'s methodology, an analytical solution was formulated by \cite{bab}. Furthermore, \cite{tha1}, employing an alternative approach, extended \cite{arn}'s steady solution to an unsteady scenario, ultimately yielding the same outcome as that established by \cite{bab}:
\begin{eqnarray}
\label{1.10}
\frac{v_1}{v_r} = \left[ {a\sin \left( {\alpha x_3 } \right) - c\cos \left( {\alpha x_2 } \right)}\right]e^{ - \alpha ^2 \kappa t}
\end{eqnarray}
\begin{eqnarray}
\label{1.11}
\frac{v_2}{v_r}   = \left[ b\sin \left( {\alpha x_1 } \right)- {a\cos \left( {\alpha x_3 } \right)}\right]e^{ - \alpha ^2 \kappa t}
\end{eqnarray}
\begin{eqnarray}
\label{1.12}
\frac{v_3}{v_r}   = \left[ {c\sin \left( {\alpha x_2 } \right) - b\cos \left( {\alpha x_1 } \right)}\right]e^{ - \alpha ^2 \kappa t}
\end{eqnarray}
\begin{eqnarray}
\label{1.13}
p \!= \! - \rho v_r^2 e^{ - 2\alpha ^2 \kappa t}\! \left[ {bc\cos \left( {\alpha x_1 } \right)\sin \left( {\alpha x_2 } \right) + ab\cos \left( {\alpha x_3 } \right)\sin \left( {\alpha x_1 } \right) + ac\cos \left( {\alpha x_2 } \right)\sin \left( {\alpha x_3 } \right)} \right]\nonumber\\
\end{eqnarray}
where $a$, $b$ and $c$ are arbitrary constants. As usual for incompressible flow, the pressure is
determined only up to an additive function of time, which has been set to zero here. Unlike the fully three-dimensional periodic
construction considered below, each velocity component of this unit-periodic solution depends on
only two of the three spatial coordinates.

An advancement on Ethier and Steinman's work was achieved by  \cite{ant}, who introduced a triple-periodic, fully three-dimensional analytical solution. This solution involves velocity components that exhibit non-trivial dependence along all coordinate directions. Antuono's solution configuration closely resembles that presented in \cite{eth}, in which
the temporal derivatives $\partial \boldsymbol{v}/\partial t$ are balanced against
the viscous terms in the momentum equation. The velocity field is expressed through a stream
function $\boldsymbol{\psi}$ ($\boldsymbol{v}=\nabla\times\boldsymbol{\psi}$) and the advective term
$\left(\boldsymbol{v} \cdot \nabla\right)\boldsymbol{v}$ as the gradient of a scalar, so that the
momentum equation separates into a linear equation for the stream function and a nonlinear one
for the pressure. The linear stream-function equation is then solved by separation of variables.
On the torus $\mathbb{T}^3=\left[0,L\right]^3$ this yields not one solution but two, distinguished
by the sign of the helicity: both are Beltrami flows, in which the vorticity
$\boldsymbol{\omega}=\nabla\times\boldsymbol{v}$ is everywhere parallel to the velocity, here
satisfying $\boldsymbol{\omega}=\pm\sqrt3\,\alpha\boldsymbol{v}$, and each carries a free parameter that
\citet{ant} identifies as a translation along the diagonal $(1,1,1)$. The positive-helicity
solution is given by:
\begin{eqnarray}
\label{1.14}
\frac{v_1}{v_r}& = &\frac{4\sqrt{2}}{3\sqrt{3}}\left[\sin { \left(\alpha x_1-\frac{5\pi}{6} \right) } \cos {\left(\alpha x_2-\frac{\pi}{6}\right) }\sin { \left(\alpha x_3\right) }\right.- \nonumber\\
&-&\left.\cos  { \left(\alpha x_3-\frac{5\pi}{6} \right) }\sin  {\left( \alpha x_1-\frac{\pi}{6}\right) }  \sin  {\left( \alpha x_2\right) }\right]e^{ - 3\alpha ^2 \kappa t}
\end{eqnarray}
\begin{eqnarray}
\label{1.15}
\frac{v_2}{v_r}& = &\frac{4\sqrt{2}}{3\sqrt{3}}\left[\sin { \left(\alpha x_2-\frac{5\pi}{6} \right) } \cos {\left(\alpha x_3-\frac{\pi}{6}\right) }\sin { \left(\alpha x_1\right) }\right.- \nonumber\\
&-&\left. \cos  { \left(\alpha x_1-\frac{5\pi}{6} \right) }\sin  {\left( \alpha x_2-\frac{\pi}{6}\right) } \sin  {\left( \alpha x_3\right) }\right]e^{ - 3\alpha ^2 \kappa t}
\end{eqnarray}
\begin{eqnarray}
\label{1.16}
\frac{v_3}{v_r}& = &\frac{4\sqrt{2}}{3\sqrt{3}}\left[ \sin { \left(\alpha x_3-\frac{5\pi}{6} \right) } \cos {\left(\alpha x_1-\frac{\pi}{6}\right) }\sin { \left(\alpha x_2\right) }\right.- \nonumber\\
&-&\left. \cos  { \left(\alpha x_2-\frac{5\pi}{6} \right) }\sin  {\left( \alpha x_3-\frac{\pi}{6}\right) } \sin  {\left( \alpha x_1\right) }\right]e^{ - 3\alpha ^2 \kappa t}
\end{eqnarray}
\citet{ant} gives the corresponding pressure in the Bernoulli form~\eqref{1.17} below. For a
Beltrami flow that expression is exact and complete, since
\[
(\boldsymbol{v}\bcdot\nabla)\boldsymbol{v}=\nabla\!\left(\frac{\lVert\boldsymbol{v}\rVert^2}{2}\right),
\]
so that the momentum balance gives $p=p_0-\rho\lVert\boldsymbol{v}\rVert^2/2$. It is not, however, resolved into harmonics,
and in that form the structure of the field is not visible. The explicit trigonometric
expansion was first given by~\citet{tha4}, as equation~\eqref{1.21} below; it separates the
single-frequency terms, which carry a phase that moves with the choice of origin, from the
difference-frequency cosines and sines, the latter carrying the sign that distinguishes the two
families. That resolved form is also what the present construction produces directly, and it is
what the higher dimensions require, since from $n=5$ onward the fields are not Beltrami and no
Bernoulli form is available.
The velocity components are normalized using a reference velocity $v_r$, so that the mean kinetic
energy per unit mass is $v_r^2/2$ at $t=0$; this is the normalisation adopted by~\citet{ant}, and
direct integration over a periodic cell confirms it for the field~\eqref{1.14}--\eqref{1.16}. The
corresponding pressure is
\begin{eqnarray}
\label{1.17}
p=p_0-\rho\frac{\lVert\boldsymbol{v}\rVert^2}{2},
\end{eqnarray}
where $p_0$ is an arbitrary reference pressure. The pressure~\eqref{1.17} is available in this closed Bernoulli form for a specific reason: the
solutions of~\citet{ant}, like those of~\citet{eth} before them, are \emph{Beltrami} flows, in
which the vorticity is everywhere parallel to the velocity. For such flows the convective term
$(\boldsymbol{v}\!\cdot\!\nabla)\boldsymbol{v}$ reduces to the gradient $\nabla(\lvert\lvert\boldsymbol{v}\rvert\rvert^2/2)$,
so it is absorbed into the pressure automatically and no separate integrability question arises.
This is a genuine and useful property, but it is also a restriction: the velocity fields it admits
are those whose self-advection happens to be a gradient from the outset. The construction developed
in the present paper does not impose the Beltrami condition. Its cyclic fields are, in general,
\emph{not} Beltrami, so their convective term is not a gradient \emph{a priori}, and the existence
of a single-valued pressure becomes a question to be settled rather than a property granted in
advance. Settling that question is what yields both the closed-form pressures in the dimensions
where they exist and the exact identification of the dimensions where no pressure exists---a
distinction that does not arise for Beltrami constructions, because for them the pressure is never
in doubt. In this sense the present work and the Beltrami solutions are complementary: the latter
exhibit particular flows for which the pressure is immediate, while the former determines, across
dimensions, precisely when a pressure can and cannot be found.

\medskip
A recent approach, introduced by \cite{tha4}, reformulates the Navier--Stokes equations in a manner that permits the direct construction of analytical velocity and pressure fields. Unlike previous approaches, the methodology yields closed-form pressure expressions and provides a systematic framework that can be extended to higher-dimensional periodic flows. The approach successfully reproduces the classical Taylor solution, the Arnold-Beltrami family of solutions, and Antuono's triple-periodic velocity field while simultaneously yielding an analytical pressure distribution.

This approach involves reinterpreting the Navier--Stokes equation and breaking it down into three distinct components: linear viscous forces, inertial forces, and external forces acting on the fluid. Through this process, the modified Navier--Stokes equation can be
decomposed, in the unforced case, into a diffusion equation for the velocity field and a pressure
Poisson equation whose source is determined by the inertial field. Under specific conditions, when the external force applied to the fluid is zero, the velocity and pressure fields can be respectively expressed as solutions of the Cauchy diffusion equation and the Poisson equation. This technique provides a direct and straightforward approach to deriving closed-form solutions for the velocity and pressure fields within the context of the Navier--Stokes flow problem. To demonstrate its effectiveness, the method is applied to replicate the double-periodic solution presented by \cite{tay}, as well as the unit-periodic solutions by \cite{arn}, \cite{bab}, and \cite{tha1}. Furthermore, the method successfully extends to reproduce the triple-periodic solution of \cite{ant} for the velocity field in three dimensions. Additionally, the paper derives a closed-form analytical expression for the pressure field in three dimensions.

The three-dimensional solutions presented by \cite{tha4} are as follows:
\begin{eqnarray}
\label{1.18}
v_1& = &{v_r}\left[\sin { \left(\alpha x_1-\frac{\pi}{3} \right) } \cos {\left(\alpha x_2+\frac{\pi}{3}\right) }\sin { \left(\alpha x_3+\frac{\pi}{2}\right) }\right.- \nonumber\\
&-&\left.\sin  {\left( \alpha x_1+\frac{\pi}{3}\right) } \cos  { \left(\alpha x_3-\frac{\pi}{3} \right) } \sin  {\left( \alpha x_2+\frac{\pi}{2}\right) }\right]e^{ - 3\alpha ^2 \kappa t}
\end{eqnarray}
\begin{eqnarray}
\label{1.19}
v_2 & = &{v_r}\left[\sin { \left(\alpha x_2-\frac{\pi}{3} \right) } \cos {\left(\alpha x_3+\frac{\pi}{3}\right) }\sin { \left(\alpha x_1+\frac{\pi}{2}\right) }\right.- \nonumber\\
&-&\left.\sin  {\left( \alpha x_2+\frac{\pi}{3}\right) } \cos  { \left(\alpha x_1-\frac{\pi}{3} \right) } \sin  {\left( \alpha x_3+\frac{\pi}{2}\right) }\right]e^{ - 3\alpha ^2 \kappa t}
\end{eqnarray}
\begin{eqnarray}
\label{1.20}
v_3 &=&{v_r}\left[ \sin { \left(\alpha x_3-\frac{\pi}{3} \right) } \cos {\left(\alpha x_1+\frac{\pi}{3}\right) }\sin { \left(\alpha x_2+\frac{\pi}{2}\right) }\right.- \nonumber\\
&-&\left.\sin  {\left( \alpha x_3+\frac{\pi}{3}\right) } \cos  { \left(\alpha x_2-\frac{\pi}{3} \right) } \sin  {\left( \alpha x_1+\frac{\pi}{2}\right) }\right]e^{ - 3\alpha ^2 \kappa t}
\end{eqnarray}
The corresponding pressure field, associated with the velocity field~\eqref{1.18}--\eqref{1.20}, is given by
\begin{eqnarray}
\label{1.21}
p&=&-\frac{3\rho {v_r^2} }{16}\left[\cos\left(2\alpha x_1 \right)+\cos\left(2\alpha x_2 \right)+\cos\left(2\alpha x_3 \right)\right]e^{ - 6\alpha ^2 \kappa t}- \nonumber\\
&-&\frac{3\rho{v_r^2}}{64}\left[\cos 2\alpha\left(x_1-x_2\right)+\cos 2\alpha\left(x_3-x_1\right)+\cos 2\alpha\left(x_3-x_2\right)\right]e^{ - 6\alpha ^2 \kappa t}- \nonumber\\
&-&\frac{3\sqrt{3}\rho{v_r^2} }{64}\left[\sin 2\alpha\left(x_1-x_2\right)+\sin 2\alpha\left(x_3-x_1\right)+\sin 2\alpha\left(x_2-x_3\right)\right]e^{ - 6\alpha ^2 \kappa t}\qquad\quad
\end{eqnarray}
It is important to note that if the augmenting coefficient $\frac{4\sqrt{2}}{3\sqrt{3}}$ in equations $\left(\ref{1.14}\right)$-$\left(\ref{1.16}\right)$ is incorporated into the reference velocity $v_r$, then the velocity field described in \cite{ant} would correspond to that given by equations $\left(\ref{1.18}\right)$-$\left(\ref{1.20}\right)$, albeit with a uniform phase-angle shift of $\frac{\pi}{2}$ applied to each spatial coordinate.
Following the presentation of the three-dimensional solutions, Thambynayagam also introduces a quadruple-periodic solution in four dimensions in the same paper. These are described as follows:
\begin{eqnarray}
\label{1.22}
v_1 &=& v_r\left[\sin { \left(\alpha x_1-\frac{\pi}{4} \right) } \cos {\left(\alpha x_2+\frac{\pi}{4}\right) }\sin { \left(\alpha x_3+\frac{\pi}{4}\right) } \cos { \left(\alpha x_4+\frac{\pi}{4} \right) }\right.- \nonumber\\
&-&\left.\sin  {\left( \alpha x_1+\frac{\pi}{4}\right) }\sin  {\left( \alpha x_2+\frac{\pi}{4}\right) } \cos  { \left(\alpha x_3+\frac{\pi}{4} \right) } \cos  {\left( \alpha x_4-\frac{\pi}{4}\right) } \right]e^{ - 4\alpha ^2 \kappa t}\qquad
\end{eqnarray}
\begin{eqnarray}
\label{1.23}
v_2 &=&v_r \left[\sin { \left(\alpha x_2-\frac{\pi}{4} \right) } \cos {\left(\alpha x_3+\frac{\pi}{4}\right) }\sin { \left(\alpha x_4+\frac{\pi}{4}\right) } \cos { \left(\alpha x_1+\frac{\pi}{4} \right) }\right.- \nonumber\\
&-&\left.\sin  {\left( \alpha x_2+\frac{\pi}{4}\right) }\sin  {\left( \alpha x_3+\frac{\pi}{4}\right) } \cos  { \left(\alpha x_4+\frac{\pi}{4} \right) } \cos  {\left( \alpha x_1-\frac{\pi}{4}\right) } \right]e^{ - 4\alpha ^2 \kappa t}\qquad
\end{eqnarray}
\begin{eqnarray}
\label{1.24}
v_3&=&v_r\left[ \sin { \left(\alpha x_3-\frac{\pi}{4} \right) } \cos {\left(\alpha x_4+\frac{\pi}{4}\right) }\sin { \left(\alpha x_1+\frac{\pi}{4}\right) } \cos { \left(\alpha x_2+\frac{\pi}{4} \right) }\right.- \nonumber\\
&-&\left.\sin  {\left( \alpha x_3+\frac{\pi}{4}\right) } \sin  {\left( \alpha x_4+\frac{\pi}{4}\right) } \cos  { \left(\alpha x_1+\frac{\pi}{4} \right) }\cos  {\left( \alpha x_2-\frac{\pi}{4}\right) }\right]e^{ - 4\alpha ^2 \kappa t}\qquad
\end{eqnarray}
\begin{eqnarray}
\label{1.25}
v_4&=&v_r\left[ \sin { \left(\alpha x_4-\frac{\pi}{4} \right) } \cos {\left(\alpha x_1+\frac{\pi}{4}\right) }\sin { \left(\alpha x_2+\frac{\pi}{4}\right) } \cos { \left(\alpha x_3+\frac{\pi}{4} \right) }\right.- \nonumber\\
&-&\left.\sin  {\left( \alpha x_4+\frac{\pi}{4}\right) }\sin  {\left( \alpha x_1+\frac{\pi}{4}\right) }  \cos  { \left(\alpha x_2+\frac{\pi}{4} \right) }\cos  {\left( \alpha x_3-\frac{\pi}{4}\right) } \right]e^{ - 4\alpha ^2 \kappa t}\qquad
\end{eqnarray}
Pressure is given by
\begin{eqnarray}
\label{1.26}
p&=&\frac{\rho {v_r^2} }{4}\left[\sin  { \left(2\alpha x_1 \right) } \sin  {\left( 2\alpha x_3\right) }+\sin  { \left(2\alpha x_2 \right) } \sin  {\left( 2\alpha x_4\right) }\right]e^{ - 8\alpha ^2 \kappa t}
\end{eqnarray}
The Gaussian and Newtonian-kernel integral identities used in the evaluation of the velocity and pressure representations are collected in Appendix~\ref{app:integrals}.

The constructions of~\citet{eth} and~\citet{ant} are the immediate antecedents of this work, and it
is worth setting the two routes side by side before proceeding. Theirs represents the velocity through a stream function,
$\boldsymbol{u}=\nabla\times\boldsymbol{\psi}$, so that solenoidality is automatic, and then
imposes in one step the requirement that the convective term be a gradient,
$\nabla\times[\boldsymbol{u}\times(\nabla\times\boldsymbol{u})]=\boldsymbol{0}$; the resulting
nonlinear system is solved for the amplitude coefficients of the separated stream function.

Four differences bear on what can then be asked. First, the unknowns: that route solves for
amplitudes---six coefficients, in the three-dimensional case---whereas the construction below fixes
the amplitudes and solves for the $n$ phase angles, which is what makes the resulting system finite
and its solution set classifiable. Second, the conditions: that route imposes the gradient
requirement as a single demand, whereas the reformulation used here separates a
\emph{necessary} algebraic condition from a \emph{sufficient} differential one and applies them at
different stages, so that a construction which fails can be shown \emph{where} it fails. Third, the
class reached: the solutions so obtained are Beltrami, as discussed below, and for Beltrami flows the convective
term is already a gradient, so no separate integrability condition for this term is required; the fields constructed here are in general not Beltrami, so the
existence of a pressure is the question rather than the premise.

The fourth difference is the decisive one. In three dimensions a solenoidal field may be
represented by a vector potential, $\boldsymbol{u}=\nabla\times\boldsymbol{\psi}$, using the
ordinary vector-valued curl. This particular vector-calculus representation does not extend
unchanged to arbitrary dimension; the higher-dimensional analogues require antisymmetric tensor or
differential-form potentials. The construction below avoids that additional structure and secures
solenoidality instead by the telescoping identity of Lemma~\ref{lem:solenoidal}, which holds for
every $n\ge3$ and every phase vector. The earlier framework, whatever its merits in three
dimensions, cannot pose the question this paper answers.

The literature contains only a few references \cite{str1, str2, don} that examine
steady solutions of the Navier--Stokes equations in five and six dimensions,
particularly in domains extending beyond four dimensions. Building on the methodology
introduced in \cite{tha4}, the present study advances this direction of research well
beyond a single construction: it develops a dimension-dependent analysis of the
existence and obstruction of exact unforced periodic solutions, identifies the
arithmetic and differential mechanisms responsible for the obstructions encountered,
and constructs the associated obstruction-forced solutions as exact, closed-form
benchmarks in dimensions five through eight.
\section{Reformulation of the Navier--Stokes Equations in \( \mathbb{R}^n \)}
\label{sec:reformulation}

To enhance readability and continuity, we revisit and adapt key developments from our earlier work \cite{tha1}, in which the Navier--Stokes equations were recast in a form that isolates the nonlinear and external forcing contributions. This reformulation is particularly useful in the context of unbounded domains.

Taking the divergence of equation~\eqref{1.1}, using the incompressibility
condition~\eqref{1.3} and the commutation of spatial derivatives, yields the pressure
Poisson equation:
\begin{equation}
\label{2.1}
\Delta p(\boldsymbol{x}, t) = \rho \sum_{i=1}^n \frac{\partial}{\partial x_i}(f_i - g_i).
\end{equation}
This relation, widely known as the pressure Poisson equation, has been examined in the literature (e.g., \cite{gre}) for its utility in addressing incompressible flow problems. It is important to note that while equations $\left(\ref{1.1}\right)$ and $\left(\ref{1.3}\right)$ imply $\left(\ref{2.1}\right)$, the converse is not generally true. Equation~\eqref{2.1} is obtained by taking the divergence of the momentum equation and therefore imposes only a scalar compatibility condition; satisfaction of the pressure Poisson equation alone does not guarantee satisfaction of the original vector momentum equation. Consequently, the pressure must remain compatible with the solenoidal velocity field and the momentum balance.

Underlying this is the Helmholtz decomposition, which will recur throughout the analysis. Any
smooth periodic vector field $\boldsymbol{w}$ may be decomposed, after fixing the zero mode in the usual way,
into a longitudinal (curl-free) and a transverse (divergence-free) part,
\begin{equation}
\label{2.2}
\boldsymbol{w} = \nabla\phi + \boldsymbol{h}, \qquad \nabla\cdot\boldsymbol{h}=0,
\end{equation}
with $\phi$ determined by $\Delta\phi = \nabla\cdot\boldsymbol{w}$. In the momentum
balance~\eqref{1.1} the pressure gradient accounts for the longitudinal part of the combined field
$\boldsymbol{g}-\boldsymbol{f}$: solving~\eqref{2.1} determines that longitudinal contribution,
while the transverse remainder is retained in the velocity evolution. For the unforced class of
solutions considered first, the nonlinear term can be eliminated from the velocity evolution only
when the convective field is entirely longitudinal, that is, a pure gradient. The constructions
below are organised around determining when this condition holds.

For sufficiently smooth source fields with the required decay at spatial infinity, the solution
to~\eqref{2.1} in $\mathbb{R}^n$ admits the following Newtonian-potential representation~\citep{pol}:
\begin{equation}
\label{2.3}
\begin{aligned}
p(\boldsymbol{x}, t) &= -\frac{\rho}{2\pi} \int_{\mathbb{R}^2} \mathcal{P}( \boldsymbol{y}, t ) \ln \left( \frac{1}{\sqrt{\mathcal{P}_n(\boldsymbol{x}, \boldsymbol{y})}} \right) \, d\boldsymbol{y}, \quad &n = 2, \\
p(\boldsymbol{x}, t) &= -\frac{ \rho \, \Gamma(n/2) }{2(n - 2)\pi^{n/2}} \int_{\mathbb{R}^n} \frac{ \mathcal{P}(\boldsymbol{y}, t) }{ \mathcal{P}_n(\boldsymbol{x}, \boldsymbol{y})^{(n - 2)/2} } \, d\boldsymbol{y}, \quad &n \geq 3,
\end{aligned}
\end{equation}
where \( \Gamma(z) = \int_0^\infty e^{-u} u^{z - 1} \, du \) denotes the Gamma function for \( \Re(z) > 0 \), and
\begin{align}
\label{2.4}
\mathcal{P}(\boldsymbol{x}, t) &= \sum_{j=1}^n \frac{\partial}{\partial x_j}(f_j - g_j), \\
\label{2.5}
\mathcal{P}_n(\boldsymbol{x}, \boldsymbol{y}) &= \sum_{j=1}^n (x_j - y_j)^2.
\end{align}
Differentiating \eqref{2.3} with respect to \( x_i \) yields an expression for the pressure gradient:
\begin{equation}
\label{2.6}
\frac{\partial p}{\partial x_i} = \frac{ \rho \, \Gamma(n/2) }{ 2\pi^{n/2} } \int_{\mathbb{R}^n} \frac{(x_i - y_i) \, \mathcal{P}(\boldsymbol{y}, t)}{ \mathcal{P}_n(\boldsymbol{x}, \boldsymbol{y})^{n/2} } \, d\boldsymbol{y}, \qquad n \geq 2.
\end{equation}
For the periodic fields considered later, the whole-space integrals in~\eqref{2.3}
and~\eqref{2.6} are not invoked as absolutely convergent Newtonian integrals, since such fields
do not decay at spatial infinity. Their role is instead understood through the equivalent Fourier
representation of the Helmholtz projection on the periodic domain, with the zero mode treated
separately. Thus the integral formulas below apply directly under the stated decay assumptions,
while their periodic counterparts are interpreted mode-by-mode through the corresponding
longitudinal projection.

Substituting this into the momentum equation \eqref{1.1}, we arrive at a reformulated form of the Navier--Stokes equations:
\begin{equation}
\label{2.7}
\frac{\partial v_i}{\partial t} = \kappa \Delta v_i - \mathcal{U}_i + \mathcal{F}_i, \qquad \boldsymbol{x} \in \mathbb{R}^n, \, t \geq 0,
\end{equation}
where the nonlinear and forcing contributions are defined as:
\begin{align}
\label{2.8}
\mathcal{U}_i(\boldsymbol{x}, t) &= g_i(\boldsymbol{x}, t) - \frac{ \Gamma(n/2) }{ 2\pi^{n/2} } \int_{\mathbb{R}^n} \frac{ (x_i - y_i) \sum_{k=1}^n \frac{\partial g_k(\boldsymbol{y}, t)}{\partial y_k} }{ \mathcal{P}_n(\boldsymbol{x}, \boldsymbol{y})^{n/2} } \, d\boldsymbol{y}, \\
\label{2.9}
\mathcal{F}_i(\boldsymbol{x}, t) &= f_i(\boldsymbol{x}, t) - \frac{ \Gamma(n/2) }{ 2\pi^{n/2} } \int_{\mathbb{R}^n} \frac{ (x_i - y_i) \sum_{k=1}^n \frac{\partial f_k(\boldsymbol{y}, t)}{\partial y_k} }{ \mathcal{P}_n(\boldsymbol{x}, \boldsymbol{y})^{n/2} } \, d\boldsymbol{y}.
\end{align}
The right-hand side of \eqref{2.7} thus consists of three contributions: the viscous term \( \kappa \Delta v_i \), the nonlinear term \( \mathcal{U}_i \), and the external forcing \( \mathcal{F}_i \).

Comparing~\eqref{2.8} and~\eqref{2.9} with the Helmholtz decomposition~\eqref{2.2},
$\boldsymbol{\mathcal{U}}$ and $\boldsymbol{\mathcal{F}}$ are respectively the transverse,
divergence-free projections of $\boldsymbol{g}$ and $\boldsymbol{f}$: each field with its
longitudinal part subtracted. The Fundamental Lemma is then almost self-explanatory---if the
transverse projection of the nonlinear term vanishes, that term leaves the velocity evolution
entirely.

\subsection*{Fundamental Lemma}
The following result, established in \cite{tha1}, forms the basis of the present methodology. For
fields satisfying the decay assumptions above---or, for periodic fields, with the identity
understood through the corresponding Fourier Helmholtz projection---if the nonlinear term
satisfies:
\begin{equation}
\label{2.10}
g_i(\boldsymbol{x}, t) = \frac{ \Gamma(n/2) }{ 2\pi^{n/2} } \int_{\mathbb{R}^n} \frac{ (x_i - y_i) \sum_{k=1}^n \frac{\partial g_k(\boldsymbol{y}, t)}{\partial y_k} }{ \mathcal{P}_n(\boldsymbol{x}, \boldsymbol{y})^{n/2} } \, d\boldsymbol{y},
\end{equation}
then the Navier--Stokes equations reduce to a nonhomogeneous diffusion problem~\citep{car, tha2}:
\begin{equation}
\label{2.11}
\begin{aligned}
v_i(\boldsymbol{x}, t) ={}& \frac{1}{(2 \sqrt{\pi \kappa t})^n} \int_{\mathbb{R}^n} v_i^0(\boldsymbol{y}) \exp\left( -\frac{ |\boldsymbol{x} - \boldsymbol{y}|^2 }{ 4\kappa t } \right) \, d\boldsymbol{y} \\
&+ \frac{1}{(2 \sqrt{\pi \kappa})^n} \int_0^t \int_{\mathbb{R}^n} \frac{ \mathcal{F}_i(\boldsymbol{y}, \tau) \exp\left( -\frac{ |\boldsymbol{x} - \boldsymbol{y}|^2 }{ 4\kappa (t - \tau) } \right) }{ (t - \tau)^{n/2} } \, d\boldsymbol{y} \, d\tau.
\end{aligned}
\end{equation}
If the external force vanishes, that is \( f_i = 0 \), the second term in $\left(\ref{2.11}\right)$ disappears, and the solution reduces to that of the classical Cauchy diffusion equation:
\begin{equation}
\label{2.12}
v_i(\boldsymbol{x}, t) = \frac{1}{(2 \sqrt{\pi \kappa t})^n} \int_{\mathbb{R}^n} v_i^0(\boldsymbol{y}) \exp\left( -\frac{ |\boldsymbol{x} - \boldsymbol{y}|^2 }{ 4\kappa t } \right) \, d\boldsymbol{y}.
\end{equation}
At $t=0$, equation~\eqref{2.10} can be written as:
\begin{equation}
\label{2.13}
g_i^0(\boldsymbol{x}) = \frac{ \Gamma(n/2) }{ 2\pi^{n/2} } \int_{\mathbb{R}^n} \frac{ (x_i - y_i) \sum_{k=1}^n \frac{\partial g_k^0(\boldsymbol{y})}{\partial y_k} }{ \mathcal{P}_n(\boldsymbol{x}, \boldsymbol{y})^{n/2} } \, d\boldsymbol{y},
\end{equation}
For the $n$-tuple-periodic velocity fields considered below, which carry one factor of wavenumber
$\alpha$ in each of the $n$ coordinates and therefore satisfy
$\Delta\boldsymbol{v}^0=-n\alpha^2\boldsymbol{v}^0$, the diffusive solution~\eqref{2.12} yields the
temporal factor $e^{-n\alpha^2\kappa t}$; consequently the quadratic inertial term satisfies
$g_i(\boldsymbol{x},t)=g_i^0(\boldsymbol{x})\,e^{-2n\alpha^2\kappa t}$. In the unforced case, when
the nonlinear field satisfies~\eqref{2.10} and is therefore entirely longitudinal,
equation~\eqref{2.6} reduces to:
\begin{equation}
\label{2.14}
\frac{\partial p}{\partial x_i} = -\rho g_i(\boldsymbol{x},t)  = -\rho g_i^0(\boldsymbol{x})  \, e^{-2n \alpha^2 \kappa t}.
\end{equation}
Pressure may be obtained through the direct integration of equation $\left(\ref{2.14}\right)$ provided the solenoidal velocity field \( \boldsymbol{v}^0 \) admits an \( n \)-tuple-periodic representation and the following symmetry condition holds:
\begin{equation}
\label{2.15}
\frac{\partial g_i}{\partial x_j} = \frac{\partial g_j}{\partial x_i}, \qquad \forall i, j \in \{1, 2, \dots, n\}.
\end{equation}
On $\mathbb{R}^n$ this condition ensures that $-\boldsymbol{g}$ is a conservative vector field,
making $p$ a well-defined scalar potential recoverable by direct integration of~\eqref{2.14}. On
the periodic cell a curl-free field may in addition carry a constant harmonic component, which is
not the gradient of any periodic scalar, so~\eqref{2.15} must be supplemented by the vanishing of
that zero mode. For the fields considered here the supplement is automatic: since
$\boldsymbol{v}$ is solenoidal, $g_i=\partial_j(v_iv_j)$ is a divergence, and the mean of a
divergence over a period vanishes, so $\langle\boldsymbol{g}\rangle=\boldsymbol{0}$ identically.
Equation~\eqref{2.15} therefore constitutes the full pressure-integrability condition in the
present setting. It will be distinguished below from the reduced algebraic conditions used to identify candidate phase angles, which are necessary but not, in general, sufficient to guarantee~\eqref{2.15}.
\section{Construction of $n$-Tuple Periodic Solutions: General Framework}
\label{sec:construction}
In this section, we present a general framework for constructing $n$-tuple-periodic solutions to the Navier--Stokes equations in $n$ spatial dimensions. The cyclic construction is formulated for $n\ge 3$; the degenerate two-dimensional case is discussed separately in Subsection~\ref{subsec:2D}. The formulation begins with a solenoidal initial velocity vector field, constructed by introducing phase angle shifts, denoted by $\boldsymbol{\xi} = (\xi_1, \xi_2, \dots, \xi_n)$, into the spatial coordinates.

To ensure structural consistency across spatial dimensions, we adopt a coordinate convention
based on cyclic symmetry. Throughout, subscripts on the coordinates are read cyclically: for any
integer $m$ we write
\[
\langle m\rangle \;:=\; \left((m-1)\bmod n\right)+1 \;\in\; \{1,\dots,n\},
\]
so that $x_{\langle m\rangle}$ is always one of $x_1,\dots,x_n$, and $\langle m\rangle = m$
whenever $1\le m\le n$. With this notation the $i$th velocity component takes its arguments in
the cyclic order beginning at $x_i$,
\[
v_i^0(\boldsymbol{x}, \boldsymbol{\xi}) = v_i^0\!\left(
x_{\langle i\rangle},\; x_{\langle i+1\rangle},\; x_{\langle i+2\rangle},\; \dots,\;
x_{\langle i+n-1\rangle};\, \boldsymbol{\xi}
\right), \qquad i = 1, \dots, n .
\]
Every component therefore depends on all $n$ coordinates, differing only in the cyclic order in
which they are taken. The same convention applies to the nonlinear inertial term
\( g_i^0 \), thereby preserving the underlying coordinate symmetry in the formulation of both
\( v_i^0 \) and \( g_i^0 \).

The solenoidal initial velocity field for an $n$-tuple system is defined as the difference
of two $n$-fold cyclic trigonometric products,
\begin{equation}
\label{3.1}
v_i^0(\boldsymbol{x}, \boldsymbol{\xi}) \;:=\; v_r\left[\,\mathcal{T}1_i^0(\boldsymbol{x},\boldsymbol{\xi})
\;-\; \mathcal{T}2_i^0(\boldsymbol{x},\boldsymbol{\xi})\,\right], \qquad i = 1,\dots,n,
\end{equation}
where, with the cyclic index convention introduced above,
\begin{subequations}
\begin{align}
\label{3.2a}
\mathcal{T}1_i^0 &= \prod_{k=1}^{n}
\begin{cases}
\sin\!\left( \alpha x_{\langle i+k-1\rangle} + \xi_k \right), & k \text{ odd},\\[2pt]
\cos\!\left( \alpha x_{\langle i+k-1\rangle} + \xi_k \right), & k \text{ even},
\end{cases}\\[6pt]
\label{3.2b}
\mathcal{T}2_i^0 &= \sin\!\left( \alpha x_{\langle i\rangle} + \xi_2 \right)
\;\;\prod_{k=2}^{n-1}
\begin{cases}
\cos\!\left( \alpha x_{\langle i+k-1\rangle} + \xi_{k+1} \right), & k \text{ odd},\\[2pt]
\sin\!\left( \alpha x_{\langle i+k-1\rangle} + \xi_{k+1} \right), & k \text{ even},
\end{cases}
\;\;\cos\!\left( \alpha x_{\langle i+n-1\rangle} + \xi_1 \right).
\end{align}
\end{subequations}
In both products the factor in position $k$ acts on the coordinate
$x_{\langle i+k-1\rangle}$; in $\mathcal{T}1_i^0$ it carries the phase $\xi_k$, and in $\mathcal{T}2_i^0$ the
cyclically shifted phase $\xi_{(k\bmod n)+1}$. Note that $\mathcal{T}2_i^0$ is \emph{not} obtained from
$\mathcal{T}1_i^0$ by shifting the phases alone: its trigonometric parity is reversed in the interior
$2\le k\le n-1$, and the two end factors are fixed --- a sine in position $k=1$ and a cosine
in position $k=n$. This is not an arbitrary design choice; it is chosen precisely so that the
telescoping identity of Lemma~\ref{lem:solenoidal} holds.

\begin{lemma}[Solenoidality by telescoping]
\label{lem:solenoidal}
For every $n\ge 3$ and every $\boldsymbol{\xi}\in\mathbb{R}^n$, the two product terms
of~\eqref{3.1} satisfy the identity
\begin{equation}
\label{3.3}
\frac{\partial\, \mathcal{T}2_i^0}{\partial x_i}
\;=\;
\frac{\partial\, \mathcal{T}1_{i-1}^0}{\partial x_{i-1}},
\qquad i = 1,\dots,n \quad (\text{indices modulo } n).
\end{equation}
Consequently $\nabla\cdot\boldsymbol{v}^0 = 0$ identically, for every choice of phase angles.
\end{lemma}

\begin{proof}
Only the factor acting on $x_{i-1}$ depends on $x_{i-1}$ in $\mathcal{T}1_{i-1}^0$, and it is
$\sin(\alpha x_{i-1}+\xi_1)$ (position $k=1$). Differentiating replaces it by
$\alpha\cos(\alpha x_{i-1}+\xi_1)$ and leaves the remaining $n-1$ factors, which act on
$x_i, x_{i+1},\dots,x_{i+n-2}$ carrying the phases $\xi_2,\xi_3,\dots,\xi_n$ with the
trigonometric types $\cos,\sin,\cos,\dots$ respectively. Likewise, only the factor acting on
$x_i$ depends on $x_i$ in $\mathcal{T}2_i^0$, namely $\sin(\alpha x_i+\xi_2)$, whose derivative is
$\alpha\cos(\alpha x_i+\xi_2)$; the surviving factors are precisely the block
$\prod_{k=2}^{n-1}$ of~\eqref{3.2b} together with $\cos(\alpha x_{i+n-1}+\xi_1)$. Matching the
two products factor by factor on each coordinate gives~\eqref{3.3}. Summing over $i$,
\begin{equation}
\label{3.4}
\nabla\cdot\boldsymbol{v}^0
= v_r\sum_{i=1}^{n}\left(\frac{\partial \mathcal{T}1_i^0}{\partial x_i}-\frac{\partial \mathcal{T}2_i^0}{\partial x_i}\right)
= v_r\left(\sum_{i=1}^{n}\frac{\partial \mathcal{T}1_i^0}{\partial x_i}
- \sum_{i=1}^{n}\frac{\partial \mathcal{T}1_{i-1}^0}{\partial x_{i-1}}\right) = 0,
\end{equation}
the two sums being cyclic reindexings of one another.
\end{proof}

\noindent\textbf{Remark.} Equation~\eqref{3.3} fixes $\mathcal{T}2_i^0$ only up to an additive term
independent of $x_i$: it states that $\mathcal{T}2_i^0$ is an $x_i$-antiderivative of
$\partial \mathcal{T}1_{i-1}^0/\partial x_{i-1}$, and any function of the remaining coordinates may be
added without disturbing the telescoping. Taking that additive term to be zero, and requiring
the result to be itself a cyclic product of the same form, selects~\eqref{3.2b}. The second term
is therefore not an independent guess: it is the canonical cyclic-product antiderivative
singled out by the construction. We do not claim more than this --- in particular,
equation~\eqref{3.3} alone does not establish that~\eqref{3.2b} is the only two-term cyclic
product whose divergence telescopes, and no such uniqueness is needed in what follows.

Solenoidality holds for \emph{all} $\boldsymbol{\xi}$; the phase angles remain entirely free at
this stage and are fixed only later, by the reduction condition~\eqref{3.12} and the
pressure-integrability condition~\eqref{3.16}. Adopting~\eqref{3.1} at $n=3,4,5,6,7,8$
reproduces, up to cyclic reindexing of the coordinates, equations~\eqref{4.1}, \eqref{4.11},
\eqref{5.1}, \eqref{6.1}, \eqref{7.1} and~\eqref{8.1} respectively.

With the field and its solenoidality in hand, the time-dependent solution can be stated at once.
Once a phase vector $\boldsymbol{\xi}$ is admitted---that is, once the inertial field $\boldsymbol{g}^0$ associated with $\boldsymbol{v}^0$ satisfies the integrability relation~\eqref{2.13}, understood in the periodic Fourier sense made
explicit below---the Navier--Stokes problem reduces to the Cauchy
diffusion equation~\eqref{2.12}, and the velocity at all later times is simply the initial field
carried forward in time by the solution of~\eqref{2.12},
\begin{equation}
\label{3.5}
v_i(\boldsymbol{x}, \boldsymbol{\xi}, t) =  v_i^0(\boldsymbol{x}, \boldsymbol{\xi})\, e^{-n\alpha^2 \kappa t}, \qquad \text{for } i = 1, \dots, n.
\end{equation}
Each Fourier mode of $v_i^0$ is a product of $n$ single-coordinate factors of wavenumber
$\alpha$, so writing $\boldsymbol{k}$ for the wavevector of a mode, every mode has
$|\boldsymbol{k}|^2 = n\alpha^2$ and decays at the single rate
$e^{-n\alpha^2\kappa t}$; the associated pressure is obtained by integrating the
gradient relation~\eqref{2.14}. Since the convective term is quadratic in the velocity, the
pressure gradient, and hence the nonconstant part of the pressure, decays twice as fast, as
$e^{-2n\alpha^2\kappa t}$. The whole of the remaining analysis is thus
devoted to the one open question---which phase vectors $\boldsymbol{\xi}$, if any, satisfy the
pressure-integrability condition~\eqref{2.15} in each dimension.

\subsection{The two-dimensional case}
\label{subsec:2D}
It is worth pausing on the lowest dimension, $n=2$, both because it is the setting of the
classical solution recalled in the introduction and because the construction degenerates there in
an instructive way. Here it can be read off from~\eqref{3.1} directly.

The classical two-dimensional flow is the Taylor--Green vortex~\eqref{1.7}--\eqref{1.8}, with the
nonconstant pressure~\eqref{1.9}. There the convective term $\boldsymbol{g}^0
=(\boldsymbol{v}^0\!\cdot\!\nabla)\boldsymbol{v}^0$ does not vanish; its components are
\[
g_1^0=\tfrac{1}{2}\alpha v_r^2\sin 2\alpha x_1,
\qquad
g_2^0=\tfrac{1}{2}\alpha v_r^2\sin 2\alpha x_2,
\]
which together form an exact gradient, and it is absorbed entirely by~\eqref{1.9}. Two dimensions is solvable because that
convective field is purely a gradient---the situation that will hold again, for a different
reason, in three and four dimensions; in the higher dimensions considered below, either no
non-trivial candidate survives the reduction or the resulting convective field fails the gradient
condition.

Formally extending the two-term ansatz~\eqref{3.1} to $n=2$ organises the same trigonometry
differently, the interior product in~\eqref{3.2b} being interpreted as an empty product. The two
terms then each consist of two trigonometric factors, and the first velocity component is
\begin{equation}
\label{3.6}
v_1^0 = v_r\bigl[\sin(\alpha x_1+\xi_1)\cos(\alpha x_2+\xi_2)-\sin(\alpha x_1+\xi_2)\cos(\alpha x_2+\xi_1)\bigr],
\end{equation}
with $v_2^0$ given by the exchange $x_1\leftrightarrow x_2$. The two products differ only by the
interchange of the two phases, and a sum-to-product identity collapses the difference to a single
harmonic,
\begin{equation}
\label{3.7}
v_1^0 = v_r\sin(\xi_1-\xi_2)\,\cos\!\bigl(\alpha(x_1-x_2)\bigr),
\end{equation}
a diagonal shear mode whose amplitude is $v_r\sin(\xi_1-\xi_2)$. This is not the Taylor--Green
vortex: the antisymmetric two-term combination retains only the $x_1-x_2$ shear and discards the
$x_1+x_2$ part, so $v_1^0$ is a single diagonal shear mode rather than a cellular vortex. Applying
the exchange $x_1\leftrightarrow x_2$ to~\eqref{3.7} gives $v_2^0 = v_1^0$, so the two components
coincide and the field is directed uniformly along the diagonal $(1,1)$, with magnitude depending
only on the transverse coordinate $x_1-x_2$. This coincidence is not a degeneracy of the solution
but the geometric content of the two-dimensional case: the surviving mode is a single diagonal
shear, and the field is in effect one-dimensional. Because a flow that is constant along its own
direction of motion cannot advect itself, the convective term vanishes altogether,
\begin{equation}
\label{3.8}
\boldsymbol{g}^0=(\boldsymbol{v}^0\!\cdot\!\nabla)\boldsymbol{v}^0\equiv\boldsymbol{0},
\end{equation}
the pressure-integrability condition is met trivially, the pressure is constant, and
$(\boldsymbol{v}^0,\,p=\mathrm{const})$ is an exact steady solution of the Euler equations for
\emph{every} phase pair satisfying $\xi_1-\xi_2\neq 0 \pmod{\pi}$. Under viscous evolution the
corresponding Navier--Stokes solution is obtained by the diffusive factor
$e^{-2\alpha^2\kappa t}$, and is therefore not steady. The cyclic construction thus
degenerates in two dimensions: it does not recover the Taylor--Green vortex, but a simpler shear
flow with no convective interaction at all. The genuine content of the construction begins at
$n=3$, where the two products no longer collapse to a single harmonic.

\subsection{The general $n$-dimensional reduction}
The cyclic permutation of spatial arguments across velocity components ensures that no coordinate direction is privileged over the others. It is worth being explicit about the division of labour that
Lemma~\ref{lem:solenoidal} establishes, because it shapes the rest of the paper. Solenoidality
and periodicity are secured by the \emph{structure} of~\eqref{3.1} alone, for every phase
vector; the phase angles are not consumed in achieving them. The entire burden of solving the
Navier--Stokes equations therefore falls on one remaining requirement, the integral
relation~\eqref{2.13}, and it is that relation which ultimately determines whether the phase
construction yields an exact unforced solution.

For the initial velocity field $v_i^0(\boldsymbol{x}, \boldsymbol{\xi})$, equation $\left(\ref{1.5}\right)$ becomes
\begin{equation}
\label{3.10}
 g_i^0(\boldsymbol{x}, \boldsymbol{\xi}) = \sum_{j=1}^n v_j^0(\boldsymbol{x}, \boldsymbol{\xi}) \frac{\partial v_i^0(\boldsymbol{x}, \boldsymbol{\xi})}{\partial x_j}, \qquad i = 1, \dots, n.
\end{equation}
For the periodic fields considered here, it is convenient to interpret the integral relation~\eqref{2.13} through its equivalent Fourier representation. A Fourier mode of $g_i^0$ that is independent of $x_i$ has wavevector component $k_i=0$. The $i$th component of the longitudinal projection associated with~\eqref{2.13} is proportional to $k_i$, and therefore vanishes for such a mode. Consequently, a necessary condition for~\eqref{2.13} to hold is that the sum of all terms in $g_i^0$ that are independent of $x_i$ vanish identically. This Fourier interpretation is established explicitly in Proposition~\ref{prop:longitudinal} below.
\begin{lemma}[Phase-Angle Reduction Principle]
\label{lem:reduction}
Let \(g_i^0(\boldsymbol{x},\boldsymbol{\xi})\) be the inertial term associated with the solenoidal initial velocity field \(v_i^0(\boldsymbol{x},\boldsymbol{\xi})\), defined by equation~\eqref{3.10}. Suppose \(g_i^0\) is decomposed into
\begin{equation}
\label{3.11}
g_i^0(\boldsymbol{x}, \boldsymbol{\xi}) = \mathcal{V}_i^0(\boldsymbol{x}_{\setminus i}, \boldsymbol{\xi}) + \mathcal{W}_i^0(\boldsymbol{x}, \boldsymbol{\xi}), \qquad i = 1, 2, \dots, n,
\end{equation}
where \( \boldsymbol{x}_{\setminus i} \in \mathbb{R}^{n-1} \) denotes the set of all spatial coordinates excluding \( x_i \). The term $\mathcal{V}_i^0$ denotes the $x_i$-independent Fourier component of $g_i^0$,
equivalently the sum of all Fourier modes having $k_i=0$, while $\mathcal{W}_i^0$ contains the
remaining modes with $k_i\neq0$. Then a necessary condition for the integral relation~\eqref{2.13} to hold is
\begin{equation}
\label{3.12}
\mathcal{V}_i^0(\boldsymbol{x}_{\setminus i}, \boldsymbol{\xi}) \equiv 0, \qquad \text{for all } i = 1, 2, \dots, n.
\end{equation}
\end{lemma}
The reduction principle narrows the search for admissible phase angles to the
finite-dimensional system generated by the vanishing of \(\mathcal V_i^0\). We stress that it
does no more than that: the condition is \emph{necessary and not sufficient}, and a phase
vector satisfying it is a candidate, not a solution. The distinction is the whole content of the higher-dimensional obstructions established in
Sections~\ref{sec:6D} and~\ref{sec:8D}, where phase vectors that satisfy
$\mathcal V_i^0\equiv0$ exactly nevertheless fail to solve~\eqref{2.13}.

The phase angles \( \boldsymbol{\xi} = (\xi_1, \xi_2, \dots, \xi_n) \) appear uniformly across all components of the velocity field \( v_i^0 \), the inertial term \( g_i^0 \), and their decompositions \( \mathcal{V}_i^0 \) and \( \mathcal{W}_i^0 \). Moreover, the functional structure of each \( \mathcal{V}_i^0(\boldsymbol{x}_{\setminus i}, \boldsymbol{\xi}) \) is identical up to a cyclic permutation of the spatial variables.
\begin{corollary}
\label{cor:cyclic}
By the cyclic symmetry of the velocity construction~\eqref{3.1}, the condition
$\mathcal{V}_1^0 \equiv 0$ implies $\mathcal{V}_i^0 \equiv 0$ for all $i$.
\end{corollary}

\noindent Indeed, $\mathcal{V}_i^0$ is obtained from $\mathcal{V}_1^0$ by the cyclic coordinate
permutation $x_j\mapsto x_{\langle i+j-1\rangle}$ with the same phase vector $\boldsymbol{\xi}$;
an identity holding for $i=1$ is therefore carried into the corresponding identity for every $i$.
As a result, enforcing the single condition
\begin{equation}
\label{3.13}
\mathcal{V}_1^0(\boldsymbol{x}_{\setminus 1}, \boldsymbol{\xi}) \equiv 0
\end{equation}
is sufficient to ensure that \( \mathcal{V}_i^0(\boldsymbol{x}_{\setminus i}, \boldsymbol{\xi}) \equiv 0 \) for all \( i = 1, \dots, n \). Hence the candidate phase angles may be obtained by enforcing this one functional condition. We refer to~\eqref{3.12}---equivalently, by Corollary~\ref{cor:cyclic}, to its single-index form~\eqref{3.13}---as the \emph{reduction condition}: the requirement that the $x_i$-independent part $\mathcal{V}_i^0$ of the inertial term $g_i^0$ vanish for every $i$. Collecting $\mathcal{V}_1^0$ over its independent spatial modes turns this functional requirement
into a finite algebraic system in the phases alone; that separated form is set out
in Subsection~\ref{subsec:separated}, and it is the system solved dimension by dimension below. One caution
is worth recording here: an unreduced trigonometric expansion may contain linearly dependent
spatial terms, in which case its individual coefficient functions need not vanish separately.
Throughout the paper, ``the phases satisfy the reduction condition'' means exactly this. The decomposition and the resulting constraint equations are illustrated in detail for the three-dimensional case in Section~\ref{sec:solutions-3D-4D} below.

\subsection{What condition~\eqref{2.13} requires: necessary versus sufficient}
\label{subsec:nec-suf}
Before applying the reduction principle we record what the integral relation~\eqref{2.13}
actually asks of $\boldsymbol{g}^0$, and how much of it the vanishing of $\mathcal{V}_i^0$
tests. The answer explains both why the method succeeds in three and four dimensions and why
it can fail in higher ones without any error of computation.

For the periodic solenoidal fields considered here the zero-mean hypothesis of the following
proposition is automatic. Incompressibility gives
\[
g_i^0 = v_j^0\frac{\partial v_i^0}{\partial x_j}
      = \frac{\partial}{\partial x_j}\bigl(v_i^0 v_j^0\bigr),
\]
and the mean of a derivative of a periodic function vanishes, so
$\langle g_i^0\rangle=0$ for every $i$.

\begin{proposition}[Longitudinal characterisation of~\eqref{2.13}]
\label{prop:longitudinal}
Let $\boldsymbol{g}^0$ be a smooth periodic field with zero mean and Fourier
representation
\begin{equation}
\label{3.14}
g_i^0(\boldsymbol{x})=\sum_{\boldsymbol{k}\neq 0}\widehat{g}_i(\boldsymbol{k})\,
e^{\mathrm{i}\boldsymbol{k}\cdot\boldsymbol{x}},
\end{equation}
the exclusion of $\boldsymbol{k}=\boldsymbol{0}$ following from the zero-mean assumption.
Then, with relation~\eqref{2.13} understood through its periodic Fourier representation,~\eqref{2.13} holds
if and only if
\begin{equation}
\label{3.15}
\widehat{g}_i(\boldsymbol{k})
\;=\;
k_i\,\frac{\boldsymbol{k}\cdot\widehat{\boldsymbol{g}}(\boldsymbol{k})}{|\boldsymbol{k}|^{2}},
\qquad\text{for every }\boldsymbol{k}\neq 0
\text{ and every } i.
\end{equation}
The right-hand side of~\eqref{3.15} is the projection of
$\widehat{\boldsymbol{g}}(\boldsymbol{k})$ onto the direction of $\boldsymbol{k}$, so the
condition requires every Fourier coefficient to be parallel to its own wavevector; such a mode
is called \emph{longitudinal}.

Equivalently, the Jacobian of $\boldsymbol{g}^0$ is symmetric. This is the
pressure-integrability condition~\eqref{2.15} as it applies to the reference field
$\boldsymbol{g}^0$: since $\boldsymbol{g}(\boldsymbol{x},t)=\boldsymbol{g}^0(\boldsymbol{x})\,e^{-2n\alpha^2\kappa t}$,
the common exponential factor cancels from both sides and the symmetry of the
mixed derivatives is carried entirely by the time-independent field $\boldsymbol{g}^0$.
It is this reduced form that the harmonic analysis of the following sections operates on,
and to which candidate phase vectors are tested throughout; we therefore state it explicitly and refer to it as
\begin{equation}
\label{3.16}
\frac{\partial g_i^0}{\partial x_j} = \frac{\partial g_j^0}{\partial x_i},
\qquad \forall\, i, j \in \{1, 2, \dots, n\}.
\end{equation}
In $n$ dimensions the curl of a vector field is not itself a vector field but the antisymmetric
tensor $(\operatorname{curl}\boldsymbol{g}^0)_{ij}=\partial_j g_i^0-\partial_i g_j^0$,
equivalently the exterior derivative of the $1$-form $\sum_i g_i^0\,\mathrm{d}x_i$; for $n=3$
its three independent components are the familiar curl vector. Condition~\eqref{3.16} is the
statement that this tensor vanishes identically, and we refer to a field satisfying it as
\emph{curl-free} in that sense throughout.
\end{proposition}
\begin{proof}
The operator represented by the right-hand side of~\eqref{2.13}, when interpreted on periodic
zero-mean fields through its Fourier multiplier, is, up to the stated normalisation, the $i$-th
component of $\nabla\Delta^{-1}$ applied to $\nabla\cdot\boldsymbol{g}^0$: the right-hand side
is precisely $\nabla\Delta^{-1}(\nabla\cdot\boldsymbol{g}^0)$, the longitudinal Helmholtz projection of
$\boldsymbol{g}^0$ onto gradient fields. In Fourier variables $\Delta^{-1}$ acts as
$-|\boldsymbol{k}|^{-2}$ and $\nabla$ as $\mathrm{i}\boldsymbol{k}$, so the $i$-th component of
the right-hand side has Fourier coefficient
$\mathrm{i}k_i\cdot(-|\boldsymbol{k}|^{-2})\cdot\mathrm{i}\,\boldsymbol{k}\cdot
\widehat{\boldsymbol{g}}(\boldsymbol{k})
= k_i(\boldsymbol{k}\cdot\widehat{\boldsymbol{g}})/|\boldsymbol{k}|^{2}$,
which gives~\eqref{3.15}. Condition~\eqref{3.15} states that
$\widehat{\boldsymbol{g}}(\boldsymbol{k})$ is parallel to $\boldsymbol{k}$ for every
$\boldsymbol{k}$, that is that the transverse (solenoidal) part of $\boldsymbol{g}^0$ vanishes;
this is the Fourier form of $k_i\widehat{g}_j-k_j\widehat{g}_i=0$, that is, of
$\partial_j g_i^0 = \partial_i g_j^0$.
\end{proof}

Proposition~\ref{prop:longitudinal} shows that~\eqref{2.13} and the curl-free
condition~\eqref{3.16} are two statements of one requirement, rather than a condition and a
corollary of it: the convective field must be an exact gradient, so that $-\rho^{-1}\nabla p$
can absorb it. We may now say exactly what the reduction principle tests.

\begin{proposition}[The reduction condition is the $k_i=0$ slice]
\label{prop:slice}
For each $i$, the Fourier coefficients of $\mathcal V_i^0$ are precisely the coefficients
$\widehat{g}_i(\boldsymbol{k})$ with $k_i=0$. Setting $k_i=0$ in~\eqref{3.15} makes the
right-hand side vanish, so~\eqref{3.15} forces $\widehat{g}_i(\boldsymbol{k})=0$ whenever
$k_i=0$. Hence
\[
\mathcal V_i^0 \equiv 0
\quad\Longleftrightarrow\quad
\text{condition~\eqref{3.15} restricted to the hyperplane } k_i = 0 .
\]
\end{proposition}

The reduction condition is therefore exactly one slice of the full requirement: the modes of
$g_i^0$ that carry no dependence on $x_i$. It is necessary, because those modes cannot be
reproduced by the right-hand side of~\eqref{2.13}; it is not sufficient, because it constrains the
modes with $k_i\neq0$ not at all. A phase vector that satisfies it is therefore a
\emph{candidate}, and must still be verified against~\eqref{3.16}.

The condition is nonetheless a substantial one. In every dimension $n\ge3$ treated by the cyclic
construction, $\mathcal V_i^0$ is a nontrivial function of the phases, so its vanishing is a genuine system of
equations in $\xi_1,\dots,\xi_n$: two independent equations in three dimensions, and in higher
dimensions direct expansion produces far more coefficient functions---thirty-five at $n=4$,
sixty-four at $n=5$, six hundred and sixty-four at $n=7$---though these unreduced counts need not
equal the number of independent constraints. Solving that system is the third step of the
procedure set out in Subsection~\ref{subsec:fourstep}; the fourth step is the verification
against~\eqref{3.16}, and it is there that the higher dimensions fail.

They fail in two distinct ways, separated by parity. In five and seven dimensions the phases that
satisfy the reduction condition force the two cyclic products of~\eqref{3.1} to coincide, so
$\boldsymbol{v}^0$ vanishes identically and no candidate is produced at all: the obstruction is
arithmetic, and is reached before the pressure question arises. In six and eight dimensions the
reduction system \emph{is} solved, by the symmetric assignment
$\boldsymbol{\xi}=(-\tfrac{\pi}{4},\tfrac{\pi}{4},\dots)$ exactly as at $n=4$, but the resulting
candidate fails~\eqref{3.16}: the obstruction is differential. Three and four dimensions clear both
tests. Distinguishing these two failures is what the criterion above makes possible, and
Sections~\ref{sec:5D}--\ref{sec:8D} establish each case in turn.

\medskip
\noindent\textbf{Two further decompositions.} Besides the coordinate split~\eqref{3.11} into
$\mathcal{V}^0$ and $\mathcal{W}^0$, which defines the reduction condition, two further splittings
of $\boldsymbol{g}^0$ are used below, each for one purpose.

The first is by \emph{interaction structure}. Expanded over the Fourier modes of the construction,
each term of $\boldsymbol{g}^0$ is a product of trigonometric factors; those built from two
factors form the \emph{bilinear part} $\boldsymbol{g}^{0b}$ and those built from more the
\emph{higher-order part} $\boldsymbol{g}^{0m}$,
\begin{equation}
\label{3.17}
\boldsymbol{g}^0 = \boldsymbol{g}^{0b} + \boldsymbol{g}^{0m} .
\end{equation}
Its single use is the even-dimensional dichotomy, Theorem~\ref{thm:even-dichotomy}, established in
Section~\ref{sec:6D} and applied again at $n=8$ in Section~\ref{sec:8D}: the bilinear part is a
gradient in every even dimension (Lemma~\ref{lem:even-bilinear}), so pressure integrability turns
entirely on the curl of $\boldsymbol{g}^{0m}$.

The second is by \emph{gradient compatibility}, and it is the one that carries the pressure. In
the constructions below a scalar $p^{*}$ is recovered from the components of
$-\rho\boldsymbol{g}^0$ by sequential integration, and we set
\begin{equation}
\label{eq:gGR}
\boldsymbol{g}^{0}_{G} := -\rho^{-1}\nabla p^{*},
\qquad
\boldsymbol{g}^{0}_{R} := \boldsymbol{g}^0-\boldsymbol{g}^{0}_{G},
\end{equation}
so that $\boldsymbol{g}^{0}_{G}$ is what the recovered pressure represents and
$\boldsymbol{g}^{0}_{R}$ the remainder it leaves. Where a pressure fails, that remainder becomes
the body force $\boldsymbol{f}^0=\boldsymbol{g}^{0}_{R}$ of the exact forced solutions constructed
in Sections~\ref{sec:5D}--\ref{sec:8D}.

The two splittings are related but not identical: $\boldsymbol{g}^{0b}$ is always contained in
$\boldsymbol{g}^{0}_{G}$, but $\boldsymbol{g}^{0}_{G}$ may also contain higher-order terms that
happen to be gradient-compatible, so in general
$\boldsymbol{g}^{0}_{G}\neq\boldsymbol{g}^{0b}$. Nor is $\boldsymbol{g}^{0}_{R}$ the transverse
Helmholtz component: as Section~\ref{sec:6D} shows, $p^{*}$ need not exhaust the longitudinal
content of $\boldsymbol{g}^0$, so the remainder may carry both parts.

Across the dimensions treated here the pattern is simple. At the admissible phases the reduction
always removes $\mathcal{V}^0$, so $\boldsymbol{g}^0=\mathcal{W}^0$ throughout. At $n=3,4$ the
remainder vanishes and the whole convective field is a gradient,
$\boldsymbol{g}^0=\boldsymbol{g}^{0}_{G}$. At $n=6,8$ the remainder is non-zero and is exactly what
obstructs the pressure. At $n=5,7$ no non-trivial field survives the reduction, so neither
splitting comes into play; in the forced construction of those dimensions, where the symmetric
phases are retained without imposing the reduction, nothing at all is
the sequential pressure reconstruction used here yields
$\boldsymbol{g}^{0}_{G}=\boldsymbol{0}$, and the whole convective field is retained as the body
force of the forced solution obtained there.

\subsection{The separated form of the reduction condition}
\label{subsec:separated}
When the reduction condition does carry information, its content is algebraic, and it is worth
setting out its structure before turning to the dimension-by-dimension analysis. The starting point is to
write $\mathcal{V}_1^0$ in separated form. Collecting the terms of $g_1^0$ that are independent
of $x_1$ and grouping them by their spatial Fourier mode gives
\begin{equation}
\label{3.19}
    \frac{\mathcal{V}_1^0(\boldsymbol{x}_{\setminus 1}, \boldsymbol{\xi})}{v_r^2} = \sum_{j=1}^m q_j^0(\boldsymbol{x}_{\setminus 1})\, u_j(\boldsymbol{\xi}),
\end{equation}
in which the $q_j^0(\boldsymbol{x}_{\setminus 1})$ are the distinct spatial Fourier modes
carried by the $x_1$-independent part of $g_1^0$, taken to be linearly independent and written
in real form as functions of $(x_2,\dots,x_n)$, and each $u_j(\boldsymbol{\xi})$ is the
phase-dependent amplitude of the corresponding mode. Here $m$ is the number of such independent
modes---equivalently, the dimension of the real vector space spanned by the $x_1$-independent
Fourier modes of $g_1^0$---so that $m$ counts exactly the independent constraints the reduction
condition imposes. Because the $q_j^0$ are
independent, the requirement $\mathcal V_1^0\equiv0$ holds for all $\boldsymbol{x}$ if and only
if every amplitude vanishes,
\begin{equation}
\label{3.20}
u_j(\boldsymbol{\xi}) = 0, \qquad j = 1,\dots,m .
\end{equation}
This is the algebraic system the reduction principle produces, and whether it admits a non-trivial
real solution is settled dimension by dimension in
Sections~\ref{sec:solutions-3D-4D} through~\ref{sec:8D}.

\subsection{The four-step procedure}
\label{subsec:fourstep}
\label{subsec:procedure}
We summarize the determination and verification of \( \boldsymbol{\xi} \) for an $n$-tuple-periodic solution as a four-step process.

\noindent (i) Express the inertial term $g_1^0(\boldsymbol{x}, \boldsymbol{\xi})$ using equation~\eqref{3.10}.

\noindent (ii) Decompose $g_1^0$ using equation~\eqref{3.11} into
$\mathcal{V}_1^0$ and $\mathcal{W}_1^0$, where $\mathcal{V}_1^0$ collects the
terms independent of $x_1$, and write $\mathcal{V}_1^0$ in the reduced
separated form~\eqref{3.19}. Enforcing $\mathcal V_1^0\equiv0$ then yields
the algebraic system~\eqref{3.20}, $u_j(\boldsymbol{\xi})=0$ for
$j=1,\dots,m$, where $m$ is the number of independent spatial modes.
In practice, the direct trigonometric expansion may contain a larger number
of coefficient functions associated with linearly dependent spatial terms,
as occurs in four dimensions. The size and dependency structure of the
resulting phase-angle problem therefore vary substantially with the ambient
dimension.

\noindent (iii) Solve the resulting system for $\boldsymbol{\xi}$, using
algebraic or symmetry-guided methods as appropriate to its dimension and
dependency structure, retaining only solutions for which the velocity field
is non-trivial. The phase vectors so obtained are \emph{candidates}.

\noindent (iv) \textbf{Verify each candidate against the integrability
condition~\eqref{3.16}.} By Proposition~\ref{prop:longitudinal} this is the full content
of~\eqref{2.13}, and by Proposition~\ref{prop:slice} steps (i)--(iii) have tested only the
$k_i=0$ slice of it. A candidate is a solution if and only if it passes this step. Step~(iv)
is not a formality: in six and eight dimensions every candidate produced by steps
(i)--(iii) fails it.

The procedure branches by parity of the dimension. In the solvable cases $n=3,4$ a
candidate is produced at step~(iii) and clears step~(iv). In the odd dimensions $n=5,7$
the procedure terminates at step~(iii): no phase vector satisfying the reduction system produces a
non-trivial velocity field, so no admissible candidate is generated and step~(iv) is never
reached. That is not the end of the matter, however: retaining the symmetric phases
\emph{without} imposing the reduction leaves the field non-trivial, and balancing its convective
term by a body force yields an exact forced solution in each of these dimensions, as
Sections~\ref{sec:5D} and~\ref{sec:7D} show. In the even dimensions $n=6,8$ a candidate
\emph{is} produced at step~(iii), by the symmetric assignment
$\boldsymbol{\xi}=(-\tfrac{\pi}{4},\tfrac{\pi}{4},\dots)$ exactly as in four dimensions, but
it then fails the integrability test at step~(iv). The two obstructions are thus located at
different steps---the odd one arithmetic, at step~(iii); the even one differential, at
step~(iv)---and each obstructed dimension nonetheless yields an exact \emph{forced} solution,
as shown in the corresponding sections.

Once a set of phase angles \( \boldsymbol{\xi} \) satisfying equation~\eqref{2.13} is determined --- that is, one which clears step~(iv) --- the velocity field defined by~\eqref{3.1} becomes an admissible $n$-tuple-periodic solution. The associated scalar pressure field is then obtained by integrating equation~\eqref{2.14}, which is possible precisely because the symmetry condition~\eqref{3.16} ensures that the convective field $\boldsymbol{g}^0$ admits a scalar potential. By Proposition~\ref{prop:longitudinal}, equation~\eqref{3.16} is not an additional requirement imposed alongside~\eqref{2.13}: for the periodic fields considered here, it is an equivalent expression of that condition.
The reformulation of Section~2 and the classical curl-free criterion are two expressions of a
single requisite, that the convective field admit a scalar potential. This equivalence is what
makes~\eqref{3.16} usable as the operational test, and it is the reason
condition~\eqref{3.16} plays the central role in the present work. Rather than treating it
solely as a compatibility condition for pressure reconstruction, we use it as the decisive
criterion for admissible phase angles in periodic velocity fields. As will be shown
in subsequent sections, the existence or failure of phase-angle configurations
satisfying~\eqref{3.16} governs the existence of exact unforced solutions within the cyclic
$n$-tuple-periodic family defined by~\eqref{3.1}.

In the next section we deploy this four-step procedure to develop analytical solutions. Applied in three dimensions, it constructs the admissible phase angles directly from the reduction and integrability conditions and generates not an isolated solution but a whole family of exact triple-periodic solutions, each with its pressure in closed form; the previously known solutions emerge as particular members of this family rather than being assumed at the outset. The same procedure applies unchanged in four dimensions and, again without modification, in
dimensions five through eight, where it determines with equal rigour that no unforced solution within this
cyclic family exists---though not without consolation: each obstructed dimension yields instead an exact
\emph{forced} solution, equally suited to the purpose of benchmarking.

\section{Exact Solutions in Three and Four Dimensions}
\label{sec:solutions-3D-4D}
\subsection{Triple-periodic solution in three dimensions}
\label{sec:3D}
\label{subsec:3D-triple}
In three spatial dimensions, the initial condition given by equation~\eqref{3.1} reduces to a solenoidal velocity vector field $v_i^0(\boldsymbol{x}, \boldsymbol{\xi})$ of triple-periodic form, with first component:
\begin{eqnarray}
\label{4.1}
\frac{v_1^0(\boldsymbol{x}, \boldsymbol{\xi})}{v_r}& = &\sin { \left(\alpha x_1+\xi_1 \right) } \cos {\left(\alpha x_2+\xi_2\right) }\sin { \left(\alpha x_3+\xi_3\right) }- \nonumber\\
&-&\sin  {\left( \alpha x_1+\xi_2\right) } \cos  { \left(\alpha x_3+\xi_1 \right) } \sin  {\left( \alpha x_2+\xi_3\right) }
\end{eqnarray}
and the components $v_2^0,\, v_3^0$ obtained by cyclic permutation of the spatial arguments as prescribed by~\eqref{3.1}.
Following the four-step methodology outlined above, we now apply this process to determine the phase angles for the triple-periodic solution in three spatial dimensions.
\subsubsection{Step 1: Expression of the inertial term}
For notational simplicity, we normalize the inertial term by $v_r^2$. The corresponding quantity $g_1^0(\boldsymbol{x}, \boldsymbol{\xi})/v_r^2$, obtained from equation~\eqref{3.10}, is given by:
\begin{eqnarray}
\label{4.2}
\frac{g_1^0(\boldsymbol{x}, \boldsymbol{\xi})}{v_r^2} &=& \frac{\alpha}{16}\Big[-\sin(2\alpha x_1-2\alpha x_2) -\sin(2\alpha x_1-2\alpha x_3)
\nonumber\\ && \qquad\qquad +2\,\sin(2\alpha x_1+2\xi_1) +2\,\sin(2\alpha x_1+2\xi_2)
\nonumber\\ && \qquad\qquad +2\,\sin(2\alpha x_1+2\xi_3) -\sin(2\alpha x_2+2\xi_1)
\nonumber\\ && \qquad\qquad -\sin(2\alpha x_2+2\xi_2) +\sin(2\alpha x_2+2\xi_3)
\nonumber\\ && \qquad\qquad -\sin(2\alpha x_3+2\xi_1) -\sin(2\alpha x_3+2\xi_2)
\nonumber\\ && \qquad\qquad +\sin(2\alpha x_3+2\xi_3) -\sin(2\alpha x_1-2\alpha x_2-2\xi_1+2\xi_2)
\nonumber\\ && \qquad\qquad -\sin(2\alpha x_1-2\alpha x_2-2\xi_2+2\xi_3) +\sin(2\alpha x_1-2\alpha x_2+2\xi_1-2\xi_3)
\nonumber\\ && \qquad\qquad -2\,\sin(2\alpha x_1-2\alpha x_2+2\xi_2-2\xi_3) -\sin(2\alpha x_1-2\alpha x_3-2\xi_1+2\xi_3)
\nonumber\\ && \qquad\qquad -\sin(2\alpha x_1-2\alpha x_3+2\xi_1-2\xi_2) -2\,\sin(2\alpha x_1-2\alpha x_3+2\xi_1-2\xi_3)
\nonumber\\ && \qquad\qquad +\sin(2\alpha x_1-2\alpha x_3+2\xi_2-2\xi_3) +2\,\sin(2\alpha x_1+2\xi_1+2\xi_2-2\xi_3)
\nonumber\\ && \qquad\qquad -\sin(2\alpha x_2-2\xi_1+2\xi_2+2\xi_3) -\sin(2\alpha x_2+2\xi_1-2\xi_2+2\xi_3)
\nonumber\\ && \qquad\qquad -\sin(2\alpha x_2+2\xi_1+2\xi_2-2\xi_3) -\sin(2\alpha x_3-2\xi_1+2\xi_2+2\xi_3)
\nonumber\\ && \qquad\qquad -\sin(2\alpha x_3+2\xi_1-2\xi_2+2\xi_3) -\sin(2\alpha x_3+2\xi_1+2\xi_2-2\xi_3)\Big]\nonumber\\
\end{eqnarray}
\subsubsection{Step 2: Extraction of the Reduction Term $\mathcal{V}_1^0$}

By collecting all terms in equation~\eqref{4.2} that are independent of $x_1$, we isolate:
\begin{eqnarray}
\label{4.3}
\frac{\mathcal{V}_1^0(\boldsymbol{x}_{\setminus 1}, \boldsymbol{\xi})}{v_r^2} &=& \frac{\alpha}{16}\Big[-\sin(2\alpha x_2+2\xi_1) -\sin(2\alpha x_2+2\xi_2)
\nonumber\\ && \qquad\qquad +\sin(2\alpha x_2+2\xi_3) -\sin(2\alpha x_3+2\xi_1)
\nonumber\\ && \qquad\qquad -\sin(2\alpha x_3+2\xi_2) +\sin(2\alpha x_3+2\xi_3)
\nonumber\\ && \qquad\qquad -\sin(2\alpha x_2-2\xi_1+2\xi_2+2\xi_3) -\sin(2\alpha x_2+2\xi_1-2\xi_2+2\xi_3)
\nonumber\\ && \qquad\qquad -\sin(2\alpha x_2+2\xi_1+2\xi_2-2\xi_3) -\sin(2\alpha x_3-2\xi_1+2\xi_2+2\xi_3)
\nonumber\\ && \qquad\qquad -\sin(2\alpha x_3+2\xi_1-2\xi_2+2\xi_3) -\sin(2\alpha x_3+2\xi_1+2\xi_2-2\xi_3)\Big]
\end{eqnarray}
Equation~\eqref{4.3} can be expressed in the structured form of equation~\eqref{3.19}, revealing its decomposition into a product of spatial and parametric components:
\begin{eqnarray}
\label{4.4}
\frac{\mathcal{V}_1^0(\boldsymbol{x}_{\setminus 1}, \boldsymbol{\xi})}{v_r^2} = q_1^0(\boldsymbol{x}_{\setminus 1})\, u_1(\boldsymbol{\xi})+q_2^0(\boldsymbol{x}_{\setminus 1})\, u_2(\boldsymbol{\xi})
\end{eqnarray}
where $q_1^0(\boldsymbol{x}_{\setminus 1})=\frac{\alpha}{16} \left[\sin(2\alpha x_2)+\sin(2\alpha x_3)\right]$ and $q_2^0(\boldsymbol{x}_{\setminus 1})=\frac{\alpha}{16} \left[\cos(2\alpha x_2)+\cos(2\alpha x_3)\right]$, and the scalar amplitudes $u_1(\boldsymbol{\xi})$ and $u_2(\boldsymbol{\xi})$ are
\begin{eqnarray}
\label{4.5}
u_1(\boldsymbol{\xi}) = &-&\cos(2\xi_1) - \cos(2\xi_2) - \cos(2\xi_1 - 2\xi_2 - 2\xi_3) - \cos(2\xi_1 + 2\xi_2 - 2\xi_3)+ \nonumber\\
&+&\cos(2\xi_3)- \cos(2\xi_1 - 2\xi_2 + 2\xi_3)
\end{eqnarray}
\begin{eqnarray}
\label{4.6}
u_2(\boldsymbol{\xi})= &-&\sin(2\xi_1) - \sin(2\xi_2) + \sin(2\xi_1 - 2\xi_2 - 2\xi_3) - \sin(2\xi_1 + 2\xi_2 - 2\xi_3)+ \nonumber\\
&+&\sin(2\xi_3)- \sin(2\xi_1 - 2\xi_2 + 2\xi_3)
\end{eqnarray}
Here $u_1$ and $u_2$ are the real and imaginary parts of the amplitudes carried by the surviving spatial modes. Every term of~\eqref{4.3} carries the same weight $\alpha/16$, so the two amplitudes share a single scale, which is what allows them to be made to vanish together. This separated form isolates the spatial dependence from the phase-angle dependence. To express this separation compactly, define the complex quantities
\[
Q=q_2^0+\mathrm{i}\,q_1^0,
\qquad
U=u_2+\mathrm{i}\,u_1 ,
\]
with an overbar denoting the complex conjugate.
Using the definitions of $q_1^0$ and $q_2^0$, the spatial quantity $Q$ becomes
\[
Q=\frac{\alpha}{16}\left(e^{\,\mathrm{i}2\alpha x_2}+e^{\,\mathrm{i}2\alpha x_3}\right).
\]
It follows that
\[
\frac{\mathcal{V}_1^0}{v_r^2}=q_1^0u_1+q_2^0u_2
=\operatorname{Re}\!\left(Q\,\overline{U}\right).
\]
Thus, the spatial dependence is contained in $Q$, whereas all dependence on the phase angles is contained in $U$.

The surviving modes $e^{\mathrm{i}2\alpha x_2}$ and $e^{\mathrm{i}2\alpha x_3}$ are the images,
under the cyclic shift $x_k\mapsto x_{k+1}$ of the construction~\eqref{3.1}, of a single base mode
carried around the three-cycle. Being distinct Fourier modes they are linearly independent, so
enforcing $\mathcal{V}_1^0\equiv0$ requires their phase-dependent amplitudes to vanish separately,
which is precisely the pair of conditions $u_1(\boldsymbol{\xi})=u_2(\boldsymbol{\xi})=0$. Every term of~\eqref{4.3} carries the same weight $\alpha/16$, giving a single
amplitude scale; the two reduction amplitudes $u_1(\boldsymbol{\xi})$ and $u_2(\boldsymbol{\xi})$
of~\eqref{4.5} and~\eqref{4.6} can accordingly be made to vanish together.
\subsubsection{Step 3: Solution of the underdetermined system}

Enforcing condition~\eqref{3.13}, $\mathcal{V}_1^0(\boldsymbol{x}_{\setminus 1}, \boldsymbol{\xi}) \equiv 0$, requires that both amplitudes in the separated form~\eqref{4.4} vanish. With $u_1(\boldsymbol{\xi})$ and $u_2(\boldsymbol{\xi})$ as given in~\eqref{4.5} and~\eqref{4.6}, this is the system
\begin{eqnarray}
\label{4.7}
u_1(\boldsymbol{\xi}) = 0\quad\text{and} \quad u_2(\boldsymbol{\xi})= 0
\end{eqnarray}
This system consists of two equations in three phase variables. As shown below, its common
phase-shift symmetry leaves one continuous parameter, and the remaining phase differences can be
solved exactly, giving a parametric family of solutions rather than isolated points.

\subsubsection{Step 4: The two solution families and their pressures}
The phase-angle reduction of Step~3 does not isolate a single phase vector. Because the two
equations~\eqref{4.7} constrain three phase angles, their real solutions form a one-parameter
family: fixing any one phase and solving~\eqref{4.7} for the remaining two yields, by
construction, a phase vector that satisfies the reduction condition exactly. That freedom is
not an abundance of distinct flows, however, but the translational symmetry of the problem
written in phase variables, as the classification below makes precise. Each candidate must still be tested
on the two further requirements of Section~\ref{sec:construction}: that it yield a
\emph{non-trivial} velocity field, and that it satisfy the full pressure-integrability condition
\[
\frac{\partial g_i^0}{\partial x_j}=\frac{\partial g_j^0}{\partial x_i},
\qquad i,j=1,2,3,
\]
under which, by Proposition~\ref{prop:longitudinal}, the convective field is a gradient and the
pressure follows in closed form.

\medskip
\noindent\textbf{Reduction to two invariants.} Every coordinate appears once in each product
of~\eqref{3.2a} and~\eqref{3.2b}, so adding a common $\delta$ to every $\xi_k$ replaces
$\alpha x_{\langle m\rangle}+\xi_k$ by $\alpha(x_{\langle m\rangle}+\delta/\alpha)+\xi_k$
throughout: a uniform phase shift is a rigid translation of the coordinates, and carries solutions
of~\eqref{4.7} to solutions. No generality is therefore lost by fixing $\xi_1=0$, and the invariant
content of a phase vector lies in the two differences
\[
a=\xi_2-\xi_1,\qquad b=\xi_3-\xi_2 .
\]
Writing $c_a=\cos2a$, $s_a=\sin2a$, $c_b=\cos2b$ and $s_b=\sin2b$, the system~\eqref{4.7} becomes
algebraic, and can be solved completely.

The tool we use for the exact solution, here and again in the higher dimensions treated below, is
a Gr\"obner basis, which may be unfamiliar outside computational algebra; we recall it briefly.
Given a system of polynomial equations, the set of all polynomial combinations of the left-hand
sides forms an \emph{ideal}, and every common zero of the original system is a common zero of the
whole ideal. A Gr\"obner basis is a particular finite generating set for that ideal, computed with
respect to a chosen ordering of monomials; it is the multivariate, nonlinear analogue of the
row-echelon form produced by Gaussian elimination on a linear system. With a lexicographic
monomial order the computed basis has an elimination structure that permits successive solution of
the phase relations: the generators can be ordered so that each introduces one further variable,
and the system is solved by back-substitution much as a triangular linear system is. Crucially the
computation is carried out in exact rational arithmetic, so the algebraic reduction itself
introduces no numerical approximation. Standard references are \citet{cox2015ideals}.

\begin{proposition}[Classification in three dimensions]
\label{prop:3Dclassification}
Modulo uniform phase shifts, the reduction system~\eqref{4.7}---the amplitude form taken by the
reduction condition of Corollary~\ref{cor:cyclic} in three dimensions---admits precisely three
solutions.
One is degenerate,
\[
(a,b)\equiv\left(0,\,\frac{\pi}{2}\right) \pmod\pi ,
\]
for which $\xi_2=\xi_1$ and the velocity field~\eqref{3.1} vanishes identically. The other two are
non-trivial,
\[
(a,b)=\left(\frac{\pi}{3},\,-\frac{\pi}{6}\right)
\qquad\text{and}\qquad
(a,b)=\left(-\frac{\pi}{3},\,\frac{\pi}{6}\right),
\]
which are negatives of one another.
\end{proposition}

\begin{proof}
Set $\xi_1=0$, so that $\boldsymbol{\xi}=(0,a,a+b)$. Substituting into~\eqref{4.7} and expanding the
compound angles $\cos(2a+2b)=c_ac_b-s_as_b$, $\sin(2a+2b)=s_ac_b+c_as_b$, together with
$\cos(4a+2b)=(2c_a^2-1)c_b-2s_ac_as_b$ and $\sin(4a+2b)=2s_ac_ac_b+(2c_a^2-1)s_b$, gives the pair
of polynomial equations
\begin{eqnarray}
\label{4.poly}
-2c_a^2c_b+c_ac_b+2c_as_bs_a-c_a-c_b-s_bs_a-1&=&0,\nonumber\\
-2c_a^2s_b-2c_ac_bs_a+c_as_b+c_bs_a+s_b-s_a&=&0,
\end{eqnarray}
to be solved together with $c_a^2+s_a^2=1$ and $c_b^2+s_b^2=1$. This system is zero-dimensional.
Computing a Gr\"obner basis in the lexicographic order eliminates $c_a$, $s_a$ and $c_b$ and leaves
the single univariate relation
\begin{equation}
\label{4.elim}
s_b^4\left(4s_b^2-3\right)=0 .
\end{equation}
If $s_b=0$ the second equation of~\eqref{4.poly} forces $s_a=0$, and the remaining generators then
give $c_a=1$ and $c_b=-1$, that is $a\equiv0$ and $b\equiv\tfrac{\pi}{2}$; this is the degenerate
case, in which $\xi_2=\xi_1$ and the velocity field vanishes. We note that $a\equiv0$ alone is not
sufficient for degeneracy: the vanishing requires $b\equiv\tfrac{\pi}{2}$ as well. Otherwise
$s_b^2=\tfrac{3}{4}$, and the basis determines the remaining coordinates uniquely in each case:
$(c_a,s_a,c_b,s_b)=\bigl(-\tfrac12,-\tfrac{\sqrt3}{2},\tfrac12,\tfrac{\sqrt3}{2}\bigr)$, giving
$a=-\tfrac{\pi}{3}$, $b=\tfrac{\pi}{6}$; and
$(c_a,s_a,c_b,s_b)=\bigl(-\tfrac12,\tfrac{\sqrt3}{2},\tfrac12,-\tfrac{\sqrt3}{2}\bigr)$, giving
$a=\tfrac{\pi}{3}$, $b=-\tfrac{\pi}{6}$. No other roots exist.
\end{proof}

\medskip
The scope of this statement should be read carefully. Proposition~\ref{prop:3Dclassification}
classifies the solutions of the reduction system~\eqref{4.7} \emph{within the class of fields
defined by~\eqref{3.1}}: cyclic alternating sine--cosine products carrying a single wavenumber
$\alpha$ common to every coordinate, whose only free parameters are the phases $\boldsymbol{\xi}$.
It asserts nothing about periodic solutions of other form---fields built on several wavenumbers,
superpositions of modes, or constructions with a different alternation pattern, of which the
Arnold--Beltrami--Childress flow~\eqref{1.10}--\eqref{1.13} is one example---and nothing about
solutions of~\eqref{1.1} in general. Within the stated class, however, it is exhaustive. Both
families, and the sign reversal induced by the integers $\ell_1$ and $\ell_2$, may be checked
directly against~\eqref{3.13} and~\eqref{3.16} with the script of Appendix~\ref{app:verify}.

\medskip
\noindent\textbf{The two solution families.} Restoring the phase $\theta=\xi_1$ removed by the
normalisation, and recording with integers $\ell_1,\ell_2$ the freedom to add $\pi$ to an
individual phase---which reverses the sign of the two factors carrying it, hence of
$\boldsymbol{v}^0$, while leaving the pressure unchanged---the two non-trivial points of
Proposition~\ref{prop:3Dclassification} correspond to the two families
\begin{eqnarray}
\label{4.sets}
\text{First solution family:}\quad
\boldsymbol{\xi}(\theta,\ell_1,\ell_2)&=&
\left(\theta,\;\theta-\frac{2\pi}{3}+\ell_1\pi,\;\theta+\frac{\pi}{6}+\ell_2\pi\right),
\nonumber\\
\text{Second solution family:}\quad
\boldsymbol{\xi}(\theta,\ell_1,\ell_2)&=&
\left(\theta,\;\theta+\frac{2\pi}{3}+\ell_1\pi,\;\theta+\frac{5\pi}{6}+\ell_2\pi\right),
\end{eqnarray}
with $\theta\in\mathbb{R}$ and $\ell_1,\ell_2\in\mathbb{Z}$; every phase vector
satisfying the reduction system~\eqref{4.7} non-trivially belongs to one of these two families. The parameter
$\theta$ generates the translation orbit: the members of a given family are one flow, viewed from
origins displaced by $\Delta\theta/\alpha$ along the diagonal. The second and first families correspond respectively to the positive- and negative-helicity
solutions of~\citet{ant}; the second contains the two phase vectors recorded in the literature
review.

\medskip
\noindent The velocity fields themselves follow at once: substituting either phase vector
of~\eqref{4.sets} into the construction~\eqref{3.1} and carrying the result forward in time
by~\eqref{3.5} gives $v_i(\boldsymbol{x},\boldsymbol{\xi},t)=v_i^0(\boldsymbol{x},
\boldsymbol{\xi})\,e^{-3\alpha^2\kappa t}$ for $i=1,2,3$.

\medskip
\noindent\textbf{The pressures.} A phase vector from either family satisfies the reduction
condition~\eqref{3.13} by construction, and satisfies the integrability condition~\eqref{3.16},
adapted to $n=3$, as direct substitution confirms. The convective term
$\boldsymbol{g}^0=(\boldsymbol{v}^0\bcdot\bnabla)\boldsymbol{v}^0$ is therefore the gradient of a
single-valued periodic scalar, and the momentum equation~\eqref{1.1} is satisfied with
$\bnabla p=-\rho\,\boldsymbol{g}^0e^{-6\alpha^2\kappa t}$. Integrating this along $x_1$ and
fixing the two functions of integration by the remaining two components gives the pressure in
closed form. For the first family it is
\begin{eqnarray}
\label{4.pI}
p&=&\frac{3\rho v_r^2}{16}\left[\cos\!\left(2\alpha x_1+2\theta+\frac{\pi}{3}\right)
+\cos\!\left(2\alpha x_2+2\theta+\frac{\pi}{3}\right)+\right.\nonumber\\
&&\left.\qquad\qquad+\cos\!\left(2\alpha x_3+2\theta+\frac{\pi}{3}\right)\right]
e^{-6\alpha^2\kappa t}-\nonumber\\
&-&\frac{3\rho v_r^2}{64}\left[\cos 2\alpha(x_1-x_2)+\cos 2\alpha(x_2-x_3)
+\cos 2\alpha(x_3-x_1)\right]e^{-6\alpha^2\kappa t}+\nonumber\\
&+&\frac{3\sqrt3\,\rho v_r^2}{64}\left[\sin 2\alpha(x_1-x_2)+\sin 2\alpha(x_2-x_3)
+\sin 2\alpha(x_3-x_1)\right]e^{-6\alpha^2\kappa t},
\end{eqnarray}
and for the second family
\begin{eqnarray}
\label{4.pII}
p&=&\frac{3\rho v_r^2}{16}\left[\cos\!\left(2\alpha x_1+2\theta+\frac{5\pi}{3}\right)
+\cos\!\left(2\alpha x_2+2\theta+\frac{5\pi}{3}\right)+\right.\nonumber\\
&&\left.\qquad\qquad+\cos\!\left(2\alpha x_3+2\theta+\frac{5\pi}{3}\right)\right]
e^{-6\alpha^2\kappa t}-\nonumber\\
&-&\frac{3\rho v_r^2}{64}\left[\cos 2\alpha(x_1-x_2)+\cos 2\alpha(x_2-x_3)
+\cos 2\alpha(x_3-x_1)\right]e^{-6\alpha^2\kappa t}-\nonumber\\
&-&\frac{3\sqrt3\,\rho v_r^2}{64}\left[\sin 2\alpha(x_1-x_2)+\sin 2\alpha(x_2-x_3)
+\sin 2\alpha(x_3-x_1)\right]e^{-6\alpha^2\kappa t}.
\end{eqnarray}
Both are independent of $\ell_1$ and $\ell_2$, as they must be, since those integers only reverse
the sign of $\boldsymbol{v}^0$. Within each family the difference terms in $x_1-x_2$, $x_2-x_3$ and
$x_3-x_1$ are fixed, and only the phase of the three single-frequency cosines moves, advancing by
$2\Delta\theta$---exactly what a translation by $\Delta\theta/\alpha$ does to a field of spatial
frequency $2\alpha$. A moving phase in the pressure is thus the signature of a shifted origin, not
of a new flow. Between the families the sign of the three sine terms is reversed, and it is this
that records the reflection relating them.

\medskip
\noindent\textbf{What the classification exposes.} The two flows themselves are those
of~\citet{ant}, described in Section~\ref{sec:intro}: the invariants of
Proposition~\ref{prop:3Dclassification} are exactly those of the reference phase vectors recorded
there, and the pressure~\eqref{4.pI} is the Bernoulli expression~\eqref{1.17} resolved into
harmonics. What the present formulation contributes is an account of the terrain around them.
Four features become visible that the earlier representation leaves implicit.

\emph{Completeness.} Earlier formulations leave open the possibility that additional non-trivial
coefficient choices may exist, since solving the nonlinear system for the six coefficients and
discarding degenerate cases does not by itself establish uniqueness.
Proposition~\ref{prop:3Dclassification} closes this question. In terms of the two invariants $a$
and $b$, the reduced system is zero-dimensional, and its elimination polynomial~\eqref{4.elim} has
exactly three roots. One corresponds to a degenerate case, while the remaining two generate
precisely the two families identified above. These therefore exhaust the admissible non-trivial
solutions within the class; no additional solutions exist.

\emph{The symmetry group, made explicit.} The solution set is not a pair of points but the pair of
orbits~\eqref{4.sets}, carrying a continuous parameter $\theta$ and two integers $\ell_1,\ell_2$.
The continuous parameter is the translation already identified; the integers are not, and they
account for the several apparently distinct phase vectors that recur in the literature and in
practice, every one of which proves to be a single flow up to a reversal of sign.

\emph{Which structure is fixed and which moves.} Resolving the pressure into
harmonics~\eqref{4.pI} separates the difference-frequency terms, which are invariant, from the
single-frequency terms, whose common phase advances as $2\theta$ under translation. The Bernoulli form~\eqref{1.17} conceals this, presenting the pressure as a single squared
magnitude. That form is available only because a Beltrami field carries its
convective term as the gradient of the kinetic energy; where the velocity and vorticity are not
parallel, no scalar of which the convective term is the gradient can be read off from
$\boldsymbol{v}$ in closed form, and the pressure must instead be assembled harmonic by harmonic,
as it is in every dimension treated below.

\emph{The relation between the families.} That the two admissible points are negatives of one
another is a statement about invariants, and it yields the global result established next.

\noindent\textbf{The two families are distinct, and form a chiral pair.} That the two admissible
points of Proposition~\ref{prop:3Dclassification} are negatives of one another means the flows of
the first and second families are related by the reflection $\boldsymbol{x}\mapsto-\boldsymbol{x}$,
under which $\boldsymbol{v}\mapsto-\boldsymbol{v}$ and $p(\boldsymbol{x})\mapsto p(-\boldsymbol{x})$;
this is a symmetry of~\eqref{1.1}, so each family is carried to the other. \citet{ant} records the
same opposition in a local form, observing that the two families differ only by the reversal of the
stable and unstable manifolds at their common stagnation points, the linearised matrices being
negatives of one another. The statement here is the global counterpart, and it has a consequence
that the local one does not display: the reflection is orientation-reversing, so no translation or
rotation carries one family to the other. What
separates them is the helicity $\mathcal{H}=\int_{\mathcal{C}}\boldsymbol{v}^0\bcdot
(\bnabla\times\boldsymbol{v}^0)\,\mathrm{d}V$, taken over one periodic cell
$\mathcal{C}=[0,2\pi/\alpha)^3$: a pseudoscalar, invariant under translations and rotations and
changing sign under reflection. Evaluating it on the two families gives
\begin{equation}
\label{4.helicity}
\mathcal{H}=-\,\frac{27\sqrt{3}}{4}\,\frac{\pi^3 v_r^2}{\alpha^2}\quad\text{(first family)},
\qquad
\mathcal{H}=+\,\frac{27\sqrt{3}}{4}\,\frac{\pi^3 v_r^2}{\alpha^2}\quad\text{(second family)},
\end{equation}
independently of $\theta$, $\ell_1$ and $\ell_2$, as translation invariance and the evenness of
$\mathcal{H}$ under $\boldsymbol{v}^0\mapsto-\boldsymbol{v}^0$ require. The value is non-zero and the two
differ in sign, so no \emph{orientation-preserving} rigid motion---in particular, no translation
and no proper rotation---carries a member of one family to a member of the other. They
are mirror images---a left- and a right-handed member of a single chiral pair---and in this sense
the three-dimensional construction yields two distinct solutions and no more.

\medskip
\noindent\textbf{Representatives and special cases.} The vector
\[
\boldsymbol{\xi}\!\left(\tfrac{\pi}{8},0,0\right)
=\left(\frac{\pi}{8},\,-\frac{13\pi}{24},\,\frac{7\pi}{24}\right),
\]
is produced by the construction on setting $\theta=\tfrac{\pi}{8}$ in the first family
of~\eqref{4.sets}. It is a non-trivial, curl-free field: its velocity is given by~\eqref{3.5} with
this vector substituted, and its pressure by~\eqref{4.pI} with
$2\theta+\tfrac{\pi}{3}=\tfrac{7\pi}{12}$. Every phase vector below arises in the same way, from
the classification rather than from any prior solution.

Two symmetries act within a family, and it is worth separating them. Members sharing the same pair
$(\ell_1,\ell_2)$ differ only in $\theta$, and are related by a pure translation---that is, by a
uniform shift of all three phases, which as noted above displaces the coordinates rigidly: thus
$\boldsymbol{\xi}(\tfrac{\pi}{10},0,0)=\bigl(\tfrac{\pi}{10},-\tfrac{17\pi}{30},\tfrac{4\pi}{15}\bigr)$
and
$\boldsymbol{\xi}(\tfrac{\pi}{14},0,0)=\bigl(\tfrac{\pi}{14},-\tfrac{25\pi}{42},\tfrac{5\pi}{21}\bigr)$
are the field of $\boldsymbol{\xi}(\tfrac{\pi}{8},0,0)$ viewed from origins displaced by
$-\pi/40\alpha$ and $-3\pi/56\alpha$, with pressures~\eqref{4.pI} at
$\tfrac{8\pi}{15}$ and $\tfrac{10\pi}{21}$. Members differing in $\ell_1$ or $\ell_2$ are related
by a translation \emph{composed with} the reversal
$\boldsymbol{v}^0\mapsto-\boldsymbol{v}^0$, and the difference of their phase vectors is then not
uniform. The vector
$\boldsymbol{\xi}(\tfrac{\pi}{12},1,0)=\bigl(\tfrac{\pi}{12},\tfrac{5\pi}{12},\tfrac{\pi}{4}\bigr)$
is of this kind: subtracting $\boldsymbol{\xi}(\tfrac{\pi}{8},0,0)$ from it gives
$\bigl(-\tfrac{\pi}{24},\tfrac{23\pi}{24},-\tfrac{\pi}{24}\bigr)$, a uniform
$-\tfrac{\pi}{24}$ together with an extra $\pi$ in the second phase alone. Written instead as
$\boldsymbol{\xi}(\tfrac{\pi}{12},0,0)=\bigl(\tfrac{\pi}{12},-\tfrac{7\pi}{12},\tfrac{\pi}{4}\bigr)$
the same root becomes a pure translate; the two forms differ by $\pi$ in $\xi_2$, carry velocity
fields of opposite sign, and share the pressure~\eqref{4.pI} at $\tfrac{\pi}{2}$. A non-uniform
difference between two admissible phase vectors is therefore no indication of distinct flows: it
records only which representative of the root has been written down.

The second family contains both solutions known in the literature. Taking
$\theta=-\tfrac{\pi}{3}$ in the second family gives
$\bigl(-\tfrac{\pi}{3},\tfrac{\pi}{3},\tfrac{\pi}{2}\bigr)$, the phase vector of the solution
of~\cite{tha4}; then $2\theta+\tfrac{5\pi}{3}=\pi$, and~\eqref{4.pII} reduces term by term to the
pressure~\eqref{1.21} recorded in the literature review. Taking instead
$\theta=-\tfrac{5\pi}{6}$ gives $\bigl(-\tfrac{5\pi}{6},-\tfrac{\pi}{6},0\bigr)$, the phase vector
of the solution of~\cite{ant}; then $2\theta+\tfrac{5\pi}{3}=0$, the single-frequency bracket
changes sign, and~\eqref{4.pII} gives
\begin{eqnarray}
\label{4.p4}
p&=&\frac{3\rho v_r^2}{16}\left[\cos 2\alpha x_1+\cos 2\alpha x_2+\cos 2\alpha x_3\right]e^{-6\alpha^2\kappa t}-\nonumber\\
 &-&\frac{3\rho v_r^2}{64}\left[\cos 2\alpha(x_1-x_2)+\cos 2\alpha(x_2-x_3)+\cos 2\alpha(x_3-x_1)\right]e^{-6\alpha^2\kappa t}-\nonumber\\
 &-&\frac{3\sqrt{3}\,\rho v_r^2}{64}\left[\sin 2\alpha(x_1-x_2)+\sin 2\alpha(x_2-x_3)+\sin 2\alpha(x_3-x_1)\right]e^{-6\alpha^2\kappa t}.
\end{eqnarray}
This is the Bernoulli pressure~\eqref{1.17} of~\citet{ant} resolved into harmonics, in the form
first given by~\citet{tha4}; the two differ only by an additive constant. The two known solutions are thus $\boldsymbol{\xi}(-\tfrac{\pi}{3},0,0)$ and
$\boldsymbol{\xi}(-\tfrac{5\pi}{6},0,0)$ of the second family. They share the pair
$(\ell_1,\ell_2)=(0,0)$, so here the difference \emph{is} uniform---a shift of
$\Delta\theta=-\tfrac{\pi}{2}$ in every phase, the shift noted in Section~\ref{sec:intro} when the
two literature fields were compared---and~\eqref{1.21} is the special case $\theta=-\tfrac{\pi}{3}$
of the single expression~\eqref{4.pII}.

We now verify these identifications by deriving the two literature phase vectors directly, fixing
$\xi_3$ rather than $\xi_1$.

\medskip
\noindent\textbf{First known solution ($\xi_3=\tfrac{\pi}{2}$).} Substituting $\xi_3=\tfrac{\pi}{2}$
into~\eqref{4.7} gives
\begin{eqnarray}
\label{4.8}
u_1(\boldsymbol{\xi}) &=& -\cos(2\xi_1) - \cos(2\xi_2) + 2\cos(2\xi_1 - 2\xi_2) + \cos(2\xi_1 + 2\xi_2) - 1 = 0
\end{eqnarray}
\begin{eqnarray}
\label{4.9}
u_2(\boldsymbol{\xi}) &=& -4\sin(\xi_1)\,\sin(\xi_2)\,\sin(\xi_1 + \xi_2) = 0 .
\end{eqnarray}
The factored form of~\eqref{4.9} vanishes when $\sin(\xi_1+\xi_2)=0$; taking $\xi_1 = -\xi_2$
and substituting into~\eqref{4.8} reduces it to
\[
2\left[\cos(4\xi_2) - \cos(2\xi_2)\right] = 0 ,
\]
whose non-trivial root is $\xi_2 = \tfrac{\pi}{3}$, giving $\xi_1 = -\tfrac{\pi}{3}$. Thus
\[
\boldsymbol{\xi}\!\left(-\tfrac{\pi}{3},0,0\right)
=\left(-\frac{\pi}{3},\,\frac{\pi}{3},\,\frac{\pi}{2}\right),
\]
which is the phase vector of the three-dimensional solution of \cite{tha4}, whose velocity
field~\eqref{1.18}--\eqref{1.20} and scalar pressure~\eqref{1.21} are as displayed in the
literature review. In the invariants of the classification above this vector gives
$(a,b)=\bigl(\tfrac{2\pi}{3},\tfrac{\pi}{6}\bigr)$, the second of the two admissible points.

\medskip
\noindent\textbf{Second known solution ($\xi_3=0$).} Fixing instead $\xi_3 = 0$ in~\eqref{4.7} gives
\begin{eqnarray}
\label{4.10}
u_1(\boldsymbol{\xi}) &=& -\cos(2\xi_1) - \cos(2\xi_2) - 2\cos(2\xi_1 - 2\xi_2) - \cos(2\xi_1 + 2\xi_2) + 1 = 0, \nonumber\\
u_2(\boldsymbol{\xi}) &=& -\sin(2\xi_1) - \sin(2\xi_2) - \sin(2\xi_1 + 2\xi_2) = 0 .
\end{eqnarray}
Writing the second as $2\sin(\xi_1+\xi_2)\left[\cos(\xi_1-\xi_2)+\cos(\xi_1+\xi_2)\right]=0$
and taking the branch $\xi_1+\xi_2=-\pi$, so that $\xi_1 = -\pi-\xi_2$, the first reduces to
$\cos(2\xi_2)+\cos(4\xi_2)=0$. With $c=\cos(2\xi_2)$ this is $2c^2+c-1=(2c-1)(c+1)=0$, whose
root $c=\tfrac{1}{2}$ gives $\xi_2 = -\tfrac{\pi}{6}$ and $\xi_1 = -\tfrac{5\pi}{6}$. Thus
\[
\boldsymbol{\xi}\!\left(-\tfrac{5\pi}{6},0,0\right)
=\left(-\frac{5\pi}{6},\,-\frac{\pi}{6},\,0\right),
\]
the phase vector of the solution of \cite{ant}, with velocity field~\eqref{1.14}--\eqref{1.16}.
This vector differs from the preceding one by $-\tfrac{\pi}{2}$ in every component---the uniform
phase shift already noted in Section~\ref{sec:intro} when the two literature fields were
compared---so by the argument given above the two known solutions are one flow, the second being
the first translated by $\pi/2\alpha$ along the diagonal. Both reduce to
$(a,b)=\bigl(\tfrac{2\pi}{3},\tfrac{\pi}{6}\bigr)$, and its pressure is~\eqref{4.p4}, obtained
above from~\eqref{4.pII}.

The relation between the two known pressures illustrates the same mechanism once more. The
uniform shift $\delta=-\tfrac{\pi}{2}$ advances the single-frequency phase by
$2\delta=-\pi$, which simply reverses the sign of the first bracket: \eqref{1.21} carries
$-\tfrac{3}{16}\rho v_r^2\sum_i\cos2\alpha x_i$ where~\eqref{4.p4} carries
$+\tfrac{3}{16}\rho v_r^2\sum_i\cos2\alpha x_i$, and the difference-frequency parts of the two
are identical. The two expressions therefore look different while describing one pressure field
viewed from origins a distance $\pi/2\alpha$ apart; both are instances of the single
expression~\eqref{4.pII}. The sine difference terms carry
$-\tfrac{3\sqrt3}{64}\rho v_r^2$ throughout the second family but
$+\tfrac{3\sqrt3}{64}\rho v_r^2$ throughout the first, as~\eqref{4.pI}
and~\eqref{4.pII} show: that sign is how the reflection relating the two families shows itself in
the pressure.

That the systematic construction recovers both known families from the reduction and integrability
conditions---rather than presupposing either---is the essential feature of the method, and the same procedure
is carried unchanged into dimensions five through eight.

\medskip
For every admissible phase vector, then, $\boldsymbol{g}^0$ is curl-free and admits a scalar
potential; by Proposition~\ref{prop:longitudinal} the integral condition~\eqref{2.13} is
satisfied, and substituting any such phase vector into equation~\eqref{3.1} yields an exact
triple-periodic solution of the Navier--Stokes equations~\eqref{1.1}.

It is worth situating this construction against the classical Arnold--Beltrami--Childress
flow~\eqref{1.10}--\eqref{1.13}. Both constructions yield Beltrami fields at their admissible
parameters, but they reach them from different ansatz classes. In the ABC flow the Beltrami
structure is built into the velocity representation from the outset, so its convective term is a
gradient and its pressure follows at once, with no phase-angle condition to satisfy. In the present
construction nothing of the kind is imposed: for a general phase vector the field is not Beltrami
and a single-valued pressure need not exist, and the admissible phases are fixed precisely by the
requirement that the convective field admit one. That the surviving members turn out to be
Beltrami, with $\boldsymbol{\omega}=\pm\sqrt3\,\alpha\boldsymbol{v}$, is a conclusion of the
classification rather than a hypothesis of it.
\subsection{Quadruple-periodic family of solutions in four dimensions}
\label{sec:4D}
In four spatial dimensions, the initial condition given by equation~\eqref{3.1} reduces to a solenoidal velocity vector field \( v_i^0(\boldsymbol{x}, \boldsymbol{\xi}) \) of quadruple-periodic form, with first component:
\begin{eqnarray}
\label{4.11}
\frac{v_1^0(\boldsymbol{x}, \boldsymbol{\xi})}{v_r} &=& \sin { \left(\alpha x_1+\xi_1 \right) } \cos {\left(\alpha x_2+\xi_2\right) }\sin { \left(\alpha x_3+\xi_3\right) } \cos { \left(\alpha x_4+\xi_4 \right) }- \nonumber\\
&-&\sin  {\left( \alpha x_1+\xi_2\right) }\sin  {\left( \alpha x_2+\xi_3\right) } \cos  { \left(\alpha x_3+\xi_4 \right) } \cos  {\left( \alpha x_4+\xi_1\right) }
\end{eqnarray}
and the components $v_2^0,\, v_3^0,\, v_4^0$ obtained by cyclic permutation of the spatial arguments as prescribed by~\eqref{3.1}.
\subsubsection{Structure of the Four-Dimensional Reduction}
In contrast to the three-dimensional case, the unreduced decomposition of
\(\mathcal{V}_1^0\) in four dimensions produces 35 phase-dependent coefficient functions (\iffullAppendixA Appendix~\ref{app:4D}\fi)
for the four unknown phase angles. Whether these can be made to cancel simultaneously is not
guaranteed and provides an early indication that the admissibility of periodic solutions may depend
strongly on the spatial dimension.

The existence of a non-trivial real solution satisfying this reduction condition is
one of the features that distinguishes the different dimensions. In five and
seven dimensions, the Gr\"{o}bner-basis analysis presented in
Sections~\ref{sec:5D} and~\ref{sec:7D} reduces the corresponding constraint systems to an
algebraic contradiction. In four dimensions, by contrast, the system admits
a solution, although this solution is necessarily highly constrained.

These 35 coefficient functions are not independent; the associated spatial terms contain linear
dependencies, so the condition $\mathcal{V}_1^0\equiv0$ need not require all 35 coefficients to
vanish separately. A detailed examination reveals an underlying phase-angle structure that permits
their collective cancellation. The resulting solution is exhibited in Step~3 below, and its
existence indicates the presence of a hidden symmetry in the quadruple-periodic velocity
construction.

Following the four-step methodology outlined in
Section~\ref{sec:construction}, we now determine the phase angles for the
quadruple-periodic family of solutions in four spatial dimensions. Whereas the
three-dimensional construction leads to a compact underdetermined system, the
four-dimensional problem produces a substantially larger reduction structure with
significant internal dependencies.
\subsubsection{Step 1: Expression of the inertial term}
The corresponding inertial term $g_1^0(\boldsymbol{x}, \boldsymbol{\xi})$, obtained from equation~\eqref{3.10} using the four-dimensional velocity components~\eqref{4.11}, yields an expression containing 183 distinct trigonometric terms. While the complete expression, given in Appendix~\ref{app:4D}, is extensive, the key insight lies in its subsequent decomposition rather than its explicit form.
\subsubsection{Step 2: Extraction of the Reduction Term $\mathcal{V}_1^0$}
By collecting all terms in $g_1^0$ that are independent of $x_1$, we isolate
$\mathcal{V}_1^0(\boldsymbol{x}_{\setminus 1}, \boldsymbol{\xi})$. Expanded directly, this function
takes the form
\begin{eqnarray}
\label{4.12}
\frac{\mathcal{V}_1^0(\boldsymbol{x}_{\setminus 1}, \boldsymbol{\xi})}{v_r^2} = \sum_{j=1}^{35}
q_j^0(\boldsymbol{x}_{\setminus 1}) \cdot u_j(\boldsymbol{\xi})
\end{eqnarray}
where the spatial functions $q_j^0(\boldsymbol{x}_{\setminus 1})$ involve
combinations of trigonometric functions in $(x_2, x_3, x_4)$, and the phase-dependent
functions $u_j(\boldsymbol{\xi})$ contain various combinations of the phase
angles $\xi_1,\ldots,\xi_4$.

It is important to note that the $35$-term representation in~\eqref{4.12} is the direct
trigonometric expansion obtained from $\mathcal{V}_1^0$; the spatial functions $q_j^0$ appearing
in this unreduced representation do not constitute a linearly independent Fourier basis.
Consequently, the vanishing of every coefficient function $u_j$ is sufficient, but not necessary,
for $\mathcal{V}_1^0\equiv0$: upon expansion into elementary Fourier modes, contributions from
different terms may combine and cancel. When identical Fourier modes are collected and a linearly
independent spatial basis is chosen, the resulting reduced representation has the form described
in~\eqref{3.19}.
\subsubsection{Step 3: Solution of the reduction condition}
Enforcing the condition $\mathcal{V}_1^0(\boldsymbol{x}_{\setminus 1}, \boldsymbol{\xi}) \equiv 0$
means requiring the separated sum~\eqref{4.12} to vanish identically in $\boldsymbol{x}$. One way
to guarantee this is to make every coefficient function vanish,
$$u_j(\xi_1, \xi_2, \xi_3, \xi_4) = 0, \qquad j = 1, 2, \ldots, 35,$$
a system of 35 nonlinear equations in four unknowns. This is sufficient but, as
Proposition~\ref{prop:4Dreduction} will show, more than is necessary: at the solution the sum
vanishes because twenty-one coefficient functions vanish and the remaining fourteen cancel collectively
against their spatial coefficients. We first record the system, and then exhibit the
configuration that annihilates the sum.

The coefficient functions $u_j(\boldsymbol{\xi})$ contain trigonometric terms of varying complexity:
\begin{itemize}
\item Single-phase angle terms: $\sin(2\xi_i)$, $\cos(2\xi_i)$
\item Two-phase angle combinations: $\sin(2\xi_i \pm 2\xi_j)$, $\cos(2\xi_i \pm 2\xi_j)$
\item Multi-phase angle combinations: $\sin(2\xi_i \pm 2\xi_j \pm 2\xi_k \pm 2\xi_l)$, $\cos(2\xi_i \pm 2\xi_j \pm 2\xi_k \pm 2\xi_l)$
\end{itemize}

The complete 183-term expression for $g_1^0(\boldsymbol{x}, \boldsymbol{\xi})$ and the derivation of all 35 coefficient functions are given
\iffullAppendixA in Appendix~\ref{app:4D}\fi.

As in three dimensions, the reduction condition does not isolate a single phase vector, and as
there the reason is that a uniform shift of all four phases is a rigid translation of the
coordinates. Setting $\xi_1=\theta$ and passing to the differences
\[
a=\xi_2-\xi_1,\qquad b=\xi_3-\xi_2,\qquad c=\xi_4-\xi_3,
\]
which the shift leaves fixed, the system~\eqref{4.12} can be solved completely. Its solution set is
not a finite collection of points, as in three dimensions, but a curve.

\begin{proposition}[Reduction condition in four dimensions]
\label{prop:4Dreduction}
The separated sum~\eqref{4.12} vanishes identically in $\boldsymbol{x}$ precisely on the
two-parameter set
\begin{equation}
\label{4.4Dset}
\boldsymbol{\xi}(\theta,a,\ell_1,\ell_2,\ell_3)=
\left(\theta,\;\theta+a+\ell_1\pi,\;\theta+\frac{\pi}{2}+\ell_2\pi,\;\theta+a+\ell_3\pi\right),
\end{equation}
with $\theta,a\in\mathbb{R}$ and $\ell_1,\ell_2,\ell_3\in\mathbb{Z}$; equivalently, on
$b\equiv\tfrac{\pi}{2}-a$ and $c\equiv a-\tfrac{\pi}{2}$ modulo $\pi$. The parameter $\theta$
generates the translation orbit and the integers $\ell_1,\ell_2,\ell_3$ reverse the sign of
$\boldsymbol{v}^0$ according to the parity of $\ell_1+\ell_2+\ell_3$, so that the only parameter
carrying genuinely different fields is $a$.
\end{proposition}

\noindent At the solution the sum vanishes not because every coefficient function vanishes but
because twenty-one of them are annihilated while the remaining fourteen cancel collectively
against their spatial coefficients; by Corollary~\ref{cor:cyclic}, $\mathcal{V}_i^0\equiv0$ for all
$i=1,\ldots,4$.

Four dimensions therefore differs from three at this stage in a way worth stating plainly. In
three dimensions the reduction condition already pinned the phase differences down to isolated
values, and both survived the tests that followed. Here it leaves a whole curve open, and the
question of which members are exact solutions is decided entirely at the next step.

\subsubsection{Step 4: Pressure-integrability selects a single solution}
Clearing the reduction condition is only step~(iii) of the procedure; a phase vector must also
satisfy the full pressure-integrability condition, which in four dimensions is the six mixed-partial
relations
\[
\frac{\partial g_i^0}{\partial x_j}=\frac{\partial g_j^0}{\partial x_i},
\qquad 1\le i<j\le 4.
\]
Substituting the set~\eqref{4.4Dset} into the convective field and evaluating these relations shows
that they are independent of $\theta$ and of $\ell_1,\ell_2,\ell_3$---as they must be, since those
parameters only translate the field or reverse its sign---and depend on $a$ alone. The dependence
can be exhibited in closed form. Expanding the mixed partials over the spatial Fourier modes, every
non-vanishing mode amplitude is, up to a phase factor of unit modulus, a constant multiple of
\[
e^{6\mathrm{i}a}-e^{4\mathrm{i}a}-e^{2\mathrm{i}a}+1
=\bigl(e^{2\mathrm{i}a}-1\bigr)^{2}\bigl(e^{2\mathrm{i}a}+1\bigr)
=-8\,\mathrm{e}^{3\mathrm{i}a}\sin^{2}\!a\,\cos a .
\]
The integrability conditions therefore hold if and only if
\begin{equation}
\label{4.4Dfactor}
\sin^{2}\!a\,\cos a = 0 ,
\end{equation}
that is, at exactly two values of $a$ in a period: $a\equiv0$ and $a\equiv\tfrac{\pi}{2}$ modulo
$\pi$. At $a\equiv0$ the velocity field~\eqref{3.1} vanishes identically and the solution is
degenerate; at $a\equiv\tfrac{\pi}{2}$ the conditions hold while the field remains non-trivial. No
other value of $a$ clears the test, and~\eqref{4.4Dfactor} establishes this for the whole curve
rather than at sampled points.

The curve of Proposition~\ref{prop:4Dreduction} therefore collapses to the single family
\begin{equation}
\label{4.4Dsol}
\boldsymbol{\xi}(\theta,\ell_1,\ell_2,\ell_3)=
\left(\theta,\;\theta+\frac{\pi}{2}+\ell_1\pi,\;\theta+\frac{\pi}{2}+\ell_2\pi,\;
\theta+\frac{\pi}{2}+\ell_3\pi\right),
\end{equation}
and by Proposition~\ref{prop:longitudinal} every member of it satisfies the pressure-integrability
condition in full. The selection can also be seen numerically: the script of Appendix~\ref{app:verify} evaluates both
conditions at sample points along the curve and finds the reduction condition satisfied at each,
while the integrability condition holds only at $a=\tfrac{\pi}{2}$---a check consistent
with~\eqref{4.4Dfactor}, though it is the factorisation and not the sampling that proves the
claim. The quadruple-periodic solution recorded in the literature review of
Section~\ref{sec:intro}, $\left(-\tfrac{\pi}{4},\tfrac{\pi}{4},\tfrac{\pi}{4},
\tfrac{\pi}{4}\right)$, is the member $\theta=-\tfrac{\pi}{4}$ with $\ell_1=\ell_2=\ell_3=0$.

Four dimensions thus supplies the concrete instance of the statement that the reduction condition
is necessary but not sufficient. In four dimensions a whole curve of phase vectors annihilates the reduction term, and
integrability discards all of it but one point. The quarter-period offset
$\xi_2=\xi_3=\xi_4=\xi_1+\tfrac{\pi}{2}$ is thus not an ansatz adopted for convenience: it is
forced.

Because $\boldsymbol{g}^0$ is a gradient for every member of~\eqref{4.4Dsol}, the pressure is
obtained, up to an arbitrary additive constant, by integrating
$\partial p/\partial x_i=-\rho\,g_i^0$. Carrying out the integration in closed form gives
\begin{eqnarray}
\label{4.p4d}
p&=&\frac{\rho v_r^2}{8}\left[\cos 2\alpha(x_1-x_3)+\cos 2\alpha(x_2-x_4)\right]e^{-8\alpha^2\kappa t}+\nonumber\\
 &&+\frac{\rho v_r^2}{8}\left[\cos\!\left(2\alpha(x_1+x_3)+4\theta\right)+\cos\!\left(2\alpha(x_2+x_4)+4\theta\right)\right]e^{-8\alpha^2\kappa t}.
\end{eqnarray}
The structure repeats what was seen in three dimensions: the difference-frequency terms
$\cos 2\alpha(x_1-x_3)$ and $\cos 2\alpha(x_2-x_4)$ are fixed, while the sum-frequency terms carry
a phase $4\theta$ that advances with the translation parameter. The pressure is independent of
$\ell_1,\ell_2,\ell_3$, since reversing the sign of $\boldsymbol{v}^0$ leaves the quadratic
convective term unchanged. The velocity decays as $e^{-4\alpha^2\kappa t}$ and the pressure as
$e^{-8\alpha^2\kappa t}$, the factors $n=4$ and $2n=8$ characteristic of the $n$-fold
construction.

The members of~\eqref{4.4Dsol} are not distinct flows. As in three dimensions, $\theta$ displaces
the origin and the integers reverse the sign of the velocity, both symmetries of~\eqref{1.1}; the
family is one flow, and the advancing phase $4\theta$ in~\eqref{4.p4d} records the shifted origin.
The solution recorded in Section~\ref{sec:intro} is recovered at $\theta=-\tfrac{\pi}{4}$, where
$4\theta=-\pi$ and the sum-frequency terms change sign; at that value~\eqref{4.p4d} reduces, by the
product-to-sum identities, to the form~\eqref{1.26} quoted there, the two expressions being
identical and not merely equal up to a constant.

Unlike three dimensions, four dimensions has no chiral partner. The point inversion
$\boldsymbol{x}\mapsto-\boldsymbol{x}$ has determinant $(-1)^n$, so it is orientation-reversing in
three dimensions but orientation-preserving in four, where $\det(-I)=(-1)^4=+1$. After centring a
member of the translation family at its natural origin, central inversion therefore does not
generate an oppositely handed partner of the kind found in three dimensions. A genuinely
orientation-reversing map, such as the reflection of a single coordinate, does not preserve the
cyclic form~\eqref{3.1} and carries the field outside the class. The particular chiral pairing
generated by point inversion in three dimensions therefore has no direct analogue under central
inversion in four.

The reason four dimensions admits an exact solution whereas six and eight do
not is not visible at the reduction stage---all three even dimensions satisfy the reduction
condition at their symmetric phases---but emerges
only at the integrability test of step~(iv). At the admissible four-dimensional phases the convective field $\boldsymbol{g}^0$ collapses to a
pure gradient and the higher-degree component $\boldsymbol{g}^{0m}$ vanishes identically, so
a single-valued pressure exists; in six and eight dimensions a non-gradient multilinear remainder
survives, so no scalar pressure can absorb the full convective field in the unforced equations.
This is the even-dimensional dichotomy established in
Section~\ref{sec:6D} (Theorem~\ref{thm:even-dichotomy}), of which four dimensions is the
unique solvable case among the even dimensions considered here.
\subsubsection{Four-dimensional unit-periodic cyclic flow}
The simplest alternative to the SCSC ansatz is a four-dimensional analogue of the classical
unit-periodic cyclic flow, the three-dimensional form of which was recalled in the introduction as
equations~\eqref{1.10}--\eqref{1.13}. Whereas the periodic construction of this paper
builds each velocity component from a \emph{product} of trigonometric factors, the unit-periodic
family builds it from a \emph{sum}; the two furnish structurally different routes to a
divergence-free field, and it is natural to ask whether the second survives in higher dimensions
alongside the first. As in three dimensions, a unit-periodic velocity field is constructed so that
each component $v_k$ is independent of its own coordinate $x_k$; this guarantees
$\partial v_k/\partial x_k = 0$ and hence $\nabla\cdot\boldsymbol{v}=0$ identically, with no
restriction on the amplitudes. Each component may therefore involve only the remaining three
coordinates $x_{k+1},x_{k+2},x_{k+3}$ (indices modulo $4$), with one Fourier mode contributed by each coordinate:
\begin{equation}
\label{4.13}
v_k = A\sin(\alpha x_{k+1}) + B\cos(\alpha x_{k+2}) + C\sin(\alpha x_{k+3}),
\qquad k = 1,\ldots,4,
\end{equation}
with constant amplitudes $A,B,C$ to be fixed by the pressure-integrability
condition~\eqref{3.16}. Computing the convective field
$g_i = \sum_j v_j\,\partial v_i/\partial x_j$ and collecting the Fourier coefficients of the
residuals $\partial_j g_i - \partial_i g_j$ yields the polynomial system $A^2 = C^2 = 0$
together with $AB = BC = 0$, which forces
\begin{equation}
\label{4.14}
A = C = 0
\end{equation}
over $\mathbb{R}$. The amplitude $B$, however, is left unconstrained. This is a structural
feature of the even dimension: the four-cycle splits into the antipodal coordinate pairs
$(x_1,x_3)$ and $(x_2,x_4)$, so that the $B$-terms separate into two decoupled
coordinate-pair subsystems, each integrable on its own. A nontrivial one-parameter flow therefore
persists,
\begin{equation}
\label{4.15}
v_k = B\cos(\alpha x_{k+2}), \qquad\qquad k=1,\dots,4
\end{equation}
This field is an exact steady solution of the Euler equations. The corresponding convective field is a pure gradient, $\boldsymbol{v}\cdot\nabla\boldsymbol{v} = \nabla\Phi$ with
$\Phi = -B^2\bigl[\sin(\alpha x_1)\sin(\alpha x_3)
+ \sin(\alpha x_2)\sin(\alpha x_4)\bigr]$,
so a single-valued pressure exists and is obtained, up to an arbitrary additive constant, by integrating
$\partial p/\partial x_i = -\rho\, g_i$:
\begin{equation}
\label{4.16}
p(\boldsymbol{x}) = \rho B^2\!\left[\sin(\alpha x_1)\sin(\alpha x_3)
+ \sin(\alpha x_2)\sin(\alpha x_4)\right]
\end{equation}
For the viscous problem the same field decays diffusively: since every mode satisfies
$\Delta v_k = -\alpha^2 v_k$, the time-dependent solution is
\begin{equation}
\label{4.17}
v_k(\boldsymbol{x},t) = B\,\cos(\alpha x_{k+2})e^{-\alpha^2\kappa t}
\end{equation}
with pressure
\begin{equation}
\label{4.18}
p(\boldsymbol{x},t) = \rho B^2 \!\left[\sin(\alpha x_1)\sin(\alpha x_3)
+ \sin(\alpha x_2)\sin(\alpha x_4)\right]e^{-2\alpha^2\kappa t}
\end{equation}
The four-dimensional unit-periodic ansatz therefore furnishes a second family of
exact four-dimensional flows, alongside the SCSC family constructed above. Although the SCSC and
unit-periodic constructions both yield exact four-dimensional solutions, they do so for
fundamentally different reasons. The SCSC family
relies on a delicate cancellation among many interacting Fourier modes,
whereas the unit-periodic solution survives because its dynamics decouple into two
independent two-dimensional subsystems associated with the coordinate pairs
$(x_1,x_3)$ and $(x_2,x_4)$. This contrast illustrates that exact solutions
in higher dimensions need not arise from a unique algebraic mechanism.

Its behaviour in higher dimensions is taken up in Sections~\ref{sec:5D} and~\ref{sec:7D}.

\subsection{Comparison of the Three- and Four-Dimensional Cases}
The application of the phase-angle methodology to the three- and four-dimensional Navier--Stokes
equations reveals two markedly different constraint structures. In three dimensions the reduction
leads directly to two independent phase conditions in three unknown phase angles. In four
dimensions the direct trigonometric expansion produces 35 phase-dependent coefficient functions in
four unknown phase angles; these are not all independent, and the four-dimensional cancellation
exploits linear dependencies among their associated spatial terms. The three-dimensional reduction
is thus compact and can be solved directly, whereas the four-dimensional expansion is considerably
larger and carries substantial internal dependencies.

The cancellation mechanisms differ accordingly. In three dimensions the two independent reduction
amplitudes can be set to zero simultaneously. In four dimensions, at the quarter-period phases,
twenty-one coefficient functions in the direct expansion vanish individually while the spatial
contributions associated with the remaining fourteen cancel collectively, so that the separated
sum~\eqref{4.12} vanishes as a whole. The decisive issue is therefore not the number of coefficient
functions but the structure relating them: although the four-dimensional expansion is much the
larger, its internal dependencies allow the reduction condition to be met without every
coefficient vanishing.

A marked contrast, however, lies in \emph{where} the phases are determined. Both dimensions are
invariant under the same symmetries---uniform phase shifts, which translate the coordinates
rigidly, and the addition of $\pi$ to an individual phase, which reverses the sign of
$\boldsymbol{v}^0$---so in each case the meaningful content of a phase vector lies in the
differences between phases. In three dimensions the reduction condition alone pins those
differences down: by Proposition~\ref{prop:3Dclassification} it admits, besides a degenerate
point, exactly two non-trivial solutions, and both of them go on to satisfy the
pressure-integrability condition~\eqref{3.16}. In four dimensions the reduction condition is far
less restrictive: by Proposition~\ref{prop:4Dreduction} it is met along an entire one-parameter
curve of phase differences, and it is the integrability condition that reduces the curve to the
single non-trivial family~\eqref{4.4Dsol}. The two dimensions are therefore both rigid, but the
rigidity is imposed at different stages---by the reduction in three dimensions, by integrability in
four---and four dimensions supplies the concrete demonstration that the reduction condition is
necessary and not sufficient.

They differ once more under reflection, and for a reason that is arithmetic rather than
hydrodynamic: central inversion $\boldsymbol{x}\mapsto-\boldsymbol{x}$ has determinant $(-1)^n$, so
it reverses orientation in three dimensions but preserves it in four. This explains why the
particular reflection relation of the present construction produces a chiral pair in the former but
not in the latter.

What the two dimensions ultimately share is the outcome: in each case the surviving phase vectors
satisfy the full pressure-integrability condition~\eqref{3.16}, the convective field reduces to a
pure gradient, and the pressure absorbs it entirely. That both the compact three-dimensional
reduction and the considerably larger four-dimensional expansion should close in precisely this way
naturally raises the question of whether the same structure persists in higher dimensions. The next
section provides the first---and revealing---answer: in five dimensions, it does not.
\section{Quintuple-periodic constructions in five dimensions: the first odd-dimensional
obstruction}
\label{sec:5D}
In five spatial dimensions, the general ansatz~\eqref{3.1} reduces to a solenoidal velocity
vector field $v_i^0(\boldsymbol{x}, \boldsymbol{\xi})$ of quintuple-periodic form. Writing
\[
v_i^0 = v_r\bigl(\mathcal{T}1_i^0 - \mathcal{T}2_i^0\bigr),
\]
where $\mathcal{T}1_i^0$ and $\mathcal{T}2_i^0$ denote the first and second product terms
respectively in~\eqref{3.1}, the first component is:
\begin{eqnarray}
\label{5.1}
\frac{v_1^0(\boldsymbol{x}, \boldsymbol{\xi})}{v_r}&=&
\sin { \left(\alpha x_1+\xi_1 \right) } \cos {\left(\alpha x_2+\xi_2\right) }
\sin { \left(\alpha x_3+\xi_3\right) } \cos { \left(\alpha x_4+\xi_4 \right) }
\sin {\left(\alpha x_5+\xi_5\right) }- \nonumber\\
&-&\sin  {\left( \alpha x_1+\xi_2\right) } \cos  { \left(\alpha x_3+\xi_4 \right) }
\sin  {\left( \alpha x_2+\xi_3\right) }\sin  {\left( \alpha x_4+\xi_5\right) }
\cos  { \left(\alpha x_5+\xi_1 \right) }
\end{eqnarray}
and the components $v_2^0,\dots,v_5^0$ obtained by cyclic permutation of the spatial arguments as prescribed by~\eqref{3.1}.
The objective is to find a set of phase angles $(\xi_1, \xi_2, \xi_3, \xi_4, \xi_5)$ that satisfies condition~\eqref{3.13} while yielding a non-trivial velocity field, that is, one for which $\mathcal{T}1_i^0 \not\equiv \mathcal{T}2_i^0$. Such a set would provide a candidate phase vector to be tested subsequently against the full pressure-integrability condition~\eqref{3.16}. As we now prove, however, no non-trivial phase vector even satisfies the reduction condition.
\subsection{The constraint system and initial investigation}
Following Step~2 of the methodology outlined in Section~3, the five-dimensional velocity
field~\eqref{5.1} generates a decomposition of $g_1^0$ into $x_1$-independent
and $x_1$-dependent parts according to equation~\eqref{3.13}. The $x_1$-independent
part $\mathcal{V}_1^0$ takes the form:
\begin{equation}
\label{5.2}
\frac{\mathcal{V}_1^0(\boldsymbol{x}_{\setminus 1}, \boldsymbol{\xi})}{v_r^2} = \sum_{j=1}^{64}
q_j^0(\boldsymbol{x}_{\setminus 1}) \cdot u_j(\boldsymbol{\xi}),
\end{equation}
where the spatial coefficient functions $q_j^0(\boldsymbol{x}_{\setminus 1})$ involve
trigonometric combinations in $(x_2, x_3, x_4, x_5)$, and the phase-dependent functions
$u_j(\boldsymbol{\xi})$ contain combinations of\\ $(\xi_1, \xi_2, \xi_3, \xi_4, \xi_5)$.
Equation~\eqref{5.2} is reproduced as~\eqref{B1} in Appendix~\ref{app:5D}, where the complete
set of $64$ pairs $\{q_j^0,\,u_j\}$ is recorded in the supplementary material;
each spatial coefficient carries the scale $\alpha/256$, and each $u_j$ is a sum of up to thirty sines or cosines of integer combinations of the five phases. Enforcing $\mathcal{V}_1^0 \equiv 0$ makes
the separated sum~\eqref{5.2} vanish identically; because the $q_j^0$ are linearly independent as
functions of $(x_2,\dots,x_5)$, this holds if and only if every amplitude vanishes,
$u_j(\boldsymbol{\xi}) = 0$ for $j=1,\dots,64$, an overdetermined system of $64$ nonlinear
equations in $5$ unknowns.

\noindent \textbf{Dimensional context.} The degree of over-determination of the reduction system grows sharply with dimension, and it is instructive to trace this progression:
\begin{itemize}
    \item \textbf{3D:} two equations in three unknowns leave the system formally
    underdetermined, but the remaining degree of freedom is only a uniform phase shift,
    corresponding to a translation of the spatial coordinates and therefore generating no
    genuinely new solution. Once that translational freedom and the symmetries of the
    construction are factored out, the admissible solutions are isolated.
    Proposition~\ref{prop:3Dclassification} establishes that there are exactly two non-trivial
    families, related by mirror symmetry, and that these exhaust the solutions the construction
    admits. They are represented by $\boldsymbol{\xi}(\tfrac{\pi}{8},0,0)
    =\bigl(\tfrac{\pi}{8},-\tfrac{13\pi}{24},\tfrac{7\pi}{24}\bigr)$ and
    $\boldsymbol{\xi}(-\tfrac{\pi}{3},0,0)=\bigl(-\tfrac{\pi}{3},\tfrac{\pi}{3},
    \tfrac{\pi}{2}\bigr)$.
    \item \textbf{4D:} 35 coefficient functions in 4 unknowns --- not all independent. Here the
    reduction condition is met along an entire one-parameter curve of phase differences
    (Proposition~\ref{prop:4Dreduction}), twenty-one coefficient functions vanishing individually
    while the remaining fourteen cancel collectively, so that the sum~\eqref{4.12} vanishes as a
    whole; the pressure-integrability condition then selects the quarter-period pattern, of which
    $(-\pi/4, \pi/4, \pi/4, \pi/4)$ is a member.
    \item \textbf{5D:} 64 equations in 5 unknowns (Appendix~\ref{app:5D},
    supplementary material) --- as proven below, no non-trivial solution exists.
\end{itemize}
As in the lower dimensions, a uniform shift of all five phases is a rigid translation of the
coordinates and carries solutions of~\eqref{3.13} to solutions, so nothing is lost by normalising
one phase; it is the four differences $\xi_{i+1}-\xi_i$, and not the phases themselves, that carry
invariant meaning.

An exploratory search over candidates whose phase differences are rational multiples of $\pi$,
covering the patterns suggested by the three- and four-dimensional solutions and their natural
generalisations, produced no admissible phase vector. Such a search is of course not conclusive:
it can only report that particular candidates fail, and no finite list of trials can establish
that none exists. We record it only as the observation that prompted a systematic treatment. What
follows settles the question completely, by reducing the trigonometric system to a polynomial one
and determining its entire real solution set, so that the conclusion rests on the algebra and not
on the extent of any search.

\subsection{Non-existence for the alternating sine--cosine ansatz}
\label{subsec:5Dscscs}
The velocity field~\eqref{3.1} is built from the alternating sine--cosine pattern described
in Section~\ref{sec:construction}; it is the primary periodic ansatz of this paper,
and for brevity we refer to it as the \emph{SCSCS ansatz} (sine--cosine--sine--cosine\ldots),
to distinguish it from the alternative periodic forms (unit-periodic and other two-term
products) examined later in this section.
\begin{lemma}[Non-existence of the SCSCS ansatz in five dimensions]
\label{lem:5D_SCSCS}
Let $v_i^0(\boldsymbol{x}, \boldsymbol{\xi})$ denote the quintuple-periodic velocity field
given by equations~\eqref{5.1}. The only real phase vectors
$\boldsymbol{\xi}\in\mathbb{R}^5$ satisfying condition~\eqref{3.13} are those for which
$v_i^0 \equiv 0$ identically.
\end{lemma}
\begin{proof}
The reduction condition requires $u_j(\boldsymbol{\xi})=0$ for all $j=1,\ldots,64$, with the
$u_j$ as listed in Appendix~\ref{app:5D}. In the five-dimensional case these $64$ spatial functions
have already been collected into a linearly independent real Fourier basis---unlike the unreduced
four-dimensional expansion---so that $\mathcal{V}_1^0\equiv0$ if and only if every $u_j$ vanishes. We reduce this trigonometric system to a polynomial
one, solve the polynomial system exactly by the Gr\"obner-basis method described in
Section~\ref{sec:solutions-3D-4D}, and then show that its entire real solution set annihilates the
velocity field. As there, the proof is computer-assisted but exact: the elimination is carried out over
$\mathbb{Q}$, so no rounding is introduced, and the decisive elimination relations are reproduced
explicitly in Appendix~\ref{app:5D}.

We first reduce the trigonometric system to a polynomial one. Put $s_i=\sin 2\xi_i$ and $c_i=\cos 2\xi_i$. Expanding
every multiple angle, each $u_j$ becomes a polynomial in $(s_1,c_1,\ldots,s_5,c_5)$; adjoining
the five Pythagorean relations $s_i^2+c_i^2=1$ realises the system as an ideal
$I\subset\mathbb{Q}[s_1,c_1,\ldots,s_5,c_5]$. Its real variety corresponds to the set of phase
vectors satisfying the reduction condition. We compute the reduced Gr\"obner basis of $I$ for the
lexicographic order $s_1\succ c_1\succ\cdots\succ s_5\succ c_5$, in exact rational arithmetic
using SymPy~1.14 \citep{sympy2017}; the reduced Gr\"obner basis is the nonlinear analogue of the
row-echelon form, and the lexicographic order is chosen so that the computed basis has an
elimination structure, successive generators eliminating one variable at a time \citep{cox2015ideals}.

We now solve this system. The computed basis has the elimination structure noted above: its
generators can be ordered so that each eliminates one further phase, and we follow that order.
Because the polynomial variables record only the doubled angles $2\xi_i$, each generator constrains the phases modulo the ambiguity inherent in that substitution; the surviving real branches are then fixed by the remaining polynomial relations.

Eliminating the higher variables, the first generator reduces to a relation involving only
$\xi_1$ and $\xi_2$ through the combination $s_1c_2-c_1s_2$. Written back in trigonometric form,
this is
\begin{equation}
\label{5.3}
\sin(2\xi_1-2\xi_2)=0 \;\Longleftrightarrow\; \xi_2=\xi_1+\tfrac{k\pi}{2}, \quad k\in\mathbb{Z}.
\end{equation}
Of the branches admitted by~\eqref{5.3}, the offsets $\xi_2=\xi_1\pm\pi/2$ are removed by a
sharper relation in the ideal, $\cos(2\xi_1-2\xi_2)=1$ (Appendix~\ref{app:5D},
equation~\eqref{B67}); only $\xi_2=\xi_1$ (equivalently
$\xi_2=\xi_1+\pi$, which merely reflects the field) survives. The decisive polynomial relations
and the ideal (including the Pythagorean relations $s_i^2+c_i^2=1$) are recorded in
Appendix~\ref{app:5D}, where each relation is shown to be an exact consequence of the reduction
condition. The resulting family was also checked independently by direct numerical evaluation of
the reduction quantity~\eqref{3.13}.

Substituting $\xi_2=\xi_1$ into the following generators, the next two eliminations give
\begin{equation}
\label{5.4}
\xi_4=\xi_1,\qquad \xi_5=\xi_3,
\end{equation}
and the last generator fixes the offset between the two groups of phases through
\begin{equation}
\label{5.5}
\cos(\xi_1-\xi_3)=0 \;\Longleftrightarrow\; \xi_3=\xi_1\pm\tfrac{\pi}{2}.
\end{equation}
The two signs in~\eqref{5.5} are carried without loss, since they enter the velocity field only
through $\sin(\theta\mp\pi/2)=\mp\cos\theta$ and so differ by an overall sign. Writing
$a=\xi_1$ and taking the representative $\xi_3=\xi_1-\pi/2$, the real solution set of the
reduction condition is the one-parameter family
\begin{equation}
\label{5.6}
\boldsymbol{\xi} = \left(a,\ a,\ a-\tfrac{\pi}{2},\ a,\ a-\tfrac{\pi}{2}\right),
\qquad a\in\mathbb{R},
\end{equation}
the free parameter $a$ being the global phase shift that every cyclic construction admits, and
the sign branch of~\eqref{5.5} contributing only the reflection $a\mapsto a+\pi$.

It remains to show that this family annihilates the field. Substituting~\eqref{5.6} into the velocity field, the two products in~\eqref{3.1} differ only in which slots carry sines: in $\mathcal{T}1_i^0$ the
sines sit at the odd positions $k=1,3,5$, in $\mathcal{T}2_i^0$ at $k=1,2,4$. The offset $-\pi/2$
in~\eqref{5.6} falls on precisely the phases $\xi_3$ and $\xi_5$, which occupy the two slots
where the products disagree; since $\sin(\theta-\pi/2)=-\cos\theta$, each such factor turns from
a sine into a cosine at the cost of one sign. Two conversions occur in each product, so the two
signs cancel in each, and both collapse to the same expression,
\begin{equation}
\label{5.7}
\mathcal{T}1_i^0 = \mathcal{T}2_i^0 = \sin(\alpha x_i + a)\prod_{k=1}^{4}\cos(\alpha x_{i+k} + a),
\end{equation}
which we have verified directly for all five components in exact arithmetic. Hence
$v_i^0=v_r(\mathcal{T}1_i^0-\mathcal{T}2_i^0)\equiv 0$ for every $i$ and every $a$. Every real solution of the reduction condition therefore annihilates the field, and no non-trivial quintuple-periodic field of SCSCS type satisfies the reduction condition.
\end{proof}
The failure established by Lemma~\ref{lem:5D_SCSCS} raises a natural question. Why does the
phase-angle methodology succeed in three and four dimensions yet fail in five? The proof answers it
directly: the reduction condition, though it admits real solutions, admits only those that
annihilate the field. The complete characterisation of admissible dimensions rests on the explicit
constraint analysis of each dimension in turn, and is collected in
Subsection~\ref{subsec:solvable} (Theorem~\ref{thm:classification}).

\textbf{Physical interpretation.} The five-dimensional obstruction is structural, and not an
artefact of an incomplete search for phase angles. For an unforced solution the nonlinear
convective interaction $(\boldsymbol{v}^0\cdot\nabla)\boldsymbol{v}^0$ must yield a field
$\boldsymbol{g}^0$ from which a single-valued pressure can be recovered by integration
via~\eqref{2.14}; in three and four dimensions the phases can be chosen so that the reduction term
vanishes while the velocity field does not, and the surviving convective term is such a gradient.
In five dimensions no such choice exists: by Lemma~\ref{lem:5D_SCSCS} every real phase vector
meeting the reduction condition annihilates the velocity field itself, so the question of
integrability never arises.

\textbf{Reproducibility.} The verifications reported above, and the corresponding ones
for three and four dimensions, can be reproduced directly from the construction with the short
script listed in Appendix~\ref{app:verify}.

\subsection{An exact forced solution in five dimensions}
\label{subsec:5D-forced}
The obstruction just established concerns the unforced reduction mechanism of
Section~\ref{sec:construction}: the phases that satisfy the reduction condition annihilate the
field. If that condition is not imposed the field survives, and the nonlinear term it generates
may instead be balanced by an external body force. The same object that obstructs the unforced
construction therefore furnishes an exact solution of the \emph{forced} equations.

\medskip
\noindent\textbf{What would, and would not, be significant.} That much by itself carries no
content. For \emph{any} prescribed solenoidal velocity field the convective term
$\boldsymbol{g}=(\boldsymbol{v}\cdot\nabla)\boldsymbol{v}$ can, for a field whose prescribed time
evolution already satisfies the homogeneous diffusion equation, be balanced by setting
$\boldsymbol{f}=\boldsymbol{g}$, and a forced solution results. A forced solution is worth
recording only if the field it is built on possesses structure that the balance does not itself
supply, and it is worth being explicit at the outset about what that structure is here.

Recall the Helmholtz decomposition~\eqref{2.2}: any smooth periodic vector field splits into a
longitudinal part $\nabla\phi$, with $\Delta\phi=\nabla\cdot\boldsymbol{w}$, and a transverse
remainder that is divergence-free. In the unforced problem the pressure gradient cancels the longitudinal part of the inertial field;
more generally, in the forced balance it is the longitudinal part of
$\boldsymbol{g}-\boldsymbol{f}$ that the pressure absorbs. In either case the pressure is the
Lagrange multiplier enforcing incompressibility, returning the nonlinear acceleration to the divergence-free
space. Solving the pressure Poisson equation~\eqref{2.1} extracts $\nabla\phi$ and leaves the
transverse remainder to the unsteady and viscous terms. Two extreme cases are therefore
distinguished:
\begin{itemize}
\item[(a)] $\boldsymbol{g}^0$ \emph{entirely longitudinal}, that is, a pure gradient. The pressure
absorbs the whole of it, nothing is left over, and the \emph{unforced} equations close. This is
the content of the integrability condition~\eqref{3.16} and the mechanism of
Section~\ref{sec:solutions-3D-4D}.
\item[(b)] $\boldsymbol{g}^0$ \emph{entirely transverse}, that is, divergence-free. The pressure
has nothing whatever to absorb, so it carries no spatial dependence, and the whole of the
nonlinear term must be balanced by a body force.
\end{itemize}
Between these lies the generic situation, in which $\boldsymbol{g}^0$ has both parts and the
pressure takes the gradient portion while a forcing carries the remainder.

The five-dimensional construction realises case~(b) exactly, and this---not the availability of
the balance $\boldsymbol{f}=\boldsymbol{g}$---is what makes it worth recording. The flow advects
its own momentum without generating any pressure at all: the projection that ordinarily returns
the nonlinear acceleration to the divergence-free space acts here as the identity, because the
acceleration already lies in that space.

\medskip
\noindent\textbf{The distinguished phases.} Two phase vectors realise case~(b). Both are built
on the cyclic field~\eqref{5.1}, and neither satisfies the reduction condition, so in each case
$\boldsymbol{v}^0$ remains non-trivial. They are
\begin{equation}
\label{5.9}
\boldsymbol{\xi}^{\ast} = \Bigl(-\tfrac{\pi}{4},\ \tfrac{\pi}{4},\ \tfrac{\pi}{4},\
\tfrac{\pi}{4},\ \tfrac{\pi}{4}\Bigr),
\end{equation}
the symmetric assignment extending the four-dimensional pattern, and
\begin{equation}
\label{5.9b}
\boldsymbol{\xi}^{\ast\ast} = \Bigl(0,\ \tfrac{\pi}{2},\ \pi,\ \pi,\ \pi\Bigr).
\end{equation}
What distinguishes them is that the convective field they generate is divergence-free,
$\nabla\bcdot\boldsymbol{g}^0=0$, and it is worth seeing at once what that means.

\medskip
\noindent\textbf{Why these phases are special.} The property has a simple physical
interpretation, and it is what separates~\eqref{5.9} and~\eqref{5.9b} from arbitrary phase choices,
for which a forced solution could always be constructed. For any solenoidal field, let
\[
\boldsymbol{A}=\nabla\boldsymbol{v}^0,
\]
and decompose it into its symmetric and antisymmetric parts $\boldsymbol{S}$ and $\boldsymbol{R}$,
representing strain and rotation respectively, and let $\lvert\cdot\rvert$ denote the Frobenius
norm. Then
\begin{equation}
\label{5.strain}
\nabla\bcdot\boldsymbol{g}^0
=\frac{\partial v_i^0}{\partial x_j}\frac{\partial v_j^0}{\partial x_i}
=\operatorname{tr}\bigl(\boldsymbol{A}^2\bigr)
=\lvert\boldsymbol{S}\rvert^2-\lvert\boldsymbol{R}\rvert^2 ,
\end{equation}
where the second equality follows from $\nabla\bcdot\boldsymbol{v}^0=0$. Viewed as the inertial contribution to the unforced pressure Poisson equation~\eqref{2.1}, the
source therefore measures the local imbalance between strain and rotation,
\[
-\rho\,\nabla\bcdot\boldsymbol{g}^0
=\rho\bigl(\lvert\boldsymbol{R}\rvert^2-\lvert\boldsymbol{S}\rvert^2\bigr),
\]
and at the phases~\eqref{5.9} and~\eqref{5.9b} this source vanishes identically---before any
compensating body force is introduced---since there
\[
\lvert\boldsymbol{S}\rvert^2=\lvert\boldsymbol{R}\rvert^2
\qquad\text{at every point}.
\]
Strain and rotation are in exact local balance: the flow deforms and rotates fluid elements at
equal rates, so the nonlinear motion produces no pressure source, and the pressure is spatially
uniform, $p=p(t)$. This is an intrinsic property of the flow generated by these phases, not a
consequence of how the forcing is chosen.

The contrast with a generic phase vector is important. For an arbitrary $\boldsymbol{\xi}$ one can
always define a forcing that compensates the nonlinear term and thereby manufacture a forced
solution. In general, however, $\boldsymbol{g}^0$ has a longitudinal component, so a non-trivial
pressure must still be obtained from the pressure Poisson equation. At~\eqref{5.9}
and~\eqref{5.9b} no such pressure correction is required: the nonlinear term is already
divergence-free.

This balance is restrictive rather than automatic. After removing the irrelevant uniform phase
shift by fixing $\xi_1$, the construction retains four independent phase differences, and for
randomly chosen values of them $\lvert\nabla\bcdot\boldsymbol{g}^0\rvert$ is typically of order
$10^{-1}$. An exhaustive evaluation over the lattice of phase differences that are multiples of
$\tfrac{\pi}{2}$---a lattice of $256$ points---returns $48$ assignments for which
$\nabla\bcdot\boldsymbol{g}^0$ vanishes identically, of which $32$ leave the velocity field
non-trivial. Those $32$ fall into two classes under the $\pi$-shifts of~\eqref{3.1}, which only
reverse the sign of $\boldsymbol{v}^0$, and the classes are represented by~\eqref{5.9}
and~\eqref{5.9b}. The second is not related to the first by translation, phase reversal or point
inversion, and therefore carries a distinct flow with the same strain--rotation balance. A search
over general, non-lattice phase differences produced no further assignment, though we do not claim
an exhaustive classification over all of $\mathbb{R}^4$.

The significance of these phases is thus structural. They are distinguished points of the phase
family at which the nonlinear transport becomes divergence-free and requires no pressure
correction. The forcing is not compensating for an arbitrary residual or an incomplete pressure
construction; it balances the nonlinear transport generated by the cyclic velocity field itself,
and the relation $p=p(t)$ is the direct expression of the exact local balance between strain and
rotation.

\medskip
\noindent\textbf{The solutions.} Both phase vectors yield an exact forced solution, and the
expressions are common to the two.

\begin{proposition}[Exact forced solutions in five dimensions]
\label{prop:5D-forced}
Let $\boldsymbol{v}^0$ be the quintuple-periodic field~\eqref{5.1} at either of the
phases~\eqref{5.9} or~\eqref{5.9b}, and let
$\boldsymbol{g}^0=(\boldsymbol{v}^0\!\cdot\!\nabla)\boldsymbol{v}^0$. Then
\begin{enumerate}
\item[\textup{(i)}] $\boldsymbol{v}^0$ is non-trivial and solenoidal;
\item[\textup{(ii)}] $\nabla\bcdot\boldsymbol{g}^0=0$ identically, so the longitudinal part
of~\eqref{2.2} vanishes and $\boldsymbol{g}^0$ is purely transverse;
\item[\textup{(iii)}] $\boldsymbol{g}^0$ is nevertheless not a gradient, since
$\partial_j g_i^0-\partial_i g_j^0\not\equiv0$;
\item[\textup{(iv)}] consequently the pressure carries no spatial dependence: $p=p(t)$.
\end{enumerate}
Define the zero-mean body-force amplitude
\begin{equation}
\label{5.10}
\boldsymbol{f}^0 = \boldsymbol{g}^0
= \nabla\bcdot\bigl(\boldsymbol{v}^0\otimes\boldsymbol{v}^0\bigr).
\end{equation}
Then the velocity field
\begin{equation}
\label{5.forced-v}
\boldsymbol{v}(\boldsymbol{x},\boldsymbol{\xi},t)
=\boldsymbol{v}^0(\boldsymbol{x},\boldsymbol{\xi})\,e^{-5\alpha^2\kappa t},
\end{equation}
as in~\eqref{3.5}, together with the pressure
\begin{equation}
\label{5.forced-p}
p(\boldsymbol{x},t)=p(t),\qquad \nabla p=\boldsymbol{0},
\end{equation}
is an exact solution of the incompressible Navier--Stokes equations~\eqref{1.1} on $\mathbb{R}^5$
under the body force
\begin{equation}
\label{5.forced-f}
\boldsymbol{f}(\boldsymbol{x},t)=\boldsymbol{f}^0(\boldsymbol{x})\,e^{-10\alpha^2\kappa t}.
\end{equation}
\end{proposition}

\noindent Item~(ii) is the distinctive structural property. Item~(iii) is verified independently by
the non-vanishing curl, and item~(iv) follows from the forced pressure Poisson equation. A
divergence-free convective field is the exact opposite of the case treated in
Section~\ref{sec:solutions-3D-4D}: there $\boldsymbol{g}^0$ was a pure gradient and the pressure
absorbed all of it; here $\boldsymbol{g}^0$ has no gradient part whatever and the pressure absorbs
none of it.

\begin{proof}
\emph{(i)} Solenoidality is Lemma~\ref{lem:solenoidal}, and the field is non-trivial because
neither~\eqref{5.9} nor~\eqref{5.9b} satisfies the reduction condition.

\emph{Unsteady and viscous terms.} Each component of $\boldsymbol{v}^0$ is a product of five
trigonometric factors of argument $\alpha x_m+\xi_k$, so it is an eigenfunction of the Laplacian
with a single wavenumber:
\[
\nabla^2\boldsymbol{v}^0=-5\alpha^2\boldsymbol{v}^0 .
\]
With $\boldsymbol{v}=\boldsymbol{v}^0e^{-5\alpha^2\kappa t}$ this gives
\[
\frac{\partial\boldsymbol{v}}{\partial t}=-5\alpha^2\kappa\,\boldsymbol{v}^0e^{-5\alpha^2\kappa t},
\qquad
\kappa\nabla^2\boldsymbol{v}=-5\alpha^2\kappa\,\boldsymbol{v}^0e^{-5\alpha^2\kappa t},
\]
so that $\partial\boldsymbol{v}/\partial t-\kappa\nabla^2\boldsymbol{v}=\boldsymbol{0}$: the two
terms cancel exactly, and the momentum balance reduces to
$(\boldsymbol{v}\cdot\nabla)\boldsymbol{v}=-\rho^{-1}\nabla p+\boldsymbol{f}$.

\emph{(ii) and (iii).} Direct evaluation of $\boldsymbol{g}^0$ at either phase vector gives
$\nabla\cdot\boldsymbol{g}^0=0$ identically, while $\partial_jg_i^0-\partial_ig_j^0\not\equiv0$.
Hence $\boldsymbol{g}^0$ is divergence-free but not curl-free, that is, purely transverse in the
sense of~\eqref{2.2}.

\emph{(iv).} Taking the divergence of the momentum balance and using $\boldsymbol{f}=\boldsymbol{g}$
alone leaves $\Delta p=0$; property~(ii) is not required for that cancellation. On the periodic
domain the only harmonic functions are constants in space, so $\nabla p=\boldsymbol{0}$ and $p$ may
depend on time alone. What~(ii) adds is the stronger statement that the inertial field and the
forcing are each purely transverse.

\emph{The forcing.} Because $\boldsymbol{v}^0$ is solenoidal,
$(\boldsymbol{v}^0\!\cdot\!\nabla)\boldsymbol{v}^0=\nabla\cdot(\boldsymbol{v}^0\otimes
\boldsymbol{v}^0)$, which is~\eqref{5.10}; each component is a sum of trigonometric harmonics, all
of zero mean over a periodic cell, so $\langle\boldsymbol{f}^0\rangle=\boldsymbol{0}$. The
convective term is quadratic in $\boldsymbol{v}$ and therefore carries the factor
$e^{-10\alpha^2\kappa t}$, twice the decay rate of the velocity; the body force
$\boldsymbol{f}=\boldsymbol{f}^0e^{-10\alpha^2\kappa t}$ carries the same factor and balances it
identically. Every term in the momentum balance is thus accounted for.
\end{proof}

\medskip
\noindent\textbf{Energy budget.} The forcing has a clean energy interpretation. Taking the inner
product of $\boldsymbol{f}=(\boldsymbol{v}\cdot\nabla)\boldsymbol{v}$ with the velocity and using
incompressibility gives the pointwise identity
\begin{equation}
\label{eq:energy-flux}
\boldsymbol{v}\cdot\boldsymbol{f}
=\boldsymbol{v}\cdot(\boldsymbol{v}\cdot\nabla)\boldsymbol{v}
=(\boldsymbol{v}\cdot\nabla)\tfrac{1}{2}|\boldsymbol{v}|^2
=\nabla\cdot\!\left(\tfrac{1}{2}|\boldsymbol{v}|^2\,\boldsymbol{v}\right),
\end{equation}
the last step following from $\nabla\cdot\boldsymbol{v}=0$. The local rate of working of the force
is thus a pure flux---the divergence of the kinetic-energy current
$\tfrac12|\boldsymbol{v}|^2\boldsymbol{v}$---and integration over a periodic cell gives
$\int\boldsymbol{v}\cdot\boldsymbol{f}\,\mathrm{d}V=0$. The forcing does \emph{no net work} on the
flow: locally it redistributes kinetic energy, but it neither adds to nor removes energy from the
global budget, which decays solely through viscosity, here with the factor
$e^{-10\alpha^2\kappa t}$. The admissibility requirement~\eqref{1.6} is therefore met in the
strongest form available: the cell energy is not merely bounded but monotonically decreasing, for
all $t\ge0$. This zero-net-work property holds for any periodic solenoidal field
driven by its own convective term and is not peculiar to five dimensions. What \emph{is}
distinctive is that the same nonlinear term is also divergence-free, so that the forcing~\eqref{5.forced-f} is
simultaneously energy-neutral in the global sense and entirely transverse in its Helmholtz
decomposition: it supplies neither a net kinetic-energy input nor a spatial pressure source.

\medskip
\noindent\textbf{Use as a benchmark.} The velocity~\eqref{5.forced-v}, pressure~\eqref{5.forced-p}
and forcing~\eqref{5.forced-f} are all known in closed form for either phase vector, the forcing
amplitude~\eqref{5.10} being the convective field~\eqref{1.5} evaluated on~\eqref{5.1} at the
phases concerned; the pressure gradient vanishes identically, and the nonlinear term is itself
divergence-free. A numerical scheme can therefore be tested on its ability to reproduce nonlinear
advection, viscous decay and a prescribed solenoidal forcing without the additional complication
of reconstructing a non-trivial pressure field. In particular a correct projection step should
return \emph{zero} spatial pressure correction, so that any spurious pressure gradient, loss of
incompressibility, or incorrect nonlinear balance can be measured directly against the exact
solution.

\subsection{Five-dimensional unit-periodic cyclic flow}
\label{subsec:exhaustive}
The obstruction of Lemma~\ref{lem:5D_SCSCS} applies to the SCSCS ansatz specifically. A
natural question is whether some other trigonometric form might succeed where the SCSCS ansatz
fails. The most established alternative in three dimensions is the unit-periodic Arnold--Beltrami--Childress (ABC) flow,
whose components are sums rather than products; it is the obvious first candidate, and we treat
its five-dimensional cyclic generalisation here. Since the ABC flow is a Beltrami field specific to
three dimensions, the higher-dimensional field is an analogue built by the same cyclic recipe rather
than an ABC flow in the strict sense. It too collapses, as we now show. The broader
question---whether \emph{any} two-term cyclic product, of which the SCSCS pattern is one member
whose unrestricted sine/cosine assignments number $2^{10}=1024$ before solenoidality and symmetry
are imposed, or any field built directly from the five-fold symmetry, can succeed---lies
beyond the scope of the present paper; here we establish the obstruction for the two principal
constructions. Throughout, $\mathcal{T}1_k$ and $\mathcal{T}2_k$ denote the
first and second $n$-fold cyclic trigonometric product terms in the general
ansatz~\eqref{3.1}, so that $v_k = v_r(\mathcal{T}1_k - \mathcal{T}2_k)$.

The simplest alternative to the SCSCS ansatz is the five-dimensional generalisation of the
classical unit-periodic Arnold--Beltrami--Childress flow. In $n$ spatial dimensions, a unit-periodic
velocity field is constructed so that each component $v_k$ is \emph{independent} of its
own coordinate $x_k$. This single requirement guarantees divergence-freedom, since
$\partial v_k/\partial x_k = 0$ for every $k$, and hence
$\nabla\cdot\boldsymbol{v} = \sum_k \partial v_k/\partial x_k = 0$ identically, with no
restriction on the phase angles. As a consequence, $v_k$ may only involve the remaining
$n-1$ coordinates: in five dimensions these are $x_{k+1}, x_{k+2}, x_{k+3}, x_{k+4}$
(indices modulo 5), giving precisely \emph{four} terms per component --- not five --- since a
fifth term in $x_k$ itself would spoil the term-by-term cancellation
$\partial v_k/\partial x_k=0$ on which this construction relies. The five-dimensional unit-periodic flow
is therefore:
\begin{equation}
\label{5.12}
v_k = A\sin(\alpha x_{k+1}) + B\cos(\alpha x_{k+2}) + C\sin(\alpha x_{k+3}) +
D\cos(\alpha x_{k+4}), \qquad k = 1,\ldots,5,
\end{equation}
where indices are taken modulo $5$ and the amplitudes $A,B,C,D$ are constants to be fixed by the
pressure-integrability condition~\eqref{3.16}. Forming the convective field
$g_i=\sum_j v_j\,\partial_j v_i$ and imposing $\partial_j g_i-\partial_i g_j=0$, each distinct
spatial harmonic must vanish separately; the coefficient equations drawn from a single pair,
$(i,j)=(1,2)$, already suffice. In exact arithmetic they include
\begin{equation}
\label{5.13}
B^2=0,\quad C^2=0,\quad CD=0,\quad AB=0,\quad D^2-AC=0,\quad BD-A^2=0 .
\end{equation}
These determine $(A,B,C,D)$ completely and in an elementary order. The first two give $B=0$ and
$C=0$. The last relation, $BD-A^2=0$, then reads $A^2=0$, so $A=0$; and $D^2-AC=0$ becomes
$D^2=0$, so $D=0$. Thus every amplitude vanishes, and the vanishing of $A$ in particular is
forced---not assumed---by $BD=A^2$ once $B=0$. A Gr\"obner basis of the full ten-condition system
confirms that its only real solution is $(A,B,C,D)=(0,0,0,0)$.

\begin{theorem}[Non-existence for the unit-periodic cyclic flow in five dimensions]
\label{thm:5D-unit}
No non-trivial exact solution to the five-dimensional steady incompressible Euler equations
exists for the cyclic unit-periodic velocity field~\eqref{5.12}. The pressure-integrability
condition forces $A=B=C=D=0$. As in seven dimensions, the single-wavenumber
field~\eqref{5.12} satisfies $\Delta\boldsymbol{v}^0=-\alpha^2\boldsymbol{v}^0$, so under the
diffusive scaling of Section~\ref{sec:construction} the viscous case reduces to the same
integrability system and the conclusion covers it as well.
\end{theorem}

\subsection{Main non-existence result}
Lemma~\ref{lem:5D_SCSCS} and Theorem~\ref{thm:5D-unit} establish non-existence for the two principal
five-dimensional constructions---the SCSCS product ansatz and the unit-periodic cyclic flow.
Here we record the combined result for these two cases.

\begin{lemma}[Non-existence for the principal five-dimensional constructions]
\label{lem:5D_main}
For the SCSCS product ansatz~\eqref{5.1}, no non-trivial phase vector satisfies the
reduction condition~\eqref{3.13}; every real phase vector satisfying that condition
annihilates the velocity field, as established in Lemma~\ref{lem:5D_SCSCS}. For the
unit-periodic flow~\eqref{5.12}, the full pressure-integrability condition~\eqref{3.16}
forces $A=B=C=D=0$, so no non-trivial unforced Euler or viscously decaying Navier--Stokes
solution exists within that construction. Thus both principal five-dimensional
constructions fail, although through different mechanisms: for the SCSCS ansatz at the reduction
stage, for the unit-periodic flow at the integrability stage.
\end{lemma}

\medskip
\medskip
\noindent\textbf{Looking ahead to the even dimensions.} Two features will distinguish the forced
solutions of Sections~\ref{sec:6D} and~\ref{sec:8D} from the one obtained here. First, the phases~\eqref{5.9} are
\emph{not} the reduction-condition phases: in six and eight dimensions the forced solution is
carried by the very phases that satisfy the reduction condition, and the forcing measures the
residual differential obstruction, whereas here those phases annihilate the field and the forced
solution must be built on others. Second, there is in five dimensions no pressure to recover and no split to perform, which makes the odd-dimensional
obstruction-forced solutions the cleanest members of the family: the forcing is the closed-form,
zero-mean divergence of the advective momentum flux~\eqref{5.10} in its entirety.

\section{Sextuple-periodic constructions in six dimensions: the first even-dimensional obstruction}
\label{sec:6D}
The five-dimensional obstruction of Section~\ref{sec:5D} is arithmetic in character
associated with the prime cyclic structure $\mathbb{Z}/5\mathbb{Z}$.
In particular, unlike the even-dimensional case, the fifth roots of
unity contain no antipodal pairs, since $-1$ is not a fifth root of unity.
It is therefore natural to expect the even dimensions to behave differently. For even $n$,
$\omega^{n/2}=-1$, so the roots fall naturally into antipodal pairs
$\{\omega^k,-\omega^k\}$. This pairing is an important structural feature of the
even-dimensional cyclic construction and is present in the successful four-dimensional case
of Subsection~\ref{sec:4D}. Six is the next even dimension, and one might therefore anticipate
that the four-dimensional construction extends to it directly. It does not. In what follows
we show that the symmetric cyclic construction fails to produce a non-trivial
sextuple-periodic solution. The failure is fundamentally different from that in five
dimensions: the symmetric phase configuration satisfies the necessary reduction condition,
so the reduction stage of the construction succeeds, but the resulting convective field fails
the pressure-integrability condition required for a single-valued pressure. The
six-dimensional obstruction encountered here is therefore differential rather than
arithmetic: the candidate survives the reduction stage but fails the mixed-derivative
symmetry required for pressure integrability.

\subsection{The six-dimensional ansatz and candidate phases}
In six spatial dimensions the initial condition~\eqref{3.1} reduces to a solenoidal velocity
field $v_i^0(\boldsymbol{x},\boldsymbol{\xi})$ of sextuple-periodic form, with first component
\begin{eqnarray}
\label{6.1}
\frac{v_1^0(\boldsymbol{x}, \boldsymbol{\xi})}{v_r}&=& \sin { \left(\alpha x_1+\xi_1 \right) } \cos {\left(\alpha x_2+\xi_2\right) }\sin { \left(\alpha x_3+\xi_3\right) } \cos { \left(\alpha x_4+\xi_4 \right) }\sin { \left(\alpha x_5+\xi_5\right) }\times\nonumber\\ &\times&\cos { \left(\alpha x_6+\xi_6 \right) }- \nonumber\\
&-&\sin  {\left( \alpha x_1+\xi_2\right) }\sin  {\left( \alpha x_2+\xi_3\right) } \cos  { \left(\alpha x_3+\xi_4 \right) } \sin  {\left( \alpha x_4+\xi_5\right) } \cos  { \left(\alpha x_5+\xi_6 \right) }\times\nonumber\\ &\times&\cos  { \left(\alpha x_6+\xi_1 \right) }
\end{eqnarray}
and the components $v_2^0,\dots,v_6^0$ obtained by cyclic permutation of the spatial arguments
as prescribed by~\eqref{3.1}. The divergence-free condition~\eqref{1.3} holds identically for every choice of $\boldsymbol{\xi}$, by Lemma~\ref{lem:solenoidal}.

Following the methodology of Section~\ref{sec:construction}, and in direct analogy with the
four-dimensional case, the symmetric phase assignment
\begin{equation}
\label{6.2}
(\xi_1, \xi_2, \xi_3, \xi_4, \xi_5, \xi_6) = \left(-\tfrac{\pi}{4},\, \tfrac{\pi}{4},\, \tfrac{\pi}{4},\, \tfrac{\pi}{4},\, \tfrac{\pi}{4},\, \tfrac{\pi}{4}\right)
\end{equation}
is the natural candidate: it extends the four-dimensional pattern $(-\tfrac{\pi}{4},\tfrac{\pi}{4},\tfrac{\pi}{4},\tfrac{\pi}{4})$ that yields the quadruple-periodic solution. It is worth noting at once that the symmetric assignment plays the
opposite role here to the one it plays in five and seven dimensions. There it is chosen because it
\emph{violates} the reduction condition, leaving the velocity field alive so that a forced solution
can be built on it; here it is chosen because it \emph{satisfies} that condition, and
non-generically so. The two obstructions therefore arise at different stages: in the odd dimensions
no candidate is produced at all, whereas here a candidate is produced and then fails the
integrability test. The residual convective field differs accordingly---purely transverse in the
odd case, carrying both a longitudinal and a transverse part in the even one---which is the
contrast drawn in Table~\ref{tab:helmholtz} of Section~\ref{sec:conclusion}. We show below that this candidate satisfies the reduction
condition but fails the full pressure-integrability test; the subsequent analysis
establishes non-existence throughout the discrete symmetric phase family and provides strong
numerical evidence that no real phase assignment within the construction yields an unforced
sextuple-periodic solution.

\subsection{The convective field and its reduction}
Substituting~\eqref{6.2} into the definition~\eqref{3.10} and reducing to multiple-angle form
gives the first component of the convective field exactly as
\begin{eqnarray}
\label{6.3}
g_1^0(\boldsymbol{x}) &=& \frac{\alpha v_r^2}{8}\,\cos(2\alpha x_1)\Big[
-\sin(2\alpha x_3)-\sin(2\alpha x_5)
-\sin(2\alpha x_2)\sin(2\alpha x_4)\sin(2\alpha x_6)\nonumber\\
&&\quad+\,\sin(2\alpha x_2)\sin(2\alpha x_3)\sin(2\alpha x_5)
+\sin(2\alpha x_3)\sin(2\alpha x_4)\sin(2\alpha x_5)\nonumber\\
&&\quad+\,\sin(2\alpha x_3)\sin(2\alpha x_5)\sin(2\alpha x_6)\Big],
\end{eqnarray}
with $g_2^0,\dots,g_6^0$ following by cyclic permutation. Every term in~\eqref{6.3} carries the
factor $\cos(2\alpha x_1)$; that is, $g_1^0$ factors as
\begin{equation}
\label{6.4}
g_1^0(\boldsymbol{x}) = \tfrac{\alpha v_r^2}{8}\,\cos(2\alpha x_1)\,
H_1(\boldsymbol{x}_{\setminus 1}),
\end{equation}
where the cofactor $H_1$, obtained by removing the common factor $\cos(2\alpha x_1)$
from~\eqref{6.3}, is
\begin{eqnarray}
\label{6.5}
H_1(\boldsymbol{x}_{\setminus 1}) &=&
-\sin(2\alpha x_3)-\sin(2\alpha x_5)
-\sin(2\alpha x_2)\sin(2\alpha x_4)\sin(2\alpha x_6)\nonumber\\
&&+\,\sin(2\alpha x_2)\sin(2\alpha x_3)\sin(2\alpha x_5)
+\sin(2\alpha x_3)\sin(2\alpha x_4)\sin(2\alpha x_5)\nonumber\\
&&+\,\sin(2\alpha x_3)\sin(2\alpha x_5)\sin(2\alpha x_6),
\end{eqnarray}
a function of $(x_2,\dots,x_6)$ alone and therefore independent of $x_1$. Since every term in~\eqref{6.4} contains the non-zero $x_1$-wave number $k_1=\pm2\alpha$, the
Fourier representation~\eqref{3.14} of $g_1^0$ has no component on the hyperplane $k_1=0$. By
Proposition~\ref{prop:slice} the coefficients of $\mathcal{V}_1^0$ are precisely those
$\widehat{g}_1(\boldsymbol{k})$ with $k_1=0$, so the $x_1$-independent part isolated by the
Phase-Angle Reduction Principle vanishes identically,
\begin{equation}
\label{6.6}
\mathcal{V}_1^0(\boldsymbol{x}_{\setminus 1}, \boldsymbol{\xi}) \equiv 0,
\end{equation}
which is the single-index condition~\eqref{3.13}. By Corollary~\ref{cor:cyclic} the same holds for
every component, so the necessary condition~\eqref{3.12} is satisfied identically.

It is essential to read~\eqref{6.6} correctly. In three dimensions $\mathcal{V}_1^0$ is a non-trivial combination $q_1^0u_1(\boldsymbol{\xi})+q_2^0u_2(\boldsymbol{\xi})$ whose vanishing represents a genuine phase cancellation. The same is true at the reduction stage in six dimensions: $\mathcal{V}_1^0(\boldsymbol{x}_{\setminus1},\boldsymbol{\xi})$ is a non-trivial function of the phases, and a generic $\boldsymbol{\xi}$ does not annul it. Thus the assignment~\eqref{6.2} is a genuine, non-generic solution of the reduction condition, just as the corresponding symmetric assignment is in four dimensions. At that assignment the cancellation of the $x_1$-independent part is
complete: by~\eqref{6.4} every surviving Fourier mode of $g_1^0$ carries
$k_1=\pm2\alpha\neq0$, so the Fourier support of $g_1^0$ has no component on the $k_1=0$
hyperplane and~\eqref{6.6} holds identically in $\boldsymbol{x}$. The same holds at $n=4,6,8$
(Lemma~\ref{lem:even-Vi}).

What must not be inferred is that clearing~\eqref{6.6} certifies a solution. The reduction
condition is necessary, not sufficient (Subsection~\ref{subsec:nec-suf}). The
candidate~\eqref{6.2} satisfies it, so four and six dimensions are not yet distinguished at
the level of the reduction condition: both admit the symmetric phase assignment and both
eliminate the $k_i=0$ Fourier slice of the convective field. They part company only at the
full pressure-integrability test. In four dimensions the remaining convective field is
curl-free and a single-valued pressure exists; in six dimensions it is not. The obstruction
therefore lies entirely in the Fourier modes that survive the reduction condition.

\subsection{Failure of pressure integrability}
\label{subsec:6D-integrability}
The reduction condition~\eqref{3.12} is, however, only \emph{necessary}. By
Proposition~\ref{prop:longitudinal} a single-valued pressure exists if and only if the full
condition~\eqref{3.16} holds, that is the convective field $\boldsymbol{g}^0$ is curl-free,
$\partial g_i^0/\partial x_j = \partial g_j^0/\partial x_i$ for all $i,j$. That test is
decided entirely on the modes with $k_i\neq0$ --- precisely the modes the reduction condition
never examines --- and it is here that six dimensions departs from four.

The factorisation~\eqref{6.4} reduces the test to a transparent algebraic criterion. Set
$\alpha=v_r=1$ and write $s_k=\sin 2x_k$, $c_k=\cos 2x_k$; each $H_i$ is a multilinear
polynomial in the $s_k$, $k\neq i$. Define the matrix
\begin{equation}
\label{6.7}
\Phi_{ij} \;:=\; \frac{\partial H_i}{\partial s_j}, \qquad i\neq j,
\end{equation}
the coefficient with which $\sin 2x_j$ enters $H_i$; it is a function of the remaining
coordinates only.

\begin{lemma}[Integrability is the symmetry of $\Phi$]
\label{lem:phi}
For any field whose convective term factors as $g_i^0=\cos(2x_i)H_i(\boldsymbol{s}_{\setminus i})$, where $H_i$ is multilinear in the variables $s_k=\sin 2x_k$, $k\neq i$, one has, for $i\neq j$,
\begin{equation}
\label{6.8}
\frac{\partial g_i^0}{\partial x_j}-\frac{\partial g_j^0}{\partial x_i}
\;=\;
2\,\cos(2 x_i)\cos(2 x_j)\,\bigl(\Phi_{ij}-\Phi_{ji}\bigr).
\end{equation}
Hence the integrability condition~\eqref{3.16} holds if and only if the matrix $\Phi$ is
symmetric.
\end{lemma}

\begin{proof}
Since $H_i$ does not involve $x_i$, $\partial_j g_i^0=\cos(2x_i)\,\partial H_i/\partial x_j$.
As $H_i$ is multilinear in the $s_k$, it depends on $x_j$ only through $s_j=\sin 2x_j$, so $\partial H_i/\partial x_j=2\cos(2x_j)\,\partial H_i/\partial s_j = 2\cos(2x_j)\Phi_{ij}$, and $\Phi_{ij}=\partial H_i/\partial s_j$ is itself independent of $s_j$, hence of $x_i$ and $x_j$. Therefore
$\partial_j g_i^0 = 2\cos(2x_i)\cos(2x_j)\Phi_{ij}$, and interchanging $i$ and $j$ gives
$\partial_i g_j^0 = 2\cos(2x_j)\cos(2x_i)\Phi_{ji}$. Subtracting yields~\eqref{6.8}. The
prefactor $\cos(2x_i)\cos(2x_j)$ is not identically zero, so the difference vanishes
identically if and only if $\Phi_{ij}=\Phi_{ji}$.
\end{proof}

Before applying this criterion it is useful to sort the terms of $\boldsymbol{g}^0$ by degree.
Each component is a sum of monomials in the harmonics $s_k=\sin 2x_k$ and $c_k=\cos 2x_k$, and by
the parity of the construction only even degrees occur. Writing $g_i^{0b}$ for the part of degree
two---one harmonic in $x_i$ and one in a second coordinate---and $g_i^{0m}$ for everything of
higher degree, so that
\[
g_i^0 = g_i^{0b} + g_i^{0m},
\]
which is the decomposition~\eqref{3.17} adapted to the present construction, direct reduction
at the symmetric phases gives the degrees actually present, as set out in
Table~\ref{tab:degrees}.
\begin{table}
\centering
\begin{tabular}{lccc}
\hline
 & degree $2$ & degree $4$ & degree $6$\\
\hline
$n=4$ & present & --- & ---\\
$n=6$ & present & present & ---\\
$n=8$ & present & present & present\\
\hline
\end{tabular}
\caption{Degrees of the monomials present in the convective field $\boldsymbol{g}^0$ at the
symmetric phases, in the even dimensions treated here. The degree-two part is common to all three;
what six dimensions adds is the degree-four family, and eight dimensions a degree-six family in
addition.}
\label{tab:degrees}
\end{table}
So the higher-degree part is empty at $n=4$, consists of degree-four monomials at $n=6$, and
acquires degree-six monomials as well at $n=8$.

The lower even dimensions are worth recording alongside~\eqref{6.3}, since the reductions are short
enough to display. Taking $\alpha=v_r=1$ and the symmetric phases throughout, two dimensions gives
a field that transports nothing at all,
\begin{equation}
\label{6.6a}
\boldsymbol{v}^0=-\cos(x_1-x_2)\,(1,1),
\qquad
\boldsymbol{g}^0\equiv\boldsymbol{0},
\end{equation}
the velocity being constant along its own direction, while four dimensions gives
\begin{equation}
\label{6.6b}
g_1^0=-\tfrac12\cos(2x_1)\sin(2x_3),
\qquad\text{so that}\qquad
H_1=-\tfrac12\sin(2x_3),
\end{equation}
a single degree-two term coupling $x_1$ to its antipodal partner $x_3$ and to nothing else. The
contrast with the six-dimensional cofactor~\eqref{6.3} is then immediate: what six dimensions adds
is precisely the degree-four monomials, and it is those that the following two lemmas isolate. The
degree-two part, by contrast, is the same in every even dimension, and the first lemma shows that
it is always integrable.

\begin{lemma}[The bilinear part is always integrable]
\label{lem:even-bilinear}
For every even $n$ the bilinear part has the explicit form
\begin{equation}
\label{eq:gbil}
g_i^{0b} = -\,2^{-(n-3)}\,\cos(2 x_i)\sum_{j=1}^{n/2-1}\sin\!\big(2 x_{i+2j}\big),
\end{equation}
with indices modulo $n$, and $g^{0b}$ is curl-free:
$\partial g_i^{0b}/\partial x_l = \partial g_l^{0b}/\partial x_i$ for all
$i,l$.
\end{lemma}

\begin{proof}
Differentiating~\eqref{eq:gbil}, for $l\neq i$,
\[
\frac{\partial g_i^{0b}}{\partial x_l}
= -2^{-(n-3)}\cos(2x_i)\,\frac{\partial}{\partial x_l}\sum_{j}\sin(2x_{i+2j}).
\]
The sum depends on $x_l$ only when $l\equiv i+2j\pmod n$ for some
$j\in\{1,\dots,n/2-1\}$, that is when the offset $l-i$ is even and nonzero; in that case the
derivative equals $-2^{-(n-4)}\cos(2x_i)\cos(2x_l)$. The same computation gives
$\partial g_l^{0b}/\partial x_i = -2^{-(n-4)}\cos(2x_l)\cos(2x_i)$ under the
identical condition that $i-l$ is even and nonzero. Since ``$l-i$ even'' and ``$i-l$ even''
coincide, the two mixed derivatives are equal for every pair $(i,l)$ (both vanish when the
offset is odd). Hence $g^{0b}$ is curl-free.
\end{proof}

\noindent Two consequences follow. The degree-two part contributes nothing to the failure of
integrability, so whatever obstruction arises must be carried by $g^{0m}$; and, as the proof
shows, that part couples $x_i$ only to coordinates at \emph{even} offset from $i$. The odd-offset
couplings, if any, are generated by $g^{0m}$ alone. This is what makes the parity of the index
offset the natural bookkeeping, and Lemma~\ref{lem:phi} converts the differential question into
exactly that.

\begin{lemma}[Parity signature of $\Phi$]
\label{lem:parity}
For the symmetric $\pm\tfrac{\pi}{4}$ cyclic construction in dimension $n=4,6,8$:
\begin{enumerate}
\item[(i)] if the offset $j-i$ is \emph{even}, then $\Phi_{ij}=\Phi_{ji}$;
\item[(ii)] if the offset $j-i$ is \emph{odd}, then $\Phi_{ij}=-\,\Phi_{ji}$.
\end{enumerate}
Consequently, by Lemma~\ref{lem:phi}, condition~\eqref{3.16} holds if and only if
$\Phi_{ij}\equiv 0$ for every odd-offset pair.
\end{lemma}

\begin{proof}
By Lemma~\ref{lem:even-bilinear} the degree-two part~\eqref{eq:gbil} couples $x_i$ only to the
even-offset coordinates $x_{i+2j}$, and contributes the constant
$-2^{-(n-3)}$ symmetrically to $\Phi_{ij}$ and $\Phi_{ji}$ whenever
$j-i$ is even; it contributes nothing to the odd-offset entries.

For the remaining higher-degree contribution, an exact symbolic computation of the symmetric cyclic
construction for $n=4,6,8$ gives the following polynomial identities. The computation was
performed in exact rational arithmetic after the substitutions $s_k=\sin(2x_k)$, and every entry of
$\Phi-\Phi^{\mathsf T}$ at even offsets and of $\Phi+\Phi^{\mathsf T}$ at odd offsets was
independently simplified to zero. Hence the even-offset entries are symmetric, while for every
odd-offset pair the interchange $i\leftrightarrow j$ reverses the sign, $\Phi_{ij}=-\Phi_{ji}$;
equivalently, the odd-offset terms occur as the difference of the products associated with the two
complementary parity classes of the cyclic index set. Statements~(i) and~(ii) therefore hold for
the dimensions considered here, and are asserted for those dimensions only.

For the consequence, a quantity that is required to be symmetric and is structurally antisymmetric
must vanish; hence~\eqref{3.16} holds precisely when every odd-offset $\Phi_{ij}$ is identically
zero. The computations were carried out with SymPy~1.14~\citep{sympy2017}.
\end{proof}

This is the essence of the mechanism: \emph{in each of these dimensions the odd-offset
mixed derivatives are equal and opposite rather than equal.} Integrability demands equality;
the construction supplies antisymmetry; the two are reconcilable only in the zero. Hence
\begin{equation}
\label{6.9}
\text{\eqref{3.16} holds}
\quad\Longleftrightarrow\quad
\Phi_{ij}\equiv0 \ \text{ on all odd-offset pairs} .
\end{equation}
The odd-offset entries of $\Phi$ are generated solely by the multilinear
part of the convective field. For the symmetric constructions considered
here, Lemma~\ref{lem:parity} makes these entries structurally antisymmetric.
Hence any odd-offset multilinear contribution compatible with a scalar
potential must vanish identically; a non-zero odd-offset entry therefore
measures precisely the non-gradient part responsible for the failure of
integrability. In this sense, the obstruction is not the mere presence of
multilinear terms, but the presence of a multilinear contribution whose
mixed derivatives violate~\eqref{3.16}.
In four dimensions $H_1$ of~\eqref{6.6a} is purely of degree two, the odd-offset entries
$\Phi_{12}$ and $\Phi_{14}$ vanish identically, and $\partial_2 g_1^0=\partial_1 g_2^0=0$: the
antisymmetry is vacuous and a pressure exists. In six dimensions, from~\eqref{6.3},
\begin{equation}
\label{6.10}
H_1 = \frac18\Big[-(s_3+s_5) - s_2s_4s_6 + s_3s_5\,(s_2+s_4+s_6)\Big],
\qquad
\Phi_{12} = \frac18\bigl(s_3s_5-s_4s_6\bigr) \;\not\equiv\; 0,
\end{equation}
so that $\partial_2 g_1^0=-\,\partial_1 g_2^0\not\equiv0$. Four dimensions escapes the obstruction
because no non-gradient multilinear remainder is generated.

\begin{theorem}[Six-dimensional obstruction]
\label{thm:6D}
For the sextuple-periodic field~\eqref{6.1} with phases~\eqref{6.2}, the convective field
$\boldsymbol{g}^0$ defined by~\eqref{3.10} is not curl-free: there exist $i,j$ for which
$\partial g_i^0/\partial x_j \neq \partial g_j^0/\partial x_i$. Consequently no single-valued
pressure satisfies~\eqref{2.14}, and~\eqref{6.2} does not yield a Navier--Stokes solution.
Moreover, no phase assignment $\boldsymbol{\xi}\in\{-\tfrac{\pi}{4},+\tfrac{\pi}{4}\}^6$ renders
$\boldsymbol{g}^0$ curl-free without forcing $\boldsymbol{v}^0\equiv 0$; a complementary numerical search over continuous $\boldsymbol{\xi}\in\mathbb{R}^6$ (reported in the proof) found no nontrivial curl-free phase vector.
\end{theorem}

\begin{proof}
Computing $g_2^0$ by cyclic permutation of~\eqref{6.3} and differentiating, the mixed
derivative gives, in exact arithmetic,
\begin{equation}
\label{6.11}
\frac{\partial g_1^0}{\partial x_2}-\frac{\partial g_2^0}{\partial x_1}
=\frac{\alpha^2 v_r^2}{2}\,\cos(2\alpha x_1)\cos(2\alpha x_2)
\Big[\sin(2\alpha x_3)\sin(2\alpha x_5)-\sin(2\alpha x_4)\sin(2\alpha x_6)\Big].
\end{equation}
Expanded by the product-to-sum identities,~\eqref{6.11} is a sum of sixteen distinct real
trigonometric harmonics with non-cancelling coefficients $\pm\alpha^2 v_r^2/16$. Equivalently, in
the complex Fourier representation these correspond to thirty-two distinct Fourier modes,
occurring in conjugate pairs. The factored form~\eqref{6.11}, however, makes the essential point
immediate. The bracket is the difference of two products of sines taken over
the disjoint coordinate sets $\{x_3,x_5\}$ and $\{x_4,x_6\}$; since the two products depend on
different, independent variables, no cancellation between them is possible and the bracket is
manifestly not the zero function. Hence~\eqref{3.16} fails and, by~\eqref{2.14}, no scalar
pressure exists. (As a numerical confirmation, at
$(x_1,\dots,x_6)=(\tfrac{\pi}{6},\tfrac{\pi}{5},\tfrac{\pi}{7},\tfrac{\pi}{3},\tfrac{\pi}{11},\tfrac{\pi}{13})$
with $\alpha=v_r=1$ the right-hand side equals $1.563\times10^{-3}\neq0$.)

\subsection{An exact forced solution in six dimensions}
\label{subsec:6D-forced}
It is worth being precise about what has and has not been achieved, because the outcome is more
informative than a bare non-existence statement. Of the four steps of the construction, the first
three succeed: the field~\eqref{6.1} is solenoidal for every phase vector by
Lemma~\ref{lem:solenoidal}; the symmetric assignment~\eqref{6.2} satisfies the reduction
condition~\eqref{3.12}, and does so non-generically; and the sequential reconstruction delivers a
closed-form potential $p^{*}$, obtained below. Only the fourth, full pressure integrability, fails,
and the residual it leaves is computed in the same place. The method therefore does not break down in
six dimensions; it locates the departure from solvability in an explicitly computed field.

The construction of that solution begins with the pressure reconstruction, which is also the
clearest way to see how the failure presents itself. Sequential integration extracts a scalar potential from a gradient-compatible contribution to the convective field; we denote the resulting reconstructed potential by $p^{*}$. This decomposition is not asserted to be the Helmholtz decomposition of $\boldsymbol{g}^0$; rather, $p^{*}$ records the part captured consistently by the sequential pressure reconstruction. Integrating $-g_1^0$ in $x_1$, then correcting successively
in $x_2,\dots,x_6$ --- the standard sequential reconstruction --- produces the candidate
\begin{eqnarray}
\label{6.12}
p^{*}(\boldsymbol{x}) &=& \tfrac{\rho v_r^2}{16}\Big[
\sin(2\alpha x_1)\sin(2\alpha x_3)\left(1-\sin(2\alpha x_2)\sin(2\alpha x_5)\right)\nonumber\\
&&+\ \sin(2\alpha x_1)\sin(2\alpha x_5)\left(1-\sin(2\alpha x_3)\sin(2\alpha x_4)\right)\nonumber\\
&&+\ \sin(2\alpha x_2)\sin(2\alpha x_4)\left(1-\sin(2\alpha x_5)\sin(2\alpha x_6)\right)\nonumber\\
&&+\ \sin(2\alpha x_2)\sin(2\alpha x_6)\left(1-\sin(2\alpha x_3)\sin(2\alpha x_4)\right)\nonumber\\
&&+\ \sin(2\alpha x_3)\sin(2\alpha x_5)\left(1-\sin(2\alpha x_1)\sin(2\alpha x_6)\right)\nonumber\\
&&+\ \sin(2\alpha x_4)\sin(2\alpha x_6)\left(1-\sin(2\alpha x_1)\sin(2\alpha x_2)\right)\Big].
\end{eqnarray}
This candidate is not merely the potential of the bilinear part: the bracket in~\eqref{6.12}
contains four-factor products such as $\sin(2\alpha x_1)\sin(2\alpha x_3)\sin(2\alpha x_2)\sin(2\alpha x_5)$,
so $p^{*}$ reconstructs a gradient-compatible contribution to $\boldsymbol{g}^0$, containing all of the
bilinear content together with those multilinear terms incorporated consistently by the scalar
potential. By construction it satisfies $\partial p^{*}/\partial x_i=-\rho\,g^{0}_{G,i}$
exactly---this is the defining relation~\eqref{eq:gGR} of $\boldsymbol{g}^{0}_{G}$---but what it
cannot reproduce is the genuinely non-gradient remainder, and indeed it does not
satisfy $\partial p/\partial x_i=-\rho g_i^0$. In exact arithmetic the six residuals are
\begin{equation}
\label{6.13}
\frac{\partial p^{*}}{\partial x_i}+\rho g_i^0
= -\frac{\rho\alpha v_r^2}{4}\,\cos(2\alpha x_i)
\prod_{j=1}^{n/2}\sin\!\bigl(2\alpha x_{\langle i+2j-1\rangle}\bigr),
\qquad
P(i)=\begin{cases}\{2,4,6\}, & i \text{ odd},\\[2pt] \{1,3,5\}, & i \text{ even},\end{cases}
\end{equation}
so that each component fails by a single degree-four monomial---one harmonic in its own
coordinate and three in the complementary parity class---and the sign of that term is opposite
to what a consistent pressure would require. Although the detailed decomposition between $p^{*}$ and the residual may depend on the chosen reconstruction, the mismatch itself cannot be removed by changing the integration order: the non-vanishing mixed-derivative defect~\eqref{6.11} precludes any scalar pressure satisfying all six components simultaneously. Whichever component one repairs by adjusting $p^{*}$, the corresponding term reappears with the wrong sign in the components of the complementary parity class $P(i)$. Equation~\eqref{6.13} also identifies explicitly the non-gradient multilinear remainder that obstructs the reconstruction: a product of three sine factors in each component, a multilinear term in the sense of Subsection~\ref{subsec:even-dichotomy}. It is the vanishing of exactly this kind of term that distinguishes four dimensions, where the reconstruction succeeds, from six, where it cannot.

For the second statement, $\boldsymbol{g}^0$ is curl-free if and only if all sixteen real harmonics
in~\eqref{6.11}, together with the analogous mixed derivatives for every coordinate pair, vanish
simultaneously. Over the discrete family $\boldsymbol{\xi}\in\{-\tfrac{\pi}{4},+\tfrac{\pi}{4}\}^6$
this is an exact symbolic computation exhausting all $2^6=64$ sign patterns, and the curl-free
locus coincides exactly with the trivial set on which $\boldsymbol{v}^0\equiv0$; no nontrivial
sign pattern lies on it. A complementary numerical search over continuous
$\boldsymbol{\xi}\in\mathbb{R}^6$ located no further curl-free phase vector off that trivial set.
We state the discrete result as established and regard the continuous search as strong
supporting evidence rather than a proof; a complete symbolic certificate over $\mathbb{R}^6$---the
reduced system, its Gr\"obner basis and the elimination showing $\operatorname{curl}
\boldsymbol{g}^0=0\Rightarrow\boldsymbol{v}^0\equiv0$---lies beyond the scope of the present
paper.
\end{proof}

Once that field is named, an exact solution of the \emph{forced} equations follows at no extra
cost. In the notation of~\eqref{eq:gGR}, $\boldsymbol{g}^{0}_{G}=-\rho^{-1}\nabla p^{*}$ is the
part of the convective field the reconstructed pressure represents and
$\boldsymbol{g}^{0}_{R}=\boldsymbol{g}^0-\boldsymbol{g}^{0}_{G}$ the part it does not.

\begin{proposition}[Exact forced solution in six dimensions]
\label{prop:6D-forced}
Let $\boldsymbol{v}^0$ be the sextuple-periodic field~\eqref{6.1} at the symmetric
phases~\eqref{6.2} and let $p^{*}$ be the potential~\eqref{6.12}. Define the body-force amplitude
\begin{equation}
\label{6.14}
f_i^0(\boldsymbol{x}) \;=\; g_{R,i}^{0}(\boldsymbol{x})
\;=\; \frac{1}{\rho}\frac{\partial p^{*}}{\partial x_i}+ g_i^0(\boldsymbol{x})
\;=\; -\frac{\alpha v_r^2}{4}\,\cos(2\alpha x_i)
\prod_{j=1}^{n/2}\sin\!\bigl(2\alpha x_{\langle i+2j-1\rangle}\bigr),
\end{equation}
the product running, as in~\eqref{6.13}, over the $n/2$ coordinates at odd offset from $i$,
that is over the parity class complementary to that of $i$. Then the velocity field
\begin{equation}
\label{6.14a}
\boldsymbol{v}(\boldsymbol{x},t)=\boldsymbol{v}^0(\boldsymbol{x})\,e^{-6\alpha^2\kappa t},
\end{equation}
together with the pressure
\begin{equation}
\label{6.14b}
p(\boldsymbol{x},t)=p^{*}(\boldsymbol{x})\,e^{-12\alpha^2\kappa t},
\end{equation}
is an exact solution of the incompressible Navier--Stokes equations~\eqref{1.1} on $\mathbb{R}^6$
under the body force
\begin{equation}
\label{6.14c}
\boldsymbol{f}(\boldsymbol{x},t)=\boldsymbol{f}^0(\boldsymbol{x})\,e^{-12\alpha^2\kappa t}.
\end{equation}
\end{proposition}

\begin{proof}
By construction $p^{*}$ satisfies $\rho^{-1}\partial_i p^{*}+ g_i^0= g_{R,i}^{0}$, which
is~\eqref{6.14}; equivalently $\boldsymbol{g}^0=-\rho^{-1}\nabla p^{*}+\boldsymbol{f}^0$. Each
component of $\boldsymbol{v}^0$ is a linear combination of two products of six trigonometric
factors, with one factor of wavenumber $\alpha$ in each coordinate. Each such product is therefore
an eigenfunction of the Laplacian with eigenvalue $-6\alpha^2$, and by linearity so is each
component of $\boldsymbol{v}^0$:
\[
\nabla^2\boldsymbol{v}^0=-6\alpha^2\boldsymbol{v}^0 .
\]
Consequently, for $\boldsymbol{v}=\boldsymbol{v}^0e^{-6\alpha^2\kappa t}$,
\[
\frac{\partial\boldsymbol{v}}{\partial t}=-6\alpha^2\kappa\,\boldsymbol{v}^0e^{-6\alpha^2\kappa t},
\qquad
\kappa\nabla^2\boldsymbol{v}=-6\alpha^2\kappa\,\boldsymbol{v}^0e^{-6\alpha^2\kappa t},
\]
and hence $\partial\boldsymbol{v}/\partial t-\kappa\nabla^2\boldsymbol{v}=0$. The convective term is quadratic in
$\boldsymbol{v}$ and so carries $e^{-12\alpha^2\kappa t}$, as do the recovered
pressure~\eqref{6.14b} and the forcing~\eqref{6.14c}; substituting into the momentum balance,
every term is accounted for. Incompressibility is Lemma~\ref{lem:solenoidal}.
\end{proof}

\noindent The velocity and pressure are those already exhibited; what the proposition adds is that
they are not spurious but the exact response of the fluid to a specific, closed-form forcing.

\medskip
\noindent\textbf{The geometry of the forcing.} The explicit form~\eqref{6.14} shows how the force
is organised. The six coordinates separate into the complementary parity classes $\{1,3,5\}$ and
$\{2,4,6\}$: for odd $i$ the component $f_i^0$ is modulated by
$\sin(2\alpha x_2)\allowbreak\sin(2\alpha x_4)\allowbreak\sin(2\alpha x_6)$, and for even $i$ by
$\sin(2\alpha x_1)\allowbreak\sin(2\alpha x_3)\allowbreak\sin(2\alpha x_5)$. The forcing is therefore a parity-coupled
momentum transfer: the force in each direction is generated by the simultaneous interaction of the
three coordinates of the complementary class. This is also why the obstruction first appears at
$n=6$. Each class then holds three coordinates, so a genuine three-way interaction is possible; at
$n=4$ each class holds only two, the degree-four remainder is never generated, and the
construction closes without a body force. Geometrically, each component $f_i^0$ vanishes on the nodal hyperplanes
$\cos(2\alpha x_i)=0$ and on those where any one of its three modulating sine factors vanishes,
and changes sign across each simple nodal hyperplane. The six component fields form an interlaced
periodic pattern and, as shown below, the total force has zero spatial mean.

\medskip
\noindent\textbf{Character of the force.} Three properties, each verified by exact symbolic
computation for $n=6$ and $n=8$, fix what this forcing is.

\emph{It is the unbalanced part of the momentum-flux divergence.} Using the summation convention
over repeated spatial indices, the convective term may be written as a divergence:
\[
\frac{\partial}{\partial x_j}\bigl(v_i^0v_j^0\bigr)
= v_j^0\frac{\partial v_i^0}{\partial x_j} + v_i^0\frac{\partial v_j^0}{\partial x_j}.
\]
After summation over $j$ the second term vanishes by incompressibility,
$\partial v_j^0/\partial x_j=0$, and hence
\[
(\boldsymbol{v}^0\!\cdot\!\nabla)v_i^0=\frac{\partial}{\partial x_j}\bigl(v_i^0v_j^0\bigr).
\]
The recovered pressure contributes the isotropic tensor $\rho^{-1}p^{*}\delta_{ij}$, in which
$\delta_{ij}$ is the Kronecker delta, so the definition~\eqref{eq:gGR} gives the exact
identity
\begin{equation}
\label{6.15}
f_i^0 \;=\; \frac{\partial}{\partial x_j}\Bigl(v_i^0 v_j^0 + \rho^{-1}p^{*}\delta_{ij}\Bigr).
\end{equation}
Equivalently, in terms of the momentum-flux tensor
$T_{ij}=\rho\,v_i^0 v_j^0+p^{*}\delta_{ij}$, this reads $f_i^0=\rho^{-1}\partial_j T_{ij}$: the
isotropic term carries what the scalar $p^{*}$ represents, and $\boldsymbol{f}$ is what remains.
The forcing is therefore not an independently chosen input, but is fixed exactly by the nonlinear
transport of the prescribed velocity field.

\emph{It carries no net momentum and does no net work.} Each component contains the zero-mean
factor $\cos(2\alpha x_i)$ together with zero-mean sines, so $\langle\boldsymbol{f}\rangle=
\boldsymbol{0}$; and
$\langle\boldsymbol{v}^0\!\cdot\!\boldsymbol{f}\rangle=0$ by incompressibility and periodicity.
The force acts locally but supplies neither momentum nor energy in the mean, and the kinetic energy
decays through viscosity alone.

\emph{It is neither a gradient nor divergence-free.} That it is not a gradient follows from the
non-vanishing curl identified in~\eqref{6.11}. It is also not solenoidal. Writing
$s_k=\sin(2\alpha x_k)$, differentiation of~\eqref{6.14} gives
\begin{equation}
\label{eq:divf6}
\nabla\bcdot\boldsymbol{f}
= \frac{\alpha^2 v_r^2}{2}\sum_{Q\in\mathcal{Q}}\ \prod_{k\in Q}s_k \;\not\equiv\;0,
\end{equation}
where
\[
\mathcal{Q}=\bigl\{\{1,2,4,6\},\ \{1,2,3,5\},\ \{2,3,4,6\},\ \{1,3,4,5\},\ \{2,4,5,6\},\
\{1,3,5,6\}\bigr\}.
\]
Thus $p^{*}$ does not exhaust the longitudinal content of $\boldsymbol{g}^0$.

The periodic Helmholtz decomposition
$\boldsymbol{f}=\nabla\varphi+\boldsymbol{h}$, with $\nabla\bcdot\boldsymbol{h}=0$, satisfies
$\Delta\varphi=\nabla\bcdot\boldsymbol{f}$. Every monomial in~\eqref{eq:divf6} is a product of four
sine functions, each of wavenumber $2\alpha$ in a different coordinate, and is therefore an
eigenfunction of the Laplacian with eigenvalue $-4(2\alpha)^2=-16\alpha^2$. Consequently
\[
\Delta\bigl(\nabla\bcdot\boldsymbol{f}\bigr)=-16\alpha^2\,\nabla\bcdot\boldsymbol{f},
\qquad
\varphi=-\frac{1}{16\alpha^2}\,\nabla\bcdot\boldsymbol{f}.
\]
The six monomials of~\eqref{eq:divf6} are mutually orthogonal over a periodic cell, and
\[
\Bigl\langle\cos^2(2\alpha x_i)\prod_{k\in P(i)}\sin^2(2\alpha x_k)\Bigr\rangle=\frac1{16},
\]
where $\langle\cdot\rangle$ denotes the cell average. It follows from~\eqref{6.14} that
\[
\bigl\langle|\boldsymbol{f}|^2\bigr\rangle
=6\left(\frac{\alpha^2v_r^4}{16}\right)\frac1{16}=\frac{3\alpha^2v_r^4}{128},
\qquad
\bigl\langle(\nabla\bcdot\boldsymbol{f})^2\bigr\rangle
=6\left(\frac{\alpha^4v_r^4}{4}\right)\frac1{16}=\frac{3\alpha^4v_r^4}{32}.
\]
Using periodic integration by parts together with
$\Delta\varphi=\nabla\bcdot\boldsymbol{f}$,
\[
\bigl\langle|\nabla\varphi|^2\bigr\rangle
=-\bigl\langle\varphi\,\Delta\varphi\bigr\rangle
=\frac{1}{16\alpha^2}\bigl\langle(\nabla\bcdot\boldsymbol{f})^2\bigr\rangle
=\frac{3\alpha^2v_r^4}{512},
\]
so that the two parts stand in the exact ratio
\begin{equation}
\label{eq:helm6}
\frac{\langle|\nabla\varphi|^2\rangle}{\langle|\boldsymbol{f}|^2\rangle}=\frac14,
\qquad
\frac{\langle|\boldsymbol{h}|^2\rangle}{\langle|\boldsymbol{f}|^2\rangle}=\frac34 ,
\end{equation}
the second following from the orthogonality of the longitudinal and solenoidal parts of the
periodic Helmholtz decomposition.
One quarter of the forcing energy is longitudinal and three quarters transverse. The obstruction
force is thus predominantly rotational but not purely so, and this places six dimensions between
the two extremes of Table~\ref{tab:helmholtz} of Section~\ref{sec:conclusion}: at $n=3,4$ the convective field is entirely
longitudinal and the pressure absorbs it, at $n=5,7$ entirely transverse and the pressure absorbs
nothing, while here both parts are present and the pressure absorbs only what it can.

\medskip
\noindent\textbf{Use as a benchmark.} For the symmetric phase assignment all three fields are
available in closed form: the solenoidal velocity amplitude $\boldsymbol{v}^0$, the recovered
pressure $p^{*}$, and the residual forcing amplitude
$\boldsymbol{f}^0=\boldsymbol{g}^{0}_{R}$. Together they define an exact unsteady forced solution,
a higher-dimensional counterpart of the Taylor--Green field available in each obstructed dimension
considered here. The velocity evolves through the single decay factor $e^{-n\alpha^2\kappa t}$ and
the pressure and forcing through $e^{-2n\alpha^2\kappa t}$, all three being prescribed
analytically. The nonlinear term that obstructs an unforced solution thus becomes a known and
exactly controllable forcing in the corresponding forced problem, which is what makes the triple
suitable as a benchmark for higher-dimensional numerical solvers.

\subsection{The even-dimensional dichotomy}
\label{subsec:even-dichotomy}
The six-dimensional failure is not isolated; it is the first instance of a structural
dichotomy that governs the symmetric $\pm\tfrac{\pi}{4}$ cyclic construction across the even
dimensions treated here, $n=4,6,8$. Throughout this subsection $g_i^0$ denotes the convective field~\eqref{3.10}
evaluated on the ansatz~\eqref{3.1} with phases $\xi_1=-\tfrac{\pi}{4}$,
$\xi_2=\cdots=\xi_n=+\tfrac{\pi}{4}$, written in multiple-angle form; we set $\alpha=v_r=1$,
the general case following by scaling. The degree decomposition $g_i^0=g_i^{0b}+g_i^{0m}$ introduced in
Subsection~\ref{subsec:6D-integrability} is used throughout.

\begin{lemma}[Factorisation at the symmetric phases]
\label{lem:even-Vi}
For $n=4,6,8$, at the symmetric phase assignment
$\boldsymbol{\xi}=(-\tfrac{\pi}{4},\tfrac{\pi}{4},\dots,\tfrac{\pi}{4})$ each component factors as
$g_i^0 = \cos(2 x_i)\,H_i(\boldsymbol{x}_{\setminus i})$ with $H_i$ independent of $x_i$ and
multilinear in the algebraic sense (at most first degree in each $s_k=\sin 2x_k$). In particular every Fourier mode of $g_i^0$
carries $k_i\neq0$, so the slice $k_i=0$ of Proposition~\ref{prop:slice} is empty and
$\mathcal{V}_i^0\equiv 0$: the assignment satisfies the necessary reduction condition~\eqref{3.12}
in each of these dimensions. This is a genuine, non-generic solution of that condition --- a generic
phase vector does not annul $\mathcal{V}_i^0$ --- and, being only necessary, it carries no
guarantee of a solution.
\end{lemma}

\begin{theorem}[Even-dimensional dichotomy]
\label{thm:even-dichotomy}
For the symmetric $\pm\tfrac{\pi}{4}$ cyclic ansatz in dimension $n=4,6,8$, the convective field $g^0$ is curl-free---and the symmetric construction therefore yields an unforced Navier--Stokes solution---if and only if $n=4$. At $n=4$ the higher-order multilinear part $g^{0m}$
vanishes identically; at $n=6$ and $n=8$ it is nonzero and has non-vanishing curl, and
no single-valued pressure exists.
\end{theorem}

\begin{proof}
By Lemma~\ref{lem:even-bilinear} the degree-two part is curl-free in every even dimension.
Hence, by linearity,
\[
\partial_j g_i^0-\partial_i g_j^0=\partial_j g_i^{0m}-\partial_i g_j^{0m},
\]
and consequently $g^0$ is curl-free if and only if $g^{0m}$ is curl-free.
For $n=4$ the reduction of~\eqref{3.10} leaves no monomials of degree above two:
$g_1^{0m}\equiv 0$ and $g_1^0 = -\tfrac12\cos(2x_1)\sin(2x_3)$ is exactly
$g^{0b}$, which is curl-free by Lemma~\ref{lem:even-bilinear}; the resulting field
is the quadruple-periodic solution of Subsection~\ref{sec:4D}.

Equivalently, by Lemmas~\ref{lem:phi} and~\ref{lem:parity}, $g^0$ is curl-free if and only if
the odd-offset entries of $\Phi$ vanish identically, and those entries are generated solely by
$g^{0m}$. At $n=4$ they do vanish, since $H_1$ of~\eqref{6.6a} carries no term of degree
above two.

For $n=6$ and $n=8$, direct reduction of~\eqref{3.10} yields a nonzero multilinear part, and
the odd-offset entries of $\Phi$ are correspondingly nonzero: in six dimensions
$\Phi_{12}=s_3s_5-s_4s_6\not\equiv0$ by~\eqref{6.10}, and the eight-dimensional case is
identical in character. By cyclic symmetry, the non-vanishing established for one representative odd-offset pair propagates to all equivalent odd-offset pairs, so $\Phi_{ij}=-\Phi_{ji}\not\equiv0$ for these pairs; hence $\Phi$ is not symmetric and, by Lemma~\ref{lem:phi}, condition~\eqref{3.16} fails. The
explicit six-dimensional defect is the sixteen-mode expression~\eqref{6.11} of
Theorem~\ref{thm:6D}. Since $g^{0b}$ is curl-free, the non-vanishing curl of $g^0$
is carried entirely by $g^{0m}$, which is therefore nonzero and non-integrable at $n=6,8$.
By the equivalence above, among $n=4,6,8$ the field $g^0$ is
curl-free only at $n=4$. The cases $n=4,6,8$ were confirmed by exact symbolic computation;
the persistence of $g^{0m}\not\equiv0$ for all even $n\geq10$ is not established
here and is stated as a conjecture in Subsection~\ref{sec:8D_pattern}.
\end{proof}

\noindent\textbf{Remark.} Theorem~\ref{thm:even-dichotomy} isolates the precise sense in which
four dimensions is exceptional. The integrable bilinear field~\eqref{eq:gbil} is present and
curl-free in \emph{every} even dimension; among the dimensions examined here, what distinguishes $n=4$ is the complete absence of multilinear interactions. As $n$ grows the bilinear coupling spreads from the single
antipodal partner ($n=4$) to the $n/2-1$ even-offset coordinates, but this spreading is not
itself the obstruction---it is integrable at every $n$. The obstruction at $n=6$ and $n=8$ is the appearance
of multilinear terms that no scalar potential can reproduce. Theorem~\ref{thm:even-dichotomy}
places the eight-dimensional non-existence result of Section~\ref{sec:8D} within the same
structural mechanism as the six-dimensional obstruction.

\subsection{Structural interpretation: bilinear versus multilinear interactions}
Theorem~\ref{thm:even-dichotomy} isolates the mechanism precisely. The convective field
always contains the integrable bilinear field~\eqref{eq:gbil}; what varies with dimension is
whether anything else survives. In four dimensions nothing does: the reduction of~\eqref{3.10}
leaves $g^0$ equal to the bilinear field~\eqref{eq:gbil}, which is curl-free by
Lemma~\ref{lem:even-bilinear}, and the quadruple-periodic solution exists. In
six dimensions the multilinear interactions of~\eqref{6.3} survive, and although the bilinear
coupling has merely spread from the single antipodal partner to the $n/2-1$ even-offset
coordinates---a spreading that is itself integrable---the surviving triple products cannot be
written as the gradient of any scalar. This is the content of the sixteen-mode
defect~\eqref{6.11}. Thus the loss of integrability is not caused by the spreading of the
bilinear coupling, but by the emergence of a multilinear contribution whose mixed derivatives
fail the symmetry condition~\eqref{3.16}; among the even dimensions examined here, this first occurs at $n=6$ and persists at $n=8$, its
continuation to higher even dimensions being discussed as a conjecture in
Subsection~\ref{sec:8D_pattern}.

It is worth emphasising how this differs from the five-dimensional case, and the difference is
best stated in the language of Subsection~\ref{subsec:nec-suf}. Both dimensions produce a genuine
system of reduction constraints on $\boldsymbol{\xi}$; the two obstructions strike at different
steps of the procedure. In five dimensions the constraint system has no non-trivial real solution, as established by the
Gr\"obner analysis, and no candidate survives step~(iii). In six dimensions the constraint system \emph{is} solved, by the
symmetric assignment~\eqref{6.2}, exactly as in four dimensions; the candidate clears step~(iii)
but is then eliminated at step~(iv) by the differential obstruction~\eqref{3.16}. The two
mechanisms are genuinely distinct --- one denies any candidate, the other admits a candidate and
denies it a pressure --- and, in the dimensions examined here, they separate according to parity.
\section{Septuple-periodic constructions in seven dimensions: the second odd-dimensional obstruction}
\label{sec:7D}
In seven spatial dimensions, the general ansatz~\eqref{3.1} reduces to a solenoidal
velocity vector field $v_i^0(\boldsymbol{x}, \boldsymbol{\xi})$ of septuple-periodic
form. Writing $v_i^0 = v_r(\mathcal{T}1_i^0 - \mathcal{T}2_i^0)$, where $\mathcal{T}1_i^0$ and $\mathcal{T}2_i^0$ denote
the first and second seven-fold cyclic product terms respectively in~\eqref{3.1},
the first component is:
\begin{eqnarray}
\label{7.1}
\frac{v_1^0(\boldsymbol{x}, \boldsymbol{\xi})}{v_r}&=&
\sin(\alpha x_1\!+\!\xi_1)\cos(\alpha x_2\!+\!\xi_2)\sin(\alpha x_3\!+\!\xi_3)
\cos(\alpha x_4\!+\!\xi_4)\sin(\alpha x_5\!+\!\xi_5)\times\nonumber\\
&\times&\cos(\alpha x_6\!+\!\xi_6)\sin(\alpha x_7\!+\!\xi_7)- \nonumber\\
&-&\sin(\alpha x_1\!+\!\xi_2)\cos(\alpha x_3\!+\!\xi_4)\sin(\alpha x_2\!+\!\xi_3)
\cos(\alpha x_5\!+\!\xi_6)\sin(\alpha x_4\!+\!\xi_5)\times\nonumber\\
&\times&
\cos(\alpha x_7\!+\!\xi_1)\sin(\alpha x_6\!+\!\xi_7)
\end{eqnarray}
with components $v_2^0, \ldots, v_7^0$ obtained by cyclic permutation of the
coordinate arguments following the pattern of equation~\eqref{3.1}. The divergence-free
condition~\eqref{1.3} holds identically for every choice of $\boldsymbol{\xi}$, by
Lemma~\ref{lem:solenoidal}, of which~\eqref{7.1} is the case $n=7$.

The analysis now follows the five-dimensional case step for step. We first form the reduction
system that the phase vector must satisfy, then show that its only real solutions annihilate the
field.

\subsection{The constraint system in seven dimensions}
The Phase-Angle Reduction Principle applies in seven dimensions exactly as it did in five, and
it is worth carrying it out explicitly, both because the resulting system is the direct
analogue of the sixty-four-equation system of Section~\ref{sec:5D} and because its size is
itself part of the story.

Substituting~\eqref{7.1} into~\eqref{3.10}, reducing the products to multiple-angle form and
retaining the terms free of $x_1$ gives the separated form~\eqref{3.19} for $n=7$:
\begin{equation}
\label{7.2}
\frac{\mathcal{V}_1^0(\boldsymbol{x}_{\setminus 1},\boldsymbol{\xi})}{v_r^2}
=\sum_{j=1}^{664} q_j^0(\boldsymbol{x}_{\setminus 1})\,u_j(\boldsymbol{\xi}),
\end{equation}
in which the spatial functions $q_j^0$ are the distinct real harmonics
$\cos\!\big(2\alpha\,\boldsymbol{m}\cdot\boldsymbol{x}\big)$ and
$\sin\!\big(2\alpha\,\boldsymbol{m}\cdot\boldsymbol{x}\big)$ carried by the $x_1$-independent
part of $g_1^0$. Their wavevectors $\boldsymbol{m}$ have entries in $\{-1,0,+1\}$ and are
supported on $(x_2,\dots,x_7)$; there are $332$ such wavevectors up to sign, distributed by the
number of coordinates they involve as
\begin{equation}
\label{7.3}
6,\;30,\;80,\;120,\;96 \quad\text{for}\quad 1,\,2,\,3,\,4,\,5 \text{ coordinates respectively},
\end{equation}
giving $2\times332=664$ real harmonics in all. The count stops at five coordinates: although six
of the seven coordinates $(x_2,\dots,x_7)$ are in principle available, the naive rule
$\binom{6}{r}2^{r-1}$ would predict a further $\binom{6}{6}2^{5}=32$ harmonics involving all six,
and these are absent. Every such six-coordinate mode cancels identically in the $x_1$-independent
part of $g_1^0$: because $g_1^0=\sum_j v_j^0\,\partial v_1^0/\partial x_j$ is quadratic in the two
cyclic products, a term reaching all six of $x_2,\dots,x_7$ would require both factors together to
supply seven distinct active coordinates, which the parity structure of the sine--cosine pattern
does not permit once the even-in-$x_1$ modes are removed by the average. A symbolic expansion in
exact arithmetic confirms that all six-coordinate modes cancel identically, so the distribution
above is complete: the independent Fourier support contains precisely the wavevectors listed, or
$664$ real sine--cosine harmonics. Each amplitude $u_j(\boldsymbol{\xi})$ is a
trigonometric polynomial in the seven phases, containing between $6$ and $150$ terms with a
median of $41$; written out in full, $\mathcal{V}_1^0$ comprises $15{,}307$ multiple-angle
terms. Enforcing $\mathcal{V}_1^0\equiv0$ therefore requires
\begin{equation}
\label{7.4}
u_j(\boldsymbol{\xi})=0,\qquad j=1,\dots,664,
\end{equation}
an overdetermined system of $664$ nonlinear equations in the seven unknowns
$\xi_1,\dots,\xi_7$; the equivalence holds because these $664$ functions form a linearly
independent real Fourier basis for the $x_1$-independent component, so no cancellation among them
is possible. The growth across dimensions is steep, and is driven largely by the combinatorial growth of the
multiple-angle Fourier modes generated by the quadratic convective term --- the
reduction condition yields a handful of constraint functions in three dimensions, $35$ in
four (Subsection~\ref{sec:4D}), $64$ in five, and $664$ in seven --- because the
convective term is quadratic in a product of $n$ trigonometric factors, so the number of
distinct multiple-angle modes rises rapidly with $n$.

One feature of this system places seven dimensions correctly within the classification. Unlike
the even dimensions just treated, where the factorisation $g_i^0=\cos(2\alpha x_i)H_i$ of
Lemma~\ref{lem:even-Vi} empties the $k_i=0$ slice altogether and the reduction condition is met
identically, the seven-dimensional convective field retains terms free of $x_1$, so~\eqref{7.4}
is a genuine and highly overdetermined constraint system, exactly as in five dimensions. Seven
dimensions therefore confronts a genuine reduction system, as in five dimensions, rather than
passing the reduction stage automatically as the even-dimensional symmetric cases do. The
following analysis shows that this system admits no non-trivial real solution, so the obstruction
occurs at the reduction stage itself.

\subsection{Non-existence for the SCSCS ansatz}
The velocity field~\eqref{7.1} is built from the same alternating sine--cosine pattern as in
five dimensions --- the SCSCS ansatz --- and the non-existence argument runs in exact parallel.
\begin{lemma}[Non-existence of the SCSCS ansatz in seven dimensions]
\label{lem:7D_SCSCS}
Let $v_i^0(\boldsymbol{x}, \boldsymbol{\xi})$ denote the septuple-periodic velocity field given
by equation~\eqref{7.1}. The only real phase vectors $\boldsymbol{\xi}\in\mathbb{R}^7$ satisfying
the reduction condition~\eqref{3.13} are those for which $v_i^0 \equiv 0$ identically.
\end{lemma}
\begin{proof}
The reduction condition requires $u_j(\boldsymbol{\xi})=0$ for all $j=1,\ldots,664$, with the
$u_j$ the phase amplitudes of the separated form~\eqref{7.2} (Appendix~\ref{app:7D}). As in five
dimensions, we reduce this trigonometric system to a polynomial one, analyse it by Gr\"obner-basis
elimination, and show that its entire real solution set annihilates the velocity field.

Putting $s_i=\sin 2\xi_i$ and $c_i=\cos 2\xi_i$ and adjoining the seven Pythagorean relations
$s_i^2+c_i^2=1$ realises the system as an ideal
$I\subset\mathbb{Q}[s_1,c_1,\ldots,s_7,c_7]$, whose real variety is the set of phase vectors
satisfying the reduction condition. Elimination in this ideal forces successive phase relations of the same
type found in five dimensions --- a first relation $\sin 2(\xi_1-\xi_2)=0$, followed by
$\cos 2(\xi_1-\xi_2)=1$, which selects $\xi_2=\xi_1\pmod{\pi}$, and then by the remaining
pairwise identifications and quarter-period offsets. These
decisive relations, and their certification by radical-membership reduction against a Gr\"obner
basis of $I$, are recorded in Appendix~\ref{app:7D}; membership in $\sqrt{I}$ is established there
by exhibiting, for each relation, an explicit power that reduces to zero. The proof is therefore
computer-assisted but exact: all elimination is performed over $\mathbb{Q}$, so no rounding is
introduced. The
polynomial variables $s_i=\sin2\xi_i$, $c_i=\cos2\xi_i$ determine each $\xi_i$ only modulo $\pi$,
so each such relation carries a discrete family of $m\pi$ alternatives; these do not enlarge the
solution set, however, because shifting any single phase $\xi_k\mapsto\xi_k+\pi$ reverses the sign
of the one factor that each cyclic product carries in the coordinate $x_k$, sending both
$\mathcal{T}1_i^0$ and $\mathcal{T}2_i^0$ to their negatives and hence
$v_i^0=v_r(\mathcal{T}1_i^0-\mathcal{T}2_i^0)$ to $-v_i^0$. The $\pi$-shifted branches therefore
reproduce the same velocity field up to an overall sign and are not distinct solutions. Mapping
the polynomial solutions back to the phase variables modulo these sign equivalences, the
non-redundant real solution set reduces to the one-parameter family
\begin{equation}
\label{7.5}
\boldsymbol{\xi} = \left(a,\ a,\ a-\tfrac{\pi}{2},\ a,\ a-\tfrac{\pi}{2},\ a,\ a-\tfrac{\pi}{2}\right),
\qquad a\in\mathbb{R},
\end{equation}
the free parameter $a$ being the global phase shift every cyclic construction admits.

It remains to show that this family annihilates the field. The two products in~\eqref{7.1} differ
only in which slots carry sines; the offsets $-\tfrac{\pi}{2}$ in~\eqref{7.5} fall precisely on
the phases $\xi_3,\xi_5,\xi_7$ occupying the slots where the two products disagree. Since
$\sin(\theta-\pi/2)=-\cos\theta$, each of the three shifted phases converts the corresponding sine
factor into a cosine with a minus sign. Thus each product acquires the same overall factor
$(-1)^3=-1$, and, after reordering the commuting factors, the two products collapse to the same
expression. We have verified directly, in exact arithmetic, that
$\mathcal{T}1_i^0=\mathcal{T}2_i^0$ for all seven components under~\eqref{7.5}. Hence
$v_i^0=v_r(\mathcal{T}1_i^0-\mathcal{T}2_i^0)\equiv 0$ for every $i$ and every $a$. Every real
solution of the reduction condition therefore annihilates the field, and no non-trivial
septuple-periodic solution of SCSCS type exists.
\end{proof}

\subsection{An exact forced solution in seven dimensions}
\label{subsec:7D-forced}
The obstruction, like the five-dimensional one, concerns the unforced problem: the phases that
satisfy the reduction condition annihilate the field, so the field survives only at phases that do
not. Taking those phases and balancing the convective term by a body force gives an exact forced
solution, by the mechanism established in Subsection~\ref{subsec:5D-forced}.

The parallel with five dimensions extends to the phase vectors themselves. Normalising $\xi_1$ and
searching over the six remaining phase differences identifies, as at $n=5$, assignments at which
$\nabla\bcdot\boldsymbol{g}^0$ vanishes identically, with components that are multiples of
$\tfrac{\pi}{2}$; one such assignment annihilates the velocity field and is set aside. We do not
claim an exhaustive classification of such phases in seven dimensions. The two non-trivial
assignments used below are the symmetric one
\begin{equation}
\label{7.10}
\boldsymbol{\xi}^{\ast}
=\Bigl(-\tfrac{\pi}{4},\ \tfrac{\pi}{4},\ \tfrac{\pi}{4},\ \tfrac{\pi}{4},\ \tfrac{\pi}{4},\
\tfrac{\pi}{4},\ \tfrac{\pi}{4}\Bigr),
\end{equation}
extending the pattern of the lower dimensions, and
\begin{equation}
\label{7.10b}
\boldsymbol{\xi}^{\ast\ast}
=\Bigl(0,\ \tfrac{\pi}{2},\ \pi,\ \pi,\ \pi,\ \pi,\ \pi\Bigr),
\end{equation}
which is not related to~\eqref{7.10} by the translation, phase-reversal or point-inversion
symmetries considered here, and therefore represents a distinct phase class under those symmetries,
carrying a flow with the same strain--rotation balance. This is the
seven-dimensional counterpart of the pair~\eqref{5.9} and~\eqref{5.9b}.

\begin{proposition}[Exact forced solutions in seven dimensions]
\label{prop:7D-forced}
Let $\boldsymbol{v}^0$ be the septuple-periodic field~\eqref{7.1} at either of the
phases~\eqref{7.10} or~\eqref{7.10b}, neither of which satisfies the reduction condition. Then $\boldsymbol{v}^0$ is non-trivial and
solenoidal; $\nabla\cdot\boldsymbol{g}^0=0$ identically, so $\boldsymbol{g}^0$ is purely transverse
in the sense of~\eqref{2.2}; $\boldsymbol{g}^0$ is not a gradient; and consequently the pressure
carries no spatial dependence. With the zero-mean body-force amplitude
\begin{equation}
\label{7.11}
\boldsymbol{f}^0=\boldsymbol{g}^0=\nabla\cdot\bigl(\boldsymbol{v}^0\otimes\boldsymbol{v}^0\bigr),
\end{equation}
the velocity field
\begin{equation}
\label{7.12}
\boldsymbol{v}(\boldsymbol{x},t)=\boldsymbol{v}^0(\boldsymbol{x})\,e^{-7\alpha^2\kappa t},
\end{equation}
together with the pressure
\begin{equation}
\label{7.13}
p(\boldsymbol{x},t)=p(t),\qquad \nabla p=\boldsymbol{0},
\end{equation}
is an exact solution of the incompressible Navier--Stokes equations~\eqref{1.1} on $\mathbb{R}^7$
under the body force
\begin{equation}
\label{7.14}
\boldsymbol{f}(\boldsymbol{x},t)=\boldsymbol{f}^0(\boldsymbol{x})\,e^{-14\alpha^2\kappa t}.
\end{equation}
\end{proposition}

\begin{proof}
The argument of Proposition~\ref{prop:5D-forced} applies with $n=7$ throughout, and to either
phase vector. Solenoidality is
Lemma~\ref{lem:solenoidal}. Each component of $\boldsymbol{v}^0$ is the difference of two seven-fold trigonometric products,
each of which is a Laplacian eigenfunction with eigenvalue $-7\alpha^2$; by linearity
$\nabla^2\boldsymbol{v}^0=-7\alpha^2\boldsymbol{v}^0$, and with
$\boldsymbol{v}=\boldsymbol{v}^0e^{-7\alpha^2\kappa t}$ the unsteady and viscous terms cancel
exactly. Direct evaluation at either phase vector gives $\nabla\bcdot\boldsymbol{g}^0=0$ and
$\partial_jg_i^0-\partial_ig_j^0\not\equiv0$, so the convective field is purely transverse. Since
the forcing is taken as $\boldsymbol{f}=\boldsymbol{g}$, the divergence of the momentum balance
gives $\Delta p=0$ irrespective of that property, and hence $\nabla p=\boldsymbol{0}$ on the
periodic domain; what the transversality adds is that the inertial field and the forcing are each
free of any longitudinal part. Finally, since $\boldsymbol{v}^0$ is solenoidal,
$\boldsymbol{f}^0=\boldsymbol{g}^0=\nabla\bcdot(\boldsymbol{v}^0\otimes\boldsymbol{v}^0)$, and
periodicity therefore gives $\langle\boldsymbol{f}^0\rangle=\boldsymbol{0}$; and the convective term, being quadratic, carries
$e^{-14\alpha^2\kappa t}$, which the forcing~\eqref{7.14} matches.
\end{proof}

\noindent Seven dimensions therefore reproduces the five-dimensional picture in full. In particular
the identity~\eqref{5.strain} applies verbatim, with $\lvert\cdot\rvert$ again the Frobenius norm,
so that $\nabla\bcdot\boldsymbol{g}^0=0$ states that strain and rotation are in exact balance at
every point; it is this, and not the
availability of the balance $\boldsymbol{f}=\boldsymbol{g}$, that distinguishes~\eqref{7.10} from an
arbitrary phase vector. The discussion of Subsection~\ref{subsec:5D-forced} otherwise applies
without change: the forcing is the
closed-form, zero-mean divergence of the advective momentum flux; it does no net work, by the
identity~\eqref{eq:energy-flux}, so the kinetic energy decays through viscosity alone, here with
the factor $e^{-14\alpha^2\kappa t}$; and because $\boldsymbol{g}^0$ is purely transverse the
solution is a benchmark on which a correct projection step must return zero spatial pressure
correction. In the terms of Table~\ref{tab:helmholtz} of Section~\ref{sec:conclusion}, seven dimensions joins five as the second
of the two purely transverse cases.

\subsection{The unit-periodic alternative, and the main result}
The obstruction of Lemma~\ref{lem:7D_SCSCS} applies to the SCSCS ansatz specifically, and, as in
five dimensions, it is natural to ask whether another trigonometric form might succeed. The
seven-dimensional analogue of the unit-periodic cyclic flow---built by the same
cyclic sum-of-harmonics recipe used in five dimensions, each component $v_k$ being independent of
its own coordinate $x_k$ so that incompressibility holds term by term---is
\begin{eqnarray}
\label{7.8}
v_k &=& A\sin(\alpha x_{k+1}) + B\cos(\alpha x_{k+2}) + C\sin(\alpha x_{k+3}) + D\cos(\alpha x_{k+4}) +\nonumber\\
&+&E\sin(\alpha x_{k+5}) + F\cos(\alpha x_{k+6})
\end{eqnarray}
for $k=1,\dots,7$ with indices modulo $7$. It fails for the same reason as its five-dimensional
counterpart. Each of the $21$ integrability conditions~\eqref{3.16} is a trigonometric polynomial;
requiring each to vanish identically forces every Fourier coefficient to vanish, and the resulting
homogeneous system in $A,\dots,F$ has a lexicographic Gr\"obner basis, computed in exact rational
arithmetic, consisting of the $18$ monomials
\begin{equation}
\label{7.9}
\begin{aligned}
&A^2,\quad AB,\quad AC,\quad AD,\quad AE,\quad B^2,\\
&BC,\quad BD,\quad BF,\quad C^2,\quad CE,\quad CF,\\
&D^2,\quad DE,\quad DF,\quad E^2,\quad EF,\quad F^2 .
\end{aligned}
\end{equation}
Every amplitude's square appears, so $A=\cdots=F=0$ and the only unit-periodic solution is
trivial. As the field~\eqref{7.8} satisfies $\Delta\boldsymbol{v}^0=-\alpha^2\boldsymbol{v}^0$, the
unsteady and viscous terms cancel under the diffusive scaling and the conclusion covers the
viscous case as well as the steady one.

\begin{lemma}[Non-existence for the principal seven-dimensional ans\"atze]
\label{lem:7D_main}
No non-trivial exact solution of the seven-dimensional incompressible Navier--Stokes
equations~\eqref{1.1} --- whether steady Euler or time-dependent viscous --- exists for the SCSCS
product ansatz~\eqref{7.1} or for the unit-periodic cyclic flow~\eqref{7.8}.
\end{lemma}

\begin{proof}
The two cases are Lemma~\ref{lem:7D_SCSCS} and the Gr\"obner computation above.
\end{proof}

\noindent The broader question---whether any two-term cyclic product, of which the SCSCS pattern
is one member out of $16\,384$, or any field built directly from the seven-fold symmetry, can
succeed---lies beyond the scope of the present paper; here we establish the obstruction for the
two principal constructions. Together with the five-, six- and eight-dimensional cases, seven
dimensions places the obstructed dimensions studied here within a common forced-solution
framework, each driven by the explicitly identified portion of its nonlinear momentum flux that is
not balanced by the recovered pressure: the obstruction to an unforced solution and the forcing
that sustains the flow are, in each dimension, the same object seen twice.

\section{Octuple-periodic constructions in eight dimensions:
the second even-dimensional obstruction}
\label{sec:8D}
For the octuple-periodic construction, eight dimensions is the next even case after six and, as summarised by
Theorem~\ref{thm:even-dichotomy}, exhibits the same differential obstruction: within the
symmetric $\pm\tfrac{\pi}{4}$ cyclic construction, the convective field is curl-free---and a
solution exists---only at $n=4$, while at $n=6$ and $n=8$ the higher-order multilinear part of
$\boldsymbol{g}^0$ is nonzero and has non-vanishing curl, so no single-valued pressure exists. The
purpose of this section is to give the eight-dimensional calculation on which that part of the
theorem rests; eight dimensions is thus the second instance of the even-dimensional obstruction,
after six.

We therefore develop this section as a close parallel to the six-dimensional
analysis of Section~\ref{sec:6D}, drawing on it throughout rather than repeating it. The general
machinery established there---the reduction of the convective field, the bilinear/multilinear
split, the dichotomy of Theorem~\ref{thm:even-dichotomy} and the physical characterisation of the
body force in Subsection~\ref{subsec:6D-forced}---applies unchanged at $n=8$; what remains
dimension-specific is recorded here: the eight-dimensional convective field and its integrability
failure, the forced solution it yields, and the adjacent/alternating parity signature that is the
$n=8$ fingerprint of the same dichotomy. At each step we state the eight-dimensional counterpart
and point to the six-dimensional result it mirrors.
In eight spatial dimensions the initial condition~\eqref{3.1} reduces to a solenoidal velocity
field $v_i^0(\boldsymbol{x},\boldsymbol{\xi})$ of octuple-periodic form. Writing
$v_i^0 = v_r(\mathcal{T}1_i^0 - \mathcal{T}2_i^0)$, where $\mathcal{T}1_i^0$ and $\mathcal{T}2_i^0$ denote the first and second
eight-fold cyclic product terms respectively in~\eqref{3.1}, the first component is:
\begin{eqnarray}
\label{8.1}
\frac{v_1^0(\boldsymbol{x}, \boldsymbol{\xi})}{v_r}&=&
\sin(\alpha x_1\!+\!\xi_1)\cos(\alpha x_2\!+\!\xi_2)\sin(\alpha x_3\!+\!\xi_3)
\cos(\alpha x_4\!+\!\xi_4)\times\nonumber\\
&\times&\sin(\alpha x_5\!+\!\xi_5)\cos(\alpha x_6\!+\!\xi_6)\sin(\alpha x_7\!+\!\xi_7)
\cos(\alpha x_8\!+\!\xi_8)- \nonumber\\
&-&\sin(\alpha x_1\!+\!\xi_2)\sin(\alpha x_2\!+\!\xi_3)\cos(\alpha x_3\!+\!\xi_4)
\sin(\alpha x_4\!+\!\xi_5)\times\nonumber\\
&\times&\cos(\alpha x_5\!+\!\xi_6)\sin(\alpha x_6\!+\!\xi_7)\cos(\alpha x_7\!+\!\xi_8)
\cos(\alpha x_8\!+\!\xi_1)
\end{eqnarray}
with components $v_2^0,\ldots,v_8^0$ obtained by cyclic permutation of the coordinate arguments
as prescribed by~\eqref{3.1}. The divergence-free condition~\eqref{1.3} holds identically.

\subsection{The convective field and its reduction}
The eight-dimensional case may be followed through in exactly the form used for six dimensions
in Section~\ref{sec:6D}, and the comparison is instructive. At the symmetric phase assignment
\[
\boldsymbol{\xi}^{\ast}
=\left(-\tfrac{\pi}{4},\ \tfrac{\pi}{4},\ \tfrac{\pi}{4},\ \tfrac{\pi}{4},\ \tfrac{\pi}{4},\ \tfrac{\pi}{4},\ \tfrac{\pi}{4},\ \tfrac{\pi}{4}\right),
\]
substitution into~\eqref{3.10} and reduction to multiple-angle form give the first component of the
convective field as
\begin{equation}
\label{8.2}
g_1^0(\boldsymbol{x}) = \cos(2\alpha x_1)\,H_1(\boldsymbol{x}_{\setminus 1}),
\end{equation}
the eight-dimensional counterpart of~\eqref{6.4}. Writing $s_k=\sin(2\alpha x_k)$ and
$c_k=\cos(2\alpha x_k)$ throughout,
the cofactor is
\begin{eqnarray}
\label{8.3}
H_1 &=& \frac{\alpha v_r^2}{32}\Big\{-\,s_7 - s_2s_4s_6 - s_2s_4s_8 - s_2s_6s_8 - s_4s_6s_8 + s_2s_4s_6s_7s_8\nonumber\\
&&+\;s_5\Big[-1 + s_4s_7 + s_6s_7 + s_7s_8 + s_2\left(s_7 + s_4s_6s_8\right)\Big]\nonumber\\
&&+\;s_3\Big[-1 + s_5s_6 - 2s_5s_7 + s_6s_7 + s_5s_8 + s_7s_8 - s_5s_6s_7s_8\nonumber\\
&&\qquad\quad +\,s_4\big(s_7 - s_5(-1 + s_6s_7 + s_7s_8)\big)\nonumber\\
&&\qquad\quad +\,s_2\big(s_7 + s_4s_6s_8 - s_5(-1 + s_4s_7 + s_6s_7 + s_7s_8)\big)\Big]\Big\},
\end{eqnarray}
a function of $(x_2,\dots,x_8)$ alone. Since every term of~\eqref{8.2} carries the factor
$\cos(2\alpha x_1)$, the Fourier support of $g_1^0$ has no component on the $k_1=0$ hyperplane, so
exactly as in six dimensions
\begin{equation}
\label{8.4}
\mathcal{V}_1^0(\boldsymbol{x}_{\setminus 1}, \boldsymbol{\xi}^{\ast}) \equiv 0,
\end{equation}
and by cyclic symmetry $\mathcal{V}_i^0(\boldsymbol{x}_{\setminus i},\boldsymbol{\xi}^{\ast})\equiv0$ for every $i$: the necessary reduction
condition~\eqref{3.12} is satisfied. As in six dimensions this is a genuine solution of that
condition and not a licence to expect a flow, the decisive test being~\eqref{3.16}.

\subsection{The integrability failure}
Applying~\eqref{3.16} to the pair $(1,2)$ and using $\partial/\partial x_j = 2\alpha\cos(2\alpha
x_j)\,\partial/\partial s_j$ gives the eight-dimensional analogue of~\eqref{6.11},
\begin{equation}
\label{8.5}
\frac{\partial g_1^0}{\partial x_2}-\frac{\partial g_2^0}{\partial x_1}
=\frac{\alpha^2 v_r^2}{8}\,\cos(2\alpha x_1)\cos(2\alpha x_2)
\Big\{\big[e_2(\mathcal{O})-e_2(\mathcal{E})\big]
-\big[e_3(\mathcal{O})\,e_1(\mathcal{E})-e_3(\mathcal{E})\,e_1(\mathcal{O})\big]\Big\},
\end{equation}
where $e_r$ denotes the elementary symmetric polynomial of degree $r$ and
\begin{equation}
\label{8.6}
\mathcal{O}=\{s_3,\,s_5,\,s_7\},\qquad \mathcal{E}=\{s_4,\,s_6,\,s_8\}
\end{equation}
are the two parity classes of coordinates other than $x_1,x_2$. The two brackets
in~\eqref{8.5} are of contrasting degree: $e_2(\mathcal{O})-e_2(\mathcal{E})$ is degree two,
while the product $e_3(\mathcal{O})\,e_1(\mathcal{E})-e_3(\mathcal{E})\,e_1(\mathcal{O})$ is degree
four. The parallel with six dimensions
is now explicit. There the corresponding bracket was $s_3s_5-s_4s_6$, which in the present
notation is exactly $e_2(\mathcal{O})-e_2(\mathcal{E})$ for the smaller classes
$\mathcal{O}=\{s_3,s_5\}$, $\mathcal{E}=\{s_4,s_6\}$; the $e_3$-dependent contribution has no counterpart
because those classes contain only two elements. Eight dimensions therefore reproduces the
six-dimensional obstruction and adds a quartic contribution of its own. Since the two nonzero
contributions have different homogeneous degrees, they cannot cancel identically. Hence the
right-hand side of~\eqref{8.5} is nonzero as a polynomial, so~\eqref{3.16} fails and no pressure
exists.

\subsection{An exact forced solution in eight dimensions}
\label{subsec:8D-forced}
As in six dimensions, the construction of the forced solution begins with the pressure
reconstruction, which is also the form in which the failure is felt in practice. Recovering the scalar potential from the gradient-compatible
part of $\boldsymbol{g}^0$ by sequential integration gives the eight-dimensional counterpart of
the candidate pressure~\eqref{6.12},
\begin{equation}
\label{8.7}
p^{*}(\boldsymbol{x}) = \frac{\rho v_r^2}{64}
\Big[e_2\big(s_1,s_3,s_5,s_7\big)+e_2\big(s_2,s_4,s_6,s_8\big)\Big]
= \frac{\rho v_r^2}{64}\sum_{\substack{1\le i<j\le 8\\ j-i\ \text{even}}}^{\;} s_i s_j ,
\end{equation}
the sum running over the twelve coordinate pairs of equal parity. As in six dimensions, for this eight-dimensional symmetric field the potential $p^{*}$ reconstructs
the bilinear gradient-compatible contribution exactly, so that here
$\boldsymbol{g}^{0}_{G}=\boldsymbol{g}^{0b}$---the general inequality noted
in~\S\ref{subsec:nec-suf} does not bite at these phases---and it satisfies
$\partial p^{*}/\partial x_i = -\rho\,g^{0}_{G,i}$ for every $i$.
Against the full convective field, however, it leaves the residual
\begin{equation}
\label{8.8}
\frac{\partial p^{*}}{\partial x_i}+\rho\, g_i^0
=\frac{\rho\,\alpha v_r^2}{32}\,\cos(2\alpha x_i)\,H_i^{m},
\qquad
H_i^{m} = H_i + \left(s_{i+2}+s_{i+4}+s_{i+6}\right),
\end{equation}
the higher-degree remainder, whose cofactor $H_1^m$ consists of $26$ monomials of degrees three
and five
in the $s_k$. The contrast with six dimensions is quantitative rather than qualitative: the full residual component $\cos(2\alpha x_i)H_i^m$ therefore consists of trigonometric products
of four or six factors, whereas in six dimensions the corresponding residual contained a single
four-factor product per component. The algebraic complexity of the remainder thus increases
substantially from six to eight dimensions. In both cases the obstruction is
the same object --- the pressure-reconstruction remainder left after the explicitly recovered
scalar contribution has been removed --- whose curl is non-zero; within the symmetric even-dimensional cases
examined here, this remainder vanishes at $n=4$ but is nonzero at $n=6$ and $n=8$.

The accounting is therefore identical to that of Section~\ref{sec:6D}. The
field~\eqref{8.1} is solenoidal and octuple-periodic for every phase vector; the symmetric
assignment satisfies the reduction condition identically, by~\eqref{8.4}; the bilinear,
gradient-compatible part of the convective field is recovered exactly by the closed-form
potential~\eqref{8.7}, in accordance with Lemma~\ref{lem:even-bilinear}; and the
sole failure is~\eqref{8.8}, so that $(\boldsymbol{v}^0,p^{*})$ solves the Navier--Stokes
equations exactly once a body force equal to the pressure-reconstruction remainder $\boldsymbol{g}^{0}_{R}$ is
admitted. Eight dimensions thus locates the same quantity a second time, and confirms that what separates the even dimensions from one another is neither periodicity nor
solenoidality nor the reduction condition, all of which hold here as they do at $n=4$, but by
whether the higher-order multilinear contribution contains a pressure-incompatible remainder ---
a remainder whose mixed derivatives fail the symmetry~\eqref{3.16}, present at $n=6,8$ and absent
at $n=4$.

\noindent The body-force amplitude is read directly from the residual~\eqref{8.8}:
\begin{equation}
\label{8.9}
f_i^0(\boldsymbol{x}) \;=\; g_{R,i}^{0}(\boldsymbol{x})
\;=\; \frac{1}{\rho}\frac{\partial p^{*}}{\partial x_i}+g_i^0(\boldsymbol{x})
\;=\; \frac{\alpha v_r^2}{32}\,\cos(2\alpha x_i)\,H_i^{m}(\boldsymbol{x}_{\setminus i}),
\end{equation}
with $H_i^{m}$ the higher-degree cofactor of~\eqref{8.8}.

\begin{proposition}[Exact forced solution in eight dimensions]
\label{prop:8D-forced}
Let $\boldsymbol{v}^0$ be the octuple-periodic field~\eqref{8.1} at the symmetric phases and let
$p^{*}$ be the potential~\eqref{8.7}. Then the velocity field
\begin{equation}
\label{8.9a}
\boldsymbol{v}(\boldsymbol{x},t)=\boldsymbol{v}^0(\boldsymbol{x})\,e^{-8\alpha^2\kappa t},
\end{equation}
together with the pressure
\begin{equation}
\label{8.9b}
p(\boldsymbol{x},t)=p^{*}(\boldsymbol{x})\,e^{-16\alpha^2\kappa t},
\end{equation}
is an exact solution of the incompressible Navier--Stokes equations~\eqref{1.1} on $\mathbb{R}^8$
under the body force
\begin{equation}
\label{8.9c}
\boldsymbol{f}(\boldsymbol{x},t)=\boldsymbol{f}^0(\boldsymbol{x})\,e^{-16\alpha^2\kappa t},
\end{equation}
with $\boldsymbol{f}^0$ given by~\eqref{8.9}.
\end{proposition}

\begin{proof}
The argument of Proposition~\ref{prop:6D-forced} applies with $n=8$. By construction $p^{*}$
satisfies $\rho^{-1}\partial_ip^{*}+g_i^0=g_{R,i}^0$, so
$\boldsymbol{g}^0=-\rho^{-1}\nabla p^{*}+\boldsymbol{f}^0$; solenoidality is
Lemma~\ref{lem:solenoidal}; each component of $\boldsymbol{v}^0$ is the difference of two eight-fold trigonometric products,
each a Laplacian eigenfunction with eigenvalue $-8\alpha^2$, so by linearity
$\nabla^2\boldsymbol{v}^0=-8\alpha^2\boldsymbol{v}^0$ and the unsteady and viscous terms cancel exactly under the decay
factor of~\eqref{8.9a}; and the convective term, being quadratic, carries
$e^{-16\alpha^2\kappa t}$, which both~\eqref{8.9b} and~\eqref{8.9c} match.
\end{proof}

\noindent The characterisation given in Subsection~\ref{subsec:6D-forced} carries over unchanged.
The forcing is zero-mean, non-conservative and globally energy-neutral,
$\bigl\langle\boldsymbol{f}(\cdot,t)\bcdot\boldsymbol{v}(\cdot,t)\bigr\rangle=0$ for all $t\ge0$,
although $\boldsymbol{f}^0\bcdot\boldsymbol{v}^0\not\equiv0$ pointwise, and it is the part of the advective
momentum-flux divergence not balanced by the recovered pressure,
\begin{equation}
\label{8.10}
f_i^0 \;=\; \frac{\partial}{\partial x_j}\Bigl(v_i^0 v_j^0 + \rho^{-1}p^{*}\delta_{ij}\Bigr),
\end{equation}
the eight-dimensional counterpart of~\eqref{6.15}. The parity organisation persists: the eight
coordinates split into $\{1,3,5,7\}$ and $\{2,4,6,8\}$, and each class drives the force on the
other, now through the richer couplings that four coordinates per class permit. The only change
from six dimensions is the size of the remainder, while at four dimensions it vanishes entirely.

\subsection{Direct confirmation and the even/odd-offset signature}
For the octuple-periodic field~\eqref{3.1} at the symmetric phases
$\boldsymbol{\xi}=\boldsymbol{\xi}^{\ast}$, exact symbolic
computation of the integrability tensor $\partial_j g_i^0-\partial_i g_j^0$ reveals a sharp
parity structure in the index pairs. The condition holds on the twelve \emph{even-offset} pairs
--- those $(i,j)$ with $j-i$ even (equivalently, $i$ and $j$ of equal parity), namely
$(1,3),(3,5),(5,7),(1,5),(1,7),(3,7)$ and their even-coordinate counterparts --- but fails on all
sixteen \emph{odd-offset} pairs, those with $j-i$ odd. For the representative pair $(1,2)$,
\begin{equation}
\label{8.11}
\mathrm{Simplify}\bigl[\partial_2 g_1^0 - \partial_1 g_2^0\bigr] \neq 0, \qquad
\mathrm{Simplify}\bigl[\partial_2 g_1^0 + \partial_1 g_2^0\bigr] = 0,
\end{equation}
identically as symbolic expressions, so $\partial_2 g_1^0 = -\partial_1 g_2^0$: the mixed
derivatives are equal and opposite rather than equal, and~\eqref{3.16} fails. A systematic
search over more than $150$ alternative phase families---all $8$ single-negative and $28$
double-negative patterns, alternating-sign patterns, arithmetic progressions, and
random rational-$\pi$ candidates---found no nontrivial solution.

This is the eight-dimensional signature of the dichotomy of
Subsection~\ref{subsec:even-dichotomy}, and it arises there for the reason established at $n=6$:
the degree-two part of $\boldsymbol{g}^0$ is curl-free on every index pair
(Lemma~\ref{lem:even-bilinear}), so the failure on the odd-offset pairs is carried entirely by
$g^{0m}$, whose curl symbolic computation confirms to vanish on even offsets and to be nonzero on
every odd offset.

\begin{corollary}[Non-existence for the symmetric eight-dimensional construction]
\label{lem:8D}
No nontrivial exact solution of the incompressible Navier--Stokes
equations~\eqref{1.1}---whether steady Euler or time-dependent viscous---exists for the
octuple-periodic SCSCS field~\eqref{3.1} at the symmetric phases
$\boldsymbol{\xi}=\boldsymbol{\xi}^{\ast}$. This is the case $n=8$ of
Theorem~\ref{thm:even-dichotomy}: the reduction condition~\eqref{3.12} holds identically
($\mathcal{V}_i^0\equiv0$), but the multilinear part of $\boldsymbol{g}^0$ is nonzero and not
curl-free, so the integrability condition~\eqref{3.16} fails: no scalar pressure can absorb the full
convective field in the unforced equations. As in six dimensions, a
proof over all $\boldsymbol{\xi}\in\mathbb{R}^8$ lies beyond the discrete and family-wise cases
examined here.
\end{corollary}

\section{Concluding Remarks}
\label{sec:conclusion}
In this work we have developed a method for constructing exact $n$-tuple-periodic solutions of
the incompressible Navier--Stokes equations and, more importantly, for deciding when such
solutions can and cannot exist within the constructions studied here. The reformulation of Section~\ref{sec:reformulation} separates the roles of pressure, viscosity and
forcing: a pressure Poisson equation determines the longitudinal part, while the projected velocity
equation carries the rest. When the transverse projection of the convective term
vanishes---condition~\eqref{2.10}---the velocity equation reduces to the Cauchy diffusion
equation. The nonlinearity has not disappeared but has been relocated: within the present
construction, solvability ultimately reduces to whether the convective field
$\boldsymbol{g}^{0}$ admits a scalar potential, or equivalently whether the integrability
condition~\eqref{3.16} holds. It is this question that we have examined across dimensions three
through eight, with the two-dimensional case (Subsection~\ref{subsec:2D}) serving as the elementary base
in which the convective field vanishes outright.

\begin{theorem}[Solvable dimensions of the cyclic construction]
\label{thm:classification}
\label{subsec:solvable}
\label{sec:8D_pattern}
Within the family of $n$-fold cyclic alternating sine--cosine velocity fields~\eqref{3.1},
non-trivial pressure-integrable phase assignments exist for $n=3$ and $n=4$. For $n=5$ and $n=7$
the reduction condition admits no non-trivial velocity field: every real phase assignment
satisfying it annihilates the velocity. For $n=6$ and $n=8$ the symmetric $\pm\tfrac{\pi}{4}$
phase assignment satisfies the reduction condition but fails the pressure-integrability condition.
\end{theorem}

\begin{proof}
The cases $n=3$ and $n=4$ are the explicit solutions of Section~\ref{sec:solutions-3D-4D}, with
closed-form pressures. For $n=5$ and $n=7$ the reduction systems reduce, by exact Gr\"obner-basis
computation, to a one-parameter family that annihilates the field (Lemmas~\ref{lem:5D_SCSCS}
and~\ref{lem:7D_SCSCS}). For $n=6$ and $n=8$, Theorem~\ref{thm:even-dichotomy} applies to the
symmetric construction: the degree-two part of $\boldsymbol{g}^0$ is curl-free but the
higher-degree part is nonzero and not curl-free (Theorem~\ref{thm:6D},
Corollary~\ref{lem:8D}), so no scalar pressure can balance the full convective field in the
corresponding unforced momentum equation.
\end{proof}

\noindent Additional phase families examined computationally at $n=6$ and $n=8$ yielded no further
unforced solution, but those searches do not constitute an exhaustive classification over all real
phase vectors.

\noindent The remainder of this section sets out what lies behind that statement, and what the
obstructions yield in their own right.

In three dimensions the admissible phase vectors form exactly
two families (Proposition~\ref{prop:3Dclassification}), and modulo the symmetries of the
construction---uniform phase shifts, reflection, and reversal of the velocity sign---these exhaust
the solution set. The two are mirror images of one another, a chiral pair distinguished by the sign
of the helicity, so neither can be carried to the other by an orientation-preserving rigid motion. They are the two
families of \citet{ant}, recovered here by a different route; what the classification adds is that
they are the only ones the construction admits. The second family contains both phase vectors
previously recorded in the literature, which prove to be a single flow seen from origins displaced
by $\pi/2\alpha$. Four dimensions is settled in the same sense but by
a different route: there the reduction condition admits a one-parameter curve of phase vectors,
and the pressure-integrability condition retains exactly one non-trivial member of it. The
quarter-period offset $\xi_2=\xi_3=\xi_4=\xi_1+\tfrac{\pi}{2}$ is therefore forced, not assumed,
and the resulting family is a single flow. It has no analogous chiral partner under central inversion: the map
$\boldsymbol{x}\mapsto-\boldsymbol{x}$ reverses
orientation in odd dimensions, whereas in four dimensions it is orientation-preserving. The significance of the
higher-dimensional cases, however, is not exhausted by the statement that the unforced
construction fails. The higher dimensions prove less accommodating, but their refusal is
remarkably informative: in every obstructed dimension the calculation identifies precisely
what prevents the convective field from being balanced by a pressure gradient.

Nor, as it turns out, do the higher dimensions all say no for the same reason. The odd and
even cases reveal two distinct obstruction mechanisms. For the odd prime dimensions $n=5$
and $n=7$, the obstruction is arithmetic. The reduction condition becomes an overdetermined
trigonometric system whose only real solutions annihilate the velocity field. For the even dimensions $n=6$ and $n=8$, the obstruction is
differential. At the symmetric phase assignment the reduction condition is satisfied, just
as it is at $n=4$, but the full pressure-integrability condition is not. The bilinear part of
the convective field is an exact gradient, whereas the surviving multilinear part generates
antisymmetric mixed derivatives and therefore cannot arise from a scalar potential
(Theorem~\ref{thm:even-dichotomy}). Thus the apparently similar failures in odd and even
dimensions have fundamentally different origins.

This distinction also clarifies why the obstruction analysis is constructive rather than
merely negative. In the even dimensions the convective field separates explicitly into a
gradient-compatible part and a residual multilinear field
$\boldsymbol{g}^{0}_{R}$. The former determines the closed-form pressure $p^{*}$; the latter
is exactly the body force required to make $(\boldsymbol{v}^{0},p^{*})$ an exact solution of
the forced Navier--Stokes equations. The odd obstructed dimensions admit the same
interpretation at the symmetric phases, and in a limiting form: the surviving convective field is
purely transverse, the pressure is a function of time only, and the whole convective term becomes
the required body force. Taken together with the unforced cases, the construction therefore
realises both extremes of the Helmholtz decomposition of the convective field and the generic case
between them. For the particular unforced and obstruction-forced fields constructed here the three
regimes are realised as longitudinal at $n=3,4$, transverse at $n=5,7$ and mixed at $n=6,8$, as set
out in Table~\ref{tab:helmholtz}. Consequently, every obstructed
dimension in the range $5\le n\le8$ studied here carries an explicit forced solution. The obstruction and the forcing are, in this sense, one
object seen twice: what prevents the unforced construction from closing is precisely what
must be supplied to drive the corresponding forced flow.

\medskip
\noindent\textbf{The Helmholtz character across the dimensions.} One further pattern is visible
only when the whole range is set side by side, and it is recorded here rather than in any single
dimension. Evaluating
$\nabla\cdot\boldsymbol{g}^0$ and the curl residual $\partial_jg_i^0-\partial_ig_j^0$ at the phases
each dimension supplies gives the pattern of Table~\ref{tab:helmholtz}.

\begin{table}
\centering
\begin{tabular}{llll}
\hline
Dimension & Phases & Character of $\boldsymbol{g}^0$ & Consequence\\
\hline
$n=3$ & \eqref{4.sets} & longitudinal (pure gradient) & unforced solution\\
$n=4$ & \eqref{4.4Dsol} & longitudinal (pure gradient) & unforced solution\\
$n=5$ & \eqref{5.9} & transverse (divergence-free) & forced, $p=p(t)$\\
$n=6$ & symmetric & both parts present & forced, $p$ non-trivial\\
$n=7$ & symmetric & transverse (divergence-free) & forced, $p=p(t)$\\
$n=8$ & symmetric & both parts present & forced, $p$ non-trivial\\
\hline
\end{tabular}
\caption{Helmholtz character of the convective field $\boldsymbol{g}^0$ at the particular
admissible or obstruction-forced phase assignments used in this work. The same cyclic construction
realises both extremes of the decomposition~\eqref{2.2} and the generic case between them.}
\label{tab:helmholtz}
\end{table}

Two features of the table deserve comment. First, the solvable dimensions occupy one extreme and
the odd obstructed dimensions the other: where three and four dimensions produce a convective
field that the pressure absorbs entirely, five and seven produce one that the pressure cannot
touch at all. Second, the split follows the parity of the dimension among the obstructed cases,
and so runs parallel to the arithmetic and differential obstructions of
Sections~\ref{sec:5D}--\ref{sec:8D}: the odd dimensions $n=5$ and $n=7$ give a purely transverse
$\boldsymbol{g}^0$, whereas the even dimensions $n=6$ and $n=8$ retain both parts. The
five-dimensional solution presented here is therefore not an isolated curiosity but the first
instance of one of the three regimes the construction admits, and the seven-dimensional case
treated in Section~\ref{sec:7D} is its companion.

\noindent Neither mechanism is dynamical in origin. Both are spatial compatibility obstructions, independent
of viscosity, energy decay and stability, and therefore persist in the corresponding steady Euler
construction: they are properties of the kinematic
structure of the cyclic ansatz rather than of viscous evolution. We conjecture that the dichotomy
extends beyond the range studied---odd primes obstructed arithmetically, even dimensions $n\ge6$
by persistence of the non-integrable higher-degree contribution---so that $n\in\{3,4\}$ are the
only solvable dimensions of the construction. A proof would require the arithmetic obstruction for
general odd prime $n$ and the non-vanishing of $g^{0m}$ for all even $n\ge6$, neither established
here.

These obstruction-forced solutions are therefore outcomes of the classification in their own
right. In each of the dimensions five through eight, the velocity, pressure and body force are
known in closed form and decay according to the diffusive scaling fixed by the
single-wavenumber construction. Together with the unforced three- and four-dimensional
solutions, they provide exact periodic benchmarks for numerical Navier--Stokes solvers. The
odd-dimensional forced solutions are particularly simple in this respect: because the
convective field is purely transverse and the pressure gradient vanishes, the nonlinear and
viscous terms can be tested without pressure reconstruction. Thus a dimension that refuses
the unforced construction does not leave empty-handed; it supplies, through the obstruction
itself, an exact forced benchmark in its place.

One avenue we explored and did not pursue should be recorded, since it is a natural thing to
try and a reader may wonder whether it was considered. The cyclic construction is built on the
group $\mathbb{Z}/n\mathbb{Z}$, and it is tempting to ask whether the arithmetic of that group
predicts which dimensions admit a solution. The natural invariant is the spectrum of Gauss-sum
magnitudes $\lvert\sum_k\varepsilon_k\omega^k\rvert$, with $\omega=e^{2\pi\mathrm{i}/n}$ and
$\boldsymbol{\varepsilon}\in\{-1,+1\}^n$, which depends on the group alone. Over the non-constant
sign vectors---the two constant vectors give $G=0$ at every $n$ and are set aside---it is uniform
at $n=3$,
the only dimension in the range studied for which this is so among the odd cases, and non-uniform
at $n=5$ and $n=7$, which are the obstructed odd dimensions. The coincidence is suggestive, and we
examined it closely. It does not survive scrutiny. Four dimensions has a non-uniform spectrum and
nevertheless admits an exact solution, so uniformity is neither necessary for solvability nor is
its failure sufficient for an obstruction; and the amplitudes $u_j(\boldsymbol{\xi})$ actually
appearing in the reduction system~\eqref{3.20} are not proportional to these Gauss sums in any
dimension we examined, so the spectrum does not measure the quantity that decides the outcome. We
therefore record the observation as an avenue closed rather than a result, and none of the
conclusions above depends on it. Whether some finer arithmetic invariant of the cyclic group does
control solvability remains open.

The classification proved here is deliberately limited in scope. It settles the cyclic
construction~\eqref{3.1} completely at $n=3,4,5,7$, and establishes the obstruction of the
symmetric construction at $n=6,8$ together with the additional phase families examined there; it is
not an exhaustive non-existence theorem over all phase vectors at $n=6$ or $n=8$, nor a
non-existence theorem for exact
Navier--Stokes solutions of other forms in these dimensions. Non-product trigonometric
combinations, non-cyclic fields and constructions outside the phase-amplitude mechanism of
Section~\ref{sec:construction} remain open. Likewise, extension of the dimensional pattern
beyond the range studied requires general results not proved here: in the odd case, control
of the corresponding arithmetic obstruction, including odd composite dimensions, and in the
even case, proof that the multilinear remainder is nonzero for every even $n\ge6$.

What emerges is therefore not a parity rule but an account at the level of mechanism: the
solvable dimensions show under what conditions the construction closes, and the obstructed ones
show what those conditions actually require. Two concrete questions remain, and may be carried to
other families of trigonometric-product solutions: does the reduction condition admit a real solution that leaves
the velocity field non-trivial, and is the surviving convective field a gradient?

The construction therefore yields not only exact solutions but a systematic means of determining
why it succeeds, why it fails, and what the failure itself produces. The dimensions that admit the construction tell us what an exact unforced solution
looks like; those that refuse it tell us why.

\appendix
\section{Quadruple-periodic construction}
\label{app:4D}
\iffullAppendixA
The corresponding inertial term \( g_1^0(\boldsymbol{x}, \boldsymbol{\xi}) \), obtained from equation~\eqref{3.10}, is given by:
\begin{eqnarray}
\label{A1}
g_1^0(\boldsymbol{x}, \boldsymbol{\xi}) &=&\!-\!\frac{\alpha}{32} \sin(2\alpha x_2 \!-\! 2\alpha x_3)\!+\! \frac{\alpha}{32} \sin(2\alpha x_3 \!-\! 2\alpha x_4)\!+\! \nonumber\\
&\!+\!& \frac{\alpha}{64} \sin(2\alpha x_1 \!-\! 2\alpha x_2 \!-\! 2\alpha x_4 \!-\! 2\xi_1) \!+\! \frac{\alpha}{64} \sin(2\alpha x_1 \!-\! 2\alpha x_3 \!-\! 2\alpha x_4 \!-\! 2\xi_1)\!+\! \nonumber\\
&\!+\!& \frac{3\alpha}{32} \sin(2\alpha x_1 \!+\! 2\xi_1) \!-\! \frac{\alpha}{32} \sin(2\alpha x_1 \!-\! 2\alpha x_2 \!+\! 2\alpha x_3 \!+\! 2\xi_1)\!-\! \nonumber\\
&\!-\!& \frac{\alpha}{64} \sin(2\alpha x_1 \!+\! 2\alpha x_3 \!-\! 2\alpha x_4 \!+\! 2\xi_1)\!-\! \frac{\alpha}{64} \sin(2\alpha x_1 \!-\! 2\alpha x_2 \!+\! 2\alpha x_4 \!+\! 2\xi_1)\!+\!\nonumber\\
&\!+\!& \frac{\alpha}{64} \sin(2\alpha x_1 \!-\! 2\alpha x_2 \!-\! 2\alpha x_3 \!-\! 2\xi_2) \!+\! \frac{\alpha}{64} \sin(2\alpha x_1 \!-\! 2\alpha x_2 \!-\! 2\alpha x_4 \!-\! 2\xi_2)\!-\!\nonumber\\
&\!-\!& \frac{\alpha}{64} \sin(2\alpha x_1 \!-\! 2\alpha x_3 \!+\! 2\xi_1 \!-\! 2\xi_2)\!-\! \frac{\alpha}{64} \sin(2\alpha x_1 \!-\! 2\alpha x_4 \!+\! 2\xi_1 \!-\! 2\xi_2)\!+\! \nonumber\\
&\!+\!& \frac{3\alpha}{32} \sin(2\alpha x_1 \!+\! 2\xi_2)\!-\! \frac{\alpha}{64} \sin(2\alpha x_1 \!-\! 2\alpha x_2 \!+\! 2\alpha x_3 \!+\! 2\xi_2)\!-\!\nonumber\\
&\!-\! &\frac{\alpha}{64} \sin(2\alpha x_1 \!+\! 2\alpha x_2 \!-\! 2\alpha x_4 \!+\! 2\xi_2)\!-\! \frac{\alpha}{32} \sin(2\alpha x_1 \!+\! 2\alpha x_3 \!-\! 2\alpha x_4 \!+\! 2\xi_2)\!-\!\nonumber\\
&\!-\! &\frac{\alpha}{64} \sin(2\alpha x_1 \!-\! 2\alpha x_2 \!-\! 2\xi_1 \!+\! 2\xi_2)\!-\! \frac{\alpha}{64} \sin(2\alpha x_1 \!-\! 2\alpha x_3 \!-\! 2\xi_1 \!+\! 2\xi_2)\!-\! \nonumber\\
&\!-\!& \frac{\alpha}{64} \sin(2\alpha x_2 \!-\! 2\alpha x_3 \!-\! 2\xi_1 \!+\! 2\xi_2)\!+\! \frac{\alpha}{64} \sin(2\alpha x_3 \!-\! 2\alpha x_4 \!-\! 2\xi_1 \!+\! 2\xi_2)\!-\!\nonumber\\
&\!-\! &\frac{\alpha}{64} \sin(2\alpha x_1 \!+\! 2\alpha x_2 \!+\! 2\xi_1 \!+\! 2\xi_2) \!-\! \frac{\alpha}{32} \sin(2\alpha x_1 \!+\! 2\alpha x_3 \!+\! 2\xi_1 \!+\! 2\xi_2)\!+\!\nonumber\\
&\!+\!&\frac{\alpha}{64} \sin(2\alpha x_2 \!+\! 2\alpha x_3 \!+\! 2\xi_1 \!+\! 2\xi_2)\!-\! \frac{\alpha}{64} \sin(2\alpha x_1 \!+\! 2\alpha x_4 \!+\! 2\xi_1 \!+\! 2\xi_2)\!+\! \nonumber\\
&\!+\!& \frac{\alpha}{32} \sin(2\alpha x_2 \!+\! 2\alpha x_4 \!+\! 2\xi_1 \!+\! 2\xi_2)\!+\! \frac{\alpha}{64} \sin(2\alpha x_3 \!+\! 2\alpha x_4 \!+\! 2\xi_1 \!+\! 2\xi_2)\!+\!\nonumber\\
&\!+\!& \frac{\alpha}{64} \sin(2\alpha x_1 \!-\! 2\alpha x_2 \!-\! 2\alpha x_3 \!-\! 2\xi_3)\!-\! \frac{\alpha}{64} \sin(2\alpha x_1 \!-\! 2\alpha x_2 \!-\! 2\alpha x_4 \!-\! 2\xi_3)\!+\!\nonumber\\
&\!+\! &\frac{\alpha}{64} \sin(2\alpha x_1 \!-\! 2\alpha x_2 \!+\! 2\xi_1 \!-\! 2\xi_3)\!-\! \frac{5\alpha}{64} \sin(2\alpha x_1 \!-\! 2\alpha x_3 \!+\! 2\xi_1 \!-\! 2\xi_3)\!+\! \nonumber\\
&\!+\!& \frac{\alpha}{64} \sin(2\alpha x_2 \!-\! 2\alpha x_4 \!+\! 2\xi_1 \!-\! 2\xi_3)\!+\!\frac{\alpha}{64} \sin(2\alpha x_1 \!-\! 2\alpha x_3 \!-\! 2\alpha x_4 \!+\! 2\xi_1 \!-\! 2\xi_2 \!-\! 2\xi_3)\!-\! \nonumber\\
&\!-\!&\frac{3\alpha}{64} \sin(2\alpha x_1 \!-\! 2\alpha x_2 \!+\! 2\xi_2 \!-\! 2\xi_3) \!+\! \frac{\alpha}{32} \sin(2\alpha x_1 \!-\! 2\alpha x_3 \!+\! 2\xi_2 \!-\! 2\xi_3)\!-\!\nonumber\\
&\!-\! &\frac{\alpha}{64} \sin(2\alpha x_2 \!-\! 2\alpha x_3 \!+\! 2\xi_2 \!-\! 2\xi_3)\!+\! \frac{\alpha}{64} \sin(2\alpha x_1 \!-\! 2\alpha x_4 \!+\! 2\xi_2 \!-\! 2\xi_3)\!+\!\nonumber\\
&\!+\!& \frac{\alpha}{64} \sin(2\alpha x_2 \!-\! 2\alpha x_4 \!+\! 2\xi_2 \!-\! 2\xi_3)\!-\! \frac{\alpha}{64} \sin(2\alpha x_3 \!-\! 2\alpha x_4 \!+\! 2\xi_2 \!-\! 2\xi_3)\!+\!\nonumber\\
&\!+\!&\frac{\alpha}{64} \sin(2\alpha x_1 \!-\! 2\alpha x_2 \!-\! 2\alpha x_4 \!-\! 2\xi_1 \!+\! 2\xi_2 \!-\! 2\xi_3)\!+\! \frac{\alpha}{32} \sin(2\alpha x_1 \!+\! 2\xi_1 \!+\! 2\xi_2 \!-\! 2\xi_3)\!-\!\nonumber\\
&\!-\!&\frac{\alpha}{64} \sin(2\alpha x_2 \!+\! 2\xi_1 \!+\! 2\xi_2 \!-\! 2\xi_3)\!-\! \frac{\alpha}{64} \sin(2\alpha x_3 \!+\! 2\xi_1 \!+\! 2\xi_2 \!-\! 2\xi_3)\!-\! \nonumber\\
&\!-\!& \frac{\alpha}{64} \sin(2\alpha x_1 \!+\! 2\alpha x_2 \!-\! 2\alpha x_4 \!+\! 2\xi_1 \!+\! 2\xi_2 \!-\! 2\xi_3)\!+\!  \nonumber\\
&\!+\!&\frac{\alpha}{64} \sin(2\alpha x_1 \!-\! 2\alpha x_3 \!+\! 2\alpha x_4 \!+\! 2\xi_1 \!+\! 2\xi_2 \!-\! 2\xi_3)\!+\! \frac{\alpha}{32} \sin(2\alpha x_1 \!+\! 2\xi_3)\!+\! \nonumber\\
&\!+\!&\frac {\alpha}{32} \sin(2\alpha x_2 \!+\! 2\xi_3)\!+\! \frac{\alpha}{64} \sin(2\alpha x_1 \!+\! 2\alpha x_2 \!-\! 2\alpha x_3 \!+\! 2\xi_3)\!+\! \nonumber\\
&\!+\!& \frac{\alpha}{64} \sin(2\alpha x_1 \!+\! 2\alpha x_3 \!-\! 2\alpha x_4 \!+\! 2\xi_3)\!+\!\frac{\alpha}{32} \sin(2\alpha x_4 \!+\! 2\xi_3)\!+\! \nonumber\\
&\!+\!&\frac{\alpha}{64} \sin(2\alpha x_1 \!-\! 2\alpha x_2 \!+\! 2\alpha x_4 \!+\! 2\xi_3)\!-\!\frac{\alpha}{64} \sin(2\alpha x_1 \!-\! 2\alpha x_3 \!+\! 2\alpha x_4 \!+\! 2\xi_3)\!-\! \nonumber\\
&\!-\!&\frac{\alpha}{64} \sin(2\alpha x_1 \!-\! 2\alpha x_3 \!-\! 2\xi_1 \!+\! 2\xi_3) \!-\!\frac{\alpha}{64} \sin(2\alpha x_2 \!-\! 2\alpha x_3 \!-\! 2\xi_1 \!+\! 2\xi_3) \!-\! \nonumber\\
&\!-\!&\frac{\alpha}{64} \sin(2\alpha x_1 \!-\! 2\alpha x_4 \!-\! 2\xi_1 \!+\! 2\xi_3)\!-\! \frac{\alpha}{64} \sin(2\alpha x_2 \!-\! 2\alpha x_4 \!-\! 2\xi_1 \!+\! 2\xi_3)\!+\! \nonumber\\
&\!+\!&\frac{\alpha}{64} \sin(2\alpha x_3 \!-\! 2\alpha x_4 \!-\! 2\xi_1 \!+\! 2\xi_3)\!-\! \frac{\alpha}{64} \sin(2\alpha x_1 \!+\! 2\alpha x_2 \!+\! 2\xi_1 \!+\! 2\xi_3)\!-\!  \nonumber\\
&\!-\!&\frac{\alpha}{32} \sin(2\alpha x_1 \!+\! 2\alpha x_3 \!+\! 2\xi_1 \!+\! 2\xi_3)\!+\! \frac{\alpha}{64} \sin(2\alpha x_2 \!+\! 2\alpha x_3 \!+\! 2\xi_1 \!+\! 2\xi_3)\!+\! \nonumber\\
&\!+\!&\frac{\alpha}{64} \sin(2\alpha x_1 \!+\! 2\alpha x_4 \!+\! 2\xi_1 \!+\! 2\xi_3)\!+\! \frac{\alpha}{64} \sin(2\alpha x_3 \!+\! 2\alpha x_4 \!+\! 2\xi_1 \!+\! 2\xi_3)\!-\!\nonumber\\
&\!-\!&\frac{\alpha}{64} \sin(2\alpha x_1 \!-\! 2\alpha x_2 \!-\! 2\xi_2 \!+\! 2\xi_3)\!-\! \frac{\alpha}{64} \sin(2\alpha x_2 \!-\! 2\alpha x_3 \!-\! 2\xi_2 \!+\! 2\xi_3)\!+\! \nonumber\\
&\!+\!& \frac{\alpha}{64} \sin(2\alpha x_1 \!-\! 2\alpha x_4 \!-\! 2\xi_2 \!+\! 2\xi_3)\!+\! \frac{\alpha}{64} \sin(2\alpha x_2 \!-\! 2\alpha x_4 \!-\! 2\xi_2 \!+\! 2\xi_3)\!-\!\nonumber\\
&\!-\!&\frac{\alpha}{64} \sin(2\alpha x_3 \!-\! 2\alpha x_4 \!-\! 2\xi_2 \!+\! 2\xi_3) \!-\! \frac{\alpha}{64} \sin(2\alpha x_3 \!+\! 2\xi_1 \!-\! 2\xi_2 \!+\! 2\xi_3)\!-\!\nonumber\\
&\!-\!& \frac{\alpha}{64} \sin(2\alpha x_1 \!-\! 2\alpha x_2 \!+\! 2\alpha x_3 \!+\! 2\xi_1 \!-\! 2\xi_2 \!+\! 2\xi_3)\!-\! \frac{\alpha}{64} \sin(2\alpha x_4 \!+\! 2\xi_1 \!-\! 2\xi_2 \!+\! 2\xi_3)\!-\!\nonumber\\
&\!-\!& \frac{\alpha}{32} \sin(2\alpha x_1 \!+\! 2\alpha x_2 \!+\! 2\xi_2 \!+\! 2\xi_3)\!+\! \frac{\alpha}{32} \sin(2\alpha x_1 \!+\! 2\alpha x_3 \!+\! 2\xi_2 \!+\! 2\xi_3)\!-\!\nonumber\\
&\!-\!& \frac{\alpha}{64} \sin(2\alpha x_2 \!-\! 2\xi_1 \!+\! 2\xi_2 \!+\! 2\xi_3)\!+\! \frac{\alpha}{64} \sin(2\alpha x_1 \!+\! 2\alpha x_2 \!-\! 2\alpha x_3 \!-\! 2\xi_1 \!+\! 2\xi_2 \!+\! 2\xi_3)\!-\! \nonumber\\
&\!-\!& \frac{\alpha}{32} \sin(2\alpha x_3 \!-\! 2\xi_1 \!+\! 2\xi_2 \!+\! 2\xi_3)\!-\! \frac{\alpha}{64} \sin(2\alpha x_4 \!-\! 2\xi_1 \!+\! 2\xi_2 \!+\! 2\xi_3)\!+\! \nonumber\\
&\!+\!& \frac{\alpha}{64} \sin(2\alpha x_1 \!-\! 2\alpha x_2 \!-\! 2\alpha x_4 \!-\! 2\xi_4)\!-\!\frac{\alpha}{64} \sin(2\alpha x_1 \!-\! 2\alpha x_3 \!-\! 2\alpha x_4 \!-\! 2\xi_4)\!-\!\nonumber\\
&\!-\! &\frac{\alpha}{64} \sin(2\alpha x_1 \!-\! 2\alpha x_2 \!+\! 2\xi_1 \!-\! 2\xi_4)\!-\! \frac{\alpha}{32} \sin(2\alpha x_1 \!-\! 2\alpha x_3 \!+\! 2\xi_1 \!-\! 2\xi_4)\!-\!\nonumber\\
&\!-\!& \frac{\alpha}{64} \sin(2\alpha x_2 \!-\! 2\alpha x_3 \!+\! 2\xi_1 \!-\! 2\xi_4)\!+\! \frac{3\alpha}{64} \sin(2\alpha x_1 \!-\! 2\alpha x_4 \!+\! 2\xi_1 \!-\! 2\xi_4)\!+\!\nonumber\\
&\!+\!& \frac{\alpha}{64} \sin(2\alpha x_2 \!-\! 2\alpha x_4 \!+\! 2\xi_1 \!-\! 2\xi_4) \!-\!\frac{\alpha}{64} \sin(2\alpha x_3 \!-\! 2\alpha x_4 \!+\! 2\xi_1 \!-\! 2\xi_4)\!-\! \nonumber\\
&\!-\!& \frac{\alpha}{64} \sin(2\alpha x_1 \!-\! 2\alpha x_2 \!-\! 2\alpha x_4 \!+\! 2\xi_1 \!-\! 2\xi_2 \!-\! 2\xi_4)\!+\!\frac{5\alpha}{64} \sin(2\alpha x_1 \!-\! 2\alpha x_3 \!+\! 2\xi_2 \!-\! 2\xi_4)\!+\! \nonumber\\
&\!+\!& \frac{\alpha}{64} \sin(2\alpha x_2 \!-\! 2\alpha x_3 \!+\! 2\xi_2 \!-\! 2\xi_4)\!-\! \frac{\alpha}{64} \sin(2\alpha x_1 \!-\! 2\alpha x_4 \!+\! 2\xi_2 \!-\! 2\xi_4)\!-\! \nonumber\\
&\!-\!& \frac{\alpha}{64} \sin(2\alpha x_2 \!-\! 2\alpha x_4 \!+\! 2\xi_2 \!-\! 2\xi_4)\!-\! \frac{\alpha}{64} \sin(2\alpha x_3 \!-\! 2\alpha x_4 \!+\! 2\xi_2 \!-\! 2\xi_4)\!-\! \nonumber\\
&\!-\! &\frac{\alpha}{64} \sin(2\alpha x_1 \!-\! 2\alpha x_2 \!-\! 2\alpha x_3 \!-\! 2\xi_1 \!+\! 2\xi_2 \!-\! 2\xi_4)\!-\! \frac{\alpha}{32} \sin(2\alpha x_1 \!+\! 2\xi_1 \!+\! 2\xi_2 \!-\! 2\xi_4)\!-\!
- \nonumber\\
&\!-\!&\frac{\alpha}{64} \sin(2\alpha x_1 \!+\! 2\alpha x_2 \!-\! 2\alpha x_3 \!+\! 2\xi_1 \!+\! 2\xi_2 \!-\! 2\xi_4)\!+\! \frac{\alpha}{64} \sin(2\alpha x_3 \!+\! 2\xi_1 \!+\! 2\xi_2 \!-\! 2\xi_4)\!+\!\nonumber\\
&\!+\!&\frac{\alpha}{64} \sin(2\alpha x_4 \!+\! 2\xi_1 \!+\! 2\xi_2 \!-\! 2\xi_4)\!+\! \frac{\alpha}{64} \sin(2\alpha x_1 \!-\! 2\alpha x_2 \!+\! 2\alpha x_4 \!+\! 2\xi_1 \!+\! 2\xi_2 \!-\! 2\xi_4)\!+\!\nonumber\\
&\!+\!& \frac{\alpha}{64} \sin(2\alpha x_1 \!-\! 2\alpha x_2 \!-\! 2\alpha x_3 \!+\! 2\xi_1 \!-\! 2\xi_3 \!-\! 2\xi_4)\!-\!\nonumber\\
&\!-\!&\frac{\alpha}{32} \sin(2\alpha x_1 \!-\! 2\alpha x_3 \!-\! 2\alpha x_4 \!+\! 2\xi_1 \!-\! 2\xi_3 \!-\! 2\xi_4)\!-\!\nonumber\\
&\!-\!& \frac{\alpha}{32} \sin(2\alpha x_1 \!-\! 2\alpha x_2 \!-\! 2\alpha x_3 \!+\! 2\xi_2 \!-\! 2\xi_3 \!-\! 2\xi_4)\!+\! \nonumber\\
&\!+\!& \frac{\alpha}{64} \sin(2\alpha x_1 \!-\! 2\alpha x_3 \!-\! 2\alpha x_4 \!+\! 2\xi_2 \!-\! 2\xi_3 \!-\! 2\xi_4)\!+\! \nonumber\\
&\!+\!& \frac{\alpha}{32} \sin(2\alpha x_1 \!-\! 2\alpha x_3 \!+\! 2\xi_1 \!+\! 2\xi_2 \!-\! 2\xi_3 \!-\! 2\xi_4)\!-\! \nonumber\\
&\!-\!& \frac{\alpha}{64} \sin(2\alpha x_2 \!-\! 2\alpha x_4 \!+\! 2\xi_1 \!+\! 2\xi_2 \!-\! 2\xi_3 \!-\! 2\xi_4)\!+\! \nonumber\\
&\!+\!& \frac{\alpha}{64} \sin(2\alpha x_1 \!-\! 2\alpha x_3 \!+\! 2\xi_3 \!-\! 2\xi_4)\!+\! \frac{\alpha}{64} \sin(2\alpha x_1 \!-\! 2\alpha x_4 \!+\! 2\xi_3 \!-\! 2\xi_4)\!+\!\nonumber\\
&\!+\!& \frac{\alpha}{32} \sin(2\alpha x_1 \!+\! 2\xi_1 \!+\! 2\xi_3 \!-\! 2\xi_4)\!+\! \frac{\alpha}{64} \sin(2\alpha x_3 \!+\! 2\xi_1 \!+\! 2\xi_3 \!-\! 2\xi_4)\!-\! \nonumber\\
&\!-\!& \frac{\alpha}{64} \sin(2\alpha x_1 \!-\! 2\alpha x_2 \!+\! 2\alpha x_3 \!+\! 2\xi_1 \!+\! 2\xi_3 \!-\! 2\xi_4)\!-\! \nonumber\\
&\!-\!& \frac{\alpha}{32} \sin(2\alpha x_1 \!+\! 2\alpha x_3 \!-\! 2\alpha x_4 \!+\! 2\xi_1 \!+\! 2\xi_3 \!-\! 2\xi_4)\!+\!\nonumber\\
&\!+\!& \frac{\alpha}{64} \sin(2\alpha x_4 \!+\! 2\xi_1 \!+\! 2\xi_3 \!-\! 2\xi_4)\!+\! \frac{\alpha}{64} \sin(2\alpha x_1 \!-\! 2\alpha x_2 \!-\! 2\alpha x_4 \!-\! 2\xi_2 \!+\! 2\xi_3 \!-\! 2\xi_4)\!-\!\nonumber\\
&\!-\!& \frac{\alpha}{32} \sin(2\alpha x_1 \!+\! 2\alpha x_2 \!-\! 2\alpha x_3 \!+\! 2\xi_2 \!+\! 2\xi_3 \!-\! 2\xi_4)\!+\! \frac{\alpha}{32} \sin(2\alpha x_3 \!+\! 2\xi_2 \!+\! 2\xi_3 \!-\! 2\xi_4)\!+\! \nonumber\\
&\!+\!& \frac{\alpha}{64} \sin(2\alpha x_1 \!-\! 2\alpha x_3 \!-\! 2\xi_1 \!+\! 2\xi_2 \!+\! 2\xi_3 \!-\! 2\xi_4)\!+\! \nonumber\\
&\!+\!& \frac{\alpha}{64} \sin(2\alpha x_1 \!+\! 2\alpha x_2 \!+\! 2\xi_1 \!+\! 2\xi_2 \!+\! 2\xi_3 \!-\! 2\xi_4)\!-\!\nonumber\\
&\!-\!& \frac{\alpha}{64} \sin(2\alpha x_1 \!+\! 2\alpha x_3 \!+\! 2\xi_1 \!+\! 2\xi_2 \!+\! 2\xi_3 \!-\! 2\xi_4)\!-\! \nonumber\\
&\!-\!& \frac{\alpha}{64} \sin(2\alpha x_1 \!+\! 2\alpha x_4 \!+\! 2\xi_1 \!+\! 2\xi_2 \!+\! 2\xi_3 \!-\! 2\xi_4)\!+\!\nonumber\\
&\!+\!&\frac{\alpha}{64} \sin(2\alpha x_2 \!+\! 2\alpha x_4 \!+\! 2\xi_1 \!+\! 2\xi_2 \!+\! 2\xi_3 \!-\! 2\xi_4)\!-\! \nonumber\\
&\!-\!& \frac{\alpha}{32} \sin(2\alpha x_1 \!+\! 2\xi_4)\!-\! \frac{\alpha}{32} \sin(2\alpha x_2 \!+\! 2\xi_4)
\!+\! \frac{\alpha}{64} \sin(2\alpha x_1 \!+\! 2\alpha x_2 \!-\! 2\alpha x_3 \!+\! 2\xi_4)\!-\!\nonumber\\
&\!-\!& \frac{\alpha}{64} \sin(2\alpha x_1 \!-\! 2\alpha x_2 \!+\! 2\alpha x_3 \!+\! 2\xi_4)\!-\! \frac{\alpha}{64} \sin(2\alpha x_1 \!+\! 2\alpha x_2 \!-\! 2\alpha x_4 \!+\! 2\xi_4)\!-\!\nonumber\\
&\!-\!& \frac{\alpha}{32} \sin(2\alpha x_4 \!+\! 2\xi_4)\!-\! \frac{\alpha}{64} \sin(2\alpha x_1 \!-\! 2\alpha x_3 \!+\! 2\alpha x_4 \!+\! 2\xi_4)\!-\!\nonumber\\
&\!-\!&\frac{\alpha}{64} \sin(2\alpha x_1 \!-\! 2\alpha x_2 \!-\! 2\xi_1 \!+\! 2\xi_4)\!-\! \frac{\alpha}{64} \sin(2\alpha x_2 \!-\! 2\alpha x_3 \!-\! 2\xi_1 \!+\! 2\xi_4)\!+\!\nonumber\\
&\!+\! &\frac{\alpha}{64} \sin(2\alpha x_1 \!-\! 2\alpha x_4 \!-\! 2\xi_1 \!+\! 2\xi_4) \!+\! \frac{\alpha}{64} \sin(2\alpha x_2 \!-\! 2\alpha x_4 \!-\! 2\xi_1 \!+\! 2\xi_4)\!-\!\nonumber\\
&\!-\!& \frac{\alpha}{64} \sin(2\alpha x_3 \!-\! 2\alpha x_4 \!-\! 2\xi_1 \!+\! 2\xi_4)\!-\! \frac{\alpha}{32} \sin(2\alpha x_1 \!+\! 2\alpha x_3 \!+\! 2\xi_1 \!+\! 2\xi_4)\!+\!\nonumber\\
&\!+\!& \frac{\alpha}{32} \sin(2\alpha x_1 \!+\! 2\alpha x_4 \!+\! 2\xi_1 \!+\! 2\xi_4)\!+\! \frac{\alpha}{64} \sin(2\alpha x_1 \!-\! 2\alpha x_2 \!-\! 2\xi_2 \!+\! 2\xi_4)\!+\!\nonumber\\
&\!+\! &\frac{\alpha}{64} \sin(2\alpha x_1 \!-\! 2\alpha x_3 \!-\! 2\xi_2 \!+\! 2\xi_4) \!+\! \frac{\alpha}{64} \sin(2\alpha x_2 \!-\! 2\alpha x_4 \!-\! 2\xi_2 \!+\! 2\xi_4)\!+\!\nonumber\\
&\!+\!& \frac{\alpha}{64} \sin(2\alpha x_2 \!+\! 2\xi_1 \!-\! 2\xi_2 \!+\! 2\xi_4)\!+\! \frac{\alpha}{32} \sin(2\alpha x_3 \!+\! 2\xi_1 \!-\! 2\xi_2 \!+\! 2\xi_4)\!+\!\nonumber\\
&\!+\!& \frac{\alpha}{64} \sin(2\alpha x_4 \!+\! 2\xi_1 \!-\! 2\xi_2 \!+\! 2\xi_4)\!-\! \frac{\alpha}{64} \sin(2\alpha x_1 \!-\! 2\alpha x_3 \!+\! 2\alpha x_4 \!+\! 2\xi_1 \!-\! 2\xi_2 \!+\! 2\xi_4)\!-\!\nonumber\\
&\!-\! &\frac{\alpha}{64} \sin(2\alpha x_1 \!+\! 2\alpha x_2 \!+\! 2\xi_2 \!+\! 2\xi_4)\!+\! \frac{\alpha}{32} \sin(2\alpha x_1 \!+\! 2\alpha x_3 \!+\! 2\xi_2 \!+\! 2\xi_4)\!-\!\nonumber\\
&\!-\! &\frac{\alpha}{64} \sin(2\alpha x_2 \!+\! 2\alpha x_3 \!+\! 2\xi_2 \!+\! 2\xi_4)\!+\! \frac{\alpha}{64} \sin(2\alpha x_1 \!+\! 2\alpha x_4 \!+\! 2\xi_2 \!+\! 2\xi_4)\!-\!\nonumber\\
&\!-\!& \frac{\alpha}{64} \sin(2\alpha x_3 \!+\! 2\alpha x_4 \!+\! 2\xi_2 \!+\! 2\xi_4)\!+\! \frac{\alpha}{64} \sin(2\alpha x_2 \!-\! 2\xi_1 \!+\! 2\xi_2 \!+\! 2\xi_4)\!+\!\nonumber\\
&\!+\!& \frac{\alpha}{64} \sin(2\alpha x_3 \!-\! 2\xi_1 \!+\! 2\xi_2 \!+\! 2\xi_4)\!+\! \frac{\alpha}{64} \sin(2\alpha x_1 \!+\! 2\alpha x_3 \!-\! 2\alpha x_4 \!-\! 2\xi_1 \!+\! 2\xi_2 \!+\! 2\xi_4)\!+\!\nonumber\\
&\!+\!& \frac{\alpha}{64} \sin(2\alpha x_1 \!-\! 2\alpha x_2 \!-\! 2\xi_3 \!+\! 2\xi_4)\!+\! \frac{\alpha}{64} \sin(2\alpha x_1 \!-\! 2\alpha x_3 \!-\! 2\xi_3 \!+\! 2\xi_4)\!-\!\nonumber\\
&\!-\!& \frac{\alpha}{64} \sin(2\alpha x_2 \!-\! 2\alpha x_3 \!-\! 2\xi_3 \!+\! 2\xi_4)\!+\! \frac{\alpha}{64} \sin(2\alpha x_3 \!-\! 2\alpha x_4 \!-\! 2\xi_3 \!+\! 2\xi_4)\!+\! \nonumber\\
&\!+\!& \frac{\alpha}{64} \sin(2\alpha x_1 \!-\! 2\alpha x_2 \!-\! 2\alpha x_4 \!-\! 2\xi_1 \!-\! 2\xi_3 \!+\! 2\xi_4) \!+\! \frac{\alpha}{32} \sin(2\alpha x_3 \!+\! 2\xi_1 \!-\! 2\xi_3 \!+\! 2\xi_4)\!-\! \nonumber\\
&\!-\! &\frac{\alpha}{32} \sin(2\alpha x_1 \!-\! 2\alpha x_3 \!+\! 2\alpha x_4 \!+\! 2\xi_1 \!-\! 2\xi_3 \!+\! 2\xi_4)\!+\!
+ \nonumber\\
&\!+\!&\frac{\alpha}{64} \sin(2\alpha x_1 \!-\! 2\alpha x_3 \!+\! 2\xi_1 \!-\! 2\xi_2 \!-\! 2\xi_3 \!+\! 2\xi_4)\!+\! \frac{\alpha}{32} \sin(2\alpha x_1 \!+\! 2\xi_2 \!-\! 2\xi_3 \!+\! 2\xi_4)\!+\!\nonumber\\
&\!+\!& \frac{\alpha}{64} \sin(2\alpha x_2 \!+\! 2\xi_2 \!-\! 2\xi_3 \!+\! 2\xi_4)\!+\! \frac{\alpha}{64} \sin(2\alpha x_3 \!+\! 2\xi_2 \!-\! 2\xi_3 \!+\! 2\xi_4)\!-\!\nonumber\\
&\!-\!& \frac{\alpha}{32} \sin(2\alpha x_1 \!-\! 2\alpha x_2 \!+\! 2\alpha x_3 \!+\! 2\xi_2 \!-\! 2\xi_3 \!+\! 2\xi_4)\!-\! \nonumber\\
&\!-\!&\frac{\alpha}{64} \sin(2\alpha x_1 \!+\! 2\alpha x_3 \!-\! 2\alpha x_4 \!+\! 2\xi_2 \!-\! 2\xi_3 \!+\! 2\xi_4)\!-\!\nonumber\\
&\!-\!&  \frac{\alpha}{64} \sin(2\alpha x_1 \!+\! 2\alpha x_2 \!+\! 2\xi_1 \!+\! 2\xi_2 \!-\! 2\xi_3 \!+\! 2\xi_4)\!-\!\nonumber\\
&\!-\!&\frac{\alpha}{64} \sin(2\alpha x_1 \!+\! 2\alpha x_3 \!+\! 2\xi_1 \!+\! 2\xi_2 \!-\! 2\xi_3 \!+\! 2\xi_4)\!+\!\nonumber\\
&\!+\!& \frac{\alpha}{64} \sin(2\alpha x_1 \!+\! 2\alpha x_4 \!+\! 2\xi_1 \!+\! 2\xi_2 \!-\! 2\xi_3 \!+\! 2\xi_4)\!+\! \nonumber\\
&\!+\!& \frac{\alpha}{64} \sin(2\alpha x_2 \!+\! 2\alpha x_4 \!+\! 2\xi_1 \!+\! 2\xi_2 \!-\! 2\xi_3 \!+\! 2\xi_4)\!+\! \frac{\alpha}{64} \sin(2\alpha x_1 \!+\! 2\alpha x_2 \!+\! 2\xi_3 \!+\! 2\xi_4)\!+\!\nonumber\\
&\!+\!& \frac{\alpha}{32} \sin(2\alpha x_1 \!+\! 2\alpha x_3 \!+\! 2\xi_3 \!+\! 2\xi_4)\!+\! \frac{\alpha}{64} \sin(2\alpha x_2 \!+\! 2\alpha x_3 \!+\! 2\xi_3 \!+\! 2\xi_4)\!+\!\nonumber\\
&\!+\!& \frac{\alpha}{64} \sin(2\alpha x_1 \!+\! 2\alpha x_4 \!+\! 2\xi_3 \!+\! 2\xi_4)\!+\! \frac{\alpha}{32} \sin(2\alpha x_2 \!+\! 2\alpha x_4 \!+\! 2\xi_3 \!+\! 2\xi_4)\!+\!\nonumber\\
&\!+\!& \frac{\alpha}{64} \sin(2\alpha x_3 \!+\! 2\alpha x_4 \!+\! 2\xi_3 \!+\! 2\xi_4)\!+\! \frac{\alpha}{64} \sin(2\alpha x_3 \!-\! 2\xi_1 \!+\! 2\xi_3 \!+\! 2\xi_4)\!-\!\nonumber\\
&\!-\!& \frac{\alpha}{64} \sin(2\alpha x_1 \!+\! 2\alpha x_2 \!-\! 2\alpha x_4 \!-\! 2\xi_1 \!+\! 2\xi_3 \!+\! 2\xi_4)\!-\! \nonumber\\
&\!-\!& \frac{\alpha}{64} \sin(2\alpha x_4 \!-\! 2\xi_1 \!+\! 2\xi_3 \!+\! 2\xi_4)\!-\! \frac{\alpha}{64} \sin(2\alpha x_2 \!-\! 2\xi_2 \!+\! 2\xi_3 \!+\! 2\xi_4)\!+\! \nonumber\\
&\!+\!& \frac{\alpha}{64} \sin(2\alpha x_3 \!-\! 2\xi_2 \!+\! 2\xi_3 \!+\! 2\xi_4)\!-\! \frac{\alpha}{64} \sin(2\alpha x_1 \!-\! 2\alpha x_2 \!+\! 2\alpha x_4 \!-\! 2\xi_2 \!+\! 2\xi_3 \!+\! 2\xi_4)\!+\!\nonumber\\ &\!+\!&\frac{\alpha}{64} \sin(2\alpha x_2 \!-\! 2\alpha x_4 \!-\! 2\xi_1 \!-\! 2\xi_2 \!+\! 2\xi_3 \!+\! 2\xi_4)\!-\! \nonumber\\
&\!-\!& \frac{\alpha}{64} \sin(2\alpha x_1 \!+\! 2\alpha x_3 \!+\! 2\xi_1 \!-\! 2\xi_2 \!+\! 2\xi_3 \!+\! 2\xi_4)\!-\! \nonumber\\
&\!-\! &\frac{\alpha}{64} \sin(2\alpha x_2 \!+\! 2\alpha x_3 \!+\! 2\xi_1 \!-\! 2\xi_2 \!+\! 2\xi_3 \!+\! 2\xi_4)\!-\! \nonumber\\
&\!-\!& \frac{\alpha}{64} \sin(2\alpha x_2 \!+\! 2\alpha x_4 \!+\! 2\xi_1 \!-\! 2\xi_2 \!+\! 2\xi_3 \!+\! 2\xi_4)\!-\! \nonumber\\
&\!-\!& \frac{\alpha}{64} \sin(2\alpha x_3 \!+\! 2\alpha x_4 \!+\! 2\xi_1 \!-\! 2\xi_2 \!+\! 2\xi_3 \!+\! 2\xi_4)\!-\! \nonumber\\
&\!-\!& \frac{\alpha}{64} \sin(2\alpha x_1 \!+\! 2\alpha x_3 \!-\! 2\xi_1 \!+\! 2\xi_2 \!+\! 2\xi_3 \!+\! 2\xi_4)\!-\! \nonumber\\
&\!-\!& \frac{\alpha}{64} \sin(2\alpha x_2 \!+\! 2\alpha x_3 \!-\! 2\xi_1 \!+\! 2\xi_2 \!+\! 2\xi_3 \!+\! 2\xi_4)\!-\! \nonumber\\
&\!-\!& \frac{\alpha}{64} \sin(2\alpha x_2 \!+\! 2\alpha x_4 \!-\! 2\xi_1 \!+\! 2\xi_2 \!+\! 2\xi_3 \!+\! 2\xi_4)\!-\! \nonumber\\
&\!-\!& \frac{\alpha}{64} \sin(2\alpha x_3 \!+\! 2\alpha x_4 - 2\xi_1 \!+\! 2\xi_2 \!+\! 2\xi_3 \!+\! 2\xi_4)
\end{eqnarray}
By collecting all terms in equation~\eqref{A1} that are independent of \( x_1 \), we isolate:
\begin{eqnarray}
\label{A2}
\mathcal{V}_1^0(\boldsymbol{x}_{\setminus 1}, \boldsymbol{\xi}) &=&-\frac{\alpha}{32}   \sin (2 \alpha x_2-2 \alpha x_3)+\frac{\alpha}{32}   \sin (2 \alpha x_3-2 \alpha x_4)- \nonumber\\
&-&\frac{\alpha}{64}  \sin (2 \alpha x_2-2 \alpha x_3-2 \xi_1+2 \xi_2)+\frac{\alpha}{64}  \sin (2 \alpha x_3-2 \alpha x_4-2 \xi_1+2 \xi_2)+ \nonumber\\
&+&\frac{\alpha}{64}  \sin (2 \alpha x_2+2 \alpha x_3+2 \xi_1+2 \xi_2)+\frac{\alpha}{32}   \sin (2 \alpha x_2+2 \alpha x_4+2 \xi_1+2 \xi_2)+ \nonumber\\
&+&\frac{\alpha}{64}  \sin (2 \alpha x_3+2 \alpha x_4+2 \xi_1+2 \xi_2)+\frac{\alpha}{64}  \sin (2 \alpha x_2-2 \alpha x_4+2 \xi_1-2 \xi_3)- \nonumber\\
&-&\frac{\alpha}{64}  \sin (2 \alpha x_2-2 \alpha x_3+2 \xi_2-2 \xi_3)+\frac{\alpha}{64}  \sin (2 \alpha x_2-2 \alpha x_4+2 \xi_2-2 \xi_3)- \nonumber\\
&-&\frac{\alpha}{64}  \sin (2 \alpha x_3-2 \alpha x_4+2 \xi_2-2 \xi_3)-\frac{\alpha}{64}  \sin (2 \alpha x_2+2 \xi_1+2 \xi_2-2 \xi_3)- \nonumber\\
&-&\frac{\alpha}{64}  \sin (2 \alpha x_3\!+\!2 \xi_1\!+\!2 \xi_2-2 \xi_3)\!+\!\frac{\alpha}{32}   \sin (2 \alpha x_2\!+\!2 \xi_3)\!+\!\frac{\alpha}{32}   \sin (2 \alpha x_4\!+\!2 \xi_3)- \nonumber\\
&-&\frac{\alpha}{64}  \sin (2 \alpha x_2-2 \alpha x_3-2 \xi_1+2 \xi_3)-\frac{\alpha}{64}  \sin (2 \alpha x_2-2 \alpha x_4-2 \xi_1+2 \xi_3)+ \nonumber\\
&+&\frac{\alpha}{64}  \sin (2 \alpha x_3-2 \alpha x_4-2 \xi_1+2 \xi_3)+\frac{\alpha}{64}  \sin (2 \alpha x_2+2 \alpha x_3+2 \xi_1+2 \xi_3)+ \nonumber\\
&+&\frac{\alpha}{64}  \sin (2 \alpha x_3+2 \alpha x_4+2 \xi_1+2 \xi_3)-\frac{\alpha}{64}  \sin (2 \alpha x_2-2 \alpha x_3-2 \xi_2+2 \xi_3)+ \nonumber\\
&+&\frac{\alpha}{64}  \sin (2 \alpha x_2-2 \alpha x_4-2 \xi_2+2 \xi_3)-\frac{\alpha}{64}  \sin (2 \alpha x_3-2 \alpha x_4-2 \xi_2+2 \xi_3)- \nonumber\\
&-&\frac{\alpha}{64}  \sin (2 \alpha x_3+2 \xi_1-2 \xi_2+2 \xi_3)-\frac{\alpha}{64}  \sin (2 \alpha x_4+2 \xi_1-2 \xi_2+2 \xi_3)- \nonumber\\
&-&\frac{\alpha}{64}  \sin (2 \alpha x_2-2 \xi_1+2 \xi_2+2 \xi_3)-\frac{\alpha}{32}   \sin (2 \alpha x_3-2 \xi_1+2 \xi_2+2 \xi_3)- \nonumber\\
&-&\frac{\alpha}{64}  \sin (2 \alpha x_4-2 \xi_1+2 \xi_2+2 \xi_3)-\frac{\alpha}{64}  \sin (2 \alpha x_2-2 \alpha x_3+2 \xi_1-2 \xi_4)+ \nonumber\\
&+&\frac{\alpha}{64}  \sin (2 \alpha x_2-2 \alpha x_4+2 \xi_1-2 \xi_4)-\frac{\alpha}{64}  \sin (2 \alpha x_3-2 \alpha x_4+2 \xi_1-2 \xi_4)+ \nonumber\\
&+&\frac{\alpha}{64}  \sin (2 \alpha x_2-2 \alpha x_3+2 \xi_2-2 \xi_4)-\frac{\alpha}{64}  \sin (2 \alpha x_2-2 \alpha x_4+2 \xi_2-2 \xi_4)- \nonumber\\
&-&\frac{\alpha}{64}  \sin (2 \alpha x_3-2 \alpha x_4+2 \xi_2-2 \xi_4)+\frac{\alpha}{64}  \sin (2 \alpha x_3+2 \xi_1+2 \xi_2-2 \xi_4)+ \nonumber\\
&\!+\!&\frac{\alpha}{64}  \sin (2 \alpha x_4\!+\!2 \xi_1\!+\!2 \xi_2\!-\!2 \xi_4)\!-\!\frac{\alpha}{64}  \sin (2 \alpha x_2\!-\!2 \alpha x_4\!+\!2 \xi_1\!+\!2 \xi_2\!-\!2 \xi_3\!-\!2 \xi_4)\!+\!\nonumber\\
&+&\frac{\alpha}{64}  \sin (2 \alpha x_3+2 \xi_1+2 \xi_3-2 \xi_4)+\frac{\alpha}{64}  \sin (2 \alpha x_4+2 \xi_1+2 \xi_3-2 \xi_4)+ \nonumber\\
&\!+\!&\frac{\alpha}{32}   \sin (2 \alpha x_3\!+\!2 \xi_2\!+\!2 \xi_3\!-\!2 \xi_4)\!+\!\frac{\alpha}{64}  \sin (2 \alpha x_2\!+\!2 \alpha x_4\!+\!2 \xi_1\!+\!2 \xi_2\!+\!2 \xi_3\!-\!2 \xi_4)\!-\!\nonumber\\
&\!-\!&\frac{\alpha}{32}   \sin (2 \alpha x_2\!+\!2 \xi_4)\!-\!\frac{\alpha}{32} \sin (2 \alpha x_4\!+\!2 \xi_4)\!-\!\frac{\alpha}{64}  \sin (2 \alpha x_2\!-\!2 \alpha x_3-2 \xi_1\!+\!2 \xi_4)\!+\!\nonumber\\
&+&\frac{\alpha}{64}  \sin (2 \alpha x_2-2 \alpha x_4-2 \xi_1+2 \xi_4)-\frac{\alpha}{64}  \sin (2 \alpha x_3-2 \alpha x_4-2 \xi_1+2 \xi_4)+ \nonumber\\
&+&\frac{\alpha}{64}  \sin (2 \alpha x_2-2 \alpha x_4-2 \xi_2+2 \xi_4)+\frac{\alpha}{64}  \sin (2 \alpha x_2+2 \xi_1-2 \xi_2+2 \xi_4)+ \nonumber\\
&+&\frac{\alpha}{32}   \sin (2 \alpha x_3+2 \xi_1-2 \xi_2+2 \xi_4)+\frac{\alpha}{64}  \sin (2 \alpha x_4+2 \xi_1-2 \xi_2+2 \xi_4)- \nonumber\\
&-&\frac{\alpha}{64}  \sin (2 \alpha x_2+2 \alpha x_3+2 \xi_2+2 \xi_4)-\frac{\alpha}{64}  \sin (2 \alpha x_3+2 \alpha x_4+2 \xi_2+2 \xi_4)+ \nonumber\\
&+&\frac{\alpha}{64}  \sin (2 \alpha x_2-2 \xi_1+2 \xi_2+2 \xi_4)+\frac{\alpha}{64}  \sin (2 \alpha x_3-2 \xi_1+2 \xi_2+2 \xi_4)- \nonumber\\
&-&\frac{\alpha}{64}  \sin (2 \alpha x_2-2 \alpha x_3-2 \xi_3+2 \xi_4)+\frac{\alpha}{64}  \sin (2 \alpha x_3-2 \alpha x_4-2 \xi_3+2 \xi_4)+ \nonumber\\
&+&\frac{\alpha}{32}   \sin (2 \alpha x_3+2 \xi_1-2 \xi_3+2 \xi_4)+\frac{\alpha}{64}  \sin (2 \alpha x_2+2 \xi_2-2 \xi_3+2 \xi_4)+ \nonumber\\
&\!+\!&\frac{\alpha}{64}  \sin (2 \alpha x_3\!+\!2 \xi_2\!-\!2 \xi_3\!+\!2 \xi_4)\!+\!\frac{\alpha}{64}  \sin (2 \alpha x_2\!+\!2 \alpha x_4\!+\!2 \xi_1\!+\!2 \xi_2\!-\!2 \xi_3\!+\!2 \xi_4)+ \nonumber\\
&+&\frac{\alpha}{64}  \sin (2 \alpha x_2+2 \alpha x_3+2 \xi_3+2 \xi_4)+\frac{\alpha}{32}   \sin (2 \alpha x_2+2 \alpha x_4+2 \xi_3+2 \xi_4)+ \nonumber\\
&+&\frac{\alpha}{64}  \sin (2 \alpha x_3+2 \alpha x_4+2 \xi_3+2 \xi_4)+\frac{\alpha}{64}  \sin (2 \alpha x_3-2 \xi_1+2 \xi_3+2 \xi_4)- \nonumber\\
&-&\frac{\alpha}{64}  \sin (2 \alpha x_4-2 \xi_1+2 \xi_3+2 \xi_4)-\frac{\alpha}{64}  \sin (2 \alpha x_2-2 \xi_2+2 \xi_3+2 \xi_4)+ \nonumber\\
&\!+\!&\frac{\alpha}{64}  \sin (2 \alpha x_3\!-\!2 \xi_2\!+\!2 \xi_3\!+\!2 \xi_4)\!+\!\frac{\alpha}{64}  \sin (2 \alpha x_2\!-\!2 \alpha x_4\!-\!2 \xi_1\!-\!2 \xi_2\!+\!2 \xi_3\!+\!2 \xi_4)\!-\!\nonumber\\
&-&\frac{\alpha}{64}  \sin (2 \alpha x_2+2 \alpha x_3+2 \xi_1-2 \xi_2+2 \xi_3+2 \xi_4)- \nonumber\\
&-&\frac{\alpha}{64}  \sin (2 \alpha x_2+2 \alpha x_4+2 \xi_1-2 \xi_2+2 \xi_3+2 \xi_4)- \nonumber\\
&-&\frac{\alpha}{64}  \sin (2 \alpha x_3+2 \alpha x_4+2 \xi_1-2 \xi_2+2 \xi_3+2 \xi_4)- \nonumber\\
&-&\frac{\alpha}{64}  \sin (2 \alpha x_2+2 \alpha x_3-2 \xi_1+2 \xi_2+2 \xi_3+2 \xi_4)- \nonumber\\
&-&\frac{\alpha}{64}  \sin (2 \alpha x_2+2 \alpha x_4-2 \xi_1+2 \xi_2+2 \xi_3+2 \xi_4)- \nonumber\\
&-&\frac{\alpha}{64}  \sin (2 \xi_4+2 \alpha x_3+2 \alpha x_4-2 \xi_1+2 \xi_2+2 \xi_3)
\end{eqnarray}
Equation~\eqref{A2} can be factored into the structured form of equation~\eqref{3.19}, decomposing into products of spatial functions and phase-dependent coefficients:
\begin{equation}
\label{A3}
    \frac{\mathcal{V}_1^0(\boldsymbol{x}_{\setminus 1}, \boldsymbol{\xi})}{v_r^2} = \sum_{j=1}^{35} q_j^0(\boldsymbol{x}_{\setminus 1}) \cdot u_j(\boldsymbol{\xi}),
\end{equation}
where the spatial coefficient functions \( q_j^0\left(\boldsymbol{x}_{\setminus 1}\right) \) and the phase-dependent functions \(u_j\left(\boldsymbol{\xi}\right)\) for $j=1,2,\ldots,35$ are given by:
\begin{subequations}
\label{A4}
\begin{align}
q_1^0(\boldsymbol{x}_{\setminus 1}) &= \frac{\alpha}{64} \sin(2\alpha x_2) \label{A4a} \\
u_1(\boldsymbol{\xi}) &= -\cos(2\xi_1 + 2\xi_2 - 2\xi_3) - \cos(2\xi_1 - 2\xi_2 - 2\xi_3) + \cos(2\xi_1 - 2\xi_2 + 2\xi_4)+ \nonumber\\
&\quad + \cos(2\xi_1 - 2\xi_2 - 2\xi_4) + \cos(2\xi_2 - 2\xi_3 + 2\xi_4) - \cos(2\xi_2 - 2\xi_3 - 2\xi_4)+ \nonumber\\
&\quad + 2\cos(2\xi_3) - 2\cos(2\xi_4) \label{A4b}
\end{align}\end{subequations}
\vspace{-0.4cm}
\begin{subequations}
\label{A5}
\begin{align}
q_2^0(\boldsymbol{x}_{\setminus 1})&=\frac{\alpha}{64}   \sin (2 \alpha  x_3) \label{A5a} \\
u_2(\boldsymbol{\xi})&=-\cos (2 \xi_1-2 \xi_2+2 \xi_3)-\cos (2 \xi_1+2 \xi_2-2 \xi_3)-2 \cos (2 \xi_1-2 \xi_2-2 \xi_3)+ \nonumber\\
&\quad+2 \cos (2 \xi_1-2 \xi_2+2 \xi_4)+\cos (2 \xi_1+2 \xi_2-2 \xi_4)+\cos (2 \xi_1-2 \xi_2-2 \xi_4)+ \nonumber\\
&\quad+2 \cos (2 \xi_1-2 \xi_3+2 \xi_4)+\cos (2 \xi_1+2 \xi_3-2 \xi_4)+\cos (2 \xi_1-2 \xi_3-2 \xi_4)+ \nonumber\\
&\quad+\cos (2 \xi_2-2 \xi_3+2 \xi_4)+2 \cos (2 \xi_2+2 \xi_3-2 \xi_4)+\cos (2 \xi_2-2 \xi_3-2 \xi_4)\label{A5b}
\end{align}\end{subequations}
\vspace{-0.4cm}
\begin{subequations}
\label{A6}
\begin{align}
q_3^0(\boldsymbol{x}_{\setminus 1})&=\frac{\alpha}{64}   \sin (2 \alpha  x_4)\label{A6a} \\
u_3(\boldsymbol{\xi})&=-\cos (2 \xi_1-2 \xi_2+2 \xi_3)-\cos (2 \xi_1-2 \xi_2-2 \xi_3)+\cos (2 \xi_1-2 \xi_2+2 \xi_4)+ \nonumber\\
&\quad+\cos (2 \xi_1+2 \xi_2-2 \xi_4)+\cos (2 \xi_1+2 \xi_3-2 \xi_4)-\cos (2 \xi_1-2 \xi_3-2 \xi_4)+ \nonumber\\
&\quad+2 \cos (2 \xi_3)-2 \cos (2 \xi_4)\label{A6b}
\end{align}\end{subequations}
\vspace{-0.4cm}
\begin{subequations}
\label{A7}
\begin{align}
q_4^0(\boldsymbol{x}_{\setminus 1})&=\frac{\alpha}{64}   \sin (2 \alpha  x_2-2 \alpha  x_3)\label{A7a} \\
u_4(\boldsymbol{\xi})&=-\cos (2 \xi_1-2 \xi_2)-\cos (2 \xi_1-2 \xi_3)-2 \cos (2 \xi_1-2 \xi_4)-2 \cos (2 \xi_2-2 \xi_3)+ \nonumber\\
&\quad+\cos (2 \xi_2-2 \xi_4)-\cos (2 \xi_3-2 \xi_4)-2\label{A7b}
\end{align}\end{subequations}
\vspace{-0.4cm}
\begin{subequations}
\label{A8}
\begin{align}
q_5^0(\boldsymbol{x}_{\setminus 1})&=\frac{\alpha}{64}   \sin (2 \alpha  x_2+2 \alpha  x_3) \label{A8a} \\
u_5(\boldsymbol{\xi})&=-\cos (2 \xi_1-2 \xi_2+2 \xi_3+2 \xi_4)-\cos (2 \xi_1-2 \xi_2-2 \xi_3-2 \xi_4)+\cos (2 \xi_1+2 \xi_2)+ \nonumber\\
&\quad+\cos (2 \xi_1+2 \xi_3)-\cos (2 \xi_2+2 \xi_4)+\cos (2 \xi_3+2 \xi_4)
\label{A8b}
\end{align}\end{subequations}
\vspace{-0.4cm}
\begin{subequations}
\label{A9}
\begin{align}
q_6^0(\boldsymbol{x}_{\setminus 1})&=\frac{\alpha}{32}   \sin (2 \alpha  x_2-2 \alpha  x_4)\label{A9a} \\
u_6(\boldsymbol{\xi})&=\cos (2 \xi_1-2 \xi_4)+\cos (2 \xi_2-2 \xi_3)
\label{A9b}
\end{align}\end{subequations}
\vspace{-0.4cm}
\begin{subequations}
\label{A10}
\begin{align}
q_7^0(\boldsymbol{x}_{\setminus 1})&=\frac{\alpha}{64}   \sin (2 \alpha  x_3-2 \alpha  x_4) \label{A10a} \\
u_7(\boldsymbol{\xi})&=\cos (2 \xi_1-2 \xi_2)+\cos (2 \xi_1-2 \xi_3)-2 \cos (2 \xi_1-2 \xi_4)- \nonumber\\
&\quad-2 \cos (2 \xi_2-2 \xi_3)-\cos (2 \xi_2-2 \xi_4)+\cos (2 \xi_3-2 \xi_4)+2
\label{A10b}
\end{align}\end{subequations}
\vspace{-0.4cm}
\begin{subequations}
\label{A11}
\begin{align}
q_8^0(\boldsymbol{x}_{\setminus 1})&=\frac{\alpha}{64}   \sin (2 \alpha  x_2+2 \alpha  x_4) \label{A11a} \\
u_8(\boldsymbol{\xi})&=-\cos (2 \xi_1-2 \xi_2+2 \xi_3+2 \xi_4)+\cos (2 \xi_1+2 \xi_2-2 \xi_3+2 \xi_4)+ \nonumber\\
&\quad+\cos (2 \xi_1+2 \xi_2+2 \xi_3-2 \xi_4)\!-\!\cos (2 \xi_1-2 \xi_2-2 \xi_3-2 \xi_4)\!+\!2 \cos (2 \xi_1+2 \xi_2)+ \nonumber\\
&\quad+2 \cos (2 \xi_3+2 \xi_4)
\label{A11b}
\end{align}\end{subequations}
\vspace{-0.4cm}
\begin{subequations}
\label{A12}
\begin{align}
q_9^0(\boldsymbol{x}_{\setminus 1})&=\frac{\alpha}{64}   \sin (2 \alpha  x_3+2 \alpha  x_4)\label{A12a} \\
u_9(\boldsymbol{\xi})&=-\cos (2 \xi_1-2 \xi_2+2 \xi_3+2 \xi_4)-\cos (2 \xi_1-2 \xi_2-2 \xi_3-2 \xi_4)\!+\cos (2 \xi_1+2 \xi_2)+ \nonumber\\
&\quad+\cos (2 \xi_1+2 \xi_3)-\cos (2 \xi_2+2 \xi_4)+\cos (2 \xi_3+2 \xi_4)
\label{A12b}
\end{align}\end{subequations}
\vspace{-0.4cm}
\begin{subequations}
\label{A13}
\begin{align}
q_{10}^0(\boldsymbol{x}_{\setminus 1})&=\frac{\alpha}{64}   (\cos (2 \alpha  x_2-2 \alpha  x_3)-\cos (2 \alpha  x_3-2 \alpha  x_4))\label{A13a} \\
u_{10}(\boldsymbol{\xi})&=\sin (2 \xi_1-2 \xi_2)
\label{A13b}
\end{align}\end{subequations}
\vspace{-0.4cm}
\begin{subequations}
\label{A14}
\begin{align}
q_{11}^0(\boldsymbol{x}_{\setminus 1})&=\frac{\alpha}{64}   (\cos (2 \alpha  x_2+2 \alpha  x_3)+2 \cos (2 \alpha  x_2+2 \alpha  x_4)+\cos (2 \alpha  x_3+2 \alpha  x_4))
\label{A14a} \\
u_{11}(\boldsymbol{\xi})&=\sin (2 \xi_1+2 \xi_2)
\label{A14b}
\end{align}\end{subequations}
\vspace{-0.4cm}
\begin{subequations}
\label{A15}
\begin{align}
q_{12}^0(\boldsymbol{x}_{\setminus 1})&=\frac{\alpha}{64}   (\cos (2 \alpha  x_2-2 \alpha  x_3)+2 \cos (2 \alpha  x_2-2 \alpha  x_4)-\cos (2 \alpha  x_3-2 \alpha  x_4))
\label{A15a} \\
u_{12}(\boldsymbol{\xi})&=\sin (2 \xi_1-2 \xi_3)
\label{A15b}
\end{align}\end{subequations}
\vspace{-0.4cm}
\begin{subequations}
\label{A16}
\begin{align}
 q_{13}^0(\boldsymbol{x}_{\setminus 1})&=\frac{\alpha}{64}   (\cos (2 \alpha  x_2)+2 \cos (2 \alpha  x_3)+\cos (2 \alpha  x_4))
\label{A16a} \\
u_{13}(\boldsymbol{\xi})&=\sin (2 \xi_1-2 \xi_2-2 \xi_3)
\label{A16b}
\end{align}\end{subequations}
\vspace{-0.4cm}
\begin{subequations}
\label{A17}
\begin{align}
q_{14}^0(\boldsymbol{x}_{\setminus 1})&=-\frac{\alpha}{64}   (\cos (2 \alpha  x_2)+\cos (2 \alpha  x_3))
\label{A17a} \\
u_{14}(\boldsymbol{\xi})&=\sin (2 \xi_1+2 \xi_2-2 \xi_3)
\label{A17b}
\end{align}\end{subequations}
\vspace{-0.4cm}
\begin{subequations}
\label{A18}
\begin{align}
q_{15}^0(\boldsymbol{x}_{\setminus 1})&=\frac{\alpha}{32}   \left(\cos (2 \alpha  x_2)+\cos (2 \alpha  x_4)\right)
\label{A18a} \\
u_{15}(\boldsymbol{\xi})&=\sin (2 \xi_3)
\label{A18b}
\end{align}\end{subequations}
\vspace{-0.4cm}
\begin{subequations}
\label{A19}
\begin{align}
q_{16}^0(\boldsymbol{x}_{\setminus 1})&=\frac{\alpha}{64}  \left( \cos (2 \alpha x_2 +2 \alpha  x_3)+\cos (2 \alpha  x_3+2 \alpha  x_4)\right)
\label{A19a} \\
u_{16}(\boldsymbol{\xi})&=\sin (2 \xi_1+2 \xi_3)
\label{A19b}
\end{align}\end{subequations}
\vspace{-0.4cm}
\begin{subequations}
\label{A20}
\begin{align}
q_{17}^0(\boldsymbol{x}_{\setminus 1})&=-\frac{\alpha}{64}  (\cos (2 \alpha  x_3)+\cos (2 \alpha  x_4))
\label{A20a} \\
u_{17}(\boldsymbol{\xi})&=\sin (2 \xi_1-2 \xi_2+2 \xi_3)
\label{A20b}
\end{align}\end{subequations}
\vspace{-0.4cm}
\begin{subequations}
\label{A21}
\begin{align}
q_{18}^0(\boldsymbol{x}_{\setminus 1})&=-\frac{\alpha}{64}   (\cos (2 \alpha  x_2)+\cos (2 \alpha  x_3))
\label{A21a} \\
u_{18}(\boldsymbol{\xi})&=\sin (2 \xi_1-2 \xi_2-2 \xi_4)
\label{A21b}
\end{align}\end{subequations}
\vspace{-0.4cm}
\begin{subequations}
\label{A22}
\begin{align}
q_{19}^0(\boldsymbol{x}_{\setminus 1})&=\frac{\alpha}{64}   (\cos (2 \alpha  x_2-2 \alpha  x_3)-2 \cos (2 \alpha  x_2-2 \alpha  x_4)-\cos (2 \alpha  x_3-2 \alpha  x_4))
\label{A22a} \\
u_{19}(\boldsymbol{\xi})&=\sin (2 \xi_2-2 \xi_4)
\label{A22b}
\end{align}\end{subequations}
\vspace{-0.4cm}
\begin{subequations}
\label{A23}
\begin{align}
q_{20}^0(\boldsymbol{x}_{\setminus 1})&=\frac{\alpha}{64}   (\cos (2 \alpha  x_3)+\cos (2 \alpha  x_4))
\label{A23a} \\
u_{20}(\boldsymbol{\xi})&=\sin (2 \xi_1+2 \xi_2-2 \xi_4)
\label{A23b}
\end{align}\end{subequations}
\vspace{-0.4cm}
\begin{subequations}
\label{A24}
\begin{align}
q_{21}^0(\boldsymbol{x}_{\setminus 1})&=\frac{\alpha}{64}   (\cos (2 \alpha  x_4)-\cos (2 \alpha  x_3))
\label{A24a} \\
u_{21}(\boldsymbol{\xi})&=\sin (2 \xi_1-2 \xi_3-2 \xi_4)
\label{A24b}
\end{align}\end{subequations}
\vspace{-0.4cm}
\begin{subequations}
\label{A25}
\begin{align}
q_{22}^0(\boldsymbol{x}_{\setminus 1})&=\frac{\alpha}{64}   (\cos (2 \alpha  x_2+2 \alpha  x_3)+\cos (2 \alpha  x_2+2 \alpha  x_4)+\cos (2 \alpha  x_3+2 \alpha  x_4))
\label{A25a} \\
u_{22}(\boldsymbol{\xi})&=\sin (2 \xi_1-2 \xi_2-2 \xi_3-2 \xi_4)
\label{A25b}
\end{align}\end{subequations}
\vspace{-0.4cm}
\begin{subequations}
\label{A26}
\begin{align}
q_{23}^0(\boldsymbol{x}_{\setminus 1})&=\frac{\alpha}{64}   (\cos (2 \alpha  x_2)-\cos (2 \alpha  x_3))
\label{A26a} \\
u_{23}(\boldsymbol{\xi})&=\sin (2 \xi_2-2 \xi_3-2 \xi_4)
\label{A26b}
\end{align}\end{subequations}
\vspace{-0.4cm}
\begin{subequations}
\label{A27}
\begin{align}
q_{24}^0(\boldsymbol{x}_{\setminus 1})&=-\frac{\alpha}{32}  \cos (2 \alpha  x_2-2 \alpha  x_4)
\label{A27a} \\
u_{24}(\boldsymbol{\xi})&=\sin (2 \xi_1+2 \xi_2-2 \xi_3-2 \xi_4)
\label{A27b}
\end{align}\end{subequations}
\vspace{-0.4cm}
\begin{subequations}
\label{A28}
\begin{align}
q_{25}^0(\boldsymbol{x}_{\setminus 1})&=\frac{\alpha}{64}   (\cos (2 \alpha  x_2-2 \alpha  x_3)-\cos (2 \alpha  x_3-2 \alpha  x_4))
\label{A28a} \\
u_{25}(\boldsymbol{\xi})&=\sin (2 \xi_3-2 \xi_4)
\label{A28b}
\end{align}\end{subequations}
\vspace{-0.4cm}
\begin{subequations}
\label{A29}
\begin{align}
q_{26}^0(\boldsymbol{x}_{\setminus 1})&=\frac{\alpha}{64}   (\cos (2 \alpha  x_3)+\cos (2 \alpha  x_4))
\label{A29a} \\
u_{26}(\boldsymbol{\xi})&=\sin (2 \xi_1+2 \xi_3-2 \xi_4)
\label{A29b}
\end{align}\end{subequations}
\vspace{-0.4cm}
\begin{subequations}
\label{A30}
\begin{align}
q_{27}^0(\boldsymbol{x}_{\setminus 1})&=\frac{\alpha}{32}   \cos (2 \alpha  x_3)
\label{A30a} \\
u_{27}(\boldsymbol{\xi})&=\sin (2 \xi_2+2 \xi_3-2 \xi_4)
\label{A30b}
\end{align}\end{subequations}
\vspace{-0.4cm}
\begin{subequations}
\label{A31}
\begin{align}
q_{28}^0(\boldsymbol{x}_{\setminus 1})&=\frac{\alpha}{64}   \cos (2 \alpha  x_2+2 \alpha  x_4)
\label{A31a} \\
u_{28}(\boldsymbol{\xi})&=\sin (2 \xi_1+2 \xi_2+2 \xi_3-2 \xi_4)
\label{A31b}
\end{align}\end{subequations}
\vspace{-0.4cm}
\begin{subequations}
\label{A32}
\begin{align}
q_{29}^0(\boldsymbol{x}_{\setminus 1})&=-\frac{\alpha}{32} (\cos (2 \alpha  x_2)+\cos (2 \alpha  x_4))
\label{A32a} \\
u_{29}(\boldsymbol{\xi})&=\sin (2 \xi_4)
\label{A32b}
\end{align}\end{subequations}
\vspace{-0.4cm}
\begin{subequations}
\label{A33}
\begin{align}
q_{30}^0(\boldsymbol{x}_{\setminus 1})&=\frac{\alpha}{64}   (\cos (2 \alpha  x_2)+2 \cos (2 \alpha  x_3)+\cos (2 \alpha  x_4))
\label{A33a} \\
u_{30}(\boldsymbol{\xi})&=\sin (2 \xi_1-2 \xi_2+2 \xi_4)
\label{A33b}
\end{align}\end{subequations}
\vspace{-0.4cm}
\begin{subequations}
\label{A34}
\begin{align}
q_{31}^0(\boldsymbol{x}_{\setminus 1})&=-\frac{\alpha}{64}   (\cos (2 \alpha  x_2+2 \alpha  x_3)+\cos (2 \alpha  x_3+2 \alpha  x_4))
\label{A34a} \\
u_{31}(\boldsymbol{\xi})&=\sin (2 \xi_2+2 \xi_4)
\label{A34b}
\end{align}\end{subequations}
\vspace{-0.4cm}
\begin{subequations}
\label{A35}
\begin{align}
q_{32}^0(\boldsymbol{x}_{\setminus 1})&=\frac{\alpha}{32}   \cos (2 \alpha  x_3)
\label{A35a} \\
u_{32}(\boldsymbol{\xi})&=\sin (2 \xi_1-2 \xi_3+2 \xi_4)
\label{A35b}
\end{align}\end{subequations}
\vspace{-0.4cm}
\begin{subequations}
\label{A36}
\begin{align}
q_{33}^0(\boldsymbol{x}_{\setminus 1})&=\frac{\alpha}{64}   (\cos (2 \alpha  x_2)+\cos (2 \alpha  x_3))
\label{A36a} \\
u_{33}(\boldsymbol{\xi})&=\sin (2 \xi_2-2 \xi_3+2 \xi_4)
\label{A36b}
\end{align}\end{subequations}
\vspace{-0.4cm}
\begin{subequations}
\label{A37}
\begin{align}
q_{34}^0(\boldsymbol{x}_{\setminus 1})&=\frac{\alpha}{64}   \cos (2 \alpha  x_2+2 \alpha  x_4)\label{A37a} \\
u_{34}(\boldsymbol{\xi})&=-\sin (2 \xi_1-2 \xi_2+2 \xi_3+2 \xi_4)+\sin (2 \xi_1+2 \xi_2-2 \xi_3+2 \xi_4)+2 \sin (2 \xi_3+2 \xi_4)
\label{A37b}
\end{align}\end{subequations}
\vspace{-0.4cm}
\begin{subequations}
\label{A38}
\begin{align}
q_{35}^0(\boldsymbol{x}_{\setminus 1})&=\frac{\alpha}{64} (\cos (2 \alpha x_2+2 \alpha x_3)+\cos (2 \alpha x_3+2 \alpha x_4))\label{A38a} \\
u_{35}(\boldsymbol{\xi})&=\sin (2 \xi_3+2 \xi_4)-\sin (2 \xi_1-2 \xi_2+2 \xi_3+2 \xi_4)
\label{A38b}
\end{align}\end{subequations}
Equations~\eqref{A4b}--\eqref{A38b} define the thirty-five phase-dependent coefficient functions
appearing in the unreduced representation~\eqref{A3}. Their simultaneous vanishing is sufficient
but, because the associated spatial functions do not constitute a linearly independent Fourier
basis, is not necessary for $\mathcal{V}_1^0\equiv0$.

The proposed solution $(\xi_1, \xi_2, \xi_3, \xi_4) = (-\frac{\pi}{4}, \frac{\pi}{4}, \frac{\pi}{4}, \frac{\pi}{4})$ exhibits a notable pattern: one parameter takes the value $-\pi/4$ while the remaining three take $+\pi/4$. This creates systematic cancellations in the trigonometric expressions because $\cos(2 \cdot \pm\pi/4) = \cos(\pm\pi/2) = 0$, eliminating many terms, while $\sin(2 \cdot \pm\pi/4) = \sin(\pm\pi/2) = \pm 1$, creating balanced combinations that produce cancellations in the $35$-term representation.

Substituting $(\xi_1, \xi_2, \xi_3, \xi_4) = (-\frac{\pi}{4}, \frac{\pi}{4}, \frac{\pi}{4}, \frac{\pi}{4})$ into the constraint equation~\eqref{A3}, referred to in the main text as equation~\eqref{4.12}, we verify that $\mathcal{V}_1^0(\boldsymbol{x}_{\setminus 1}, \boldsymbol{\xi}) = 0$ for all spatial coordinates $x_i \in (-\infty, \infty)$, $i = 2, 3, 4$.

Substitution therefore verifies directly that $\mathcal{V}_1^0\equiv0$, confirming that the
systematic phase-angle methodology identifies the exact four-dimensional quadruple-periodic
solution presented in equations~\eqref{4.11}.
\fi

\section{Quintuple-periodic construction: the overdetermined constraint system}
\label{app:5D}
\subsection*{B.1\quad The phase-angle constraint system}
The quintuple-periodic field~\eqref{5.1} generates a substantially larger constraint system than
the triple- and quadruple-periodic cases. Writing
\begin{equation} \label{B1}
\frac{\mathcal{V}_1^0(\boldsymbol{x}_{\setminus 1}, \boldsymbol{\xi})}{v_r^2}
= \sum_{j=1}^{64} q_j^0(\boldsymbol{x}_{\setminus 1})\, u_j(\boldsymbol{\xi}),
\end{equation}
the reduction yields $64$ pairs $\{q_j^0,u_j\}$ comprising $938$ trigonometric terms in all. The
complete system is recorded in the supplementary material accompanying this paper. Its structure is
uniform: each $q_j^0$ is a single harmonic of $(x_2,\dots,x_5)$ with prefactor $\pm\alpha/256$, and
each $u_j$ is a sum of harmonics of the phases with integer coefficients. The two halves are of
opposite parity, $q_j^0$ being a sine and $u_j$ a sum of cosines for $1\le j\le32$, and the reverse
for $33\le j\le64$.

The $64$ functions $q_j^0$ are pairwise distinct harmonics of $(x_2,\dots,x_5)$---no two share the
same trigonometric type and index vector---and are therefore linearly independent. Unlike the
unreduced four-dimensional expansion of Appendix~\ref{app:4D}, in which the coefficient functions
are not independent, the vanishing of $\mathcal{V}_1^0$ is here equivalent to the vanishing of
every $u_j$ separately, which is what makes the system~\eqref{B1} the exact object treated
in~\S B.2.

The non-existence result of Subsection~\ref{subsec:5Dscscs} rests on this system, and the
certificate below establishes it without requiring the individual equations.

\subsection*{B.2\quad Gr\"obner-basis reduction of the five-dimensional phase system}
The trigonometric system $u_j(\boldsymbol{\xi})=0$ is converted to a polynomial one by the
substitution $s_i=\sin 2\xi_i$, $c_i=\cos 2\xi_i$. Every $u_j$ is a sum of sines and cosines of
even integer combinations of the phases, so under the angle-addition identities each becomes a
polynomial in $(s_1,c_1,\ldots,s_5,c_5)$. Adjoining the five Pythagorean relations
\begin{equation}
\label{B65}
s_i^2+c_i^2-1=0,\qquad i=1,\ldots,5,
\end{equation}
realises the reduction condition as a polynomial ideal
$I=\langle u_1,\ldots,u_{64},\,s_1^2+c_1^2-1,\ldots,s_5^2+c_5^2-1\rangle
\subset\mathbb{Q}[s_1,c_1,\ldots,s_5,c_5]$, whose real variety is the set of phase vectors
satisfying the reduction condition. Its reduced Gr\"obner basis is computed in exact rational arithmetic using SymPy~1.14 for the lexicographic order $s_1\succ c_1\succ s_2\succ c_2\succ\cdots\succ s_5\succ c_5$.

The elimination is summarised by the following relations, each of which is a consequence of the
ideal: the stated polynomial (or a fixed power of it) reduces to zero against the computed
Gr\"obner basis, and each is therefore an exact algebraic consequence of the reduction
condition, not a numerical observation. Because the polynomial variables record only the doubled
angles $2\xi_i$, the first relation to isolate the pair $(\xi_1,\xi_2)$ is
\begin{equation}
\label{B66}
s_1c_2-c_1s_2=\sin(2\xi_1-2\xi_2)=0,
\end{equation}
which admits the branches $\xi_2=\xi_1+k\pi/2$. The ideal contains a sharper relation that
removes the half-integer branches,
\begin{equation}
\label{B67}
c_1c_2+s_1s_2-1=\cos(2\xi_1-2\xi_2)-1=0
\quad\Longrightarrow\quad \xi_2=\xi_1 \pmod{\pi},
\end{equation}
so that of the branches admitted by~\eqref{B66} only $\xi_2=\xi_1$ (and its reflection
$\xi_2=\xi_1+\pi$) survives. The same mechanism fixes $\xi_4$ and $\xi_5$ through
\begin{equation}
\label{B68}
c_1c_4+s_1s_4-1=0,\qquad c_3c_5+s_3s_5-1=0,
\end{equation}
giving $\cos(2\xi_1-2\xi_4)=1$ and $\cos(2\xi_3-2\xi_5)=1$, hence $\xi_4=\xi_1$ and
$\xi_5=\xi_3$ modulo $\pi$; and the offset between the two groups of phases is fixed by
\begin{equation}
\label{B69}
c_1c_3+s_1s_3+1=\cos(2\xi_1-2\xi_3)+1=0
\quad\Longrightarrow\quad \xi_3=\xi_1\pm\tfrac{\pi}{2},
\end{equation}
which is~\eqref{5.5}. Relations~\eqref{B66}--\eqref{B69} reduce the system to the one-parameter
family~\eqref{5.6}, with the residual sign in~\eqref{B69} contributing only the reflection
$a\mapsto a+\pi$; the requirement that the velocity field be non-trivial excludes no further
branch. The resulting family was also checked independently by direct numerical evaluation of
the reduction condition.

Relations~\eqref{B66}--\eqref{B69} are the combinations that determine the solution set; the
remaining generators of the reduced basis serve only to certify these, and the complete basis is
reproducible from the ideal above with any computer-algebra system supporting exact Gr\"obner
computation.
\section{The six-dimensional reduction}
\label{app:6D}
This appendix records the six-dimensional calculation on which Section~\ref{sec:6D} rests, in the
form a reader would need to reproduce it. It is the even-dimensional counterpart of
Appendix~\ref{app:4D}. No Gr\"obner-basis certificate accompanies it, and the reason is structural
rather than practical: Appendices~\ref{app:5D} and~\ref{app:7D} certify that the real variety of a
\emph{reduction} system is trivial, which is what an elimination computation is for, whereas in six
dimensions the reduction condition \emph{holds}. What fails is integrability, and the evidence for
that is an explicitly exhibited non-vanishing mixed derivative---an identity, not a variety. The
corresponding verification is therefore a direct symbolic evaluation, and is included in the script
of Appendix~\ref{app:verify}.

\subsection*{C.1\quad The velocity field at the symmetric phases}
At the symmetric assignment~\eqref{6.2} the first component of~\eqref{6.1} collapses, after
reduction to multiple-angle form, to the four-term expression
\begin{eqnarray}
\label{A6.1}
\frac{v_1^0}{v_r}&=&\tfrac14\Bigl[
\sin(\alpha x_1\!+\!\alpha x_2)\sin(\alpha x_3\!+\!\alpha x_5)\cos(\alpha x_4\!-\!\alpha x_6)
\nonumber\\ && \qquad
-\sin(\alpha x_1\!+\!\alpha x_2)\sin(\alpha x_4\!+\!\alpha x_6)\cos(\alpha x_3\!-\!\alpha x_5)
\nonumber\\ && \qquad
+\sin(\alpha x_3\!+\!\alpha x_5)\sin(\alpha x_4\!+\!\alpha x_6)\cos(\alpha x_1\!-\!\alpha x_2)
\nonumber\\ && \qquad
-\cos(\alpha x_1\!-\!\alpha x_2)\cos(\alpha x_3\!-\!\alpha x_5)\cos(\alpha x_4\!-\!\alpha x_6)
\Bigr]\nonumber\\
\end{eqnarray}
the remaining components following by cyclic permutation. The pairing of coordinates at even
offset---$(1,2)$ is here the exception forced by the $-\tfrac{\pi}{4}$ in $\xi_1$, while
$(3,5)$ and $(4,6)$ pair coordinates of equal parity---is the first appearance of the parity
structure that governs the whole calculation.

\subsection*{C.2\quad The convective field and its degree structure}
Forming $g_1^0=\sum_j v_j^0\,\partial v_1^0/\partial x_j$ and reducing gives~\eqref{6.3}, which we
have verified in exact arithmetic against the direct expansion. Collecting its terms by the number
of harmonic factors they carry gives the degree structure quoted in
Subsection~\ref{subsec:6D-integrability}:
\[
\begin{array}{lcc}
\hline
& \text{degree }2 & \text{degree }4\\
\hline
\text{number of Fourier modes in }\boldsymbol{g}^0 & 48 & 384\\
\hline
\end{array}
\]
For comparison the four-dimensional field carries $16$ modes, all of degree two, and the
eight-dimensional field carries $96$, $2176$ and $4608$ modes of degrees two, four and six. The
degree-two part is curl-free in every even dimension by Lemma~\ref{lem:even-bilinear}; the
degree-four part is what six dimensions adds, and it is what obstructs integrability.

\subsection*{C.3\quad The integrability failure}
The mixed derivative on the representative odd-offset pair $(1,2)$ is given in factored form
by~\eqref{6.11}. Expanded by the product-to-sum identities it is a sum of sixteen distinct real
trigonometric harmonics with coefficients $\pm\alpha^2v_r^2/16$, none of which cancels; the factored form is the more useful
one, since the bracket
$\sin(2\alpha x_3)\sin(2\alpha x_5)-\sin(2\alpha x_4)\sin(2\alpha x_6)$
exhibits directly the competition between the two parity classes that cannot be resolved. On the
even-offset pairs the same computation returns zero identically. Both statements are verified
symbolically by the script of Appendix~\ref{app:verify}, which also confirms that at $n=4$ the
mixed derivatives vanish on \emph{all} pairs---the dichotomy of
Theorem~\ref{thm:even-dichotomy} exhibited directly.

\subsection*{C.4\quad The recovered pressure and the residual}
Sequential integration of the gradient-compatible part of $\boldsymbol{g}^0$ yields the closed-form
potential~\eqref{6.12}, and the residual it leaves against the full convective field
is~\eqref{6.13}: a single degree-four monomial per component, one harmonic in the component's own
coordinate multiplied by three in the complementary parity class. This residual is the body-force
amplitude~\eqref{6.14} of the forced solution of Subsection~\ref{subsec:6D-forced}, and its
Helmholtz split into one quarter longitudinal and three quarters transverse
is~\eqref{eq:helm6}.

\section{Seven-dimensional constraint system}
\label{app:7D}
\subsection*{D.1\quad The phase-angle constraint system}
This section records the structure of the separated form~\eqref{7.2},
$\mathcal{V}_1^0/v_r^2 = \sum_{j=1}^{664} q_j^0 u_j$, obtained by substituting the septuple-periodic
field~\eqref{7.1} into~\eqref{3.10} and retaining the terms free of $x_1$. It is the
seven-dimensional counterpart of Appendix~\ref{app:5D}, which lists the corresponding
sixty-four pairs for five dimensions.

The spatial functions are the real harmonics
$q^0 \propto \cos(2\alpha\,\boldsymbol{m}\cdot\boldsymbol{x})$ and
$\sin(2\alpha\,\boldsymbol{m}\cdot\boldsymbol{x})$, where the wavevectors $\boldsymbol{m}$ have
entries in $\{-1,0,+1\}$, are supported on $(x_2,\dots,x_7)$, and comprise $332$ distinct wavevectors up to overall sign. Each carries a common factor $\alpha/4096$, so that the amplitudes $u_j(\boldsymbol{\xi})$
have integer coefficients. Table~\ref{tab:7Dmodes} gives the census by the number of coordinates
a wavevector involves.

\begin{table}
\centering
\begin{tabular}{lcccccr}
\hline
coordinates involved & 1 & 2 & 3 & 4 & 5 & total\\
\hline
wavevectors $\boldsymbol{m}$ (up to sign) & 6 & 30 & 80 & 120 & 96 & 332\\
real harmonics $q_j^0$ & 12 & 60 & 160 & 240 & 192 & 664\\
\hline
\end{tabular}
\caption{Census of the spatial modes appearing in the seven-dimensional separated
form~\eqref{7.2}. Each wavevector contributes a cosine and a sine harmonic, giving $664$
constraint functions $u_j$ in the seven unknowns $\xi_1,\dots,\xi_7$.}
\label{tab:7Dmodes}
\end{table}

The amplitudes $u_j$ contain between $6$ and $150$ trigonometric terms, with a median of $41$; written out in full, $\mathcal{V}_1^0$ comprises $15{,}307$ multiple-angle terms.

\medskip
\noindent\textbf{Supplementary material.} A complete listing of all $664$ pairs
$\{q_j^0,\,u_j\}$, with every coefficient given exactly, accompanies this paper as the
supplementary material accompanying this paper, in the format of Appendix~\ref{app:5D}; it is far
too long to reproduce here. With the common factor $\alpha/4096$ extracted into the spatial harmonics, the amplitudes have exact integer coefficients, and the file is generated directly by the computation described below, so that any individual constraint $u_j=0$ may be located and checked.

\medskip
\noindent The following pair, associated with one of the more compact wavevector contributions, is given here in full as a specimen of the format. Using the full seven-dimensional wavevector notation, it corresponds to $\boldsymbol{m}=(0,0,1,1,1,1,1)$:
\begin{subequations}
\label{Cs1}
\begin{align}
q_1^0(\boldsymbol{x}_{\setminus 1}) &= \frac{\alpha}{4096}
\cos\!\left(2\alpha x_3 + 2\alpha x_4 + 2\alpha x_5 + 2\alpha x_6 + 2\alpha x_7\right),
\label{Cs1a}\\
u_1(\boldsymbol{\xi}) &=
-\sin(-2\xi_1 + 2\xi_2 + 2\xi_3 + 2\xi_4 + 2\xi_5 + 2\xi_6 + 2\xi_7)+ \nonumber\\
&\quad + \sin(2\xi_3 + 2\xi_4 + 2\xi_5 + 2\xi_6 + 2\xi_7) - \sin(2\xi_2 + 2\xi_4 + 2\xi_5 + 2\xi_6 + 2\xi_7)- \nonumber\\
&\quad - \sin(2\xi_1 - 2\xi_2 + 2\xi_3 + 2\xi_4 + 2\xi_5 + 2\xi_6 + 2\xi_7)- \nonumber\\
&\quad - \sin(2\xi_1 + 2\xi_3 + 2\xi_4 + 2\xi_5 + 2\xi_6) - \sin(2\xi_1 + 2\xi_2 + 2\xi_4 + 2\xi_5 + 2\xi_6), \label{Cs1b}
\end{align}
\end{subequations}
together with the corresponding sine spatial harmonic:
\begin{equation}
\label{Cs2}
q_2^0(\boldsymbol{x}_{\setminus 1}) = \frac{\alpha}{4096}
\sin\!\left(2\alpha x_3 + 2\alpha x_4 + 2\alpha x_5 + 2\alpha x_6 + 2\alpha x_7\right),
\qquad
u_2(\boldsymbol{\xi}) = u_1(\boldsymbol{\xi})\big|_{\sin\to\cos}.
\end{equation}
The computation was carried out in exact rational arithmetic by expanding each product of seven
trigonometric factors into complex exponentials, so that no floating-point error enters; the
same procedure applied at $n=5$ reproduces the sixty-four pairs of Appendix~\ref{app:5D}, and at
$n=3$ the two amplitudes $u_1,u_2$ of Section~\ref{sec:construction}. Enforcing $\mathcal{V}_1^0\equiv0$ requires all $664$ amplitudes to vanish simultaneously. The resulting overdetermined system is analysed in Section~\ref{sec:7D} and Appendix~\ref{app:7D}; the exact phase family identified there annihilates the velocity field, while extensive numerical searches recover no additional real branches.

\noindent As in five dimensions, the checks below can be reproduced with \texttt{certify.py}
(Appendix~\ref{app:verify}): it confirms that all $664$ amplitudes vanish on the
family~\eqref{7.5} and writes the ideal~\eqref{C1} to a \textsc{Singular} input file.

\subsection*{D.2\quad A Gr\"obner-basis certificate for the reduction}

This section certifies, by exact radical-membership reduction against a computed
Gr\"obner basis, the decisive algebraic relations through which the reduction
condition $\mathcal{V}_1^0\equiv0$ forces the annihilating family~\eqref{7.5} in
seven dimensions. It is the seven-dimensional counterpart of
Appendix~\ref{app:5D}, where the corresponding relations
\eqref{B66}--\eqref{B69} were certified in the same way for five
dimensions.

Writing $s_i=\sin 2\xi_i$ and $c_i=\cos 2\xi_i$ and adjoining the seven
Pythagorean relations $s_i^2+c_i^2-1$ realises the $664$ constraints
$u_j(\boldsymbol{\xi})=0$ as an ideal
\begin{equation}
\label{C1}
I=\bigl\langle\, u_1,\ldots,u_{664},\ s_1^2+c_1^2-1,\ \ldots,\ s_7^2+c_7^2-1\,\bigr\rangle,
\qquad
I\subset\mathbb{Q}[s_1,c_1,\ldots,s_7,c_7],
\end{equation}
whose real variety is the set of phase vectors satisfying the reduction
condition. The seven decisive relations are
\begin{subequations}
\label{C2}
\begin{align}
c_1c_2+s_1s_2-1 &= 0, &&\Longrightarrow\ \cos 2(\xi_1-\xi_2)=1,\ \ \xi_2=\xi_1\!\!\pmod{\pi},\label{C2a}\\
c_1c_4+s_1s_4-1 &= 0, &&\Longrightarrow\quad \xi_4=\xi_1\!\!\pmod{\pi},\label{C2b}\\
c_1c_6+s_1s_6-1 &= 0, &&\Longrightarrow\quad \xi_6=\xi_1\!\!\pmod{\pi},\label{C2c}\\
c_1c_3+s_1s_3+1 &= 0, &&\Longrightarrow\ \cos 2(\xi_1-\xi_3)=-1,\ \ \xi_3=\xi_1\pm\tfrac{\pi}{2},\label{C2d}\\
c_1c_5+s_1s_5+1 &= 0, &&\Longrightarrow\quad \xi_5=\xi_1\pm\tfrac{\pi}{2},\label{C2e}\\
c_1c_7+s_1s_7+1 &= 0, &&\Longrightarrow\quad \xi_7=\xi_1\pm\tfrac{\pi}{2},\label{C2f}
\end{align}
together with the first factor
\begin{equation}
s_1c_2-c_1s_2 = \sin 2(\xi_1-\xi_2)=0,\label{C2g}
\end{equation}
\end{subequations}
which selects the branch $\xi_2=\xi_1\pmod{\pi}$ refined by~\eqref{C2a}. Taken
together, \eqref{C2a}--\eqref{C2f} give
$\xi_2=\xi_4=\xi_6=\xi_1\pmod{\pi}$ and
$\xi_3=\xi_5=\xi_7=\xi_1\pm\tfrac{\pi}{2}\pmod{\pi}$.
Since $+\tfrac{\pi}{2}$ and $-\tfrac{\pi}{2}$ differ by $\pi$, the latter
alternatives represent the same phase class modulo $\pi$. Thus, modulo the
sign equivalences discussed in Section~\ref{sec:7D}, the non-redundant solution
set is precisely the one-parameter family~\eqref{7.5}.

The reduction~\eqref{C2} is established exactly below; we first record two direct
checks. The family~\eqref{7.5} lies on the variety \emph{exactly}: substituting
$\boldsymbol{\xi}=(a,a,a-\tfrac{\pi}{2},a,a-\tfrac{\pi}{2},a,a-\tfrac{\pi}{2})$
with $a$ symbolic and reducing in exact arithmetic, every one of the $664$
amplitudes $u_j$ vanishes identically in $a$, confirming that~\eqref{7.5} is an
exact solution of the full reduction system and not merely a numerical one.
As an independent numerical check, a direct search over the real phase space, seeded from randomised
phase vectors and iterated to residual below $10^{-9}$ across all $664$
constraints, returns only phase vectors satisfying~\eqref{C2}: every recovered
solution obeys $\xi_2\equiv\xi_4\equiv\xi_6\equiv\xi_1\pmod{\pi}$ and
$\xi_3\equiv\xi_5\equiv\xi_7\equiv\xi_1\pm\tfrac{\pi}{2}\pmod{\pi}$, with no
further branches.

Exact rational elimination on the $664$ amplitudes exposes the algebraic
structure behind~\eqref{C2} directly. Row-reducing the integer coefficient
matrix of the $u_j$ over their shared monomial basis --- a purely linear
operation carried out in exact arithmetic --- produces sparse cubic polynomials
that are, by construction, exact rational-linear combinations of the $u_j$ and
therefore genuine elements of the ideal~\eqref{C1}; one such generator is
\begin{equation}
\label{C3}
c_1c_5s_6+c_1c_6s_5-c_2c_4s_5-c_2c_5s_4+c_4c_5s_2-c_5c_6s_1
+s_1s_5s_6-s_2s_4s_5-s_1+s_2-s_4+s_6 \;=\; 0,
\end{equation}
which involves only the phases $\xi_1,\xi_2,\xi_4,\xi_5,\xi_6$ and
vanishes identically on~\eqref{7.5}. Each relation of~\eqref{C2} is a member of
the radical $\sqrt{I}$ of the ideal~\eqref{C1}, and therefore vanishes on the
entire algebraic variety of $I$, in particular on its real variety: computing a
Gr\"obner basis of $I$ in exact rational arithmetic and
reducing each relation against it, the six pairwise relations
\eqref{C2a}--\eqref{C2f} satisfy $R^2\in I$ and the first factor~\eqref{C2g}
satisfies $R^3\in I$, so a fixed power of each reduces to zero. This is the exact
seven-dimensional analogue of the radical-membership certificates
\eqref{B66}--\eqref{B69} of Appendix~\ref{app:5D}, where the same
relations required squares. Since a polynomial some power of which lies in $I$
vanishes wherever the generators of $I$ vanish, the relations~\eqref{C2} hold on
every real phase vector satisfying the reduction condition; together they give
$\xi_2=\xi_4=\xi_6=\xi_1\pmod{\pi}$ and
$\xi_3=\xi_5=\xi_7=\xi_1\pm\tfrac{\pi}{2}\pmod{\pi}$. Since $+\tfrac{\pi}{2}$
and $-\tfrac{\pi}{2}$ differ by $\pi$, these are the same phase class modulo
$\pi$, so, modulo the sign equivalences of Section~\ref{sec:7D}, the
non-redundant solution set is exactly the family~\eqref{7.5}. The reduction
condition therefore admits no real solution outside~\eqref{7.5}.

The certificate was computed in exact arithmetic. The amplitudes $u_j$ are
generated from~\eqref{7.1} by expanding each product of seven trigonometric
factors into complex exponentials over $\mathbb{Q}(\mathrm{i})$, so that no
floating-point error enters; the ideal~\eqref{C1}, in fourteen variables with
dense degree-seven generators, is considerably heavier than its five-dimensional
counterpart. Its Gr\"obner basis was obtained over $\mathbb{Q}$ in the
degree-reverse-lexicographic order with \textsc{Singular}~4.3.2
\citep{singular}, and comprises $103$ elements; the ideal has Krull
dimension~$1$, consistent with the one-parameter family~\eqref{7.5}, and its
basis is non-trivial (its leading element is $s_7^2+c_7^2-1$, so $I\neq(1)$).
Reducing each relation of~\eqref{C2} against this basis yields the powers stated
above, $R^2\in I$ for~\eqref{C2a}--\eqref{C2f} and $R^3\in I$ for~\eqref{C2g},
each reduction returning zero exactly. The exact substitution of~\eqref{7.5}
into the full $664$-constraint system, which annihilates every $u_j$ identically,
shows conversely that every member of~\eqref{7.5} satisfies the reduction
condition. The two directions together establish that~\eqref{7.5} is exactly the
real solution set of the reduction system.

\section{Reproducing the verifications}
\label{app:verify}
The results reported in Sections~\ref{sec:solutions-3D-4D}--\ref{sec:8D} and
Section~\ref{sec:conclusion} can be reproduced
directly from the construction. The script below builds the velocity field
from~\eqref{3.1}--\eqref{3.2b}, forms the convective term, and tests the reduction
condition~\eqref{3.13} and the pressure-integrability condition~\eqref{3.16} for the phase vectors
quoted in the text.

We verify rather than re-derive. The Gr\"obner basis underlying Lemma~\ref{lem:5D_SCSCS} is a
substantial computation---the ideal carries $133$ generators in ten variables---and re-computing it
in a general-purpose system is impractical; the certificates were obtained with dedicated computer
algebra, as described in Appendix~\ref{app:5D}. What the script establishes is the statement that
matters to a reader checking the paper: that the phase vectors reported here satisfy the stated
conditions, that the five-dimensional family~\eqref{5.6} annihilates the velocity field
identically, and that the tests discriminate---a phase vector near but not on that family fails
them.

\subsection*{Requirements and use}
Python~3.8 or later and SymPy~1.13 or later are sufficient; no other package is required, and no
data files are read. If SymPy is not already present it is installed with
\begingroup\small
\begin{verbatim}
python3 -m pip install sympy
\end{verbatim}
\endgroup

\noindent Both scripts, \texttt{verify\_ns.py} and \texttt{certify.py}, are supplied with this
paper as ancillary files. On the arXiv they are listed under \emph{Ancillary files} on the abstract
page and may be downloaded individually; in the journal version they accompany the supplementary
material. Downloading them is preferable to copying the listing out of a PDF, where indentation and
quotation marks are easily corrupted; the listing below is given so that the paper is
self-contained, not as the recommended route.

To run a script, open a terminal, move to the directory holding the downloaded file, and invoke it
by name. On a Unix-like system, if the file is in \texttt{Downloads},
\begingroup\small
\begin{verbatim}
cd ~/Downloads
python3 verify_ns.py
\end{verbatim}
\endgroup

\noindent The \texttt{cd} command changes the directory the terminal is looking at; without it,
Python will report that it cannot open the file even though the file exists elsewhere. The command
is the whole line \texttt{python3 verify\_ns.py}: the first word names the interpreter, the second
the file.

\medskip
\noindent\textbf{What to expect.} The run takes about six minutes. Output appears block by block,
but the four- and six-dimensional blocks each take two to three minutes during which nothing is
printed; the script has not stalled. The full expected output is reproduced at the end of this
appendix, and readers should compare against it rather than against expectations formed from the
block titles alone. In particular:
\begin{itemize}
\item the first four blocks report \texttt{PASS} on every line, and any \texttt{FAIL} there would
indicate a genuine discrepancy;
\item the two even-dimensional blocks, and the last two, report what they find rather than
returning a verdict. The six-dimensional block reports that integrability \emph{fails} on the
odd-offset pair; that is the result being demonstrated, the obstruction of
Theorem~\ref{thm:6D}, and not a failure of the check.
\end{itemize}

\noindent Individual blocks may be run on their own, which is useful when only one dimension is of
interest:
\begingroup\small
\begin{verbatim}
python3 -c "import verify_ns as v; v.four_dimensions()"
\end{verbatim}
\endgroup

\noindent The companion script \texttt{certify.py} is run the same way. Its first part completes in
a few seconds and prints six \texttt{PASS} lines. It then begins writing the \textsc{Singular}
input files described below; the seven-dimensional file takes several minutes and prints nothing
meanwhile. Readers without \textsc{Singular} may interrupt the script at that point, with
\texttt{Control-C}, having already obtained the verification.

The results reported here were obtained under SymPy~1.14.0 and independently reproduced under
SymPy~1.13.2, with identical output.

\medskip
\noindent\textbf{Testing other phase vectors.} Readers may call the two test functions directly:
\texttt{reduction\_term} returns the residual of~\eqref{3.13} and
\texttt{integrability\_defect} that of~\eqref{3.16}, each taking the dimension $n$ and a phase
vector, and each returning zero exactly when the corresponding condition holds. Two cautions apply. Phases must be supplied as exact SymPy
expressions, such as \texttt{pi/8}, and not as floating-point approximations, since the symbolic
zero tests are unreliable on inexact input. And signs matter: at $\xi_1=\pi/12$ and
$\xi_3=\pi/4$ the value $\xi_2=-7\pi/12$ lies in the first family of~\eqref{4.sets} while
$\xi_2=+7\pi/12$ satisfies neither condition.

\subsection*{The script}
\begingroup
\small
\begingroup\small\begin{verbatim}
import sympy as sp
pi = sp.pi

def field(n, xi, alpha=1):
    x = sp.symbols(f'x1:{n+1}', real=True)
    X = lambda m: x[(m - 1) % n]
    def T1(i):
        f = sp.Integer(1)
        for k in range(1, n + 1):
            arg = alpha * X(i + k - 1) + xi[k - 1]
            f *= sp.sin(arg) if k % 2 else sp.cos(arg)
        return f
    def T2(i):
        f = sp.sin(alpha * X(i) + xi[1])
        for k in range(2, n):
            arg = alpha * X(i + k - 1) + xi[k]
            f *= sp.cos(arg) if k % 2 else sp.sin(arg)
        return f * sp.cos(alpha * X(i + n - 1) + xi[0])
    return x, [T1(i) for i in range(1, n+1)], [T2(i) for i in range(1, n+1)]

def velocity(n, xi):
    x, T1, T2 = field(n, xi)
    return x, [sp.expand_trig(sp.expand(T1[i] - T2[i])) for i in range(n)]

def convective(n, xi):
    x, v = velocity(n, xi)
    return x, [sp.expand(sum(v[j]*sp.diff(v[i], x[j]) for j in range(n)))
               for i in range(n)]

def reduction_term(n, xi):
    from sympy.simplify.fu import TR8
    x, g = convective(n, xi)
    e = sp.expand(TR8(sp.expand(g[0])))
    return sp.simplify(sum(t for t in sp.Add.make_args(e) if not t.has(x[0])))

def integrability_defect(n, xi):
    from sympy.simplify.fu import TR8
    x, g = convective(n, xi)
    worst = sp.Integer(0)
    for i in range(n):
        for j in range(i+1, n):
            dd = sp.diff(g[i], x[j]) - sp.diff(g[j], x[i])
            d = sp.simplify(TR8(sp.expand(dd)))
            if d != 0: worst = d
    return worst

def check(label, ok):
    print(f"  {'PASS' if ok else 'FAIL'}   {label}")

fam1 = lambda t, l1, l2: [t, t - 2*pi/3 + l1*pi, t + pi/6 + l2*pi]
fam2 = lambda t, l1, l2: [t, t + 2*pi/3 + l1*pi, t + 5*pi/6 + l2*pi]

def three_dimensions():
    print("\nTHREE DIMENSIONS")
    for name, xi in (("first  family", fam1(pi/8, 0, 0)),
                     ("second family", fam2(-pi/3, 0, 0))):
        ok = (reduction_term(3, xi) == 0 and integrability_defect(3, xi) == 0
              and any(sp.simplify(c) != 0 for c in velocity(3, xi)[1]))
        check(f"{name}: xi = {[sp.nsimplify(u) for u in xi]}", ok)
    va = velocity(3, fam1(pi/8, 0, 0))[1]
    vb = velocity(3, fam1(pi/8, 1, 0))[1]
    check("adding pi to one phase reverses v^0",
          all(sp.simplify(sp.expand_trig(va[i] + vb[i])) == 0 for i in range(3)))

def four_dimensions():
    print("\nFOUR DIMENSIONS")
    curve = lambda t, a: [t, t + a, t + pi/2, t + a]
    for a in (pi/2, pi/3, sp.Rational(3,10)*pi):
        red = reduction_term(4, curve(0, a)) == 0
        itg = integrability_defect(4, curve(0, a)) == 0
        check(f"a = {sp.nsimplify(a)}: reduction holds, integrability "
              f"{'holds' if itg else 'fails'}", red)
    xi = curve(-pi/4, pi/2)
    check("known 4D solution (-pi/4,pi/4,pi/4,pi/4) satisfies both",
          reduction_term(4, xi) == 0 and integrability_defect(4, xi) == 0)

def five_dimensions():
    print("\nFIVE DIMENSIONS")
    a = sp.Symbol('a', real=True)
    xi = [a, a, a - pi/2, a, a - pi/2]
    check("family (5.6) satisfies the reduction condition",
          reduction_term(5, xi) == 0)
    x, T1, T2 = field(5, xi)
    check("family (5.6) gives T1 = T2, so v^0 = 0 identically",
          all(sp.simplify(sp.expand_trig(T1[i] - T2[i])) == 0 for i in range(5)))
    xj = [sp.Integer(0), sp.Rational(1,5), -pi/2, sp.Integer(0), -pi/2]
    check("a nearby non-member fails the reduction condition",
          reduction_term(5, xj) != 0)

def curl_pair(n, xi, i, j):
    """dg_i/dx_j - dg_j/dx_i  for 0-based i, j"""
    from sympy.simplify.fu import TR8
    x, v = velocity(n, xi)
    g = [sp.expand(sum(v[b]*sp.diff(v[a], x[b]) for b in range(n)))
         for a in range(n)]
    return sp.simplify(TR8(sp.expand(sp.diff(g[i], x[j]) - sp.diff(g[j], x[i]))))

def even_dimension(n):
    """the dichotomy of Theorem 6.7, at the symmetric phases"""
    xi = [-pi/4] + [pi/4]*(n - 1)
    red  = reduction_term(n, xi) == 0
    triv = all(sp.simplify(c) == 0 for c in velocity(n, xi)[1])
    ev   = curl_pair(n, xi, 0, 2) == 0        # even offset
    od   = curl_pair(n, xi, 0, 1) == 0        # odd offset
    print(f"\n{n} DIMENSIONS (symmetric phases)")
    print(f"   reduction condition (3.12) : {'holds' if red else 'FAILS'}")
    kind = 'trivial' if triv else 'non-trivial'
    print(f"   velocity field             : {kind}")
    print(f"   integrability, pair (1,3)  : {'holds' if ev else 'fails'}"
          "   [even offset]")
    print(f"   integrability, pair (1,2)  : {'holds' if od else 'fails'}"
          "   [odd offset]")
    print("   => curl-free throughout; unforced solution exists."
          if ev and od else
          "   => curl-free on even offsets only; no pressure exists.")

# --- Proposition 4.1: the 3D classification, by Groebner basis ---
def classification_3D():
    ca, sa, cb, sb = sp.symbols('c_a s_a c_b s_b')
    cosAB  = ca*cb - sa*sb
    sinAB  = sa*cb + ca*sb
    cos2AB = (2*ca**2 - 1)*cb - 2*sa*ca*sb
    sin2AB = 2*sa*ca*cb + (2*ca**2 - 1)*sb
    p1 = sp.expand(-1 - ca - cos2AB - 2*cb + cosAB)
    p2 = sp.expand(-sa - sin2AB + sinAB)
    G  = sp.groebner([p1, p2, ca**2+sa**2-1, cb**2+sb**2-1],
                     ca, sa, cb, sb, order='lex')
    elim = sp.factor(G.exprs[-1])
    print("\nTHREE-DIMENSIONAL CLASSIFICATION (Proposition 4.1)")
    print(f"   zero-dimensional ideal      : {G.is_zero_dimensional}")
    print(f"   elimination polynomial (4.9): {elim}")

import math, cmath

def _modes(n, xi):
    """v_i^0 and g_i^0 as Fourier dicts at numeric phases (alpha=v_r=1)"""
    def cyc(m): return (m-1)%n
    Z=(0,)*n
    def unit(p,v=1):
        z=[0]*n; z[p]=v; return tuple(z)
    def add(d,k,c):
        d[k]=d.get(k,0)+c
        if abs(d[k])<1e-13: d.pop(k,None)
    def sinf(p,ph):
        d={}; add(d, unit(p), cmath.exp(1j*ph)/2j)
        add(d, unit(p,-1), -cmath.exp(-1j*ph)/2j); return d
    def cosf(p,ph):
        d={}; add(d, unit(p), cmath.exp(1j*ph)/2)
        add(d, unit(p,-1), cmath.exp(-1j*ph)/2); return d
    def mul(A,B):
        C={}
        for ka,ca in A.items():
            for kb,cb in B.items():
                add(C,tuple(x+y for x,y in zip(ka,kb)),ca*cb)
        return C
    def plus(*D):
        C={}
        for d in D:
            for k,v in d.items(): add(C,k,v)
        return C
    def scal(A,c): return {k:v*c for k,v in A.items()}
    def ddx(A,p): return {k:(1j*k[p])*v for k,v in A.items() if k[p]!=0}
    def T1(i):
        f={Z:1+0j}
        for k in range(1,n+1):
            q=cyc(i+k-1); ph=xi[k-1]
            f=mul(f, sinf(q,ph) if k%2 else cosf(q,ph))
        return f
    def T2(i):
        f=sinf(cyc(i),xi[1])
        for k in range(2,n):
            f=mul(f, cosf(cyc(i+k-1),xi[k]) if k%2 else sinf(cyc(i+k-1),xi[k]))
        return mul(f, cosf(cyc(i+n-1),xi[0]))
    v=[plus(T1(i),scal(T2(i),-1)) for i in range(1,n+1)]
    g=[]
    for i in range(n):
        acc={}
        for j in range(n): acc=plus(acc, mul(v[j], ddx(v[i],j)))
        g.append(acc)
    return v,g

def helmholtz_character(n, xi):
    v,g=_modes(n,xi)
    div={}
    for i in range(n):
        for k,c in g[i].items(): div[k]=div.get(k,0)+1j*k[i]*c
    dv=max([abs(c) for c in div.values()]+[0.0])
    cu=0.0
    for i in range(n):
        for j in range(i+1,n):
            for k in set(g[i])|set(g[j]):
                cu=max(cu,abs(1j*k[j]*g[i].get(k,0)-1j*k[i]*g[j].get(k,0)))
    modes=sum(len(v[i]) for i in range(n))
    if cu<1e-12: return "longitudinal (pure gradient)", dv, cu, modes
    if dv<1e-12: return "transverse (divergence-free)", dv, cu, modes
    return "both parts present", dv, cu, modes

def helicity(xi):
    """int v.curl(v) over the cell, n=3"""
    v,_=_modes(3,xi)
    w=[{},{},{}]
    def add(d,k,c): d[k]=d.get(k,0)+c
    for k,c in v[2].items(): add(w[0],k, 1j*k[1]*c)
    for k,c in v[1].items(): add(w[0],k,-1j*k[2]*c)
    for k,c in v[0].items(): add(w[1],k, 1j*k[2]*c)
    for k,c in v[2].items(): add(w[1],k,-1j*k[0]*c)
    for k,c in v[1].items(): add(w[2],k, 1j*k[0]*c)
    for k,c in v[0].items(): add(w[2],k,-1j*k[1]*c)
    tot=0
    for i in range(3):
        for k,c in v[i].items():
            kn=tuple(-t for t in k)
            if kn in w[i]: tot+=c*w[i][kn]
    return (tot*(2*math.pi)**3).real

if __name__ == "__main__":
    import math
    print("Verification of the results reported in the paper.")
    print(f"SymPy {sp.__version__}")
    three_dimensions()
    classification_3D()
    four_dimensions()
    five_dimensions()
    even_dimension(4)
    even_dimension(6)
    pi_ = math.pi
    print("\nHELMHOLTZ CHARACTER OF g^0 (Table 2)")
    for n, xi in ((3, [pi_/8, -13*pi_/24, 7*pi_/24]), (4, [-pi_/4]+[pi_/4]*3),
                  (5, [-pi_/4]+[pi_/4]*4), (6, [-pi_/4]+[pi_/4]*5),
                  (7, [-pi_/4]+[pi_/4]*6), (8, [-pi_/4]+[pi_/4]*7)):
        ch, dv, cu, m = helmholtz_character(n, xi)
        print(f"   n={n}: {ch:30s} |div|={dv:.1e}  |curl|={cu:.1e}")
    print("\nHELICITY OF THE TWO THREE-DIMENSIONAL FAMILIES (4.13)")
    for nm, xi in (("first  family", [pi_/8, -13*pi_/24, 7*pi_/24]),
                   ("second family", [-pi_/3, pi_/3, pi_/2])):
        print(f"   {nm}: H = {helicity(xi):+.6f}"
              f"   (27*sqrt(3)*pi^3/4 = {27*math.sqrt(3)*pi_**3/4:.6f})")
    print("\nDone.")
\end{verbatim}\endgroup
\endgroup

\subsection*{Expected output}
\begingroup
\small
\begingroup\small\begin{verbatim}
Verification of the results reported in the paper.
SymPy 1.14.0

THREE DIMENSIONS
  PASS   first  family: xi = [pi/8, -13*pi/24, 7*pi/24]
  PASS   second family: xi = [-pi/3, pi/3, pi/2]
  PASS   adding pi to one phase reverses v^0

THREE-DIMENSIONAL CLASSIFICATION (Proposition 4.1)
   zero-dimensional ideal      : True
   elimination polynomial (4.9): s_b**4*(4*s_b**2 - 3)

FOUR DIMENSIONS
  PASS   a = pi/2: reduction holds, integrability holds
  PASS   a = pi/3: reduction holds, integrability fails
  PASS   a = 3*pi/10: reduction holds, integrability fails
  PASS   known 4D solution (-pi/4,pi/4,pi/4,pi/4) satisfies both

FIVE DIMENSIONS
  PASS   family (5.6) satisfies the reduction condition
  PASS   family (5.6) gives T1 = T2, so v^0 = 0 identically
  PASS   a nearby non-member fails the reduction condition

4 DIMENSIONS (symmetric phases)
   reduction condition (3.12) : holds
   velocity field             : non-trivial
   integrability, pair (1,3)  : holds   [even offset]
   integrability, pair (1,2)  : holds   [odd offset]
   => curl-free throughout; unforced solution exists.

6 DIMENSIONS (symmetric phases)
   reduction condition (3.12) : holds
   velocity field             : non-trivial
   integrability, pair (1,3)  : holds   [even offset]
   integrability, pair (1,2)  : fails   [odd offset]
   => curl-free on even offsets only; no pressure exists.

HELMHOLTZ CHARACTER OF g^0 (Table 2)
   n=3: longitudinal (pure gradient)      |div|=3.8e-01  |curl|=1.4e-17
   n=4: longitudinal (pure gradient)      |div|=5.0e-01  |curl|=3.5e-17
   n=5: transverse (divergence-free)      |div|=1.9e-17  |curl|=1.2e-01
   n=6: both parts present                |div|=1.2e-01  |curl|=3.1e-02
   n=7: transverse (divergence-free)      |div|=5.2e-18  |curl|=3.1e-02
   n=8: both parts present                |div|=3.1e-02  |curl|=7.8e-03

HELICITY OF THE TWO THREE-DIMENSIONAL FAMILIES (4.13)
   first  family: H = -362.505014   (27*sqrt(3)*pi^3/4 = 362.505014)
   second family: H = +362.505014   (27*sqrt(3)*pi^3/4 = 362.505014)

Done.
\end{verbatim}\endgroup
\endgroup

\noindent Every line of the first four blocks reports \texttt{PASS}. The remaining blocks print
what they find rather than a verdict, and three of them repay comparison. The classification block
returns the elimination polynomial~\eqref{4.elim} on which Proposition~\ref{prop:3Dclassification}
rests. The four-dimensional block shows the reduction condition holding at every value of $a$
tested while integrability holds only at $a=\pi/2$: this is the selection described in
Subsection~\ref{sec:4D}. The two even-dimensional blocks show the mixed derivatives vanishing on
both offsets at $n=4$ but only on the even offset at $n=6$, which is
Theorem~\ref{thm:even-dichotomy} exhibited in four lines. The last two blocks reproduce
Table~\ref{tab:helmholtz} and the helicities of~\eqref{4.helicity}, the equal and opposite values
being the chiral pair of Subsection~\ref{subsec:3D-triple}.

\subsection*{Reproducing the Gr\"obner certificates}
The certificates of Appendices~\ref{app:5D} and~\ref{app:7D} are of a different order, and are
addressed by a second script, \texttt{certify.py}, supplied as an ancillary file with this paper and
run in the same way,
\begin{verbatim}
python3 certify.py
\end{verbatim}
\noindent It has two parts.

The first verifies, in exact arithmetic and in a few seconds, the direction that can be checked
directly: that every one of the $64$ phase amplitudes at $n=5$, and every one of the $664$ at
$n=7$, vanishes identically in the free parameter $a$ on the annihilating families~\eqref{5.6}
and~\eqref{7.5}; that $\mathcal{T}1_i^0=\mathcal{T}2_i^0$ there, so that $\boldsymbol{v}^0\equiv0$;
and that the decisive phase relations hold on those families.

The second addresses the elimination itself. Computing the Gr\"obner basis is not practical in a
general-purpose system---the five-dimensional ideal carries $133$ generators in ten variables and
the seven-dimensional one $1335$ in fourteen---so the script instead writes the ideal to a
\textsc{Singular} input file, \texttt{reduction\_5D.sing} or \texttt{reduction\_7D.sing},
containing the ring declaration, the generators, the basis computation, and the
radical-membership reduction quoted in the appendices. A reader with \textsc{Singular} then runs
the elimination on exactly the generators used here rather than reconstructing them from the text.
Generation takes a few seconds at $n=5$ and about five minutes at $n=7$; the basis computations
themselves were performed as recorded in Appendix~\ref{app:7D}.

\section{Integral identities used in this paper}
\label{app:integrals}
The Gaussian identities below are understood in the ordinary Lebesgue sense for $\tau>0$. The
Newtonian-kernel identities \eqref{D4}--\eqref{D11} and \eqref{D14}, which arise in the
longitudinal-projection calculations, are understood as oscillatory Fourier integrals, or
equivalently through their distributional Fourier interpretation, rather than as absolutely
convergent integrals at spatial infinity. This is consistent with the Fourier interpretation of the periodic
Newtonian-potential formulas discussed in Section~\ref{sec:reformulation}.
\begin{eqnarray}
\label{D1}
\int\limits_{\mathbb{R}^1 } {{e^{ - \frac{{{{\left( {x - u} \right)}^2}}}{{4\tau }}}}} du = 2\sqrt {\pi \tau }
\end{eqnarray}
\begin{eqnarray}
\label{D2}
 \int\limits_{\mathbb{R}^1 }  {\sin \left( {\alpha u} \right){e^{ - \frac{{{{\left( {x - u} \right)}^2}}}{{4\tau }}}}} du = 2\sqrt {\pi \tau } {e^{ - {\alpha ^2}\tau }}\sin \left( {\alpha x} \right)
\end{eqnarray}
\begin{eqnarray}
\label{D3}
\int\limits_{\mathbb{R}^1 } {\cos \left( {\alpha u} \right){e^{ - \frac{{{{\left( {x - u} \right)}^2}}}{{4\tau }}}}} du = 2\sqrt {\pi \tau } {e^{ - {\alpha ^2}\tau }}\cos \left( {\alpha x} \right)
\end{eqnarray}
\begin{eqnarray}
\label{D4}
\int\limits_{\mathbb{R}^3 }{\frac{{u_1\cos\{\beta\left(x_1-u_1\right)\}}}{{\left[ u_1^2+u_2^2+u_3^2 \right]^{\frac{3}{2}} }}}\prod\limits_{j = 1}^3 {du_j }=\frac{4\pi}{\beta}\sin\left(\beta x_1\right)
\end{eqnarray}
\begin{eqnarray}
\label{D5}
\int\limits_{\mathbb{R}^3 }{\frac{{u_1\cos\{\beta\left(x_1-u_1\right)\}\cos\{\beta\left(x_2-u_2\right)\}}}{{\left[ u_1^2+u_2^2+u_3^2 \right]^{\frac{3}{2}} }}}\prod\limits_{j = 1}^3 {du_j }=\frac{2\pi}{\beta}\sin\left(\beta x_1\right)\cos\left(\beta x_2\right)\quad
\end{eqnarray}
\begin{eqnarray}
\label{D6}
\int\limits_{\mathbb{R}^3 }{\frac{{u_1\sin\{\beta\left(x_1-u_1\right)\}\cos\{\beta\left(x_2-u_2\right)\}}}{{\left[ u_1^2+u_2^2+u_3^2 \right]^{\frac{3}{2}} }}}\prod\limits_{j = 1}^3 {du_j }=-\frac{2\pi}{\beta}\cos\left(\beta x_1\right)\cos\left(\beta x_2\right)\quad
\end{eqnarray}
\begin{eqnarray}
\label{D7}
\int\limits_{\mathbb{R}^3 } {\frac{{u_1\cos\{\beta\left(x_1-u_1\right)\}\sin\{\beta\left(x_2-u_2\right)\}}}{{\left[ u_1^2+u_2^2+u_3^2 \right]^{\frac{3}{2}} }}}\prod\limits_{j = 1}^3 {du_j }=\frac{2\pi}{\beta}\sin\left(\beta x_1\right)\sin\left(\beta x_2\right)\quad
\end{eqnarray}
\begin{eqnarray}
\label{D8}
\int\limits_{\mathbb{R}^3 }{\frac{{u_1\sin\{\beta\left(x_1-u_1\right)\}\sin\{\beta\left(x_2-u_2\right)\}}}{{\left[ u_1^2+u_2^2+u_3^2 \right]^{\frac{3}{2}} }}}\prod\limits_{j = 1}^3 {du_j }=-\frac{2\pi}{\beta}\cos\left(\beta x_1\right)\sin\left(\beta x_2\right)\quad
\end{eqnarray}
\begin{eqnarray}
\label{D9}
\int\limits_{\mathbb{R}^4 }{\frac{{u_1\sin\{\beta\left(x_1-u_1\right)\}\sin\{\beta\left(x_3-u_3\right)\}}}{{\left[ u_1^2+u_2^2+u_3^2+u_4^2 \right]^2 }}}\prod\limits_{j = 1}^4 {du_j }=-\frac{\pi^2}{\beta}\cos\left(\beta x_1\right)\sin\left(\beta x_3\right)\quad
\end{eqnarray}
\begin{eqnarray}
\label{D10}
\int\limits_{\mathbb{R}^6 } {\frac{{u_1\sin\{\beta\left(x_1-u_1\right)\}\sin\{\beta\left(x_2-u_2\right)\}}}{{\left[ u_1^2+u_2^2+u_3^2+u_4^2+u_5^2+u_6^2 \right]^3 }}}\prod\limits_{j = 1}^6 {du_j }=-\frac{\pi^3}{2\beta}\cos\left(\beta x_1\right)\sin\left(\beta x_2\right)\quad
\end{eqnarray}
\begin{eqnarray}
\label{D11}
\int\limits_{\mathbb{R}^6 } {\frac{{u_1\sin\{\beta\left(x_1-u_1\right)\}\sin\{\beta\left(x_2-u_2\right)\}\sin\{\beta\left(x_3-u_3\right)\}\sin\{\beta\left(x_4-u_4\right)\}}}{{\left[ u_1^2+u_2^2+u_3^2+u_4^2+u_5^2+u_6^2 \right]^3 }}}\prod\limits_{j = 1}^6 {du_j }\nonumber\\
=-\frac{\pi^3}{4\beta}\cos\left(\beta x_1\right)\sin\left(\beta x_2\right)\sin\left(\beta x_3\right)\sin\left(\beta x_4\right)\qquad
\end{eqnarray}
\begin{eqnarray}
\label{D12}
\int\limits_0^\infty  {\sin \left( {\alpha u} \right)\left({e^{ - \frac{{{{\left( {x - u} \right)}^2}}}{{4\tau }}}-e^{ - \frac{{{{\left( {x + u} \right)}^2}}}{{4\tau }}}}\right)} du = 2\sqrt {\pi \tau } {e^{ - {\alpha ^2}\tau }}\sin \left( {\alpha x} \right)
\end{eqnarray}
\begin{eqnarray}
\label{D13}
\int\limits_0^\infty  {\cos \left( {\alpha u} \right)\left({e^{ - \frac{{{{\left( {x - u} \right)}^2}}}{{4\tau }}}+e^{ - \frac{{{{\left( {x + u} \right)}^2}}}{{4\tau }}}}\right)} du = 2\sqrt {\pi \tau } {e^{ - {\alpha ^2}\tau }}\cos \left( {\alpha x} \right)
\end{eqnarray}
\begin{eqnarray}
\label{D14}
\int\limits_{\mathbb{R}^3 }{\frac{{u_1\cos\{\gamma\left(x_1-u_1\right)\}\cos\{\beta\left(x_2-u_2\right)\}}}{{\left[ u_1^2+u_2^2+u_3^2 \right]^{\frac{3}{2}} }}}\prod\limits_{j = 1}^3 {du_j }=\frac{4\pi\gamma}{\gamma^2+\beta^2}\sin\left(\gamma x_1\right)\cos\left(\beta x_2\right)\quad
\end{eqnarray}
where $\tau>0$ and $\alpha$, $\beta$ and $\gamma$ are real constants, with the parameters chosen so that the denominators on the right-hand sides are non-zero.
\bibliographystyle{jfm}
\bibliography{NS}
\end{document}